\documentclass[english,reqno]{amsart}

\usepackage[margin=1in]{geometry}

\numberwithin{equation}{section}

\usepackage{graphicx}				
\usepackage{xcolor}
\usepackage{tikz}
\usetikzlibrary{decorations.pathreplacing,angles,quotes}
\usepackage{amsmath}
\usepackage{amsthm}
\usepackage{amsfonts}
\usepackage{amssymb}
\usepackage{hyperref}
\usepackage{blindtext}
\usepackage{halloweenmath}

\newtheorem{thm}{Theorem}[section]
\newtheorem{lem}[thm]{Lemma}
\newtheorem{prop}[thm]{Proposition}
\newtheorem{cor}[thm]{Corollary}
\theoremstyle{remark}
\newtheorem{rmk}[thm]{Remark}

\renewcommand{\tilde}{\widetilde}
\renewcommand{\hat}{\widehat}
\renewcommand{\bar}{\overline}

\newcommand{\nn}{\nonumber}
\newcommand{\R}{{\mathbb R}}

\newcommand{\Ss}{{\mathbb{S}}}

\newcommand{\del}{\partial}
\newcommand{\Denote}{\stackrel{\Delta}{=}}

\newcommand{\One}{\boldsymbol{1}}
\newcommand{\Ni}{\noindent}

\newcommand{\Id}{{\bf I}}

\newcommand{\Eps}{\epsilon}

\newcommand{\dt}{ \, {\rm d} t}

\newcommand{\dw}{ \, {\rm d} w}
\newcommand{\dx}{ \, {\rm d} x}
\newcommand{\dy}{ \, {\rm d} y}
\newcommand{\dz}{ \, {\rm d} z}
\newcommand{\dv}{ \, {\rm d} v}

\newcommand{\dbmu}{\, {\rm d} \bar\mu}

\newcommand{\dtau}{\, {\rm d} \tau}

\newcommand{\dxi}{\, {\rm d} \xi}

\newcommand{\dsigma}{\, {\rm d}\sigma}

\newcommand{\CalB}{{\mathcal{B}}}

\newcommand{\CalD}{{\mathcal{D}}}
\newcommand{\CalE}{{\mathcal{E}}}
\newcommand{\CalF}{{\mathcal{F}}}
\newcommand{\CalG}{{\mathcal{G}}}
\newcommand{\CalH}{{\mathcal{H}}}

\newcommand{\CalL}{{\mathcal{L}}}

\newcommand{\CalQ}{{\mathcal{Q}}}

\newcommand{\CalS}{{\mathcal{S}}}
\newcommand{\CalT}{{\mathcal{T}}}
\newcommand{\CalX}{{\mathcal{X}}}
\newcommand{\CalY}{{\mathcal{Y}}}

\newcommand{\IntTRRS}{\int_{\T^3} \!\! \int_{\R^3} \!\! \int_{\R^3} \!\! \int_{\Ss^2_+}}
\newcommand{\IntRRS}{\int_{\R^3} \!\! \int_{\R^3} \!\! \int_{\Ss^2_+}}

\newcommand{\NN}{\mathbb{N}}
\newcommand{\ZZ}{\mathbb{Z}}
\newcommand{\T}{\mathbb{T}}

\newcommand{\abs}[1]{\left\lvert#1\right\rvert}
\newcommand{\norm}[1]{\left\lVert#1 \, \right\rVert}
\newcommand{\vint}[1]{\left\langle#1\right\rangle}
\newcommand{\viint}[2]{\left\langle#1, \, #2 \,\right\rangle}
\newcommand{\vpran}[1]{\left(#1\right)}

\newcommand{\Fl}[1]{f^{(\ell)}_{#1, +}}
\newcommand{\Fk}[2]{f^{(#1)}_{#2, +}}

\newcommand{\Hl}[1]{h^{(\ell)}_{#1, +}}

\newcommand{\Tnorm}[1]{{\left\vert\kern-0.25ex\left\vert\kern-0.25ex\left\vert #1 
    \right\vert\kern-0.25ex\right\vert\kern-0.25ex\right\vert}}

\begin{document}


\date{}

\title[$L^\infty$-Solutions of Non-cutoff Boltzmann]{De Giorgi Argument for Weighted $L^2\cap L^\infty$ solutions
	to the non-cutoff Boltzmann Equation}
%
%
\author{R. Alonso \and Y. Morimoto \and W. Sun \and T. Yang}

\begin{abstract}
This paper gives the first affirmative answer to the question of the global existence of Boltzmann equations without angular cutoff in the $L^\infty$-setting. In particular, we show that when the initial data is close to equilibrium and the perturbation is small in $L^2 \cap L^\infty$ with a polynomial decay tail, the Boltzmann equation has a global solution in the weighted $L^2\cap L^\infty$-space. In order to overcome the difficulties arising from the singular cross-section and the low regularity, a De Giorgi type argument
is crafted in the kinetic context with the help of the averaging lemma. More specifically, we use a strong averaging lemma to obtain suitable $L^p$-estimates for level-set functions. These estimates are crucial for constructing an appropriate energy functional to carry out the De Giorgi argument. Similar as in \cite{AMSY}, we extend local solutions to global ones by using the spectral gap of the linearised Boltzmann operator. The convergence to the equilibrium state is then obtained as a byproduct with relaxations shown in both $L^2$ and $L^\infty$-spaces.

	
{\bf keywords: } Boltzmann, De Giorgi argument, velocity averaging, level-set estimates, spectral gap.
\end{abstract}

\maketitle

\tableofcontents

\section{Introduction}
\subsection{Setup and Objective}
We consider in this paper the nonlinear Boltzmann equation
\begin{align} \label{eq:Nonlinear-Boltzmann-intro}
   \del_t F + v \cdot \nabla_x F = Q(F, F). 
\end{align}
Solutions to this equation $F=F(t,x,v)\geq0$ are the mass density distribution of particles at a time-space point $(t,x)\in(0,\infty)\times\T^{3}$ with velocity $v\in\R^{3}$.  The equation is supplemented with the initial condition
\begin{equation}\label{initial-data}
F(0,x,v) = F_0(x,v)\geq0.
\end{equation}
The nonlinear operator $Q(F,F)$ stands for the \textit{collision operator} which is defined by the integral formula
\begin{align*}
   Q(F, F) 
= \iint_{\R^3 \times \Ss^2}
    B( v - v_\ast , \sigma) \vpran{F'_\ast F' - F_\ast F}
    \dsigma \dv_\ast.
\end{align*}
In this paper we treat the hard potential case with the collision kernel taking the form
\begin{equation}\label{cross-section}
B( v - v_\ast , \sigma) = |v - v_\ast|^{\gamma}b(\cos\theta),
\qquad
\gamma > 0. 
\end{equation}
Following a convention, we assume without loss of generality that $b(\cos\theta)$ is supported on $\cos\theta \geq 0$. This is valid due to the structure of the collision operator. We consider the so-called non-cutoff kernels satisfying
\begin{equation*}
   \sin\theta\,b(\cos\theta)\sim \frac{C}{\theta^{1+2s}},
\qquad 
   \text{for $\theta$ near $0$ and for any $s\in(0,1)$}.
\end{equation*}  
Away from the region of grazing interactions $\theta=0$, the scattering kernel $b$ is assumed to be integrable in $\mathbb{S}^{2}$. 

\medskip
\noindent
The regime close to equilibrium is considered in this work as we seek solutions of the form
\begin{equation*}
F(t,x,v) = \mu + f(t,x,v),
\qquad
\mu = \mu(v) = (2\pi)^{-3/2}\,e^{-|v|^{2}/2}.
\end{equation*}
In such a situation, the unknown $f$ satisfies the nonlinear Boltzmann equation
\begin{equation*}
\partial_{t}f = \mathcal{L}f + Q(f,f),\qquad f(0,x,v)=f_0(x,v),
\end{equation*}
where $\mathcal{L}$ stands for the linear operator
\begin{equation*}
\mathcal{L}f = Q(\mu,f) + Q(f,\mu) - v\cdot\nabla_{x}f.
\end{equation*} 
Physical solutions satisfy the laws of mass, momentum, and energy conservation, which translate to 
\begin{align}\label{eq:conservation-laws}
\int_{\T^{3}}\int_{\R^{2}}f(t,x,v){\rm d}v{\rm d}x = 0,\quad\int_{\T^{3}}\int_{\R^{2}}v\,f(t,x,v){\rm d}v{\rm d}x = 0,\quad\int_{\T^{3}}\int_{\R^{2}}|v|^{2}f(t,x,v){\rm d}v{\rm d}x = 0,
\end{align} 
for all $t\geq0$.

The goal of this work is to show the existence of solutions of the Boltzmann equation for any initial data $f_0$ satisfying $(1 + |v|^2)^{k_0}f_0 \in L^{2}_{x,v}\cap L^{\infty}_{x,v}$ in the perturbative framework.
We note that since the well-established work of constructing near-equilibrium solutions (\cite{AMUXY, AMUXY-1,AMUXY2011-AA,AMUXY2011-CMP,GS2011}) in
the case of $\mu^{-1/2}f_0 \in L^{2}_{v}H^{2}_{x}$,
a lot of efforts have been made to lower the regularity requirement on the initial data in seeking global solutions (cf. \cite{DSSS,D-S,MoSa} and the references therein). To the best of our knowledge, our work is the first to obtain a global solution in the $L^\infty$-setting for the non-cutoff Boltzmann, thus adding a missing link to the studies of the global well-posedness of the nonlinear Boltzmann equations.

%


\subsection{Significance and Main Result} 

\noindent
The problem of constructing solutions to the Boltzmann equation with initial data having minimal spatial regularity has been highly appealing to the community in both cutoff and without cutoff contexts of the Boltzmann equation.  It is desirable to create mathematical tools that can deal with singularity creation/propagation since they may connect to the physical phenomena of shock formation and/or attenuation on the macroscopic scale. 

\vspace{0.5mm}
For the cutoff Boltzmann equations, the development of the well-posedness theory for solutions near equilibrium in the  $L^\infty$-framework  can be traced back to Grad \cite{Grad} for local existence and later by Ukai \cite{Ukai1974} for global  existence under the  Grad's angular cutoff assumption. Ukai's theory relies on the spectral analysis \cite{E-P} of the linearized Boltzmann operator and a bootstrap argument. 
In addition to the $L^\infty$-framework, the well-posedness theory for cutoff Boltzmann has also been well developed in other settings. For example,  in the $L^1$-setting the classical theory on the renormalized solution was established by DiPerna-Lions in their seminal work \cite{D-L} by making essential use of the famous H-theorem and the velocity averaging lemma. The $L^2$-framework based on energy methods has also been extensively explored (\cite{Guo-0, LY, LYY}). Furthermore, an $L^2-L^{\infty}$ interplay method has been introduced in \cite{UkaiYang,Guo} 
and applied to various contexts (see \cite{GKTT} and the references therein) to obtain solutions with low regularity and close to equilibrium.  Earlier theory on solutions near equilibrium has been focused on perturbations with Gaussian tails. More recently, a big step forward is made in \cite{GMM} where the authors introduced a new framework of using spectral analysis to relax the velocity decay constraint from Gaussian to polynomial. 
In the cutoff context a key point for the control of the collision operator is to work in Banach spaces with an ``algebraic structure" in the spatial variable.  

\vspace{0.5mm}
For Boltzmann equations without angular cutoff, the well-posedness for large data in $L^1$-framework was obtained by  Alexandre-Villani in \cite{A-V}, and the $L^2$-theory for perturbative solutions around an equilibrium was first established in \cite{AMUXY-1,AMUXY2011-AA,AMUXY2011-CMP,AMUXY2012JFA,GS2011}. In \cite{AMUXY,AMUXY2013KRM}, the authors considered the space $L^{2}_{v}H^{\beta}_{x}$ with $\beta>3/2$ for local existence.  The particular range of $\beta$ seems almost optimal since 
the main idea is to use the Sobolev embedding $H^{\beta}_{x}\subseteq L^{\infty}_{x}$ to handle the quadratic nonlinearity of the collision operator. This mimics the idea implemented for the cutoff case through the ``algebraic" control $\|f\,g\|_{H^{\beta}}\leq \| f\|_{H^{\beta}}\| g \|_{H^{\beta}}$.

Contrary to the extensive studies in the $L^2$-setting, the $L^\infty$-theory for the wellposedness of the Boltzmann equation without angular cutoff has remained open. Recently in \cite{DSSS}, the authors were able to construct global-in-time solutions in a space based on the Wiener algebra in $x$ with the norm
\begin{equation*}
\| f \|_{\mathcal{W}} := \sum_{k}\sup_{t}\| \mathcal{F}_{x}\{f\}(t,k,\cdot) \|_{L^{2}_{v}},
\end{equation*}
The key point in~\cite{DSSS} is again the ``algebraic" control $\| f\,g\|_{\mathcal{W}} \leq \|f\|_{\mathcal{W}} \|g\|_{\mathcal{W}}$. The Wiener algebra is significantly more general than the Sobolev spaces $H^\beta_x$ used in earlier works. It is a considerable step toward $L^\infty_x$-spaces, although still more restrictive.  We further note that the aforementioned works are in the context of velocity with Gaussian tails, in which spectral and coercivity properties of the collision operator appear naturally.  For the recent development on the perturbation with a polynomial decay, one can refer to \cite{AMSY, GMM, HTT} and the references therein.

The goal of our paper is to give a global existence proof for the non-cutoff Boltzmann equation in the $L^\infty$-setting. Instead of following the path of exploring algebraic structures, we apply a different framework based on a De Giorgi argument. The full machinery of the De Giorgi-Nash-Moser method has been used lately in a series of remarkable developments for the Landau and non-cutoff Boltzmann equations \cite{GIMV, KGH, IS, S}. Among these works, one direction can be summarized as {\it conditional regularities} of solutions to kinetic equations \cite{GIMV, IS, S}. More specifically, solutions with properties which remain to be justified in general are shown to have quantitative regularization. In this respect, these works are mainly concerned with the regularization mechanism of kinetic equations while our goal is to establish a self-contained existence result. In \cite{KGH}, by using strategies inspired by~\cite{GIMV} the authors have proved a global well-posedness result for the Landau equation. To achieve the H\"{o}lder continuity typical in the De Giorgi-Nash-Moser argument, the initial data in \cite{KGH} needs to be better prepared than merely in weighted $L^\infty$-spaces. Moreover, compared with the Boltzmann operator, the Landau operator has a more localized structure  which is closer to classical nonlinear parabolic operators. As a consequence, it is not clear to us how to extend the argument in~\cite{KGH} to the Boltzmann equation with an $L^\infty$-data.

In this paper we are not aiming at showing the H\"{o}lder continuity of the solutions or making use of such regularization. Instead, we make essential use of the classical velocity averaging lemma to transfer the regularity from the velocity to the spatial variable. The approach we adopt in this paper is more related to a level-set method introduced for the homogeneous Boltzmann equation in \cite{A} inspired by the first part of \cite{CV}. See also the application to a linear radiative transfer equation in the forward-peak regime \cite{AS, ACuba}.  Built upon some key estimates in our recent work \cite{AMSY}, we develop the theory in the polynomial-tail setting.  The method presented here relies purely on the weak formulation of the collision operator and the averaging properties of the transport operator. We believe it is general and ready to apply to a wide range of kinetic equations including the Landau equation.

Finally we remark that unlike the results in~\cite{DSSS, KGH}, so far we are not able to obtain uniqueness in $L^\infty$-spaces. 

\smallskip
With some details of the parameters left out, the main theorem can be summarized as:

\begin{thm}
	Suppose the cross section of the Boltzmann equation satisfies~\eqref{cross-section} with $\gamma \in (0, 1)$ and $s \in (0, 1)$ and the initial data $F_0 \geq 0$ satisfies \eqref{eq:conservation-laws}. Then for $k_0, k$ large enough with $k > k_0$, there exists $\delta^\natural_{\ast} > 0$ such that if 
\begin{align*} 
   \norm{\vint{v}^{k_0}(F_0(x,v)-\mu)}_{L^{2}_{x,v}\cap L^{\infty}_{x,v}}
\leq \delta^\natural_{\ast},
\qquad
     \norm{\vint{v}^{k}(F_0(x,v)-\mu)}_{L^{2}_{x,v}}
< \infty.
\end{align*}
Then there exists a non-negative solution $F \in L^\infty(0, \infty; L^2_x L^2_{k} (\T^3 \times \R^3))$ to \eqref{eq:Nonlinear-Boltzmann-intro}. Moreover, for some $\delta_0$ and $\lambda'>0$, the solution $F$ satisfies 
\begin{equation}\label{solution-1}
  \norm{\vint{v}^{k_0}(F(t,x,v)-\mu)}_{L^{\infty}_{x,v}}
\leq \delta_0,
\qquad
  \norm{\vint{v}^{k}(F(t,x,v)-\mu)}_{L^{2}_{x,v}}
< C e^{-\lambda' t}.
\end{equation}
Furthermore, for some $C_\ast, \tilde \lambda > 0$, the weighted $L^\infty$-norm of the perturbation decays exponentially in time:
\begin{align*}
  \norm{\vint{v}^{k_0}(F(t,x,v)-\mu)}_{L^{\infty}_{x,v}}
\leq
  C_\ast e^{-\tilde \lambda t}.
\end{align*}
\end{thm}

\subsection{Notations} We employ several notations for function spaces in this paper. First, $L^2_{x,v}$ or $L^2(\T^3 \times \R^3)$ denotes the usual $L^2$-space over $\T^3 \times \R^3$ and $L^2_x H^s_v(\T^3 \times \R^3)$ denotes the space where $(I - \Delta_v)^{s/2} f \in L^2_{x,v}$. Any weight in the subindex denotes a weight in $v$ only. For example, $L^2_x L^2_k$ denotes the space where $\vint{v}^k f \in L^2_{x,v}$ where $\vint{v}$ is the Japanese bracket defined by $\vint{v}^2 = 1 + |v|^2$. 

There are many parameters in this paper. Among them, we reserve the key ones for designated meanings that will not change throughout this paper:
\begin{itemize}
\item $\gamma$: power in the hard potential.

\item $s$: strength of the singularity in the collision kernel.

\item $s'$: regularity in $x$ (and $t$) derived from  the averaging lemma.

\item $\gamma_0$: dissipation coefficient in Lemma~\ref{lem:decomp-Q-ell}.

\item $c_0$: dissipation coefficient in Proposition~\ref{prop:coercivity-1}. 

\item $k_0$: moment in the $L^\infty$-bound (in $t, x, v$) of the solution.

\item $\delta_0$: smallness of the $L^\infty$-norm (in $t, x$) of the solution for the energy estimates to close.

\item $\Eps$: strength of the regularizing operator $\Eps L_\alpha$ with $L_\alpha$ defined in~\eqref{def:L-alpha}.


\item $\CalE_k$: $k$-th energy level.

\item $\ell_0$: minimal order of moments needed for the inhomogeneous embedding in~\eqref{estT1-1}.

\item $\lambda_0$: spectral gap of the linearized Boltzmann operator.

\end{itemize}

\Ni Other parameters such as $p, q, p', k, \beta, \beta', s'', \ell, \theta, \eta$ may change from statement to statement.  Constants denoted by $C, C_\ell, C_k$ may change from line to line.  

\Ni Since we assume that the collision kernel $b(\cos\theta)$ is supported on $\cos\theta \geq 0$, the integration limits for $b$ can be either $\Ss^2$ or $\Ss^2_+$ and we use them interchangeably. 

\subsection{Methodology and organization}    
A brief outline of the strategy implemented in this paper is as follows:  the underlying condition for the validity of the \textit{a priori} estimates, in addition to sufficient regularity, is the smallness condition
\begin{equation} \label{intro:smallness}
\sup_{t,x}\| f \|_{L^{1}_{w_0}\cap L^{2}}\leq \delta_0,
\end{equation}
where $w_0>0$ is a threshold of polynomial decay and $\delta_0>0$ is a sufficiently small quantity.     
With this condition, $L^{2}_{x,v}$ and $L^{2}_{x}H^{s}_{v}$ energy estimates with general weights in velocity can be proved. The bound in $L^2_x H^s_v$ demonstrates the natural regularization in the velocity variable reminiscent of a fractional Laplace's equation.  
Using a time-localized averaging lemma, one can ``complete" the velocity energy estimate of the equation to include the regularization in the spatial variable using the norm $H^{s'}_{x}L^{2}_{v}$ for some $s' \in(0,s)$.  This confirms the hypoelliptic properties of the equation as expected. The hypoellipticity paves the way to apply the De Giorgi argument through embeddings of Sobolev spaces into various $L^p$ spaces. In particular, we construct the crucial energy functional
\begin{align*}
     \CalE_{p}(K,T_1,T_2)
:= \sup_{ t \in [ T_1 , T_2 ] } 
       \norm{\Fl{K}}^{2}_{L^{2}_{x,v}} 
& +  c_0\int^{T_2}_{T_1}\int_{\T^{3}}
         \norm{\vint{\cdot}^{\gamma/2}\Fl{K}}^{2}_{H^{s}_{v}} \dx\dtau \nn
\\
& + \frac{1}{C_0}\vpran{\int^{T_2}_{T_1} \norm{(1-\Delta_{x})^{\frac {s''}{2}}\vpran{\Fl{K}}^{2}}^{p}_{L^{p}_{x,v}}\dtau}^{\frac{1}{p}},
\end{align*}
where $0 < s'' < s$ will be suitably chosen, $K$ is any positive number, $p$ is a parameter depending on $s$ and $\Fl{K}$ is the level-set function defined ~by
\begin{align} \label{def:level-set-into}
   \Fl{K} = \vpran{\vint{v}^\ell f - K} \One_{\vint{v}^\ell f - K \geq 0}.
\end{align}
With $K = K_n \to K_0$ for some $K_0$ depending on the initial data, the main step in the De Giorgi argument is to show that $\CalE_p(K_n, T_1, T_2)$ satisfies an inhomogeneous (in degree) iterative relation (see~\eqref{key-estimate-linear} and~\eqref{key-estimate}). This key iterative relation leads to the limit $\CalE_p(K_n, T_1, T_2) \to 0$ as $n \to \infty$, thus proving the weighted $L^\infty$-bound. To enforce the smallness condition~\eqref{intro:smallness} when constructing the approximate solutions, we introduce a cutoff function $\chi$ and consider the modified collision operator 
\begin{align*}
   Q(\mu + f\chi(\vint{v}^{k_0} f), \mu + f),
\end{align*}
so that the smallness condition is satisfied naturally for the approximate solutions. Such cutoff function automatically disappears after one applies the De Giorgi method and shows {\it a posteriori} that the smallness condition holds intrinsically when the initial data is small enough. 

The strategy described above is applied first to the linearized equation and then to the nonlinear equation to obtain local solutions with $L^\infty$-bounds to the original Boltzmann equation. A major part of this paper is dedicated to the linear analysis. Although obtaining a solution to the linearized equation is fairly straightforward, significant effort has been 
carried out to show the $L^\infty$-bounds of the solution. A delicate issue is to handle the moments required in various estimates. Interestingly, the moment requirement imposed on solutions for linear and nonlinear estimates drastically differs, with the nonlinear case much easier to handle. The key factor at play is the quadratic structure of the collision operator. This structure reduces the moment needed on solutions for the $L^{2}_{x,v}$ estimates which is essential for closing the argument. Finally, combining the local existence with the spectral gap we obtain a global solution to the original Boltzmann equation. 

This paper is laid out as follows. After including a technical toolbox in Section 2, we establish the well-posedness for the linearized Boltzmann equation in Sections~\ref{Sec:a-priori-linear} and~\ref{Sec:linear-wellposedness}: Section~\ref{Sec:a-priori-linear} consists of a priori $L^2_{x,v}$ and $L^2_x H^s_v$-estimates and $L^1$-estimates for the collision term for the linear equation. These estimates are used in Section~\ref{Sec:linear-wellposedness} to show the existence of solutions to the linear equation.  
In Section~\ref{Sec:nonlinear-local} we establish the nonlinear counterparts of the estimates of those in Section~\ref{Sec:a-priori-linear} and apply them to establish the local existence of the nonlinear Boltzmann equation. In Section~\ref{sec:nonlinear-global} we combine the results in Section~\ref{Sec:nonlinear-local} and the spectral gap property of the linearized Boltzmann operator to establish the global existence of the nonlinear Boltzmann equation. 
The result proved in Section~\ref{sec:nonlinear-global} is only for the weakly singular kernels. We extend the result to the strong singularity in Section~\ref{Sec:strong-singularity}.  

\smallskip
\noindent

\smallskip
\noindent
To help the reader better understand the structure of the proof, we show a flow chart of the main steps in Figure~\ref{strategy}.  Starting from the smallness assumption, the $L^{2}$-theory is performed.  The velocity regularization appears in a standard way whereas spatial regularization is obtained through velocity averaging.  The $L^{2}$-theory comprises both $f$ and $K$-levels $\Fl{K}$. A higher moment of order $k$ is needed (to be precise, $k = k_0+\ell_0+2$ for some $\ell_0$ depending on $s$) in the $L^{2}$-estimates to prove algebraic $k_0$ moments in the $L^{\infty}$-estimates.     

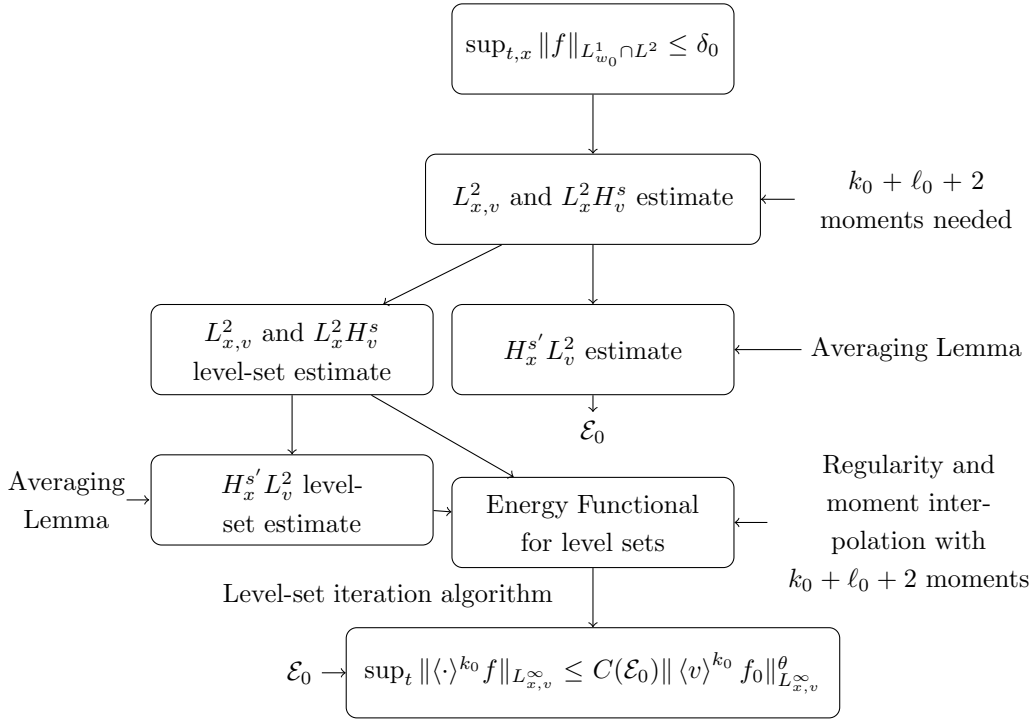
\begin{figure}[!h]
\centering
\begin{tikzpicture}
\node[draw=black, text width=10em, text centered, rounded corners, minimum height=4em, minimum size=12mm] (init) {$\sup_{t,x}\| f \|_{L^{1}_{w_0}\cap L^{2}}\leq \delta_0$};
\node[draw=black, text width=12em, text centered, rounded corners, minimum size=12mm, node distance=2cm, below of=init]  (velEnergy) {$L^{2}_{x,v}$ and $L^2_{x}H^s_{v}$ estimate};
\node[draw=black, text width=10em, text centered, rounded corners, minimum size=12mm, node distance=2cm, below of=velEnergy]   (spatialEnergy) {$H^{s'}_{x}L^{2}_{v}$ estimate};
\node[node distance=1.1cm, below of=spatialEnergy] (E0) {$\mathcal{E}_0$};
\node[draw=black, text width=10em, text centered, rounded corners, minimum size=12mm, node distance=4.0cm, left of=spatialEnergy]  (LevelVel) {$L^{2}_{x,v}$ and $L^2_{x}H^s_{v}$ level-set estimate};
\node[draw=black, text width=10em, text centered, rounded corners, minimum size=12mm, node distance=2.0cm, below of=LevelVel]  (LevelSpatial) {$H^{s'}_{x}L^{2}_{v}$ level-set estimate};
\node[draw=black, text width=10em, text centered, rounded corners, minimum size=12mm, node distance=2.3cm, below of=spatialEnergy]  (EnergyFunct) {Energy Functional for level sets};
\node[draw=black, text width=18em, text centered, rounded corners, minimum size=12mm, node distance=2cm, below of=EnergyFunct]  (infinity) {$\sup_{t}\| \langle\cdot\rangle^{k_0}f \|_{L^{\infty}_{x,v}} \leq C(\mathcal{E}_0)\|\vint{v}^{k_0} f_0 \|^{\theta}_{L^{\infty}_{x,v}}$};
\node[text width=8.5em, text centered, node distance=4.3cm, right of=velEnergy] (a) {$k_0+\ell_0+2$ moments needed};
\node[node distance=2cm, below of=a] (avlemma) {Averaging Lemma};
\node[text width=6em, text centered, node distance=3cm, left of=LevelSpatial] (avlemma1) {Averaging Lemma};
\node[text width=10em, text centered, node distance=4.2cm, right of=EnergyFunct] (inter) {Regularity and moment interpolation with $k_0+\ell_0+2$ moments};
\node[text width=15em, text centered, node distance=1.8cm, below right of=LevelSpatial] (iteration) {Level-set iteration algorithm};
\node[node distance=3.9cm, left of=infinity] (E0v2) {$\mathcal{E}_0$};
\path [draw, ->] (init) -- (velEnergy);
\path [draw, ->] (velEnergy) -- (spatialEnergy);
\path [draw, ->] (velEnergy) -- (LevelVel);
\path [draw, ->] (LevelVel) -- (LevelSpatial);
\path [draw, ->] (LevelSpatial) -- (EnergyFunct);
\path [draw, ->] (LevelVel) -- (EnergyFunct);
\path [draw, ->] (EnergyFunct) -- (infinity);
\path [draw, ->] (a) -- (velEnergy);
\path [draw, ->] (avlemma) -- (spatialEnergy);
\path [draw, ->] (-6.2,-6) -- (-5.9,-6cm);
\path [draw, ->] (inter) -- (EnergyFunct);
\path [draw, ->] (spatialEnergy) -- (E0);
\path [draw, ->] (E0v2) -- (infinity);

\end{tikzpicture}
\caption{Flow chart of the strategy.  Moments are related as $k_0>w_0>0$ and so does regularity as $s>s'>0$.  The constant $C(\mathcal{E}_0)$ is independent of the smallness parameter $\delta_0$.}\label{strategy}
\end{figure}


\section{Technical Toolbox}
\subsection{Function Spaces}
In this paper we use two classical function spaces, namely, Bessel potential and Sobolev-Slobodeckij spaces.  Most of the work is based on the former, yet, the proof of some estimates is simpler if performed in the latter.

\medskip
\noindent
\textbf{Definition.} For $p\in[1,\infty)$ and $\beta\in\mathbb{R}$, the Bessel Potential space is
\begin{equation} \label{def:H-s-p}
H^{\beta,p}(\mathbb{R}^{d}):=\Big\{ u\in L^{p}(\mathbb{R}^{d})\, \big|\, \mathcal{F}^{-1}\big\{ (1+|\xi|^{2})^{\frac{\beta}{2}}\mathcal{F}u\big\} \in L^{p}(\mathbb{R}^{d})\Big\},
\end{equation}
where $\CalF$ is the Fourier transform. The norm that equips $H^{\beta,p}(\mathbb{R}^{d})$ is naturally 
\begin{equation*}
\| u \|_{H^{\beta,p}(\mathbb{R}^{d})} := \big\| \mathcal{F}^{-1}\big\{ (1+|\xi|^{2})^{\frac{\beta}{2}}\mathcal{F}u \big\}\big\|_{L^{p}(\mathbb{R}^{d})} = \big\| (1-\Delta)^{\frac{\beta}{2}}u \big\|_{L^{p}(\mathbb{R}^{d})}. 
\end{equation*} 
\textbf{Definition.} For $p\in[1,\infty)$ and $\beta\in(0,1)$, the Sobolev-Slobodeckij space is
\begin{equation} \label{def:W-s-p}
W^{\beta,p}(\mathbb{R}^{d}):= \Big\{ u\in L^{p}(\mathbb{R}^{d})\, \big|\, \int_{\mathbb{R}^{d}}\int_{\mathbb{R}^{d}}\frac{|u(x) - u(y)|^{p}}{|x-y|^{d+\beta p}}\dy\dx < \infty\Big\}.
\end{equation}
A natural norm that equips $W^{\beta,p}(\mathbb{R}^{d})$ is given by
\begin{equation*}
\| u \|_{W^{\beta,p}(\mathbb{R}^{d})} := \bigg(\int_{\mathbb{R}^{d}} |u(x)|^{p} \dx + \int_{\mathbb{R}^{d}}\int_{\mathbb{R}^{d}}\frac{|u(x) - u(y)|^{p}}{|x-y|^{d+\beta p}}\dy\dx\bigg)^{\frac{1}{p}}.
\end{equation*}
The Bessel potential spaces and Sobolev-Slobodeckij spaces agree for $p=2$.  More generally, the following relation holds:

\smallskip
\noindent
(i) For all $p\in(1,2]$, $\beta\in(0,1)$ it holds that $W^{\beta,p}(\mathbb{R}^{d}) \hookrightarrow H^{\beta,p}(\mathbb{R}^{d})$. 

\smallskip
\noindent
(ii) For all $p\in[2,\infty)$, $\beta\in(0,1)$ it holds that $H^{\beta,p}(\mathbb{R}^{d})  \hookrightarrow W^{\beta,p}(\mathbb{R}^{d})$. 

\smallskip
\noindent
The proof of this fact can be found in \cite{Stein}, Theorem 5 in Chapter V.  
  
\subsection{Useful facts about polynomial weights}
In this section, we list some useful estimates that are needed for later estimation. Recall that
\begin{align*}
  \norm{g}_{ H^{\beta}_{\ell} } 
= \norm{\vint{v}^\ell g}_{ H^{\beta}_{v} } ,\qquad \ell,\,\beta\in\R .
\end{align*}

First we present two lemmas related to commutator estimates of fractional derivatives. Since their proofs are technical and not directly associated with the Boltzmann operator, we leave them to Appendix~\ref{appendix:lemmas}. 
\begin{lem}[cf.,\cite{HMUY-2008}] \label{prop:equivalence}
Let $1\le p \le\infty$. Suppose $\ell,\, \theta \in \R$. Then there exists a generic constant $C$ independent of $f$ such that
\begin{align*}
    \frac{1}{C} \norm{\vint{v}^\ell \vint{D_v}^\theta f}_{ L^p_{v} }
\leq 
    \norm{\vint{D_v}^\theta \vint{v}^\ell f}_{ L^p_v } 
\leq 
   C  \norm{\vint{v}^\ell \vint{D_v}^\theta f}_{ L^p_{v} },
\end{align*}
that is, these two norms are equivalent. 
\end{lem}

We will also need a homogeneous version related to fractional derivatives. 
\begin{lem} \label{cor:commut-homo-fraction}
Suppose $\alpha \in (0, 1)$ and $f \in H^\alpha_v(\R^3)$ Then $\vint{v}^{-2} f \in H^\alpha_v(\R^3)$ with the bound
\begin{align*}
   \norm{(-\Delta_v)^{\alpha/2} \vpran{\vint{v}^{-2} f}}_{L^2_{v}(\R^3)}
\leq
   C \norm{(-\Delta_v)^{\alpha/2} f}_{L^2_{v} (\R^3)}. 
\end{align*}
\end{lem}


Next we recall the now-classical trilinear estimate.
\begin{prop} [\cite{AMUXY-1, MoSa}] \label{prop:trilinear}
Denote $a^+ = \max\{a, 0\}$. Then the bilinear operator $Q$ satisfies
\begin{align*}
   \Big| \int_{\R^3} Q(f, g) \, h \, \dv \Big| \leq C \Big( \| f \|_{ L^1_{(m - \gamma/2)^{+} + \gamma + 2s} }
 \!\!\!\! + \| f \|_{L^2} \Big)\| g \|_{ H^{s+\sigma}_{ \gamma/2 + 2s + m} }\| h \|_{H^{s-\sigma}_{\gamma/2 - m}} 
\end{align*}
for any $\sigma \in [\min\{s-1, -s\}, s]$, $m \in \R$, $\gamma \geq 0$ and $0 < s < 1$.  Here,  $f,\, g,\, h$  are any functions so that the corresponding norms are well-defined.  The constant $C$ is independent of $f,\, g,\, h$. 
\end{prop}
\begin{lem}[\cite{ADVW}] \label{prop:change-variable} 
Suppose $f$ and $b$ are functions that make sense of the integrals below. Then

\noindent
(a) (Regular change of variables)
\begin{align*}
  \int_{\R^3} \int_{\mathbb{S}^2} b(\cos\theta)&|v-v_\ast|^\gamma f(v') {\rm d}\sigma \dv 
   = \int_{\R^3} \int_{\mathbb{S}^2} b(\cos\theta) \frac{1}{\cos^{3+\gamma}(\theta/2)}|v-v_\ast|^\gamma f(v) {\rm d}\sigma \dv .
\end{align*}
\noindent
(b) (Singular change of variables)
\begin{align*}
    \int_{\R^3} \int_{\mathbb{S}^2} b(\cos\theta)&|v-v_\ast|^\gamma f(v') {\rm d}\sigma \dv_\ast
= \int_{\R^3} \int_{\mathbb{S}^2} b(\cos\theta) \frac{1}{\sin^{3+\gamma}(\theta/2)}|v-v_\ast|^\gamma f(v_\ast) {\rm d}\sigma \dv_\ast .
\end{align*}
\end{lem}
%

%
%
\begin{prop} [\cite{AMUXY2012JFA}] \label{prop:coercivity-1}
Suppose for some constants $D_0, E_0 > 0$, the function $F$ satisfies 
\begin{align*}
  F \geq 0, 
\qquad  
  \| F \|_{L^1} \geq D_0 > 0,
\qquad 
  \| F \|_{L^1_2} + \| F \|_{L\log L} \leq E_0 < \infty.
\end{align*}
Then there exist two constants $c_0$ and $C$ such that
\begin{align*}
\int_{\R^3} Q(F, f) \, f \dv \leq - c_0 \| f \|^2_{H^s_{\gamma/2}} + C \| f \|^2_{L^2_{\gamma/2}} .
\end{align*}
\end{prop}

Throughout this paper, we use $\dbmu$ to denote the measure
\begin{align*}
  \dbmu = \dsigma \dv \dv_\ast \dx.
\end{align*}

For the convenience of the later analysis, we record a simple decomposition and bound related to the nonlinear Boltzmann operator:
\begin{lem} [\cite{AMSY}] \label{lem:decomp-Q-ell}
(a) Let $G, h$ be functions that make sense of the integrals below. Then for any $s \in (0, 1)$ and $\ell \geq 0$,
\begin{align} \label{eqn:decomp-Q-ell}
    \iint_{\T^3 \times \R^3} 
     Q(G, h) h \vint{v}^{2\ell} \dv\dx
&= \int_{\T^3}\int_{\R^3} Q(G, \, \vint{v}^{\ell} h) \, \vint{v}^{\ell} h \dv\dx \nn
\\
& \quad \,
  + \IntTRRS b(\cos\theta) |v - v_\ast|^\gamma
        G_\ast h \, h' \vint{v'}^{\ell}  \vpran{\vint{v'}^{\ell} - \vint{v}^{\ell}\cos^{\ell} \tfrac{\theta}{2}} \dbmu  \nn
\\
& \quad \, 
  + \IntTRRS 
       b(\cos\theta) |v - v_\ast|^\gamma 
       G_\ast h \vint{v}^{\ell}  h' \vint{v'}^{\ell}
       \vpran{\cos^{\ell}  \tfrac{\theta}{2} - 1} 
       \dbmu. 
\end{align}

\Ni (b) Suppose in addition $G \geq 0 $ and $G = \mu + g$. Let $\gamma_0, \gamma_1$ be the positive constants satisfying 
\begin{align*}
    \int_{\R^3} |v - v_\ast|^\gamma \mu(v_\ast) \dv_\ast
\geq 
   \gamma_1 \vint{v}^\gamma
\end{align*}
and
\begin{align*}
   \gamma_0 
\geq 
   -\frac{\gamma_1}{2} \int_{\Ss^2} b(\cos\theta) \vpran{\cos^{2\ell - 3 - \gamma} \tfrac{\theta}{2} - 1} \dsigma,
\qquad
   \text{for all $\ell > \frac{3 + \gamma}{2}$}.
\end{align*}
Then we have
\begin{align} 
   \iint_{\T^3 \times \R^3} 
     Q(G, h) h \vint{v}^{2\ell} \dv\dx
&\leq
  \frac{1}{2} \IntTRRS
        b(\cos\theta) |v - v_\ast|^\gamma G_\ast |h|^2 \vint{v}^{2\ell}
        \vpran{\cos^{2\ell - 3 - \gamma} \tfrac{\theta}{2} - 1} \dbmu  \nn
\\
& \,\,
   + \IntTRRS 
       b(\cos\theta) |v - v_\ast|^\gamma 
       G_\ast |h| |h'| \vint{v'}^{\ell}
       \vpran{\vint{v'}^{\ell} - \vint{v}^{\ell} \cos^{\ell} \tfrac{\theta}{2}}  
       \dbmu \label{ineq:decomp-Q-ell-1}
\\
&\leq
   -\vpran{\gamma_0 - C_{\ell} \sup_x \norm{g}_{L^1_\gamma}}
     \norm{\vint{v}^{\ell + \gamma/2} h}_{L^2_{x, v}}^2  \nn
\\
& \quad \,
   + \IntTRRS 
       b(\cos\theta) |v - v_\ast|^\gamma 
       G_\ast |h| |h'| \vint{v'}^{\ell}
       \vpran{\vint{v'}^{\ell} - \vint{v}^{\ell} \cos^{\ell} \tfrac{\theta}{2}}  
       \dbmu.  \label{ineq:decomp-Q-ell-2}
\end{align}
\end{lem}
\begin{proof}
Part (a) follows from a direct addition-subtraction applied to the definition of $Q$. Part (b) follows from (3.15) and (3.16) in \cite{AMSY}.
\end{proof}

Cancellation plays a vital role in dealing with the strong singularity. Let us recall a useful representation for $|v'|$:
\begin{align} 
|v'|^{2} = |v|^{2}\cos^{2}\frac{\theta}{2} + |v_\ast|^{2}\sin^{2}\frac{\theta}{2}+2\cos\frac{\theta}{2}\sin\frac{\theta}{2}|v-v_*| v\cdot\omega,
\label{formula:classical-1}
\end{align}
and, as a result, 
\begin{align}
\langle v' \rangle^{2} = \langle v \rangle^{2}\cos^{2}\frac{\theta}{2} + \langle v_\ast\rangle^{2} \sin^{2}\frac{\theta}{2} + 2\cos\frac{\theta}{2}\sin\frac{\theta}{2}|v-v_*| (v_\ast \cdot \omega),
\label{formula:classical-2}
\end{align}
where 
\begin{align} \label{def:omega-direction}
  \omega = \frac{\sigma - (\sigma \cdot \hat{u}) \hat{u}}{ | \sigma - (\sigma \cdot \hat{u}) \hat{u} |} ,
\qquad 
  \hat{u} = \frac{v - v_\ast}{|v - v_\ast|} .
\end{align}
By its definition, $\omega$ satisfies that $\omega \perp (v - v_\ast)$, thus, $v \cdot \omega = v_\ast \cdot \omega$.  Consequently, one has the freedom to choose $v \cdot \omega $ or $v_\ast \cdot \omega$ in the estimates. We also introduce the notation $\tilde \omega$  for later use:
\begin{align} \label{def:tilde-omega}
   \tilde \omega
= \frac{v' - v}{|v' - v|}.
\end{align}

\begin{lem}[see \cite{AMSY}]\label{lem:diff-v-k}
Suppose $\ell > 6$ and $(v, v_\ast), (v', v'_\ast)$ are the velocity pairs before and after the collision or vice versa. Let $\omega$ be the vector defined in~\eqref{def:omega-direction}.  Then,
\begin{align}\label{precise}
\begin{split}
\langle v' \rangle^{\ell} - \langle v \rangle^{\ell} \cos^{\ell} \tfrac{\theta}{2} =   \ell \langle v \rangle^{\ell-2}&|v-v_*| \big (v \cdot \omega \big)\cos^{\ell-1} \tfrac{\theta}{2}  \sin\tfrac{\theta}{2} 
+ \langle v_\ast \rangle^{\ell}\sin^{\ell} \tfrac{\theta}{2}  + \mathfrak{R}_1 + \mathfrak{R}_2 + \mathfrak{R}_3 ,
\end{split}
\end{align} 
where there exists a constant $C_\ell$ only depending on $\ell$ such that
\begin{align} \label{bound:mathfrak-R-2-3}
\begin{split}
& | \mathfrak{R}_1 | \leq C_\ell \, \langle v \rangle \langle v_\ast\rangle^{\ell-1} \sin^{\ell-3} \tfrac{\theta}{2}, \qquad | \mathfrak{R}_2 | \leq C_\ell \, \langle v\rangle^{\ell-2} \langle v_\ast \rangle^2 \sin^2\tfrac{\theta}{2},\\
&\quad \text{and} \qquad  | \mathfrak{R}_3 | \leq C_\ell\, \langle v\rangle^{\ell-4} \langle v_\ast\rangle^4 \sin^2\tfrac{\theta}{2}.
\end{split}
\end{align}
\end{lem}

We are ready to establish an all-important commutator estimate.

\begin{prop}[Weighted commutator] \label{prop:commutator}
Suppose 
\begin{align*}
   G = \mu + g ,
\qquad
   g \in L^\infty_x L^1_{\ell+\gamma} \cap L^2_{x, v} ,
\qquad
   \ell \geq 8+\gamma, 
\qquad
   \gamma \in (0, 1),
\qquad
   s \in (0, 1).
\end{align*}

\Ni (a) For general $G, F, H$ making sense of the terms in the inequality, we have
\begin{align} \label{ineq:trilinear-1}
& \quad \,
   \iiiint_{\T^3 \times \R^6 \times \Ss^2}
    G_\ast \frac{F}{\vint{v}^\ell} H' 
    \vpran{\vint{v'}^{\ell} - \vint{v}^{\ell} \cos^\ell \tfrac{\theta}{2}} 
        b(\cos\theta) |v - v_\ast|^\gamma \dbmu   \nn
\\
&\leq
\ell \iiiint_{\T^3 \times \R^6 \times \Ss^2}
    G_\ast \vpran{F \vint{v}^{-2}  - F' \vint{v'}^{-2}} H' 
    \vpran{v_\ast \cdot \tilde \omega} \cos^{\ell}\tfrac{\theta}{2} \sin\tfrac{\theta}{2}
        b(\cos\theta) |v - v_\ast|^{1+\gamma} \dbmu \nn
\\
& \quad \,
   +    C_\ell \vpran{1 + \sup_x \norm{g}_{L^1_{\ell+\gamma}}}
   \min \left\{\norm{F}_{L^2_{x,v}} 
   \norm{H}_{L^2_{x,v}}, \,\,
   \norm{F/\vint{v}^{\ell-1-\gamma}}_{L^\infty_{x, v}}   
   \norm{H}_{L^1_{x,v}}\right\}  \nn
\\
& \quad \,
   + C_\ell \vpran{1 + \sup_x \norm{g}_{L^1_{4+\gamma}}}
   \min \left\{\norm{F}_{L^2_{x,v}} 
   \norm{H}_{L^2_{x,v}}, \,\,
   \norm{F}_{L^\infty_{x, v}}   
   \norm{H}_{L^1_{x, v}}\right\},
\end{align}
where $\tilde \omega$ is the unit vector defined in~\eqref{def:tilde-omega}. 

\smallskip

\Ni (b) Suppose $G = \mu + g$ is non-negative
and there exist $\ell, K_0$ such that
\begin{align} \label{assump:bound-g}
   g \vint{v}^\ell \leq K_0,
\qquad
   \ell > 8 + \gamma.
\end{align}
Then for any $s \in (0, 1)$ and the same $\ell$ as in~\eqref{assump:bound-g}, we have
\begin{align} \label{ineq:trilinear-1-f-minus}
& \quad \,
   \iiiint_{\T^3 \times \R^6 \times \Ss^2}
    G_\ast \frac{F}{\vint{v}^\ell} H' 
    \vpran{\vint{v'}^{\ell} - \vint{v}^{\ell} \cos^\ell \tfrac{\theta}{2}} 
        b(\cos\theta) |v - v_\ast|^\gamma \dbmu   \nn
\\
&\leq
\ell \iiiint_{\T^3 \times \R^6 \times \Ss^2}
    G_\ast \vpran{F \vint{v}^{-2}  - F' \vint{v'}^{-2}} H' 
    \vpran{v_\ast \cdot \tilde \omega} \cos^{\ell}\tfrac{\theta}{2} \sin\tfrac{\theta}{2}
        b(\cos\theta) |v - v_\ast|^{1+\gamma} \dbmu \nn
\\
& \quad \,
   +    C_\ell \vpran{1 + K_0}
   \min \left\{\norm{F}_{L^2_{x,v}} 
   \norm{H}_{L^2_{x,v}}, \,\,
   \norm{F}_{L^\infty_{x, v}}   
   \norm{H}_{L^1_{x,v}} \right\}  \nn
\\
& \quad \,
  + C_\ell (1 + K_0) \vpran{\sup_x \norm{F/\vint{v}^{\ell-1-\gamma}}_{L^1_v}}
   \norm{H}_{L^1_x L^1_\gamma}.
\end{align}
\end{prop}
\begin{proof}
The proofs for both parts follow from a revision of the proof of Proposition 3.1 in~\cite{AMSY}. Applying Lemma~\ref{lem:diff-v-k}, we decompose the integral as 
\begin{align} \label{decomp:general}
& \quad \,
   \iiiint_{\T^3 \times \R^6 \times \Ss^2}
    G_\ast \frac{F}{\vint{v}^\ell} H' 
    \vpran{\vint{v'}^{\ell} - \vint{v}^{\ell} \cos^\ell \tfrac{\theta}{2}} 
        b(\cos\theta) |v - v_\ast|^\gamma \dbmu    \nn
\\
& = \ell \iiiint_{\T^3 \times \R^6 \times \Ss^2}
    G_\ast \frac{F}{\vint{v}^2} H' 
    |v - v_\ast| (v \cdot \omega) \cos^{\ell-1}\tfrac{\theta}{2} \sin\tfrac{\theta}{2}
        b(\cos\theta) |v - v_\ast|^\gamma \dbmu   \nn
\\
& \quad \,
  + \iiiint_{\T^3 \times \R^6 \times \Ss^2}
    \vpran{G_\ast \vint{v_\ast}^\ell} \frac{F}{\vint{v}^\ell} H' 
    \sin^\ell \tfrac{\theta}{2} \,
        b(\cos\theta) |v - v_\ast|^\gamma \dbmu  \nn
\\
& \quad \,
   + \sum_{i=1}^3
       \iiiint_{\T^3 \times \R^6 \times \Ss^2}
    G_\ast \frac{F}{\vint{v}^\ell} H' \mathfrak{R}_i
     b(\cos\theta) |v - v_\ast|^\gamma \dbmu
\,\Denote \, \sum_{n=1}^5 \Gamma_n .
\end{align}
The main difference between~\eqref{ineq:trilinear-1} and~\eqref{ineq:trilinear-1-f-minus} is that in~\eqref{ineq:trilinear-1} the extra $\gamma$-weight falls on $g$ while in~\eqref{ineq:trilinear-1-f-minus} it is on $H$. 


\smallskip

\Ni (a) Deriving the bound for $\Gamma_1$ requires careful use of symmetry in the case of the strong singularity. The idea is similar to the proof of Proposition 3.1 in \cite{AMSY}. In particular, we decompose $\Gamma_1$ as 
\begin{align*}
  \Gamma_1
& = \ell \iiiint_{\T^3 \times \R^6 \times \Ss^2}
    G_\ast \vpran{F' \vint{v'}^{-2} } H' 
    \vpran{v_\ast \cdot \tilde \omega} \cos^{\ell}\tfrac{\theta}{2} \sin\tfrac{\theta}{2}
        b(\cos\theta) |v - v_\ast|^{1+\gamma} \dbmu
\\
& \quad \,
  + \ell \iiiint_{\T^3 \times \R^6 \times \Ss^2}
    G_\ast \vpran{F \vint{v}^{-2} } H' 
    \vpran{v_\ast \cdot \frac{v'-v_\ast}{|v' - v_\ast|}} \cos^{\ell-1}\tfrac{\theta}{2} \sin^2\tfrac{\theta}{2}
        b(\cos\theta) |v - v_\ast|^{1+\gamma} \dbmu
\\
& \quad \,
  + \ell \iiiint_{\T^3 \times \R^6 \times \Ss^2}
    G_\ast \vpran{F \vint{v}^{-2} - F' \vint{v'}^{-2}} H' 
    \vpran{v_\ast \cdot \tilde \omega} \cos^{\ell}\tfrac{\theta}{2} \sin\tfrac{\theta}{2}
        b(\cos\theta) |v - v_\ast|^{1+\gamma} \dbmu
\\
& \Denote 
   \Gamma_{1,1} + \Gamma_{1,2} + \Gamma_{1,3} .
\end{align*}
By symmetry, the first term $\Gamma_{1,1}$ vanishes. The second term $\Gamma_{1,2}$ is readily bounded by
\begin{align*}
   \abs{\Gamma_{1,2}}
&\leq
  C_\ell \iiint_{\T^3 \times \R^6}
    \vpran{\abs{G_\ast} \vint{v_\ast}^{2+\gamma}} |F| \abs{H'} 
    \dv\dv_\ast \dx
\\
&\leq
  C_\ell \vpran{1 + \sup_x \norm{g}_{L^1_{2+\gamma}}}
   \min \left\{\norm{F}_{L^2_{x,v}} 
   \norm{H}_{L^2_{x,v}}, \,\,
   \norm{F}_{L^\infty_{x, v}}   
   \norm{H}_{L^1_{x, v}}\right\} .
\end{align*}
We will leave $\Gamma_{1, 3}$ as is since in the later analysis, Proposition~\ref{prop:strong-sing-cancellation} will be applied in each specific case. Putting the components together gives the bound of $\Gamma_1$ as
\begin{align} \label{bound:Gamma-1}
  \Gamma_1
&\leq
  \ell \iiiint_{\T^3 \times \R^6 \times \Ss^2}
    G_\ast \vpran{F \vint{v}^{-2}  - F' \vint{v'}^{-2}} H' 
    \vpran{v_\ast \cdot \tilde \omega} \cos^{\ell}\tfrac{\theta}{2} \sin\tfrac{\theta}{2}
        b(\cos\theta) |v - v_\ast|^{1+\gamma} \dbmu \nn
\\
& \quad \,
   +    C_\ell \vpran{1 + \sup_x \norm{g}_{L^1_{4+\gamma}}}
   \min \left\{\norm{F}_{L^2_{x,v}} 
   \norm{H}_{L^2_{x,v}}, \,\,
   \norm{F}_{L^\infty_{x, v}}   
   \norm{H}_{L^1_{x, v}}\right\} .
\end{align}
\noindent
Next we show the estimate for $\Gamma_2$. Start with the direct bound using Cauchy-Schwarz and a regular change of variables stated in Lemma~\ref{prop:change-variable}:
\begin{align*}
   \Gamma_2
&\leq
 C \vpran{\iiint_{\T^3 \times \R^6}
    \abs{G_\ast} \vint{v_\ast}^{\ell+\gamma} F^2 \dv \dv_\ast \dx}^{1/2}
  \vpran{\iiint_{\T^3 \times \R^6 \times \Ss^2}
    \abs{G_\ast} \vint{v_\ast}^{\ell + \gamma} H^2 
      \dv \dv_\ast \dx  }^{1/2}
\\
&\leq
   C \vpran{\sup_{x} \norm{G \vint{v}^{\ell + \gamma}}_{L^1}}
   \norm{F}_{L^2_{x, v}} \norm{H}_{L^2_{x, v}}
\\
&\leq
   C_\ell \vpran{1 + \sup_{x} \norm{\vint{v}^{\ell + \gamma} g}_{L^1}}
   \norm{F}_{L^2_{x, v}} \norm{H}_{L^2_{x, v}}. 
\end{align*}
A second way to estimate $\Gamma_2$ is
\begin{align*} 
  \Gamma_2
&\leq
 C \vpran{\sup_x \norm{G\vint{v}^{\ell+\gamma}}_{L^1_v}} 
  \norm{F/\vint{v}^{\ell- \gamma}}_{L^\infty_{x, v}} 
  \norm{H}_{L^1_{x, v}}
\\
&\leq
  C_\ell \vpran{1 + \sup_{x} \norm{\vint{v}^{\ell+\gamma} g}_{L^1_v}}
  \norm{F/\vint{v}^{\ell- \gamma}}_{L^\infty_{x, v}}
  \norm{H}_{L^1_{x, v}},
\end{align*}
where again we have applied the regular change of variables. Overall, we have 
\begin{align} \label{bound:Gamma-2}
   \Gamma_2
\leq
   C_\ell \vpran{1 + \sup_{x} \norm{\vint{v}^{\ell + \gamma} g}_{L^1_v}}
   \min \left\{\norm{F}_{L^2_{x, v}} \norm{H}_{L^2_{x, v}}, \ \norm{F/\vint{v}^{\ell- \gamma}}_{L^\infty_{x, v}}
  \norm{H}_{L^1_{x, v}} \right\}. 
\end{align}
Estimate for $\Gamma_3$ is similar to $\Gamma_2$. By 
the bound on $\mathfrak{R}_1$ in Lemma~\ref{lem:diff-v-k}, we have
\begin{align} \label{bound:Gamma-3}
   \Gamma_3 
&\leq
  C_\ell \iiiint_{\T^3 \times \R^6 \times \Ss^2}
    \vpran{\abs{G_\ast} \vint{v_\ast}^{\ell-1}} \frac{|F|}{\vint{v}^{\ell-1}} \abs{H'}  \sin^{\ell-3} \tfrac{\theta}{2}
     b(\cos\theta) |v - v_\ast|^\gamma \dbmu   \nn 
\\
& \leq
   C_\ell \vpran{1 + \sup_{x} \norm{\vint{v}^{\ell - 1+ \gamma} g}_{L^1_v}}
   \min \left\{\norm{F}_{L^2_{x, v}} \norm{H}_{L^2_{x, v}}, \ \norm{F/\vint{v}^{\ell-1- \gamma}}_{L^\infty_{x, v}}
  \norm{H}_{L^1_{x, v}} \right\}.   
\end{align}
Estimates for $\Gamma_4$ and $\Gamma_5$ are more straightforward. Using the upper bound of $\mathfrak{R}_2$ and a regular change of variables, we have
\begin{align} \label{bound:Gamma-4}
  \abs{\Gamma_4}
&\leq
  C_\ell \iiiint_{\T^3 \times \R^6 \times \Ss^2}
    \vpran{\abs{G_\ast} \vint{v_\ast}^2} \frac{|F|}{\vint{v}^2} \abs{H'} 
    \sin^2 \tfrac{\theta}{2}
     b(\cos\theta) |v - v_\ast|^\gamma \dbmu \nn
\\
&\leq
  C_\ell \iiiint_{\T^3 \times \R^6 \times \Ss^2}
    \vpran{\abs{G_\ast} \vint{v_\ast}^{2+\gamma}} |F| \abs{H'} 
    \sin^2 \tfrac{\theta}{2}
     b(\cos\theta) \dbmu  \nn
\\
&\leq
  C_\ell \vpran{1 + \sup_x \norm{g}_{L^1_{2+\gamma}}}
   \min \left\{\norm{F}_{L^2_{x,v}} 
   \norm{H}_{L^2_{x,v}}, \,\,
   \norm{F}_{L^\infty_{x, v}}   
   \norm{H}_{L^1_{x, v}}\right\}.
\end{align}
Similarly, we can bound $\Gamma_5$ by
\begin{align} \label{bound:Gamma-5}
   \abs{\Gamma_5}
&\leq
  C_\ell \iiiint_{\T^3 \times \R^6 \times \Ss^2}
    \vpran{\abs{G_\ast} \vint{v_\ast}^4} \frac{|F|}{\vint{v}^4} \abs{H'} 
    \sin^2 \tfrac{\theta}{2}
     b(\cos\theta) |v - v_\ast|^\gamma \dbmu   \nn
\\
&\leq
    C_\ell \vpran{1 + \sup_x \norm{g}_{L^1_{4+\gamma}}}
   \min \left\{\norm{F}_{L^2_{x,v}} 
   \norm{H}_{L^2_{x,v}}, \,\,
   \norm{F}_{L^\infty_{x, v}}   
   \norm{H}_{L^1_{x, v}}\right\} .
\end{align}
The desired estimate in~\eqref{ineq:trilinear-1} is obtained by adding all the bounds for $\Gamma_1, \cdots,  \Gamma_5$ in~\eqref{bound:Gamma-1}-\eqref{bound:Gamma-5}.

\smallskip

\Ni (b) The proof is similar to part (a) with a revision based on the extra condition~\eqref{assump:bound-g} on $g$. In particular, we use the decomposition in~\eqref{decomp:general}:
\begin{align*}
    \iiiint_{\T^3 \times \R^6 \times \Ss^2}
    G_\ast \frac{F}{\vint{v}^\ell} H' 
    \vpran{\vint{v'}^{\ell} - \vint{v}^{\ell} \cos^\ell \tfrac{\theta}{2}} 
        b(\cos\theta) |v - v_\ast|^\gamma \dbmu
= \sum_{n=1}^5 {\Gamma_n},
\end{align*}
where $\Gamma_n$'s are exactly the same as in~\eqref{decomp:general}.
The estimates of $\Gamma_1, \Gamma_4, \Gamma_5$ remain the same as in part (a), which give
\begin{align} \label{est:rest-f-minus}
& \quad \,
   \abs{\Gamma_1} + \abs{\Gamma_4} + \abs{\Gamma_5}   \nn
\\
&\leq
   \ell \abs{\iiiint_{\T^3 \times \R^6 \times \Ss^2}
    G_\ast \vpran{F \vint{v}^{-2}  - F' \vint{v'}^{-2}} H' 
   \vpran{v_\ast \cdot \tilde \omega} \cos^{\ell}\tfrac{\theta}{2} 
   \sin\tfrac{\theta}{2}
        b(\cos\theta) |v - v_\ast|^{1+\gamma} \dbmu} \nn
\\
& \quad \, 
  +    C_\ell \vpran{1 + \sup_x \norm{g}_{L^1_{4+\gamma}}}
   \min \left\{\norm{F}_{L^2_{x,v}} 
   \norm{H}_{L^2_{x,v}}, \,\,
   \norm{F}_{L^\infty_{x, v}}   
   \norm{H}_{L^1_{x, v}}\right\}.
\end{align}

To prove~\eqref{ineq:trilinear-1-f-minus}, we combine~\eqref{assump:bound-g} with the positivity of $G$ and the singular change of variables to estimate $\Gamma_2$ and get
\begin{align} \label{bound:Gamma-2-f-minus-1}
   \abs{\Gamma_2}
&= \abs{\iiiint_{\T^3 \times \R^6 \times \Ss^2}
    \vpran{G_\ast \vint{v_\ast}^\ell} \frac{F}{\vint{v}^\ell} H' 
    \sin^\ell \tfrac{\theta}{2} \,
        b(\cos\theta) |v - v_\ast|^\gamma \dbmu}  \nn
\\
& \leq
   \iiiint_{\T^3 \times \R^6 \times \Ss^2}
    \vpran{\mu_\ast \vint{v_\ast}^\ell + K_0} \frac{|F|}{\vint{v}^\ell} |H'| 
    \sin^\ell \tfrac{\theta}{2} \,
        b(\cos\theta) |v - v_\ast|^\gamma \dbmu  \nn
\\
&\leq
  C_\ell (1 + K_0) \vpran{\sup_x \norm{F/\vint{v}^{\ell - \gamma}}_{L^1_v}}
   \norm{H}_{L^1_x L^1_\gamma}. 
\end{align}
Similarly, using the positivity of $G$ and~\eqref{assump:bound-g}, we have the bound of $\Gamma_3$ as
\begin{align} \label{bound:Gamma-3-f-minus-1}
  \abs{\Gamma_3}
&\leq
  C_\ell \iiiint_{\T^3 \times \R^6 \times \Ss^2}
    \vpran{G_\ast \vint{v_\ast}^{\ell-1}} \frac{|F|}{\vint{v}^{\ell-1}} |H'|  \sin^{\ell-3} \tfrac{\theta}{2}
     b(\cos\theta) |v - v_\ast|^\gamma \dbmu   \nn
\\
&\leq
  C_\ell (1 + K_0) \vpran{\sup_x \norm{F/\vint{v}^{\ell-1-\gamma}}_{L^1_v}}
   \norm{H}_{L^1_{x,v}}. 
\end{align}
Combining~\eqref{est:rest-f-minus},~\eqref{bound:Gamma-2-f-minus-1} and~\eqref{bound:Gamma-3-f-minus-1} gives~\eqref{ineq:trilinear-1-f-minus}. 
\end{proof}

We summarize explicit bounds for the first term on the right-hand side of~\eqref{ineq:trilinear-1} and~\eqref{assump:bound-g} in the following lemma:

\begin{prop} \label{prop:strong-sing-cancellation}
Let $\tilde \omega$ be the unit vector defined in~\eqref{def:tilde-omega}. Suppose $G, F, H$ are functions that make sense of the integral below. 

\smallskip
\Ni (a) If $s \in [1/2, 1)$, then for any pair of $(s_1, \gamma_1)$ satisfying 
\begin{align} \label{def:s-1-gamma-1}
   s_1 \in (2s - 1, s),
\qquad
   \frac{\gamma_1}{2}
= \frac{2+\gamma}{2} + s_1 -2
< \frac{\gamma}{2},
\end{align}
we have
\begin{align} \label{bound:cancellation-1}
& \quad \,
   \abs{\IntRRS G_\ast \vpran{F \vint{v}^{-2} - F' \vint{v}^{-2}} \, H' \big (v_\ast \cdot \tilde \omega \big)\cos^{\ell} \tfrac{\theta}{2}  \sin  \tfrac{\theta}{2} b(\cos\theta) |v - v_\ast|^{1+\gamma}
         \dsigma \dv_\ast \dv}  \nn
\\
& \leq     
  C \norm{G}_{L^1_{3+\gamma+2s} \cap L^2}
  \norm{F}_{H^{s_1}_{\gamma_1/2}}
  \norm{H}_{L^2_{\gamma/2}}. 
\end{align}

\Ni (b) If $s \in (0, 1/2)$, then 
\begin{align} \label{bound:cancellation-2}
& \quad \,
   \abs{\IntRRS G_\ast \vpran{F \vint{v}^{-2} - F' \vint{v}^{-2}} \, H' \big (v_\ast \cdot \tilde \omega \big)\cos^{\ell} \tfrac{\theta}{2}  \sin  \tfrac{\theta}{2} b(\cos\theta) |v - v_\ast|^{1+\gamma}
         \dsigma \dv_\ast \dv}  \nn
\\
& \leq     
  C \norm{G}_{L^1_{2+\gamma}}
  \min \left\{\norm{F}_{L^2_v} \norm{H}_{L^2_v}, \,\, \norm{F}_{L^\infty_v} \norm{H}_{L^1_v} \right\}. 
\end{align}

\Ni (c) If $F \in W^{1, \infty}(\R^3_v)$, then for any $s \in (0, 1)$ we have
\begin{align} \label{bound:cancellation-3}
& \quad \,
   \abs{\IntRRS G_\ast \vpran{F \vint{v}^{-2} - F' \vint{v'}^{-2}} \, H' 
   \big (v_\ast \cdot \tilde \omega \big)\cos^{\ell} \tfrac{\theta}{2}  \sin  \tfrac{\theta}{2} b(\cos\theta) |v - v_\ast|^{1+\gamma}
          \dsigma \dv_\ast \dv}  \nn
\\
& \leq     
  C \norm{F}_{W^{1, \infty}(R^3_v)}
  \norm{G}_{L^1_{3+\gamma}}
  \norm{H}_{L^1_\gamma}.
\end{align}
\end{prop}
\begin{proof}
Part (a) is an immediate application of (3.13) in \cite{AMSY} (with a reshuffle of the function names). Part~(b) is a direct bound using the fact that $b(\cos\theta) \sin\tfrac{\theta}{2}$ is integral if $s \in (0, 1/2)$. Hence,
\begin{align*}
   \text{LHS of~\eqref{bound:cancellation-2}}
&\leq
  C\IntRRS |G_\ast| \vpran{|F| \vint{v}^{-2} + |F'| \vint{v'}^{-2}} \, |H'| \vint{v_\ast} |v - v_\ast|^{1+\gamma}
         \dsigma \dv_\ast \dv
\\
& \leq
  C\IntRRS |G_\ast| \vint{v_\ast}^{2+\gamma} \vpran{|F| + |F'|} \, |H'| \dsigma \dv_\ast \dv.
\end{align*}
Depending on the property of $F$, we can obtain two types of bounds here:
\begin{align*}
  \IntRRS |G_\ast| \vint{v_\ast}^{2+\gamma} \vpran{|F| + |F'|} \, |H'| \dsigma \dv_\ast \dv
\leq
  \norm{G}_{L^1_{2+\gamma}} \norm{F}_{L^2_v} \norm{H}_{L^2_v} 
\end{align*}
and
\begin{align*}
  \IntRRS |G_\ast| \vint{v_\ast}^{2+\gamma} \vpran{|F| + |F'|} \, |H'| \dsigma \dv_\ast \dv
\leq
  \norm{G}_{L^1_{2+\gamma}} \norm{F}_{L^\infty_v} \norm{H}_{L^1_v}. 
\end{align*}
A combination of them gives~\eqref{bound:cancellation-2}.

\Ni Part (c) follows directly from the Mean Value Theorem, which gives the following bound:
\begin{align*}
   \abs{F \vint{v}^{-2} - F' \vint{v'}^{-2}} \, |v - v_\ast|^{1+\gamma}
&\leq
   \vpran{\abs{F - F'} \vint{v}^{-2} + \abs{\vint{v}^{-2} - \vint{v'}^{-2}} |F'|} 
   |v - v_\ast|^{1+\gamma}
\\
&\leq
   \vpran{\norm{\nabla_v F}_{L^\infty_v} \frac{|v - v'|}{\vint{v}^2}
 + \norm{F}_{L^\infty_v} \frac{(|v| + |v'|) |v - v'|}{\vint{v}^2 \vint{v'}^2}}
    |v - v_\ast|^{1+\gamma}
\\
&= \vpran{\norm{\nabla_v F}_{L^\infty_v} \frac{|v - v_\ast|^{2+\gamma}}{\vint{v}^2}
 + \norm{F}_{L^\infty_v} \frac{(|v| + |v'|) |v - v_\ast|^{2+\gamma}}{\vint{v}^2 \vint{v'}^2}} \sin\tfrac{\theta}{2}.
\end{align*}
Since on $\Ss^2_+$ it holds that 
\begin{align*}
    \tfrac{\sqrt{2}}{2}\abs{v - v_\ast}
\leq
   \abs{v' - v_\ast}
\leq
   \abs{v - v_\ast}, 
\end{align*}
there exists a generic constant $C$ such that
\begin{align*}
   \frac{(|v| + |v'|) |v - v_\ast|^{2+\gamma}}{\vint{v}^2 \vint{v'}^2}
&= \frac{|v| |v - v_\ast|^{2+\gamma}}{\vint{v}^2 \vint{v'}^2}
   + \frac{|v'| |v - v_\ast|^{2+\gamma}}{\vint{v}^2 \vint{v'}^2} 
\leq 
  C \vint{v_\ast}^{2+\gamma}. 
\end{align*}
Hence, the (partial) integrand satisfies 
\begin{align*}
& \quad \, 
  \abs{\vpran{F \vint{v}^{-2} - F' \vint{v'}^{-2}}
  \big (v_\ast \cdot \tilde \omega \big)\cos^{\ell} \tfrac{\theta}{2}  \sin  \tfrac{\theta}{2} b(\cos\theta)} |v - v_\ast|^{1+\gamma}
\\
& \leq
  C \vpran{\norm{\nabla_v F}_{L^\infty_v} \frac{|v - v_\ast|^{2+\gamma}}{\vint{v}^2}
 + \norm{F}_{L^\infty_v} \frac{(|v| + |v'|) |v - v_\ast|^{2+\gamma}}{\vint{v}^2 \vint{v'}^2}} \vint{v_\ast} 
\\
& \leq 
  C \norm{F}_{W^{1, \infty}_v}
  \vint{v'}^\gamma \vint{v_\ast}^{3+\gamma},
\end{align*}
when restricted on $\Ss^2_+$. Inserting such bound into the left-hand side of~\eqref{bound:cancellation-3}, we get 
\begin{gather*}
  \text{LHS of~\eqref{bound:cancellation-3}}
\leq
  C \norm{F}_{W^{1, \infty}_v}
  \iint_{\R^3 \times \R^3}
    |G_\ast| \vint{v_\ast}^{3+\gamma}
    \vint{v'}^\gamma H' \dv_\ast \dv
\\
\leq
  C \norm{F}_{W^{1, \infty}_v}
  \norm{G}_{L^1_{3+\gamma}}
  \norm{H}_{L^1_\gamma}. \qedhere
\end{gather*}

\end{proof}

We also record a proposition using the symmetry cancellation:
\begin{prop} \cite[Lemma 2.1]{CCL} \label{prop:symmetry-cancel}
Suppose $H \in W^{2, \infty}(\R^3)$. Then for any $s \in (0, 1)$, it holds that
\begin{align*}
   \abs{\int_{\Ss^2}
       \vpran{H' - H} b(\cos\theta) \dsigma}
&\leq
  C \vpran{\sup_{|u| \leq |v| + |v_\ast|} \abs{\nabla H (u)} 
        + \sup_{|u| \leq |v| + |v_\ast|} \abs{\nabla^2 H (u)}} |v - v_\ast|^2.
\end{align*}
\end{prop}

\subsection{Interpolation results}
In this section we collect several results about interpolation in fractional Sobolev spaces that will be used in the sequel.
\begin{lem}\label{app-inter-x-theta-1}
Let $\eta,\eta'\in(0,1)$.  Then for $r = r(\eta, \eta', d) > 2$ and $\alpha = \alpha(\eta, \eta', d)\in(0,1)$ defined in~\eqref{def:alpha-r}, it follows that
\begin{align} \label{ineq:interp-1}
   \norm{\varphi}_{L^{r}_{x,v}}
\leq 
  C \vpran{\int_{\T^d} 
  \norm{(-\Delta_v)^{\eta/2}\varphi(x,\cdot)}^2_{L^2_v} \dx}^{\frac{\alpha}{2}}
     \vpran{\int_{\R^{d}} \norm{(1-\Delta_x)^{\eta'/2} \varphi(\cdot, v)}^{2}_{L^{2}_x} \dv}^{\frac{1-\alpha}{2}}.
\end{align} 
The constant $C$ only depends on $\eta, \eta', d$.
\end{lem}
\begin{proof}
By the Sobolev embedding in $\R^d$ and $\T^d$ there exists $c$ depending only $\eta, \eta', d$ such that
\begin{align}
   \int_{\T^d} \norm{(-\Delta_v)^{\eta/2}\varphi(x,\cdot)}^{2}_{L^{2}_{v}} \dx 
&\geq 
   c \int_{\T^d}\Big(\int_{\R^{d}} \big| \varphi(x, v) \big|^{p} \dv\Big)^{\frac{2}{p}}\dx,  \label{interp:1}
\\
   \int_{\R^{d}}\|(1-\Delta_x)^{\eta'/2} \varphi(\cdot, v)\|^{2}_{L^{2}_x} \dv &\geq 
   c \int_{\R^{d}}\left(\int_{\T^{d}} \big|\varphi(x, v)\big|^{q} \dx\right)^{\frac{2}{q}} \dv, \label{interp:2}
\end{align}
where
\begin{align}  \label{def:p-q}
   \frac{1}{q} = \frac {1}{2} - \frac{\eta'}{d}, 
\qquad 
   \frac{1}{p} = \frac {1}{2} - \frac{\eta}{d},
\qquad
   q, \, p > 2. 
\end{align}
Set
\begin{align}
   &\alpha_1 = \frac{q - 2}{ \frac{p}{2} \, q - 2} \in (0, 1), 
\qquad 
   \alpha_2 = \frac{p}{2} \, \alpha_1 \in (0, 1),  \label{alpha12}
\\  
   & r = p \, \alpha_1 + 2(1 - \alpha_1) > 2,
\qquad
   \alpha = \frac{2\alpha_2}{r} \in (0, 1). \label{def:alpha-r}
\end{align} 
One can readily check that
\begin{align} \label{rel:parameters-1}
    \frac{\alpha_1}{\alpha_2} = \frac{2}{p},
\qquad 
   \frac{1-\alpha_1}{1-\alpha_2} = \frac{q}{2} > 1,
\qquad
   r = 2 \alpha_2 + q (1- \alpha_2),
\qquad
   \frac{q}{2r}(1- \alpha_2) = \frac{1-\alpha}{2}.
\end{align}
Then, using the H\"{o}der inequality we have 
\begin{align*}
    \int_{\T^d} \int_{\R^{d}} \abs{\varphi(x,v)}^r \dv \dx
&\leq 
   \int_{\T^d} \vpran{\int_{\R^{d}} \abs{\varphi(x,v)}^{p} \dv}^{\alpha_1}
   \vpran{\int_{\R^{d}} \abs{\varphi(x,v)}^{2} \dv}^{1-\alpha_1}\dx
\\
&\leq
    \vpran{\int_{\T^d} \vpran{\int_{\R^{d}} \abs{\varphi(x,v)}^{p} \dv}^{\frac{\alpha_1}{\alpha_2}}\dx}^{\alpha_2}
    \vpran{\int_{\T^{d}}\vpran{\int_{\R^{d}} \abs{\varphi(x,v)}^{2} \dv}^{\frac{1-\alpha_1}{1-\alpha_{2}}}\dx}^{1-\alpha_{2}}
\\
& = \vpran{\int_{\T^d} \vpran{\int_{\R^{d}} \abs{\varphi(x,v)}^{p} \dv}^{\frac{2}{p}}\dx}^{\alpha_2}
    \vpran{\int_{\T^{d}}\vpran{\int_{\R^{d}} \abs{\varphi(x,v)}^{2} \dv}^{\frac{ q }{2}}\dx}^{1-\alpha_{2}}
\\
&\leq 
   C \vpran{\int_{\T^d} \|(-\Delta_{v})^{\eta/2}\varphi(x,\cdot)\|^{2}_{L^{2}_v} \dx}^{\alpha_2}
   \vpran{\int_{\R^{d}}\vpran{\int_{\T^{d}} \abs{\varphi(x,v)}^{ q } \dx}^{\frac{2}{ q }}\dv}^{\frac{ q }{2}(1-\alpha_{2})}
\\
&\leq 
    C \vpran{\int_{\T^d} \norm{(-\Delta_{v})^{\eta/2}\varphi(x,\cdot)}_{L^{2}_v}^2 \dx}^{\alpha_2}
    \vpran{\int_{\R^{d}} \norm{(1-\Delta_x)^{\eta'/2} \varphi(\cdot, v)}^{2}_{L^{2}_x} \dv}^{\frac{q}{2}(1-\alpha_{2})},
\end{align*}
where the Minkowski's integral inequality is used in the second last step.
Inequality~\eqref{ineq:interp-1} then follows by the definition and property of $\alpha$ in~\eqref{def:alpha-r} and~\eqref{rel:parameters-1}. 
\end{proof}
\noindent
Observe that the estimates hold in the proof of Lemma~\ref{app-inter-x-theta-1}, or equivalently, the existence of $\alpha, \alpha_1, \alpha_2, r$ in the correct range is guaranteed as long as $p, q > 2$. Based on such observation, we have a second interpolation similar as Lemma~\ref{app-inter-x-theta-1}:
%
%
%
\begin{lem}\label{app-inter-x-theta}
Let $\eta, \eta' \in (0,1)$ and $m \geq 1$.  Then, for some $\tilde r = r(\eta, \eta', m, d) > 2$ and $\tilde \alpha = \tilde \alpha (\eta, \eta', m, d) \in (0, 1)$, we have
\begin{align*}
\| \varphi \|_{L^{\tilde r}_{x,v}}\leq C\bigg(\int_{\T^{d}}\| (-\Delta_v)^{\eta/2}\varphi(x,\cdot) \|^{2}_{L^{2}_{v}}\dx\bigg)^{\frac{\tilde \alpha}{2}}\bigg(\int_{\R^{d}}\big\|(1-\Delta_x)^{\eta'/2} \varphi^{2}(\cdot, v) \big\|_{L^{m}_x}\dv\bigg)^{\frac{1-\tilde \alpha}{2}}.
\end{align*} 
The constant $C$ only depends on $\eta, \eta',m, d$.
\end{lem}
\begin{proof}
Using Sobolev imbedding we have that 
\begin{align*}
   \int_{\T^d}\| (-\Delta_v)^{\eta/2}\varphi(x,\cdot) \|^{2}_{L^{2}_{v}} \dx &\geq 
     c \int_{\T^d}\Big(\int_{\R^{d}} \big| \varphi(x, v) \big|^{p} \dv\Big)^{\frac{2}{p}}\dx, 
\\
   \int_{\R^{d}}\|(1-\Delta_x)^{\eta'/2} \varphi^2(\cdot, v)\|_{L^m_x} \dv &\geq 
     c \int_{\R^{d}}\left(\int_{\T^{d}} \big|\varphi(x, v)\big|^{\tilde q} \dx\right)^{\frac{2}{\tilde q}} \dv,
\end{align*}
where $c=c(\eta,\eta',m,d)$ and
\begin{align*} 
   \frac{2}{\tilde q} = \frac {1}{m} - \frac{\eta'}{d}, 
\qquad 
   \frac{1}{p} = \frac {1}{2} - \frac{\eta}{d},
\qquad
   p, \,\, \tilde q > 2,
\end{align*}
which are in a similar form as~\eqref{interp:1} and~\eqref{interp:2}. By the comment before the statement of Lemma~\ref{app-inter-x-theta}, the desired inequality holds with
\begin{gather*}
   \tilde \alpha_1 = \frac{\tilde{q} - 2}{\frac{p}{2} \, \tilde{q} - 2} \in (0, 1), 
\qquad 
   \tilde \alpha_2 = \frac{p}{2} \, \tilde \alpha_1 \in (0, 1), 
\qquad  
   \tilde r = p \, \tilde \alpha_1 + 2(1 - \tilde \alpha_1) > 2,
\qquad
    \tilde \alpha = \frac{2 \tilde \alpha_2}{\tilde r}. \qedhere
\end{gather*} 
\end{proof}
\smallskip

Next we show a ``Leibniz" rule for fractional derivatives:
\begin{lem}\label{app-square-f}
Let $p \in (1,2)$, $0\leq \beta'<\beta\in(0,1)$,
\begin{align} \label{def:p-0}
   p'=\frac{p}{2-p}
\qquad \text{that is} \qquad
   2 p' = \frac{p}{1 - p/2}.
\end{align}
Then for any $\varphi$ making sense of the terms of the inequality below, it follows that
\begin{align*}
    \norm{(-\Delta)^{\frac {\beta'}{2} }\varphi^{2}}_{L^{p}(\mathbb{R}^{d})}
\leq 
   C \vpran{\norm{\varphi}_{ \dot{H}^{\beta}(\mathbb{R}^{d}) } \norm{\varphi^{2}}^{\frac12}_{ L^{p'}(\mathbb{R}^{d}) } + \norm{\varphi^{2}}_{L^{p}(\mathbb{R}^{d})}},
\end{align*}
where the constant $C$ only depends on $d, \beta' , \beta, p$ and $\dot{H}^\beta(\mathbb{R}^{d})$ is the homogeneous Bessel potential space. 
\end{lem}
\begin{proof}
By the continuous embedding of the Bessel potential space in the fractional Sobolev-Slobodeckij space for $p\in(1,2]$, it follows that
\begin{align*}
    \norm{(-\Delta)^{\frac {\beta'}{2} }\varphi^{2}}_{L^{p}}^{p}
& \leq C \int_{\mathbb{R}^{d}}\int_{\mathbb{R}^{d}}\frac{\abs{\varphi^{2}(x) - \varphi^{2}(y)}^{p}}{\abs{x - y}^{d + \beta' p}} \dx\dy
\\
&\leq  
   C\vpran{\int_{\mathbb{R}^{d}}\int_{|x - y |\leq 1}
      + \int_{\mathbb{R}^{d}}\int_{ | x - y |>1}}
   \frac{\abs{\varphi^{2}(x) - \varphi^{2}(y)}^{p}}
          {\abs{x - y}^{d + \beta' p}} \dx\dy
\Denote I_{1} + I_{2}.
\end{align*}
A simple computation shows that 
\begin{align*}
   I_{2}
\leq 
   C \norm{\varphi^{2}}^{p}_{L^{p}}, 
\end{align*}
where $C$ depends on $d,\beta$ and $p$.  To estimate $I_1$, decompose its integrand as the product
\begin{align*}
   \frac{\abs{\varphi^{2}(x) - \varphi^{2}(y)}^{p}}{\abs{x - y}^{\beta' p}} 
 = \frac{\abs{\varphi(x) - \varphi(y)}^{p} }{\abs{x - y}^{ \beta p}}  \,
   \frac{\abs{\varphi(x) + \varphi(y)}^{p} }{\abs{x - y}^{ -(\beta - \beta')p }}.
\end{align*}
Then a direct application of the H\"{o}lder inequality with measure $|x-y|^{-d}\dx\dy $ and pair $\big( \frac{2}{p},\frac{2}{2-p} \big)$ gives
\begin{align*}
    I_{1}
&\leq 
   \norm{\varphi}^{p}_{\dot{H}^{\beta}} 
    \vpran{\int_{\mathbb{R}^{d}} \int_{|x -y | \leq 1}
        \frac{\abs{\varphi(x) + \varphi(y)}^{2p'}}{\abs{x - y}^{d - 2p'(\beta - \beta')}}\dx\dy}^{\frac {p}{2p'}} 
\leq 
   C_{d,\beta,\beta',p} \norm{\varphi}^{p}_{\dot{H}^{\beta}} \norm{\varphi^{2}}^{\frac{p}{2}}_{L^{p'}},
\end{align*}
which combined with the estimate for $I_2$ proves the result.
\end{proof}
\subsection{Strong Averaging Lemma}
\noindent
The following result is a time-localised version of \cite[Theorem 1.3]{Bouchut} that is needed for the Cauchy problem.
\begin{prop}\label{average-lemma-p}
Fix $0 \leq T_{1} < T_{2}$, $p\in(1,\infty)$, $\beta\geq0$, and assume that $f \in C\big([T_{1},T_{2}];L^{p}_{x,v}\big)$ with $\Delta^{\beta/2}_{v}f \in L^{p}_{t,x,v}$ satisfies
\begin{align*}
\del_{t}f + v\cdot\nabla_{x} f = \mathcal{F},\qquad t\in(0,\infty).
\end{align*}
Then, for any $r\in[0,\frac{1}{p}]$, $m \in \mathbb{N}$, $\beta_{-}\in[0,\beta)$, if we define
\begin{align} \label{def:s-prime}
   s^\flat = \frac{(1- r\, p)\,\beta_{-}}{p\,(1+m+\beta)},
\end{align} 
and
\begin{align*}
   \tilde f = f \, 1_{(T_{1},T_{2})}(t),
\qquad
   \tilde{\mathcal{F}} = \mathcal{F} \, 1_{(T_{1},T_{2})}(t),
\end{align*}
then it follows that
\begin{align} \label{ineq:averaging-lemma}
   \norm{(-\Delta_x)^{\frac{s^\flat}{2}} \tilde f}_{L^{p}_{t,x,v}} 
   + \norm{(-\del_t^2)^{\frac{s^\flat}{2}} \tilde f}_{L^{p}_{t,x,v}} 
&\leq
   C \Big( \norm{\vint{v}^{1+m} (1-\Delta_{x})^{-\frac r2}(1-\Delta_{v})^{-\frac m2}f(T_1)}_{L^{p}_{x,v}} \nn 
\\
& \hspace{1cm} +  \norm{\vint{v}^{1+m} (1-\Delta_{x})^{-\frac r2}(1-\Delta_{v})^{-\frac m2}f(T_2)}_{L^{p}_{x,v}} \nn   
\\
& \hspace{1cm} 
  + \norm{\vint{v}^{1+m}(1-\Delta_{x} - \del^{2}_{t})^{-\frac r2}(1- \Delta_{v})^{-\frac m2}\tilde{\mathcal{F}}}_{L^{p}_{t,x,v}}
\\
& \hspace{1cm}
    + \norm{(-\Delta_v)^{\beta/2} \tilde f}_{L^{p}_{t,x,v}}
    + \big\| \,\tilde f \,\big\|_{L^{p}_{t,x,v}}  \Big),  \nn
\end{align}
where the constant $C$ only depends on $d,\beta,r,m,p$. 
\end{prop}

\begin{proof}
Multiplying the transport equation by $1_{(T_{1},T_{2})}(t)$ we arrive at
\begin{equation*}
  \del_t \tilde f + v \cdot \nabla_{x} \tilde f 
= \big( f(T_1)\delta(t-T_1) -   f(T_2)\delta(t-T_2) \big) + \tilde{\mathcal{F}} 
\;\;\Denote \;
    A + \tilde{\mathcal{F}}, \qquad t\in(-\infty,\infty).
\end{equation*}
Write the sources as
\begin{align*}
  A 
&= (1-\Delta_{x} - \del^{2}_{t})^{\frac {\tilde r}{2} }(1- \Delta_{v})^{\frac {m}{2} } \Big( (1-\Delta_{x} - \del^{2}_{t})^{-\frac {\tilde r}{2} }(1- \Delta_{v})^{-\frac {m}{2}} A\Big),
\\
\tilde{\mathcal{F}} &= (1-\Delta_{x} - \del^{2}_{t})^{\frac r2}(1- \Delta_{v})^{\frac m2} \Big( (1-\Delta_{x} - \del^{2}_{t})^{-\frac r2}(1- \Delta_{v})^{-\frac m2} \tilde{\mathcal{F}}\Big),
\end{align*}
for $0 \leq r < \tilde r \in (0,1]$ and $m\in\mathbb{N}$. By \cite[Theorem 1.3]{Bouchut} and the additive contribution of the sources one has that 
\begin{align*}
  \norm{(-\Delta_x)^{\frac{s^\flat}{2}} \tilde f}_{L^p_{t,x,v}} 
+  \norm{(-\del_t^2)^{\frac{s^\flat}{2}} \tilde f}_{L^p_{t,x,v}}
&\leq 
   C\Big(\|\tilde f\|_{L^p_{t,x,v}} + \norm{(-\Delta_v)^{\beta/2} \tilde f}_{L^p_{t,x,v}}
\\
& \hspace{1cm} 
   + \norm{\vint{v}^{1+m} (1-\Delta_{x} - \del^{2}_{t})^{-\frac{\tilde r}{2} }(1- \Delta_{v})^{-\frac {m}{2}}A}_{L^{p}_{t,x,v}} 
\\
& \hspace{1cm} 
   + \norm{\vint{v}^{1+m}(1-\Delta_{x} - \del^{2}_{t})^{-\frac r2}(1- \Delta_{v})^{-\frac m2}\tilde{\mathcal{F}})}_{L^{p}_{t,x,v}} \Big),
\end{align*}
for $s^\flat=\min\big\{\frac{(1-r)\beta}{m+1+\beta}, \frac{(1-\tilde r)\beta}{m+ 1+\beta}\big\}=\frac{(1-\tilde r)\beta}{m+ 1+\beta}$.  It remains to estimate the term involving $A$ on the right. First by \cite[Lemma 2.3]{Bouchut} we have
\begin{align*}
   \norm{\vint{v}^{1+m} (1-\Delta_{x} - \del^{2}_{t})^{-\frac{\tilde r}{2} }(1- \Delta_{v})^{-\frac {m}{2}}A}_{L^{p}_{t,x,v}} 
\leq
  C \norm{\vint{v}^{1+m} (1-\Delta_{x})^{-\frac{r}{2}} (1 - \del^{2}_{t})^{-\frac{\tilde r -\tilde r}{2}}(1- \Delta_{v})^{-\frac {m}{2}}A}_{L^{p}_{t,x,v}}. 
\end{align*}
By the definition of $A$, we can explicitly compute 
\begin{align*}
  (1 - \del^{2}_{t})^{-\frac{\tilde r -\tilde r}{2}} (1-\Delta_{x})^{-\frac{r}{2}} (1- \Delta_{v})^{-\frac {m}{2}}A 
& = \CalB_{\tilde r - r}(t - T_1)  \vpran{(1-\Delta_{x})^{-\frac{r}{2}} (1- \Delta_{v})^{-\frac {m}{2}}\,f(T_1)} 
\\
& \quad 
 - \CalB_{\tilde r - r}(t - T_2) \vpran{(1-\Delta_{x})^{-\frac{r}{2}} (1- \Delta_{v})^{-\frac {m}{2}}\,f(T_2)},
\end{align*}
where $\CalB_{\tilde r - r}$ is the Bessel kernel of order $\tilde r - r$ in $\R$.
The asymptotic behaviours of the Bessel kernel near $0$ and $\infty$ give
\begin{align*}
   0 
\leq 
  \CalB_{\tilde r - r}(t - T_1) 
\leq 
  C_{d,\tilde r, r} \frac{e^{-|t - T_1|}}{|t - T_1|^{1+ r - \tilde r}},
\qquad
   0 
\leq 
  \CalB_{\tilde r - r}(t - T_2)
\leq 
  C_{d,\tilde r, r} \frac{e^{-|t - T_2|}}{|t - T_2|^{1+ r - \tilde r}},
\qquad
  t \in \R.
\end{align*}
As a consequence, it follows that
\begin{align*}
& \quad \,
 \norm{\vint{v}^{1+m} \CalB_{\tilde r - r}(t - T_1) 
   \vpran{(1-\Delta_{x})^{-\frac{r}{2}} (1- \Delta_{v})^{-\frac {m}{2}} f(T_1)}}_{L^{p}_{t,x,v}} 
\\
&\leq 
  C_{d,\tilde r,p} \norm{\vint{v}^{1+m} (1-\Delta_{x})^{-\frac r2}(1-\Delta_v)^{-\frac m2}f(T_1)}_{L^{p}_{x,v}}, 
\end{align*}
under the condition $ \tilde r > 1 - \frac{1}{p} + r >0$, which together with $s^\flat=\frac{(1-\tilde r)\beta}{m+ 1+\beta}$ implies the choice \eqref{def:s-prime}.  
Analogous estimate follows at the point $t=T_2$.
\end{proof}

The next proposition is an estimate for $Q(F, \mu)$:

\begin{prop} \label{prop:Q-F-mu}
For any $F \in L^1_{\gamma+2s}(\R^3_v)$, the quantity $Q(F, \mu)$ is in $L^\infty(\R^3_v)$ with the bound
\begin{align} \label{bound:L-infty-Q-F-mu}
   \norm{Q(F, \mu)}_{L^\infty(\R^3_v)}
\leq
  C \norm{F}_{L^1_{\gamma+2s}(\R^3_v)}.
\end{align}
\end{prop}
\begin{proof}
The proof follows from a similar line of argument as the proof of Proposition 2.1 in \cite{AMUXY-1}. First we decompose $|v - v_\ast|^\gamma$ as
\begin{align*}
   |v - v_\ast|^\gamma
= \Phi_c + \Phi_{\bar c}, 
\end{align*}
where $\Phi_{\bar c}$ is smooth and $\Phi_{\bar c} = 0$ near $v = v_\ast$ while $\Phi_c = |v - v_\ast|^\gamma$ near $v = v_\ast$. The main property of $\Phi_c$ is 
\begin{align} \label{bound:Phi-c}
   \abs{\nabla^\alpha \hat\Phi_c (\xi)} 
\lesssim
  \frac{1}{\vint{\xi}^{3 + \gamma + |\alpha|}} ,
\qquad
  \forall \, |\alpha| \in \NN \cup \{0\} ,
\end{align}
where $\hat \Phi_c$ is the Fourier transform of $\Phi_c$. Denote $Q_c, Q_{\bar c}$ as the corresponding collision operators such that
\begin{align*}
   Q (F, \mu) = Q_c (F, \mu) + Q_{\bar c}(F, \mu).
\end{align*}
Then by the trilinear estimate (2.1) in~\cite{AMUXY-1}, we have
\begin{align} \label{bound:Q-reg}
   \norm{Q_{\bar c} (F, \mu)}_{L^\infty}
\leq
   C \norm{F}_{L^1_{\gamma + 2s}}. 
\end{align}
Hence we are left to bound $Q_c (F, \mu)$. Take an arbitrary $h \in L^1(\R^3)$. 
In the Fourier space, we have
\begin{align} \label{def:Q-c-bilinear}
  \vint{Q_c(F, \mu), h}
&= \iiint_{\Ss^2 \times \R^6} b\vpran{\frac{\xi}{|\xi|} \cdot \sigma}
   \vpran{\hat \Phi_c (\xi_\ast - \xi^-) - \hat\Phi_c(\xi_\ast)}
   \hat F(\xi_\ast) \hat \mu(\xi - \xi_\ast)
   \overline{\hat h(\xi)} \dxi \dxi_\ast \dsigma,
\end{align}
where 
\begin{align*}
   \xi^\pm = \frac{1}{2} \vpran{\xi \pm |\xi| \sigma}, 
\qquad
  |\xi^-| = |\xi| \sin\tfrac{\theta}{2}
\quad \text{with} \quad
  \cos\theta = \frac{\xi}{|\xi|} \cdot \sigma. 
\end{align*}
Note that $\xi^+$ is perpendicular to $\xi^-$.
%
By Taylor's theorem, we have
\begin{align} \label{eq:Taylor-hat-Phi-c}
   \hat \Phi_c (\xi_\ast - \xi^-) - \hat\Phi_c(\xi_\ast)
= -\xi^- \cdot \nabla \hat \Phi_c (\xi_\ast)
   + \vpran{\int_0^1 (1 - t) \nabla^2 \hat \Phi_c (\xi_\ast - t \xi^-) \dt} : (\xi^- \otimes \xi^-).
\end{align}
Similar as in~\cite{AMUXY-1}, we decompose $\xi^-$ as
\begin{align*}
    \xi^-
= \frac{|\xi|}{2} \vpran{\vpran{\frac{\xi}{|\xi|} \cdot \sigma} \frac{\xi}{|\xi|} - \sigma}
   + \vpran{1 - \vpran{\frac{\xi}{|\xi|} \cdot \sigma}} \frac{\xi}{2}.
\end{align*}
Inserting~\eqref{eq:Taylor-hat-Phi-c} into~\eqref{def:Q-c-bilinear}, we get
\begin{align*}
  \vint{Q_c(F, \mu), h}
& = - \iiint_{\Ss^2 \times \R^6} b\vpran{\frac{\xi}{|\xi|} \cdot \sigma}
   \vpran{\frac{|\xi|}{2} \vpran{\vpran{\frac{\xi}{|\xi|} \cdot \sigma} \frac{\xi}{|\xi|} - \sigma}} \cdot \nabla \hat \Phi_c (\xi_\ast)
   \hat F(\xi_\ast) \hat \mu(\xi - \xi_\ast)
   \overline{\hat h(\xi)} \dxi \dxi_\ast \dsigma
\\
& \quad \,
   - \iiint_{\Ss^2 \times \R^6} b\vpran{\frac{\xi}{|\xi|} \cdot \sigma}
   \vpran{1 - \vpran{\frac{\xi}{|\xi|} \cdot \sigma}} \frac{\xi}{2} \cdot \nabla \hat \Phi_c (\xi_\ast)
   \hat F(\xi_\ast) \hat \mu(\xi - \xi_\ast)
   \overline{\hat h(\xi)} \dxi \dxi_\ast \dsigma
\\
& \quad \, 
  + \int_0^1 (1 - t)  \iiint_{\Ss^2 \times \R^6} b\vpran{\frac{\xi}{|\xi|} \cdot \sigma}
     \hat F(\xi_\ast) \hat \mu(\xi - \xi_\ast)
   \overline{\hat h(\xi)}
   \vpran{\nabla^2 \hat \Phi_c (\xi_\ast - t \xi^-) : (\xi^- \otimes \xi^-)} 
\\
& \Denote
   \vint{Q_c^{(1)}(F, \mu), h}
   + \vint{Q_c^{(2)}(F, \mu), h}
   + \vint{Q_c^{(3)}(F, \mu), h}. 
\end{align*}
By symmetry $\vint{Q_c^{(1)}(F, \mu), h}$ vanishes. By the property that
\begin{align*}
    \abs{1 - \vpran{\frac{\xi}{|\xi|} \cdot \sigma}}
 = 2 \sin^2{\tfrac{\theta}{2}}, 
\qquad
  \cos\theta = \frac{\xi}{|\xi|} \cdot \sigma,
\end{align*}
we have
\begin{align} \label{bound:Q-c-2}
  \abs{\vint{Q_c^{(2)}(F, \mu), h}}
&\leq
  C \iiint_{\Ss^2 \times \R^6} 
   |\xi| \abs{\nabla \hat \Phi_c (\xi_\ast)}
   \abs{\hat F(\xi_\ast)} \abs{\hat \mu(\xi - \xi_\ast)}
   \abs{\overline{\hat h(\xi)}} \dxi \dxi_\ast \dsigma  \nn
\\
&\leq
   C \norm{F}_{L^1_v} \norm{h}_{L^1_v}
   \iint_{\R^6} \frac{\vint{\xi_\ast} \vint{\xi - \xi_\ast}}{\vint{\xi_\ast}^{4+\gamma}}
   \abs{\hat \mu(\xi - \xi_\ast)} \dxi\dxi_\ast   \nn
\\
& \leq 
   C \norm{F}_{L^1_v} \norm{h}_{L^1_v}. 
\end{align}
To bound $\vint{Q_c^{(3)}(F, \mu), h}$, we first use the property of $\xi^-$ to get
\begin{align*}
   \abs{\vint{Q_c^{(3)}(F, \mu), h}}
&\leq
   \int_0^1 (1 - t) \iiint_{\Ss^2 \times \R^6} b\vpran{\frac{\xi}{|\xi|} \cdot \sigma}
   \abs{\hat F(\xi_\ast)} \abs{\hat \mu(\xi - \xi_\ast)}
   \abs{\overline{\hat h(\xi)}}
   \abs{\nabla^2 \hat \Phi_c (\xi_\ast - t \xi^-)}  \abs{\xi^-}^2
\\
& \leq
   C \norm{F}_{L^1_v} \norm{h}_{L^1_v} 
   \int_0^1 \iiint_{\Ss^2 \times \R^6} 
   \abs{\hat \mu(\xi - \xi_\ast)}
   \abs{\nabla^2 \hat \Phi_c (\xi_\ast - t \xi^-)}  \abs{\xi}^2
   \dxi\dxi_\ast \dsigma \dt
\\
& \leq
   C \norm{F}_{L^1_v} \norm{h}_{L^1_v} 
   \int_0^1 \iiint_{\Ss^2 \times \R^6} 
   \abs{\hat \mu(\xi - \xi_\ast)}
   \frac{\abs{\xi}^2}{\vint{\xi_\ast - t \xi^-}^{3+\gamma+2}}  
   \dxi\dxi_\ast \dsigma \dt. 
\end{align*}
Make the change of variables
\begin{align*}
  w = \xi_\ast - \xi,
\qquad
  z = \xi_\ast - t \xi^-. 
\end{align*}
Since $b(\cos\theta)$ is supported on $\cos\theta \geq 0$, we have $\theta \in [0, \pi/2]$ and
\begin{align*}
   \abs{\frac{\del(w, z)}{\del(\xi_\ast, \xi)}}
= \abs{\det 
    \begin{pmatrix}
      I & -I  \\[2pt]
      I & -\frac{t}{2} \vpran{I - \sigma \otimes \frac{\xi}{|\xi|}}
    \end{pmatrix}}
= (1 - t/2)^2 \vpran{1 - t \sin^2 \tfrac{\theta}{2}}
\geq 
    \frac{1}{4} \vpran{1 - t \sin^2 \tfrac{\theta}{2}}
\geq 
   1/8. 
\end{align*}
Similarly, by $\sin\frac{\theta}{2} \leq \sqrt{2}/2$ and the fact that $\xi^+ \perp \xi^-$, we have 
\begin{align*}
   |w - z| 
= \abs{\xi - t \xi^-}
= \abs{\xi^+ + (1 - t) \xi^-}
\geq
   \abs{\xi^+}
\geq
    \tfrac{\sqrt{2}}{2} |\xi|, 
\end{align*}
which gives
\begin{align*}
   |\xi| \leq \sqrt{2} |w - z|. 
\end{align*}
Applying the change of variables $(\xi, \xi_\ast) \to (w, z)$ in $Q_c^{(3)}$, we have
\begin{align*}
  \abs{\vint{Q_c^{(3)}(F, \mu), h}}
&\leq
     C \norm{F}_{L^1_v} \norm{h}_{L^1_v} 
   \int_0^1 \iiint_{\Ss^2 \times \R^6} 
   \abs{\hat \mu(w)}
   \frac{\vint{w}^2 \vint{z}^2}{\vint{z}^{3+\gamma+2}}  
   \dw\dz \dsigma \dt
\leq
  C \norm{F}_{L^1_v} \norm{h}_{L^1_v}. 
\end{align*}
Combining the estimates for $Q_c^{(1)}, Q_c^{(2)}, Q_c^{(3)}$ and $Q_{\bar c}$ gives the bound in~\eqref{bound:L-infty-Q-F-mu}. 
\end{proof}

\begin{rmk} \label{rmk:weight-L-infty}
It is clear from the proof of Proposition~\ref{prop:Q-F-mu} that we can replace $\mu$ by $\mu \vint{v}^\ell$ for any $\ell$ and obtain that
\begin{align} \label{bound:L-infty-Q-F-mu-weight}
   \norm{Q(F, \mu \vint{v}^\ell)}_{L^\infty(\R^3_v)}
\leq
  C_\ell \norm{F}_{L^1_{\gamma+2s}(\R^3_v)}.
\end{align}
\end{rmk}

Finally we state some elementary interpolations and the specific form of the Gronwall's inequality used frequently in later sections. 

\begin{lem} \label{lem:basic-interpolation}
For  any $\alpha > 0$ and $k \in \R$, we have
\begin{align*}
   L^\infty_{k} (\R^3) \hookrightarrow L^1_{k-3-\alpha} (\R^3), 
\qquad
   L^2_{k} (\R^3) \hookrightarrow L^1_{k-3/2-\alpha} (\R^3),
\qquad
  L^\infty_{k} (\R^3) \hookrightarrow L^2_{k-3/2-\alpha} (\R^3).
\end{align*}
\end{lem}

\begin{lem} \label{lem:Gronwall}
Let $C_1, C_2$ be two positive constants. Suppose $u(t) \geq 0$ satisfies
\begin{align*}
    \frac{\rm d}{\dt} u^2(t) \leq C_1 u(t) + C_2 u^2(t), 
\qquad
   u |_{t= 0} = u_0. 
\end{align*}
Then
\begin{align*}
   u^2(t)
\leq 
  e^{(1+ C_2) t} \vpran{u^2_0 + C_1^2 \, t}.
\end{align*}
Note that the coefficient in the second term $C_1^2$ is independent of $C_2$.
\end{lem}
\begin{proof}
First by the Cauchy-Schwarz inequality, we have
\begin{align*}
   \frac{\rm d}{\dt} u^2(t) 
\leq 
  C_1^2 + \vpran{1+ C_2} u^2(t). 
\end{align*}
Then by the usual Gronwall's inequality, 
\begin{gather*}
   u^2(t)
\leq
  e^{(1+C_2) t} \vpran{u_0^2 + C_1^2 \frac{1 - e^{-(1 + C_2) t}}{1+C_2}}
\leq
  e^{(1+C_2) t} \vpran{u_0^2 + C_1^2 \, t}.  \qedhere
\end{gather*}
\end{proof}

%
%


\section{Linear Local Theory: A priori Estimates} \label{Sec:a-priori-linear}
We start with the theory for the linear equation
\begin{equation}\label{BEe1-1}
  \partial_{t}f + v\cdot\nabla_{x}f 
= Q(G,F) + \Eps\,L_{\alpha}F 
\,\Denote \,
   \tilde{Q}(G,F),
\qquad 
   (t,x,v)\in (0,T) \times \T^{3}\times\R^{3},
\end{equation}
where $T > 0$ is fixed and $G$ is a fixed nonnegative function and we write
\begin{equation*}
    G(t,x,v) = \mu(v) + g(t,x,v) \geq  0. 
\end{equation*}
The operator $L_{\alpha}$ is a regularising linear operator defined by
\begin{align} \label{def:L-alpha}
  L_{\alpha}\psi(v) 
= - \vpran{\vint{v}^{2\alpha} \psi - \nabla_{v} \cdot \vpran{\langle v \rangle^{2\alpha} \nabla_{v}\psi}},
\qquad 
   \alpha \geq 0,
\end{align}
where $\alpha > 0$ will be specified and fixed in the sequel.

The goal of this section is to establish a priori estimates in various $L^2$-based spaces. Hence we suppose $F(t,x,v)$ is a sufficiently smooth nonnegative solution to~\eqref{BEe1-1} and let $f(t,x,v)$ be its perturbation around the global Maxwellian, 
\textit{i.e.},
\begin{equation*}
F(t,x,v) = \mu(v) + f(t,x,v) \geq 0.
\end{equation*}
Then for any $\Eps\in (0,1]$, the pair $(F,f)$ 
satisfies the equation
\begin{equation}\label{BEe1}
  \partial_{t}f + v\cdot\nabla_{x}f 
= \Eps \,L_{\alpha}F + Q(G,F)
= \Eps \,L_{\alpha} (\mu + f) + Q(G, \mu + f),
\qquad (t,x,v)\in (0,T)\times\T^{3}\times\R^{3}.
\end{equation}

\subsection{Local in time $L^{2}$-estimates}
First we derive a uniform-in-$\Eps$ $L^2$-estimate for equation~\eqref{BEe1}.  

\begin{prop}[Bilinear uniform-in-$\Eps$ estimate]\label{bilinear-zero-level}
Suppose $G = \mu + g \geq 0$ satisfies that
\begin{align} \label{cond:coercivity-1}
    \inf_{t, x} \norm{G}_{L^1_{v}} \geq D_0 > 0, 
\qquad
   \sup_{t, x} \vpran{\norm{G}_{L^1_2} + \norm{G}_{L\log L}} 
   < E_0 < \infty.
\end{align}
Suppose $s \in (0, 1)$ and $\ell > 8 + \gamma$. Let $F = \mu + f$ be a solution to equation~\eqref{BEe1}. Then
\begin{align} \label{ineq:energy-basic-1}
   \frac{\rm d}{\dt} \norm{\vint{v}^{\ell}  f}_{L^2_{x,v}}^2
&\leq 
    -\vpran{\frac{\gamma_0}{2} - C_{\ell} \sup_{x} \norm{g}_{L^1_{\gamma}}}
  \norm{\vint{v}^{\ell+\gamma/2} f}_{L^2_{x,v}}^2
   + C_{\ell} \vpran{1 + \sup_{x} \norm{g}_{L^1_{\ell+\gamma}}^{b_0}} \norm{\vint{v}^{\ell} f}^2_{L^2_{x, v}}  \nn
\\
& \quad \,
  - \frac{c_0 \delta_2}{4} \int_{\T^3} \norm{\vint{v}^{\ell} f}^2_{H^s_{\gamma/2}} \dx
  - \frac{\Eps}{2} \norm{\vint{v}^{\ell+\alpha} f}_{L^2_x H^1_v}^2   \nn
\\
& \quad \,
   + C_{\ell} \vpran{\Eps + \sup_x \norm{g}_{L^1_{\ell+\gamma+2s} \cap L^2}}
   \norm{\vint{v}^\ell f}_{L^2_{x,v}},
\end{align}
where $\delta_2$ is a small enough constant satisfying~\eqref{cond:delta-2}, $\gamma_0, c_0$ are the positive constants in Lemma~\ref{lem:decomp-Q-ell} and Proposition~\ref{prop:coercivity-1} respectively, and $b_0$ is a constant that only depends on $s, \gamma$. All the coefficients $c_0, \gamma_0, \delta_2, b_0, C_\ell$ are independent of ~$\Eps$. 
Furthermore, for any $0 \leq T_1 < T_2 < T$ and $0 < s' < \frac{s}{2(s+3)}$, we have the regularisation 
\begin{align} \label{bound:velocity-avg-basic-1}
& \quad \, 
     \int^{T_2}_{T_{1}}
     \norm{(1 - \Delta_{t})^{s'/2}  f}^{2}_{L^{2}_{x,v}} \dtau
  + \int^{T_2}_{T_{1}}
     \norm{(1 - \Delta_{x})^{s'/2}  f}^{2}_{L^{2}_{x,v}} \dtau  \nn
\\
&\leq 
  C\int^{T_2}_{T_1} 
   \vpran{\Eps^2 \norm{\vint{v}^{3+2\alpha} f}^{2}_{L^{2}_{x,v}}
   + \norm{(1 -\Delta_{v})^{s/2} f}^{2}_{L^{2}_{x,v}}} \dt  \nn
\\
& \quad \, 
  + C \vpran{1 + \sup_{t,x} \norm{g}_{L^1_{3+\gamma+2s} \cap L^2}^2}\int^{T_2}_{T_1} \norm{\vint{v}^{3+\gamma+2s} f}_{L^2_{x, v}}^2 \dt \nn
\\
& \quad \, 
  +  C \norm{\vint{v}^3 f(T_1)}^{2}_{L^{2}_{x,v}} 
     + C \norm{\vint{v}^3 f(T_2)}^{2}_{L^{2}_{x,v}} 
  + C \vpran{\Eps^2 + \sup_{t, x} \norm{g}_{L^1_{3+\gamma+2s} \cap L^2}^2} (T_2 - T_1),
\end{align}
where the coefficient $C$ is independent of $\Eps$.
\end{prop}
\begin{proof} 
Multiply~\eqref{BEe1} by $\vint{v}^{2\ell} f$ and integrate in $x, v$. The regularising term is bounded as
\begin{align} \label{est:L-alpha-1}
& \quad \,
  \Eps \iint_{\T^3 \times \R^3} L_\alpha (\mu + f) f \vint{v}^{2\ell} \dv\dx \nn
\\
&=  -\Eps \iint_{\T^3 \times \R^3} \vint{v}^{2\ell}\vpran{\vint{v}^{2\alpha} f - \nabla_v \cdot (\vint{v}^{2\alpha} \nabla_v f)} f  \dv\dx
    + \Eps \iint_{\T^3 \times \R^3} L_\alpha(\mu) f \vint{v}^{2\ell} \dv\dx  \nn
\\
&\leq 
  - \frac{\Eps}{2} \norm{f}_{L^2_{\ell+\alpha}(\T^3 \times \R^3)}^2
   - \Eps \iint_{\T^3 \times \R^3}
       \abs{\vint{v}^{\alpha+\ell} \nabla_v f}^2 \dx\dv
   + C_{\ell} \Eps \norm{\vint{v}^{\ell} f}_{L^2_{x, v}}^2
   + C_{\ell} \Eps \norm{f}_{L^1_{x, v}}  \nn
\\
&\leq 
  - \frac{\Eps}{2} \norm{f}_{L^2_{\ell+\alpha}(\T^3 \times \R^3)}^2
   - \Eps \iint_{\T^3 \times \R^3}
       \abs{\vint{v}^{\alpha+\ell} \nabla_v f}^2 \dx\dv
   + C_{\ell} \Eps \norm{\vint{v}^{\ell} f}_{L^2_{x, v}}^2
   + C_{\ell} \Eps \norm{\vint{v}^{\ell} f}_{L^2_{x, v}}.   
\end{align}
Decompose the integration of the collision term as
\begin{align} \label{def:T-0}
& \quad \, 
  \iint_{\T^3 \times \R^3} 
     Q(G, \mu + f) f \vint{v}^{2\ell} \dv\dx  \nn
\\
& = \iint_{\T^3 \times \R^3} 
     Q(G, f) f \vint{v}^{2\ell} \dv\dx
      +  \iint_{\T^3 \times \R^3} 
     Q(G, \mu) f \vint{v}^{2\ell} \dv\dx  
\Denote
   T_0 + \tilde T_0,
\end{align}
where by the trilinear estimate in Proposition~\ref{prop:trilinear}, we have
\begin{align} \label{est:tilde-T-0}
   \tilde T_0
\leq
   C_{\ell} \vpran{\sup_x \norm{g}_{L^1_{\ell + \gamma+2s} \cap L^2}}
   \norm{\vint{v}^\ell f}_{L^2_{x,v}}. 
\end{align}
To bound $T_0$,  we use~\eqref{ineq:decomp-Q-ell-2} in Lemma~\ref{lem:decomp-Q-ell} and \eqref{ineq:trilinear-1} in Proposition~\ref{prop:commutator} and get
\begin{align} 
  T_0
& \leq
   -\vpran{\gamma_0 - C_{\ell} \sup_x \norm{g}_{L^1_\gamma}}
     \norm{\vint{v}^{\ell + \gamma/2} f}_{L^2_{x, v}}^2  \nn
\\
& \quad \,
   + \IntRRS 
       b(\cos\theta) |v - v_\ast|^\gamma 
       G_\ast \frac{|f|\vint{v}^\ell}{\vint{v}^\ell} |f'| \vint{v'}^{\ell}
       \vpran{\vint{v'}^{\ell} - \vint{v}^{\ell} \cos^{\ell} \tfrac{\theta}{2}} 
       \dbmu \label{est:T-0-1-1}
\\
& \leq 
  \ell \iiiint_{\T^3 \times \R^6 \times \Ss^2}
    G_\ast \vpran{|f|\vint{v}^{\ell-2}  - |f'|\vint{v'}^{\ell-2}} |f'|\vint{v'}^{\ell}
    \vpran{v_\ast \cdot \tilde \omega} \cos^{\ell}\tfrac{\theta}{2} \sin\tfrac{\theta}{2}
        b(\cos\theta) |v - v_\ast|^{1+\gamma} \dbmu \nn
\\
& \quad \,
   +  C_\ell \vpran{1 + \sup_x \norm{g}_{L^1_{\ell+\gamma}}}
   \norm{\vint{v}^\ell f}_{L^2_{x,v}}^2 
    -\vpran{\gamma_0 - C_{\ell} \sup_x \norm{g}_{L^1_\gamma}}
     \norm{\vint{v}^{\ell + \gamma/2} f}_{L^2_{x, v}}^2. \label{est:T-0-1}
\end{align}
Here we treat the mild and strong singularities separately. If $s \in (0, 1/2)$, then we apply part (b) of Proposition~\ref{prop:strong-sing-cancellation} and bound the first term on the right-hand side of~\eqref{est:T-0-1} as
\begin{align*}
& \quad \, 
  \ell \abs{\iiiint_{\T^3 \times \R^6 \times \Ss^2}
    G_\ast \vpran{|f|\vint{v}^{\ell-2}  - |f'|\vint{v'}^{\ell-2}} |f'|\vint{v'}^{\ell}
    \vpran{v_\ast \cdot \tilde \omega} \cos^{\ell}\tfrac{\theta}{2} \sin\tfrac{\theta}{2}
        b(\cos\theta) |v - v_\ast|^{1+\gamma} \dbmu}
\\
& \leq 
   C_\ell \vpran{1 + \sup_x \norm{g}_{L^1_{2+\gamma}}}
  \norm{\vint{v}^\ell f}_{L^2_{x,v}}^2. 
\end{align*}
If $s \in [1/2, 1)$, then by part (a) of Proposition~\ref{prop:strong-sing-cancellation} and interpolation, we have
\begin{align*}
& \quad \, 
  \ell \abs{\iiiint_{\T^3 \times \R^6 \times \Ss^2}
    G_\ast \vpran{|f|\vint{v}^{\ell-2}  - |f'|\vint{v'}^{\ell-2}} |f'|\vint{v'}^{\ell}
    \vpran{v_\ast \cdot \tilde \omega} \cos^{\ell}\tfrac{\theta}{2} \sin\tfrac{\theta}{2}
        b(\cos\theta) |v - v_\ast|^{1+\gamma} \dbmu}
\\
& \leq
  C_\ell \vpran{1 + \sup_x\norm{g}_{L^1_{3+\gamma+2s} \cap L^2}}
  \norm{\vint{v}^\ell f}_{L^2_x H^{s_1}_{\gamma_1/2}}
  \norm{\vint{v}^\ell f}_{L^2_x L^2_{\gamma/2}}
\\
& \leq
   \delta_1 \norm{\vint{v}^\ell f}_{L^2_x H^s_{\gamma/2}}^2
   + C_{\delta_1} 
       \vpran{1 + \sup_x\norm{g}_{L^1_{3+\gamma+2s} \cap L^2}^{b_0}} \norm{\vint{v}^\ell f}_{L^2_{x, v}}^2, 
\end{align*}
where $\delta_1$ can be chosen arbitrarily small and $b_0 = b_0(s, \gamma, s_1, \gamma_1)$. In fact, one can show that 
\begin{align*}
    b_0 = \frac{2s_1}{s-s_1} + \frac{2s}{s-s_1} \frac{\gamma_1}{\gamma - \gamma_1} + 2 > 0. 
\end{align*}
Since the particular form of $b_0$ is not needed we omit its derivation. Moreover, since the choices of $s_1, \gamma_1$ only depend on $s, \gamma$, we can view $b_0$ as a constant that only depends on $s, \gamma$.
We can now combine both cases and get that for any $s \in (0, 1)$, 
\begin{align*}
& \quad \, 
  \ell \abs{\iiiint_{\T^3 \times \R^6 \times \Ss^2}
    G_\ast \vpran{|f|\vint{v}^{\ell-2}  - |f'|\vint{v'}^{\ell-2}} |f'|\vint{v'}^{\ell}
    \vpran{v_\ast \cdot \tilde \omega} \cos^{\ell}\tfrac{\theta}{2} \sin\tfrac{\theta}{2}
        b(\cos\theta) |v - v_\ast|^{1+\gamma} \dbmu}
\\
& \leq
   \delta_1 \norm{\vint{v}^\ell f}_{L^2_x H^s_{\gamma/2}}^2
   + C_{\delta_1} 
       \vpran{1 + \sup_x\norm{g}_{L^1_{3+\gamma+2s} \cap L^2}^{b_0}} \norm{\vint{v}^\ell f}_{L^2_{x, v}}^2.
\end{align*}
Substituting this bound into~\eqref{est:T-0-1} gives
\begin{align} \label{est:T-0}
   T_0
&\leq
  -\vpran{\gamma_0 - C_{\ell} \sup_x \norm{g}_{L^1_\gamma}}
     \norm{\vint{v}^{\ell + \gamma/2} f}_{L^2_{x, v}}^2  \nn
\\
& \quad \, 
  + C_{\ell, \delta_1} \vpran{1 + \sup_x \norm{g}_{L^1_{\ell+\gamma}}^{b_0}}
   \norm{\vint{v}^\ell f}_{L^2_{x,v}}^2 
  + \delta_1 \norm{\vint{v}^\ell f}_{L^2_x H^s_{\gamma/2}}^2, 
\end{align}
where $\delta_1$ can be taken arbitrarily small. Combining~\eqref{est:L-alpha-1}, \eqref{est:tilde-T-0} and \eqref{est:T-0} gives the energy estimate as
\begin{align} \label{est:basic-L-2-bound-1}
   \frac{\rm d}{\dt} \norm{\vint{v}^{\ell} f}_{L^2_{x,v}}^2
&\leq 
    -\vpran{\gamma_0 - C_{\ell} \sup_{x} \norm{g}_{L^1_{\gamma}}}
  \norm{\vint{v}^{\ell+\gamma/2} f}_{L^2_{x,v}}^2
   + C_{\ell, \delta_1} \vpran{1 + \sup_{x} \norm{g}_{L^1_{\ell+\gamma}}^{b_0}} \norm{\vint{v}^{\ell} f}^2_{L^2_{x, v}}  \nn
\\
& \hspace{-5mm}
   + C_{\ell} \vpran{\Eps + \sup_x \norm{g}_{L^1_{\ell + \gamma+2s} \cap L^2}}
   \norm{\vint{v}^\ell f}_{L^2_{x,v}}  \!\!
  - \frac{\Eps}{2} \norm{\vint{v}^{\ell+\alpha} f}_{L^2_x H^1_v}^2
  + \delta_1 \norm{\vint{v}^\ell f}_{L^2_x H^s_{\gamma/2}}^2, 
\end{align}
where all the constants are independent of $\Eps$.

To complete the basic energy estimate, we include the $H^s$-regularisation. To this end, we only need to perform the second kind of estimate for $T_0$, as done in the proof of Proposition 3.2 in \cite{AMSY}. By Proposition~\ref{prop:coercivity-1}, equation~\eqref{eqn:decomp-Q-ell} in Lemma~\ref{lem:decomp-Q-ell} and the same estimates in~\eqref{est:T-0-1-1}-\eqref{est:T-0} for $T_0$, we have
\begin{align} \label{est:T-0-2}
   T_0
& = \int_{\T^3}\int_{\R^3} Q(G, \, \vint{v}^{\ell} f) \, \vint{v}^{\ell} f \dv\dx \nn
\\
& \quad \,
  + \IntTRRS b(\cos\theta) |v - v_\ast|^\gamma
        G_\ast f \, f' \vint{v'}^{\ell}  \vpran{\vint{v'}^{\ell} - \vint{v}^{\ell}\cos^{\ell} \tfrac{\theta}{2}} \dbmu  \nn
\\
& \quad \, 
  + \IntTRRS 
       b(\cos\theta) |v - v_\ast|^\gamma 
       G_\ast f f' \vint{v'}^{\ell}
       \vint{v}^{\ell} \vpran{\cos^{\ell}  \tfrac{\theta}{2} - 1} 
       \dbmu \nn
\\
& \leq
  -c_0 \norm{\vint{v}^{\ell} f}_{L^2_x H^s_{\gamma/2}}^2
    + \frac{c_0}{2} \norm{\vint{v}^\ell f}_{L^2_x H^s_{\gamma/2}}^2 
  + C_{\ell} \vpran{1 + \sup_{x} \norm{g}_{L^1_{\ell+\gamma}}^{b_0}} \norm{\vint{v}^{\ell} f}^2_{L^2_{x, v}}   
\\
& \quad \,
   + C_{\ell} \vpran{1 + \sup_x \norm{g}_{L^1_\gamma}} \norm{\vint{v}^{\ell + \gamma/2} f}_{L^2_{x,v}}^2  \nn
\\
& \leq
  -\frac{c_0}{2} \norm{\vint{v}^{\ell} f}_{L^2_x H^s_{\gamma/2}}^2
  + C_{\ell} \vpran{1 + \sup_{x} \norm{g}_{L^1_{\ell+\gamma}}^{b_0}} \norm{\vint{v}^{\ell} f}^2_{L^2_{x, v}}  
   + C_{\ell} \vpran{1 + \sup_x \norm{g}_{L^1_\gamma}} \norm{\vint{v}^{\ell + \gamma/2} f}_{L^2_{x,v}}^2.   \nn
\end{align}
Choose $\delta_1, \delta_2$ small enough such that
\begin{align} \label{cond:delta-2}
   \delta_2 C_{\ell} \vpran{1 + \sup_x \norm{g}_{L^1_\gamma}}
\leq 
   \frac{\gamma_0}{2}, 
\qquad
  \delta_2 < 1, 
\qquad
   \delta_1 < \frac{c_0 \delta_2}{4}. 
\end{align}
Multiplying~\eqref{est:T-0-2} by $\delta_2$ and adding it to~\eqref{est:basic-L-2-bound-1} gives~\eqref{ineq:energy-basic-1}.

\smallskip

Finally we apply the averaging lemma in Proposition~\ref{average-lemma-p} to obtain the regularisation in $x$. In light of equation \eqref{BEe1}, if we invoke Proposition \ref{average-lemma-p} with 
\begin{align*}
    \beta=s, 
\quad m=2, 
\quad r=0,
\quad p=2, 
\quad s^\flat = \frac{s_-}{2(s+3)} =: s',
\end{align*} 
then for any $0\leq T_{1} \leq T_{2}<T$, 
\begin{align*}
& \quad \,
   \int^{T_{2}}_{T_{1}} \norm{(1 - \Delta_{t})^{s'/2}f}^{2}_{L^{2}_{x,v}} \dtau   
   + \int^{T_{2}}_{T_{1}} \norm{(1 - \Delta_{x})^{s'/2}f}^{2}_{L^{2}_{x,v}} \dtau 
\\
&\leq 
     C \norm{\vint{v}^3 f(T_1)}^{2}_{L^{2}_{x,v}} 
     + C \norm{\vint{v}^3 f(T_2)}^{2}_{L^{2}_{x,v}} 
     + C \int^{T_2}_{T_{1}} 
          \norm{(1 -\Delta_{v})^{s/2} f}^{2}_{L^{2}_{x,v}} \dtau
\\
& \quad \,
   + C \int_{T_1}^{T_2} \norm{\vint{v}^{3} (1 -\Delta_{v})^{-1}\tilde{Q}(G,F)}^{2}_{L^{2}_{x,v}} \dtau.
\end{align*}
By the trilinear estimate in Proposition~\ref{prop:trilinear}, 
it follows that
\begin{align*}
  \norm{\langle v \rangle^{3} (1 - \Delta_{v})^{-1}\tilde{Q}(G,F)}_{L^{2}_{v}}
& \leq 
  \norm{\vint{v}^{3} (1 -\Delta_{v})^{-s} \vpran{Q(G,f) + Q(g,\mu)}}_{L^{2}_{v}} 
  + \Eps \norm{\vint{v}^{3} (1 -\Delta_{v})^{-1}L_{\alpha} F}_{L^{2}_{v}} 
\\
&\leq 
   C \vpran{1 + \norm{g}_{L^{1}_{3+\gamma+2s} \cap L^2}} 
   \norm{f}_{L^{2}_{3+\gamma+2s}}
   + \norm{g}_{L^{1}_{3+\gamma+2s} \cap L^2}
   + \Eps \, C\,\| f \|_{L^{2}_{3+2\alpha}}
   + C \Eps.
\end{align*}
In this way, we are led to
the desired inequality showing the spatial regularisation of $f$.
\end{proof}

Applying the Gronwall's inequality to Proposition~\ref{bilinear-zero-level}, we obtain the following bound:

\begin{cor}\label{cor:bi-cor-basic-energy-estimate-linear}
Suppose $G = \mu + g \geq 0$ satisfies that
\begin{align*} 
    \inf_{t, x} \norm{G}_{L^1_{v}} \geq D_0 > 0, 
\qquad
   \sup_{t, x} \vpran{\norm{G}_{L^1_2} + \norm{G}_{L\log L}} 
   < E_0 < \infty.
\end{align*}
Let $F = \mu + f$ be a solution to equation~\eqref{BEe1} with $s \in (0, 1)$. Assume that the following conditions hold:
\begin{align} \label{bound:addi-g}
   \sup_{t,x} \norm{g}_{L^\infty_{k_0}(\R^3)}  < \infty,
\qquad
   \sup_{t,x} \norm{g}_{L^1_{\gamma}(\R^3)} < \delta_0 < \frac{\gamma_0}{4 C_{\ell}}, 
\qquad
   8 + \gamma < \ell \leq k_0 - 5 - \gamma.
\end{align}
Let 
\begin{align*}
   \Sigma(g) = 1 + \sup_{t,x}\| g\|_{L^{\infty}_{k_0}}^{b_0}.
\end{align*} 
where $b_0$ is the same exponent as in Proposition~\ref{bilinear-zero-level}. Then it holds that
\begin{align}\label{bi-cor-bas-e1}
   \norm{\vint{\cdot}^{\ell} f(t)}^{2}_{L^{2}_{x,v}} 
\leq 
    C_\ell e^{C_\ell\,\Sigma(g)\,t}
    \vpran{\norm{\vint{\cdot}^{\ell} f_0}^{2}_{L^{2}_{x,v}}  + \sup_{t,x}\| g\|^{2}_{L^{\infty}_{k_0}}\,t + \Eps^{2} t},\qquad t\in[0,T),
\end{align}
and
\begin{align*}
& \quad \,
   c_0\delta_2 \vpran{\int^{t}_{0}\big\| \langle v \rangle^{\ell+\gamma/2} (1 -\Delta_{v})^{s/2}f  \big\|^{2}_{L^{2}_{x,v}}{\rm d}\tau }  \nn
   + \frac{\Eps}{4}\,\int^{t}_{0}\big\| \langle v \rangle^{\ell+\alpha} (1 -\Delta_{v})^{1/2}f  \big\|^{2}_{L^{2}_{x,v}}{\rm d}\tau 
\\
& \leq 
   C_\ell e^{C_\ell\,\Sigma(g)\,t}
    \vpran{\norm{\vint{\cdot}^{\ell} f_0}^{2}_{L^{2}_{x,v}}  + \sup_{t,x}\| g\|^{2}_{L^{\infty}_{k_0}}\,t + \Eps^{2} t},\qquad t\in[0,T).
\end{align*}
Furthermore, if in addition,  
\begin{align} \label{cond:alpha-1}
   \ell \geq 3 + 2 \alpha, 
\end{align}
then for any $0 < s' < \frac{s}{2(s+3)}$ it holds that
\begin{align}\label{bi-cor-bas-e2}
& \quad \, 
  \int^{t}_{0} \norm{(1 - \Delta_{t})^{s'/2}f}^{2}_{L^{2}_{x,v}}{\rm d}\tau
+ \int^{t}_{0} \norm{(1 - \Delta_{x})^{s'/2}f}^{2}_{L^{2}_{x,v}}{\rm d}\tau   \nn
\\
&\leq 
   C e^{C\,\Sigma(g)\,t} \Big( \| \vint{\cdot}^{9} f_0 \|^{2}_{L^{2}_{x,v}} + \sup_{t,x}\| g\|^{2}_{L^{\infty}_{k_0}}t + \Eps^{2} t\Big),
\end{align}
where the exponent $9$ is chosen such that $9 > 8 + \gamma$.
\end{cor}
\begin{proof}
Applying the additional bounds in~\eqref{bound:addi-g} to~\eqref{ineq:energy-basic-1} gives
\begin{align} \label{ineq:energy-basic-2}
     \frac{\rm d}{\dt} \norm{\vint{v}^{\ell}  f}_{L^2_{x,v}}^2
&\leq 
    -\frac{\gamma_0}{4}
  \norm{\vint{v}^{\ell+\gamma/2} f}_{L^2_{x,v}}^2
  -\frac{c_0 \delta_2}{4} \norm{\vint{v}^{\ell} f}^2_{L^{2}_{x}H^s_{\gamma/2}} - \frac{\Eps}{4} \norm{\vint{v}^{\ell+\alpha} f}_{L^2_x H^1_v}^2\nn
\\
& \quad \, 
   + C_{\ell}\vpran{1 + \sup_{x} \norm{g}_{L^\infty_{k_0}}^{b_0}} \norm{\vint{v}^{\ell} f}^2_{L^2_{x, v}} 
   + C_{\ell} \vpran{\Eps + \sup_x \norm{g}_{L^\infty_{k_0}}}
   \norm{\vint{v}^\ell f}_{L^2_{x,v}}  \nn
\\
&\leq 
    -\frac{\gamma_0}{4}
  \norm{\vint{v}^{\ell+\gamma/2} f}_{L^2_{x,v}}^2
  -\frac{c_0 \delta_2}{4} \norm{\vint{v}^{\ell} f}^2_{L^{2}_{x}H^s_{\gamma/2}} - \frac{\Eps}{4} \norm{\vint{v}^{\ell+\alpha} f}_{L^2_x H^1_v}^2\nn
\\
& \quad \, 
   + C_{\ell}\vpran{1 + \sup_{x} \norm{g}_{L^\infty_{k_0}}^{b_0}} \norm{\vint{v}^{\ell} f}^2_{L^2_{x, v}} 
   + \vpran{\Eps^2 + \sup_x \norm{g}_{L^\infty_{k_0}}^2}.
\end{align}
Estimate \eqref{bi-cor-bas-e1} follows directly from applying the Gronwall's inequality to~\eqref{ineq:energy-basic-2}.  When integrating in time \eqref{ineq:energy-basic-2} one concludes that
\begin{align*}
& \quad \,
   \norm{\vint{v}^\ell f(t)}_{L^2_{x,v}}^2 
   + \frac{c_0\delta_{2}}{2} 
   \int^{t}_{0} \norm{\vint{v}^{\ell}f}^{2}_{L^{2}_{x}H^{s}_{\gamma/2}}\dtau
  + \frac{\Eps}{4} \int^{t}_0\norm{\vint{v}^{\ell+\alpha} f}_{L^2_x H^1_v}^2\dtau
\\
&\leq 
   C_{\ell} 
   \Sigma(g) \int^{t}_{0} \norm{\vint{v}^{\ell}f}^{2}_{L^{2}_{x,v}}\dtau 
  + C_{\ell}  e^{C_\ell\,\Sigma(g)\,t}
      \vpran{\Eps^{2} + \sup_{t,x} \norm{g}^{2}_{L^{\infty}_{k_0}}}\,t 
  + \norm{\vint{v}^\ell f_0}_{L^2_{x,v}}^2
\\
&\leq 
     C_\ell e^{C_\ell\,\Sigma(g)\,t} 
     \vpran{\norm{\vint{\cdot}^{\ell} f_0}^{2}_{L^{2}_{x,v}}  
           + \sup_{t,x} \norm{g}^{2}_{L^{\infty}_{k_0}} t + \Eps^{2} t}.
\end{align*}
This bound together with estimate \eqref{bound:velocity-avg-basic-1}, with $T_{1}=0$ and $T_{2}=t$, gives \eqref{bi-cor-bas-e2} for sufficiently large $C>0$ under the condition $\ell\geq 3+2\alpha$.
\end{proof}
%
%

%
%

Our main $L^\infty$-bound will be based on various $L^2$-estimates of the level-set functions defined as follows:  for any $\ell\geq0$ and $K\geq0$ define the levels $f^{(\ell)}_{K}:=f\,\vint{v}^{\ell} - K$ and 
\begin{equation*}
    \Fl{K} = f^{(\ell)}_{K}\,1_{\{ f^{(\ell)}_{K} \geq 0 \} } , 
\qquad 
   f^{(\ell)}_{K, - } = - f^{(\ell)}_{K}\,1_{ \{ f^{(\ell)}_{K} < 0 \}}.
\end{equation*}

\subsection{$L^2$-estimates for level sets}

The focus of this subsection is to prove the following natural  \textit{a priori} estimate for the level sets.  It is a building block for the energy functional presented later in the argument.   
\begin{prop}\label{thm:L2-level-set}
Suppose $G = \mu + g \geq 0$, $F = \mu + f$ and $s \in (0, 1)$. Suppose in addition $G$ satisfies that
\begin{align*} 
    \inf_{t,x} \norm{G}_{L^1_{v}} \geq D_0 > 0, 
\qquad
    \sup_{t,x}\Big(\norm{G}_{L^1_2} + \norm{G}_{L\log L} \Big) < E_0 < \infty.
\end{align*}
Then for any $\ell > 8 + \gamma$, 

\smallskip
\noindent
(a) the (bilinear) collision term satisfies, 
\begin{align} \label{est:level-set-1}
   \int_{\T^3} \int_{\R^3} Q(G, F) \Fl{K} \vint{v}^{\ell} \dv\dx
&\leq   
   -\gamma_0 \vpran{1 - C \sup_x\norm{g}_{L^1_\gamma}} 
   \norm{\Fl{K}}^2_{L^2_x L^2_{\gamma/2}}  \nn
   -\frac{c_0 \delta_4}{4} \norm{\Fl{K}}^2_{L^2_x H^s_{\gamma/2}}
\\
& \quad \, 
   +    C_\ell \vpran{1 + \sup_x \norm{g}_{L^1_{\ell+\gamma}}
                                + \sup_x\norm{g}_{L^1_{3+\gamma+2s} \cap L^2}^{b_0}}
   \norm{\Fl{K}}^2_{L^2_{x,v}}    \nn
\\
& \quad \, 
   + C_\ell (1 + K) \vpran{1 + \sup_x\norm{g}_{L^1_{\ell+\gamma}}}
   \norm{\Fl{K}}_{L^1_x L^1_\gamma}.
\end{align}
where $\delta_4$ satisfies the bound in~\eqref{bound:Eps-2-linear}.

\smallskip
\Ni (b) The regularising term satisfies
\begin{align} \label{est:level-set-3}
\int_{\T^3} \int_{\R^3} L_{\alpha}(F) \Fl{K} \vint{v}^{\ell} \dv\dx
&\leq
   - \frac{1}{2}\norm{\vint{v}^{\alpha} \Fl{K}}^2_{L^2_{x}H^{1}_{v}} + C_\ell \norm{\Fl{K}}_{L^2_{x, v}}^2  
  + C_{\ell,\alpha}\,(1+K) \norm{\Fl{K}}_{L^1_{x, v}}.
\end{align}
\end{prop}

\begin{proof}
(a) As in the proof of Proposition~\ref{bilinear-zero-level}, we make two estimates of the $Q$-term: one with $H^s$-norm in $v$ and one without.
To derive the one without the $H^s$-norm, make the decomposition
\begin{align} \label{decomp:linear}
  \int_{\T^3} \int_{\R^3} Q(G, F) \Fl{K} \vint{v}^{\ell} \dv\dx
&= \int_{\T^3} \int_{\R^3} Q \vpran{G, f - \tfrac{K}{\vint{v}^\ell}} \Fl{K} \vint{v}^{\ell} \dv\dx  \nn
\\
& \quad \,
+ \int_{\T^3} \int_{\R^3} Q \vpran{G, \tfrac{K}{\vint{v}^\ell}} \Fl{K} \vint{v}^{\ell} \dv\dx  \nn
\\
& \quad \, 
+ \int_{\T^3} \int_{\R^3} Q \vpran{G, \mu} \Fl{K} \vint{v}^{\ell} \dv\dx
\ \Denote T_1 + T_2 + T_3 .  
\end{align}
By the definition of $Q$ and the positivity of $G$, the first term $T_1$ satisfies
\begin{align} \label{bound:T-1-levl-set}
  T_1
& = \iiiint_{\T^3 \times \R^6 \times \Ss^2}
        G_\ast \vpran{f - \tfrac{K}{\vint{v}^\ell}} \vpran{\Fl{K}(v') \vint{v'}^{\ell} - \Fl{K} \vint{v}^{\ell}} 
        b(\cos\theta) |v - v_\ast|^\gamma \dbmu   \nn
\\
& \leq 
  \iiiint_{\T^3 \times \R^6 \times \Ss^2}
        G_\ast \Fl{K} \frac{1}{\vint{v}^\ell} \vpran{\Fl{K}(v') \vint{v'}^{\ell} - \Fl{K} \vint{v}^{\ell}} 
        b(\cos\theta) |v - v_\ast|^\gamma \dbmu.
\end{align}
Continuing from here, the part of the integrand involving $\Fl{K}$ satisfies
\begin{align*}
& \quad \,
   \Fl{K} \frac{1}{\vint{v}^\ell} \vpran{\Fl{K}(v') \vint{v'}^{\ell} - \Fl{K} \vint{v}^{\ell}} 
\\
&=  \frac{1}{\vint{v}^\ell} 
   \vpran{\Fl{K}(v) \Fl{K}(v') \vint{v}^\ell \cos^\ell \tfrac{\theta}{2} 
   - \vpran{\Fl{K}}^2 \vint{v}^\ell} 
   + \frac{1}{\vint{v}^\ell} \Fl{K}(v) \Fl{K}(v') \vpran{\vint{v'}^{\ell} - \vint{v}^{\ell} \cos^\ell \tfrac{\theta}{2}} 
\\
& \leq 
   \frac{1}{2} \vpran{\vpran{\Fl{K}(v')}^2 \cos^{2\ell} \tfrac{\theta}{2}
     - \vpran{\Fl{K}}^2}
   + \frac{1}{\vint{v}^\ell} \Fl{K}(v) \Fl{K}(v') \vpran{\vint{v'}^{\ell} - \vint{v}^{\ell} \cos^\ell \tfrac{\theta}{2}}  .
\end{align*}
By the regular change of variables, ~\eqref{ineq:trilinear-1} in Proposition~\ref{prop:commutator} and Proposition~\ref{prop:strong-sing-cancellation}, we bound $T_1$ as 
\begin{align} \label{bound:T-1-first-1}
   T_1
&\leq
 \frac{1}{2} \iiiint_{\T^3 \times \R^6 \times \Ss^2}
        G_\ast  \vpran{\vpran{\Fl{K}(v')}^2 \cos^{2\ell} \tfrac{\theta}{2}
     - \vpran{\Fl{K}}^2}
        b(\cos\theta) |v - v_\ast|^\gamma \dbmu   \nn
\\
& \quad \,
  + \iiiint_{\T^3 \times \R^6 \times \Ss^2}
        G_\ast \frac{\Fl{K}(v)}{\vint{v}^\ell}  \Fl{K}(v') \vpran{\vint{v'}^{\ell} - \vint{v}^{\ell} \cos^\ell \tfrac{\theta}{2}} 
        b(\cos\theta) |v - v_\ast|^\gamma \dbmu   \nn
\\
&\leq
  -\gamma_0 \vpran{1 - C\sup_x\norm{g}_{L^1_\gamma}}
   \norm{\Fl{K}}^2_{L^2_x L^2_{\gamma/2}}  
  + C_\ell \vpran{1 + \sup_x \norm{g}_{L^1_{\ell+\gamma}}}
     \norm{\Fl{K}}^2_{L^2_{x, v}}  \nn
\\
& \quad \,
   + C\vpran{1 + \sup_x\norm{g}_{L^1_{3+\gamma+2s} \cap L^2}}
  \norm{\Fl{K}}_{L^2_x H^{s_1}_{\gamma_1/2}}
  \norm{\Fl{K}}_{L^2_x L^2_{\gamma/2}}  \nn
\\
& \leq
  -\gamma_0 \vpran{1 - C\sup_x\norm{g}_{L^1_\gamma}}
   \norm{\Fl{K}}^2_{L^2_x L^2_{\gamma/2}}  
   + \delta_3 \norm{\Fl{K}}_{L^2_x H^s_{\gamma/2}}^2 \nn
\\
& \quad \,     
  + C_{\ell, \delta_3} \vpran{1 + \sup_x \norm{g}_{L^1_{\ell+\gamma}}
  + \sup_x\norm{g}_{L^1_{3+\gamma+2s} \cap L^2}^{b_0}}
     \norm{\Fl{K}}^2_{L^2_{x, v}},
\end{align}
where $b_0$ is the same exponent as in Proposition~\ref{bilinear-zero-level} and $\delta_3 > 0$ can be arbitrarily small.
In the estimate above we have combined the mild and strong singularities.
Next we estimate $T_2$. Writing out $Q$, we get
\begin{align*}
   T_2
& = K \iiiint_{\T^3 \times \R^6 \times \Ss^2}
       G_\ast \frac{1}{\vint{v}^\ell}
       \vpran{\Fl{K}(v') \vint{v'}^{\ell} - \Fl{K} \vint{v}^{\ell}} 
       b(\cos\theta) |v - v_\ast|^\gamma \dbmu
\\
& = K \iiiint_{\T^3 \times \R^6 \times \Ss^2}
       G_\ast 
       \vpran{\Fl{K}(v') - \Fl{K}(v) } 
       b(\cos\theta) |v - v_\ast|^\gamma \dbmu
\\
& \quad \, 
   + K \iiiint_{\T^3 \times \R^6 \times \Ss^2}
       G_\ast \Fl{K}(v') \frac{1}{\vint{v}^\ell}
       \vpran{\vint{v'}^\ell - \vint{v}^\ell \cos^\ell \tfrac{\theta}{2}} 
       b(\cos\theta) |v - v_\ast|^\gamma \dbmu
\\
& \quad \,
   - K \iiiint_{\T^3 \times \R^6 \times \Ss^2}
       G_\ast \Fl{K}(v') 
       \vpran{1 - \cos^\ell \tfrac{\theta}{2}} 
       b(\cos\theta) |v - v_\ast|^\gamma \dbmu.
\end{align*}
Applying the regular change of variables to the first term, then~\eqref{ineq:trilinear-1} in Proposition~\ref{prop:commutator} and part (c) in Proposition~\ref{prop:strong-sing-cancellation} with $F=1$ to the second term and a direct estimate to the third term, we get
\begin{align*}
   T_2
& \leq
      C K \vpran{1 + \sup_x\norm{g}_{L^1_\gamma}}
   \norm{\Fl{K}}_{L^1_x L^1_\gamma}
+ C_\ell K \vpran{1 + \sup_x\norm{g}_{L^1_{\ell + \gamma}}}
   \norm{\Fl{K}}_{L^1_x L^1_\gamma} 
\\
& \leq
   C_\ell K \vpran{1 + \sup_x\norm{g}_{L^1_{\ell + \gamma}}}
   \norm{\Fl{K}}_{L^1_x L^1_\gamma}. 
\end{align*}
Applying similar estimates and Remark~\ref{rmk:weight-L-infty} to $T_3$, we have 
\begin{align} \label{decomp:T-3}
   T_3
& \leq \int_{\T^3} \int_{\R^3} Q \vpran{G, \mu \vint{v}^\ell} \Fl{K} \dv\dx \nn
\\
& \quad \,
     +  \iiiint_{\T^3 \times \R^6 \times \Ss^2}
       G_\ast \Fl{K}(v') \frac{\mu \vint{v}^\ell}{\vint{v}^\ell}
       \vpran{\vint{v'}^\ell - \vint{v}^\ell \cos^\ell \tfrac{\theta}{2}} 
       b(\cos\theta) |v - v_\ast|^\gamma \dbmu \nn
\\
& \quad \,
    - \iiiint_{\T^3 \times \R^6 \times \Ss^2}
       G_\ast \Fl{K}(v') 
       \mu \vpran{1 - \cos^\ell \tfrac{\theta}{2}} 
       b(\cos\theta) |v - v_\ast|^\gamma \dbmu  \nn
\\
& \leq
   C_\ell \vpran{1 + \sup_x\norm{g}_{L^1_{\ell + \gamma}}}
   \norm{\Fl{K}}_{L^1_x L^1_\gamma}. 
\end{align}
Combining all the estimates, we obtain the first bound for the right-hand side as
\begin{align} \label{bound:RHS-1}
   \int_{\T^3} \int_{\R^3} Q(G, F) \Fl{K} \vint{v}^{\ell} \dv\dx
&\leq
   -\gamma_0 \vpran{1 - C \sup_x\norm{g}_{L^1_\gamma}}
   \norm{\Fl{K}}^2_{L^2_x L^2_{\gamma/2}}   \nn
  + \delta_3 \norm{\Fl{K}}_{L^2_x H^s_{\gamma/2}}^2
\\
& \quad \,
   +    C_{\ell, \delta_3} \vpran{1 + \sup_x \norm{g}_{L^1_{\ell+\gamma}}
           + \sup_x\norm{g}_{L^1_{3+\gamma+2s} \cap L^2}^{b_0}}
   \norm{\Fl{K}}^2_{L^2_{x,v}}     \nn
\\
& \quad \,
   + C_\ell(1 + K) \vpran{1 + \sup_x\norm{g}_{L^1_{\ell+\gamma}}}
   \norm{\Fl{K}}_{L^1_x L^1_\gamma} . 
\end{align}
Next we derive the second bound with the $H^s$-norm. To this end, we use Proposition~\ref{prop:coercivity-1}, part (a) in Lemma~\ref{lem:decomp-Q-ell}, inequality \eqref{bound:T-1-levl-set} and similar bounds in the proof of~\eqref{bound:T-1-first-1} to re-estimate $T_1$ as
\begin{align} \label{bound:RHS-2}
  T_1
& \leq
    \iiiint_{\T^3 \times \R^6 \times \Ss^2}
        G_\ast \Fl{K} \vpran{\Fl{K}(v') - \Fl{K}(v) } 
        b(\cos\theta) |v - v_\ast|^\gamma \dbmu   \nn
\\
& \quad \,
  +     \iiiint_{\T^3 \times \R^6 \times \Ss^2}
        G_\ast \Fl{K} \Fl{K}(v') \frac{1}{\vint{v}^\ell}\vpran{\vint{v'}^\ell -\vint{v}^\ell \cos^\ell \tfrac{\theta}{2}}               
        b(\cos\theta) |v - v_\ast|^\gamma \dbmu  \nn
\\
& \quad \,
  +     \iiiint_{\T^3 \times \R^6 \times \Ss^2}
        G_\ast \Fl{K} \Fl{K}(v') \vpran{1 - \cos^\ell \tfrac{\theta}{2}}               
        b(\cos\theta) |v - v_\ast|^\gamma \dbmu   \nn
\\
& \leq
   - \frac{c_0}{2} \norm{\Fl{K}}^2_{L^2_x H^s_{\gamma/2}}
   +  C_\ell \vpran{1 + \sup_x \norm{g}_{L^1_{\ell+\gamma}}
        + \sup_x\norm{g}_{L^1_{3+\gamma+2s} \cap L^2}^{b_0}}
   \norm{\Fl{K}}^2_{L^2_{x,v}}   \nn
\\
& \quad \,
   + C_\ell \vpran{1 + \sup_x \norm{g}_{L^1_{\gamma}}} \norm{\Fl{K}}^2_{L^2_x L^2_{\gamma/2}}. 
\end{align}
Multiply~\eqref{bound:RHS-2} by a small enough $\delta_4$, choose $\delta_3$ small enough and add it to~\eqref{bound:RHS-1}. This gives the desired bound in part (a). The specific requirements for $\delta_3, \delta_4$ are
\begin{align} \label{bound:Eps-2-linear}
  C_\ell \vpran{1 + \sup_x \norm{g}_{L^1_{\gamma}}} \delta_4
\leq
  \frac{1}{8} c_0,
\qquad
   \delta_3 < \frac{c_0 \delta_4}{4}.
\end{align}

\smallskip
\noindent
(b) To estimate the contribution of the $\Eps$-regularising term 
to the energy estimate of the level set, denote
\begin{align*} 
  T_R 
= \iint_{\T^3 \times \R^3} \vpran{L_{\alpha}F} \, \Fl{K} \vint{v}^\ell \dv\dx  \nn
= \iint_{\T^3 \times \R^3}
   -\vint{v}^\ell \Fl{K}\vpran{\vint{v}^{2\alpha}  - \nabla_v \cdot (\vint{v}^{2\alpha} \nabla_v )} F \dv\dx. 
\end{align*}
Decomposing $F$ gives
\begin{align} \label{decomp:L-R}
   T_R
& = \iint_{\T^3 \times \R^3}
       - \vint{v}^\ell \Fl{K}\vpran{\vint{v}^{2\alpha}  - \nabla_v \cdot (\vint{v}^{2\alpha} \nabla_v )} \mu \dv\dx  \nn
\\
& \quad \,
     + \iint_{\T^3 \times \R^3}
       - \vint{v}^\ell \Fl{K}\vpran{\vint{v}^{2\alpha}  - \nabla_v \cdot (\vint{v}^{2\alpha} \nabla_v )} \frac{K}{\vint{v}^\ell} \dv\dx  \nn
\\
& \quad \,
   + \iint_{\T^3 \times \R^3}
     - \vint{v}^\ell \Fl{K}\vpran{\vint{v}^{2\alpha}  - \nabla_v \cdot (\vint{v}^{2\alpha} \nabla_v )} \vpran{f - \frac{K}{\vint{v}^\ell}} \dv\dx  
\Denote
    T^1_R + T^2_{R} + T^3_{R}. 
\end{align}
Then $T^1_R$ is directly bounded as
\begin{align*}
   T^1_{R} \leq C_{\ell,\alpha} \, \norm{\Fl{K}}_{L^1_{x, v}}.
\end{align*}
Carrying out the computation of differentiation, we get
\begin{align*}
  T^2_{R}
\leq
   K \iint_{\T^3 \times \R^3}
     \Fl{K}\vpran{C_{\ell,\alpha}\vint{v}^{2\alpha-2}  - \vint{v}^{2\alpha}} \dv\dx
\leq
  C_{\ell,\alpha}\,K \norm{\Fl{K}}_{L^1_{x,v}},
\end{align*}
where we have applied the positivity of $\Fl{K}$ and the bound $C_{\ell,\alpha}\vint{v}^{2\alpha-2} - \vint{v}^{2\alpha} \leq C'_{\ell, \alpha} \One_{|v| \leq V_0}$ for some constant $V_0$ large enough.

To estimate $T^3_{R}$, we break it into two parts:
\begin{align*}
   T^3_{R}
&= \iint_{\T^3 \times \R^3}
   \vpran{-\vint{v}^\ell \Fl{K}\vpran{\vint{v}^{2\alpha}  - \nabla_v \cdot (\vint{v}^{2\alpha} \nabla_v )} \vpran{f - \frac{K}{\vint{v}^\ell}}} \dv\dx
\\
&= -\iint_{\T^3 \times \R^3}
   \vint{v}^\ell \Fl{K}  \vint{v}^{2\alpha} \vpran{f - \frac{K}{\vint{v}^\ell}} \dv\dx
\\
& \quad \, 
  + \iint_{\T^3 \times \R^3}
       \vint{v}^\ell \Fl{K}\vpran{\nabla_v \cdot (\vint{v}^{2\alpha} \nabla_v ) \vpran{f - \frac{K}{\vint{v}^\ell}}} \dv\dx
\Denote
   T^{3,1}_{R} + T^{3,2}_{R}.
\end{align*}
Then $T^{3,1}_{R}
= -\norm{\vint{v}^{\alpha} \Fl{K}}_{L^2_{x,v}}^2$.
Integrating by parts, we have
\begin{align*}
   T^{3,2}_{R}
&= -
    \iint_{\T^3 \times \R^3}
       \nabla_v \vpran{\vint{v}^\ell \Fl{K} } \cdot \vpran{ (\vint{v}^{2\alpha} \nabla_v ) \vpran{f - \frac{K}{\vint{v}^\ell}}} \dv\dx
\\
&= - \iint_{\T^3 \times \R^3}
      \vint{v}^{2\alpha} \nabla_v \vpran{\vint{v}^\ell \Fl{K} } \cdot \vpran{\nabla_v \frac{\Fl{K}}{\vint{v}^\ell}} \dv\dx
\\
& \leq
  - \norm{\vint{v}^{\alpha} \nabla_v \Fl{K}}_{L^2_{x, v}}^2
  + C_\ell \, \norm{\vint{v}^{\alpha - 1} \Fl{K}}_{L^2_{x, v}}^2. 
\end{align*}
Hence, 
\begin{align*}
  T^3_{R}
\leq
  - \norm{\vint{v}^{\alpha} \Fl{K}}_{L^2_{x}H^{1}_{v}}^2
  + C_\ell \, \norm{\vint{v}^{\alpha - 1} \Fl{K}}_{L^2_{x, v}}^2.
\end{align*}
Altogether we have
\begin{align*}
   T_R
\leq
   - \frac{1}{2} \norm{\vint{v}^{\alpha} \Fl{K}}^2_{L^2_{x}H^{1}_{v}} 
   + C_\ell  \,\norm{\Fl{K}}_{L^2_{x, v}}^2 
   + C_{\ell,\alpha}\,(1+K) \norm{\Fl{K}}_{L^1_{x, v}}, 
\end{align*} 
which concludes the proof of part $(b)$.
\end{proof}


To show that $|f| \vint{v}^k \leq K$ in the later part, we will need to bound not only the level set function $f^{(\ell)}_{K, +}$ but also the one for $(-f)^{(\ell)}_{K, +}$ since the former only gives $f \vint{v}^k \leq K$. Given the linearity of the Boltzmann operator $Q(G, F)$ in $F$, it is not surprising that estimate for $(-f)^{(\ell)}_{K, +}$ follows a similar line as that for $f^{(\ell)}_{K, +}$. The equation for $h = -f$ is
\begin{align} \label{eq:h-1}
   \del_t h + v \cdot \nabla_x h 
= -Q (G, \mu - h) - \Eps L_\alpha (\mu - h),
\qquad
   h|_{t=0} = -f_0(x, v). 
\end{align}

\begin{prop} \label{thm:level-set-minus-f}
Let $h = -f$. Suppose 
\begin{align*}
 & \hspace{2cm}
   G = \mu + g \geq 0, 
\qquad
   F = \mu + f = \mu - h,
\\
 &   \inf_{t,x} \norm{G}_{L^1_{v}} \geq D_0 > 0, 
\qquad
    \sup_{t,x}\Big(\norm{G}_{L^1_2} + \norm{G}_{L\log L} \Big) < E_0 < \infty.
\end{align*}
Then for any $s \in (0, 1)$ and $\ell > 8 + \gamma$, 

\noindent
(a) The bilinear collision term satisfies
\begin{align*}
   -\int_{\T^3} \int_{\R^3} Q(G, F) \Hl{K} \vint{v}^{\ell} \dv\dx
&\leq 
  -\gamma_0 \vpran{1 - C \sup_x\norm{g}_{L^1_\gamma}}
   \norm{\Hl{K}}^2_{L^2_x L^2_{\gamma/2}}  
   - \frac{c_0 \delta_4}{4} \norm{\Fl{K}}^2_{L^2_x H^s_{\gamma/2}}
\\
& \quad \,
   +    C_\ell \vpran{1 + \sup_x \norm{g}_{L^1_{\ell+\gamma}}}
   \norm{\Hl{K}}^2_{L^2_{x,v}}
\\
& \quad \,
  + C_\ell (1 + K) \vpran{1 +  \sup_x \norm{g}_{L^1_{\ell + \gamma}}}
  \norm{\Hl{K}}_{L^1_x L^1_\gamma},
\end{align*}
where $\delta_4$ is the same constant in~\eqref{bound:Eps-2-linear}.

\Ni (b) For the regularising term it holds that
\begin{align*}
  - \int_{\T^3} \int_{\R^3} L_{\alpha}(F) \Hl{K} \vint{v}^{\ell} \dv\dx
&\leq
   - \frac{1}{2} \norm{\vint{v}^{\alpha} \Hl{K}}^2_{L^2_{x}H^{1}_{v}} + C_\ell \norm{\Hl{K}}_{L^2_{x, v}}^2 
  + C_{\ell,\alpha}\,(1+K) \norm{\Hl{K}}_{L^1_{x, v}}.
\end{align*}
\end{prop}

\begin{proof}
(a) Make a similar decomposition as in~\eqref{decomp:linear}:
\begin{align} \label{decomp:linear-level-set}
  -\int_{\T^3} \int_{\R^3} Q(G, F) \Hl{K} \vint{v}^{\ell} \dv\dx
&= \int_{\T^3} \int_{\R^3} Q \vpran{G, h - \tfrac{K}{\vint{v}^\ell}} \Hl{K} \vint{v}^{\ell} \dv\dx  \nn
\\
& \quad \,
+ \int_{\T^3} \int_{\R^3} Q \vpran{G, \tfrac{K}{\vint{v}^\ell}} \Hl{K} \vint{v}^{\ell} \dv\dx  \nn
\\ 
& \quad \, 
- \int_{\T^3} \int_{\R^3} Q \vpran{G, \mu} \Hl{K} \vint{v}^{\ell} \dv\dx
\ \Denote J_1 + J_2 + J_3 .
\end{align}
Since $J_1, J_2$ have the same forms as $T_1, T_2$ in~\eqref{decomp:linear}, by taking $\delta_4$ with a bound in~\eqref{bound:Eps-2-linear}, we get
\begin{align*}
   J_1 + J_2
&\leq
  -\gamma_0 \vpran{1 - C \sup_x\norm{g}_{L^1_\gamma}}
   \norm{\Hl{K}}^2_{L^2_x L^2_{\gamma/2}}  
   - \frac{c_0 \delta_4}{4} \norm{\Fl{K}}^2_{L^2_x H^s_{\gamma/2}}
\\
& \quad \,
   +    C_\ell \vpran{1 + \sup_x \norm{g}_{L^1_{\ell+\gamma}}}
   \norm{\Hl{K}}^2_{L^2_{x,v}}
  + C_\ell K \vpran{1 + \sup_x\norm{g}_{L^1_{\ell + \gamma}}}
   \norm{\Hl{K}}_{L^1_x L^1_\gamma}.
\end{align*}
Decomposing $J_3$ similarly as $T_3$ in~\eqref{decomp:T-3} and applying Proposition~\ref{prop:Q-F-mu}, inequality~\eqref{ineq:trilinear-1} in Proposition~\ref{prop:commutator}, we have
\begin{align*} 
   J_3
& = \int_{\T^3} \int_{\R^3} Q \vpran{G, -\mu} \Hl{K} \vint{v}^{\ell} \dv\dx \nn
\\
& \leq \int_{\T^3} \int_{\R^3} Q \vpran{G, -\mu \vint{v}^\ell} \Hl{K} \dv\dx
\\
& \quad \,
     +  \iiiint_{\T^3 \times \R^6 \times \Ss^2}
       G_\ast \Hl{K}(v') \frac{(-\mu) \vint{v}^\ell}{\vint{v}^\ell}
       \vpran{\vint{v'}^\ell - \vint{v}^\ell \cos^\ell \tfrac{\theta}{2}} 
       b(\cos\theta) |v - v_\ast|^\gamma \dbmu  \nn
\\
&\quad \,
   + \iiiint_{\T^3 \times \R^6 \times \Ss^2}
       G_\ast \Hl{K}(v') \mu \vint{v}^\ell
       \vpran{1 - \cos^\ell \tfrac{\theta}{2}} 
       b(\cos\theta) |v - v_\ast|^\gamma \dbmu
\\
& \leq
  C_\ell \vpran{1 +  \sup_x \norm{g}_{L^1_{\ell + \gamma}}}
  \norm{\Hl{K}}_{L^1_x L^1_\gamma},
\end{align*}
where note that inequality~\eqref{ineq:trilinear-1} in Proposition~\ref{prop:commutator} does not require positivity of the functions in the integrand. Estimate in part (a) is a combination of the bounds for $J_1, J_2, J_3$.

\medskip
\Ni (b) Decompose the integral in part (c) in a similar way as in~\eqref{decomp:L-R}:
\begin{align*}
   - \int_{\T^3} \int_{\R^3} L_{\alpha}(F) \Hl{K} \vint{v}^{\ell} \dv\dx
& = \iint_{\T^3 \times \R^3}
   \vpran{\vint{v}^\ell \Hl{K}\vpran{\vint{v}^{2\alpha}  - \nabla_v \cdot (\vint{v}^{2\alpha} \nabla_v )} \mu} \dv\dx  \nn
\\
& \quad \,
     + \iint_{\T^3 \times \R^3}
   \vpran{- \vint{v}^\ell \Hl{K}\vpran{\vint{v}^{2\alpha}  - \nabla_v \cdot (\vint{v}^{2\alpha} \nabla_v )} \frac{K}{\vint{v}^\ell}} \dv\dx  \nn
\\
& \quad \,
  +  \iint_{\T^3 \times \R^3}
   \vpran{- \vint{v}^\ell \Hl{K}\vpran{\vint{v}^{2\alpha}  - \nabla_v \cdot (\vint{v}^{2\alpha} \nabla_v )} \vpran{h - \frac{K}{\vint{v}^\ell}}} \dv\dx.  \nn
\end{align*}
It is then clear that estimates for the three terms above are similar to those for $T^1_R$, $T^2_R$ and $T^3_R$ in Proposition~\ref{thm:L2-level-set}, since they rely on the absolute values of the terms. Hence we obtain a similar bound. 
\end{proof}

\subsection{A level sets estimate for the $L^{1}$-norm of the collisional operator}
In this part we estimate an $L^1$-norm related to $Q(G, F)$, which provides the basis for a later application of the averaging lemma. 
By subtracting $K$ from $f\,\vint{v}^{\ell}$ and multiplying equation \eqref{BEe1-1} by $\Fl{K}$, we obtain that
\begin{align}\label{e2}
\del_{t}(\Fl{K})^{2} + v\cdot\nabla_{x}(\Fl{K})^{2} = 2\,\tilde{Q}(G,F)\,\vint{v}^{\ell}\,\Fl{K},\qquad (t,x,v)\in (0,\infty)\times\T^{3}\times\R^{3}.
\end{align}
When applying the averaging lemma to the level sets in the next section, it will be important to estimate
\begin{align*}
   \int_0^T \int_{\T^3} \int_{\R^3}
    \Big|\vint{v}^{j}(1-\Delta_{v})^{-\kappa/2}\big( \tilde{Q}(G,F)\,\vint{v}^{\ell}\,\Fl{K} \big)\Big|\dv\dx\dt,
\qquad 
   j,\,\ell \geq 0, \ \ \kappa \geq 0, \ \
  T > 0. 
\end{align*}

One key observation is that the dominant part of the integrand above is its positive part.
\begin{lem}\label{L1}
Let $(F, f)$ be a pair satisfying the linearized Boltzmann equation \eqref{BEe1-1}.  Then, for any $j,\ell\geq0$, $\kappa\geq0$, $K\geq0$ and $0 \leq T_1 < T_2 < T$, it follows that
\begin{align*}
& \quad \, 
   \int_{T_1}^{T_2} \int_{\T^3} \abs{ \vint{v}^{j}(1-\Delta_{v})^{-\kappa/2} 
   \vpran{\tilde{Q}(G,F)\,\vint{v}^{\ell}\,\Fl{K}}(\cdot,\cdot,v)} \dx\dt 
\\
&\leq  
   \tfrac{1}{2}\int_{\T^{3}} \vint{v}^{j}(1-\Delta_{v})^{-\kappa/2}(\Fl{K})^{2}(T_1, \cdot, v)\dx 
\\
&\quad \,  
  + 2 \int_{T_1}^{T_2} \int_{\T^3} \Big[ \vint{v}^{j} (1-\Delta_{v})^{-\kappa/2}\big( \tilde{Q}(G,F)\,\vint{v}^{\ell}\,\Fl{K} \big)(\cdot,\cdot,v) \Big]^{+} \dx\dt,\qquad \forall\, v\in\R^{3},
\end{align*}
where $[\cdot]^+$ denotes the positive part of the term. 
\end{lem}   
\begin{proof}
First we fix $v \in \R^3$ and integrate \eqref{e2} in $(t,x)$ to obtain that
\begin{align}\label{e3}
   \int_{\T^{3}} (\Fl{K})^{2}(T_2, x, v) \dx 
= \int_{\T^{3}} (\Fl{K})^{2}(T_1, x, v) \dx 
    + 2 \int_{T_1}^{T_2} \!\! \int_{\T^3} \tilde{Q}(G,F)(t,x,v) \vint{v}^{\ell}\,\Fl{K}(t,x,v) \dx\dt, 
\end{align}
for any $v \in \R^3$. 
\smallskip
\noindent
An application of the Bessel potential in velocity to \eqref{e3} then leads us to
\begin{align}\label{e5}
  0
&\leq
   \int_{\T^{3}} \vint{v}^{j}(1-\Delta_{v})^{-\kappa/2} (\Fl{K})^{2}(T, \cdot, v)\dx   \nn
\\
&= \int_{\T^{3}}\vint{v}^{j}(1-\Delta_{v})^{-\kappa/2}(\Fl{K})^{2}(0, \cdot, v)\dx   \nn
\\
& \quad \, 
   + 2\int_{T_1}^{T_2} \int_{\T^3}\vint{v}^{j}(1-\Delta_{v})^{-\kappa/2}\big( \tilde{Q}(G,F)\,\vint{v}^{\ell}\,\Fl{K} \big)(\cdot,\cdot,v)\, \dx\dt,\qquad \forall\, v\in\R^{3}.
\end{align} 
Hence, if we denote
\begin{align*}
   \CalG=\vint{v}^{j}(1-\Delta_{v})^{-\kappa/2}\big( \tilde{Q}(G,F)\,\vint{v}^{\ell}\,\Fl{K} \big), 
\end{align*}
and $\CalG_-$ and $\CalG_+$ as its negative and positive parts respectively, then for any $v \in \R^3$, 
\begin{equation*}
   \int_{T_1}^{T_2} \int_{\T^3} \CalG_{-}(t,x,v)\,\dx\dt
\leq 
   \tfrac{1}{2}\int_{\T^{3}}\vint{v}^{j}(1-\Delta_{v})^{-\kappa/2}(\Fl{K})^{2}(0, x, v)\dx 
   + \int_{T_1}^{T_2} \int_{\T^3} \CalG_{+}(t,x,v)\,\dx\dt.
\end{equation*}
We thereby conclude that
\begin{align*}
& \quad \, 
   \int_{T_1}^{T_2} \int_{\T^3} \abs{\vint{v}^{j}(1-\Delta_{v})^{-\kappa/2} \vpran{\tilde{Q}(G,F)\,\vint{v}^{\ell}\,\Fl{K}} (\cdot,\cdot,v)} \dx\dt 
\\
&= \int_{T_1}^{T_2} \int_{\T^3} \CalG_{+}(\cdot,\cdot,v)  \dx\dt
      + \int_{T_1}^{T_2} \int_{\T^3}\CalG_{-}(\cdot,\cdot,v) \dx\dt 
\\
&\leq 
   \tfrac{1}{2}\int_{\T^{3}}\vint{v}^{j}(1-\Delta_{v})^{-\kappa/2}(\Fl{K})^{2}(T_1, \cdot, v)\dx 
   + 2\,\int_0^T \int_{\T^3} \CalG_{+}(\cdot,\cdot,v) \dx\dt 
\\
&= \tfrac{1}{2}\int_{\T^{3}}\vint{v}^{j}(1-\Delta_{v})^{-\kappa/2}(\Fl{K})^{2}(T_1, \cdot, v)\dx 
\\
&\quad  \,
  + 2\int_{T_1}^{T_2} \int_{\T^3}\Big[ \vint{v}^{j}(1-\Delta_{v})^{-\kappa/2}\big( \tilde{Q}(G,F)\,\vint{v}^{\ell}\,\Fl{K} \big)(\cdot,\cdot,v) \Big]^{+} \dx\dt,\qquad \forall\, v\in\R^{3},
\end{align*}
which proves the lemma.
\end{proof}

The counterpart for $h = -f$ states
\begin{lem}\label{L1-h}
Let $h$ be a solution to the equation
\begin{align*}
   \del_t h + v \cdot \nabla_x h 
= Q(G, -\mu + h) + \Eps L_\alpha (-\mu + h) 
= \tilde Q(G, -\mu + h). 
\end{align*}
Then, for any $0 \leq T_1 < T_2 < T$, $j,\ell\geq0$, $\kappa\geq0$, $K\geq0$, it follows that
\begin{align*}
& \quad \, 
   \int_{T_1}^{T_2} \int_{\T^3} \abs{ \vint{v}^{j}(1-\Delta_{v})^{-\kappa/2} 
   \vpran{\tilde{Q}(G, -\mu + h)\,\vint{v}^{\ell}\,\Hl{K}}(\cdot,\cdot,v)} \dx\dt 
\\
&\leq  
   \tfrac{1}{2}\int_{\T^{3}} \vint{v}^{j}(1-\Delta_{v})^{-\kappa/2}(\Hl{K})^{2}(0, \cdot, v)\dx 
\\
&\quad \,  
  + 2 \int_{T_1}^{T_2} \int_{\T^3} \Big[ \vint{v}^{j} (1-\Delta_{v})^{-\kappa/2}\big( \tilde{Q}(G, -\mu + h)\,\vint{v}^{\ell}\,\Hl{K} \big)(\cdot,\cdot,v) \Big]^{+} \dx\dt,\qquad \forall\, v\in\R^{3},
\end{align*}
where $[\cdot]^+$ denotes the positive part of the term. 
\end{lem} 

The remainder of this subsection focuses on proving the following theorem:
\begin{prop}\label{T1}
Suppose $G= \mu + g \geq 0$ and $F = \mu + f$ satisfy equation \eqref{BEe1-1}.  Then, for any 
\begin{align*}
  [T_1, T_2] \subseteq [0, T),
\quad
   s\in(0,1),
\quad
    \epsilon \in [0, 1],
\quad 
   j \geq 0,
\quad 
   \ell > 8 + \gamma,
\quad 
   \kappa > 2, 
\quad
   K > 0,
\end{align*} 
it holds that
\begin{align}\label{Qlevelaverage-2}
& \quad \,
    \int_{T_1}^{T_2} \int_{\T^3} \int_{\R^3} \abs{\vint{v}^{j}(1 - \Delta_{v})^{-\kappa/2}\big( \tilde{Q}(G,F)\,\vint{v}^{\ell}\,\Fl{K} \big)} \dv\dx\dt  \nn
\\
&\leq C\,\| \vint{v}^{j/2} \Fl{K}(T_1, \cdot, \cdot) \|^{2}_{L^{2}_{x,v}}
   + C_\ell \vpran{1 + \sup_{t,x} \norm{g}_{L^1_{\ell+\gamma}}}
     \norm{\Fl{K}}_{L^2_{t,x} L^2_j}^2 \nn
\\
& \quad \, 
   + C\vpran{1 + \sup_{t,x} \norm{g}_{L^1_{3+\gamma+2s} \cap L^2}}
  \norm{\Fl{K}}_{L^2_{t,x} H^{s}_{\gamma/2}}^2 
  + C\vpran{1 + \sup_{t,x} \norm{g}_{L^1_{j+2+\gamma}}}
  \norm{\Fl{K}}_{L^2_{t, x} L^2_{j + \gamma/2 + 1}}^2  \nn
\\
& \quad \,
  +   C (1+K) \vpran{1 + \sup_{t,x} \norm{g}_{L^1_{\ell + \gamma}}}
   \norm{\Fl{K}}_{L^1_{t, x} L^1_{j+\gamma}},
\end{align}
where $C, C_\ell$ are independent of $\Eps$ and $T_1, T_2$. 
Identical estimate holds for $\tilde Q(G, -\mu + h)$ with $\Fl{K}$ replaced by ~$\Hl{K}$.
\end{prop}

\begin{proof}
First note that for any $\kappa \geq 0$,
\begin{align*}
    \int_{\R^{3}}\int_{\T^{3}}\vint{v}^{j}(1-\Delta_{v})^{-\kappa/2}(\Fl{K})^{2}(T_1, x, v) \dx\dv 
\leq
  C  \norm{\vint{v}^{j/2} \Fl{K}(T_1, \cdot, \cdot)}^{2}_{L^{2}_{x,v}},
\end{align*}
which explains the first term in the right side of \eqref{Qlevelaverage-2}.
%
Thus, using Lemma \ref{L1} we have that for $j\geq0$, $\ell\geq0$, $\kappa\geq0$,
\begin{align*}
& \quad \, 
    \int_{T_1}^{T_2} \int_{\T^3} \int_{\R^3} \abs{\vint{v}^{j}(1-\Delta_{v})^{-\kappa/2} \vpran{\tilde{Q}(G , F) \vint{v}^{\ell} \Fl{K}}} \dv\dx\dt 
   - C\,\| \vint{v}^{j/2} \Fl{K}(T_1, \cdot, \cdot) \|^{2}_{L^{2}_{x,v}} 
\\
&\leq 
   2 \int_{T_1}^{T_2} \int_{\T^3} \int_{\R^3} \Big[\vint{v}^{j} (1-\Delta_{v})^{-\kappa/2} \vpran{\tilde{Q}(G,F)\,\vint{v}^{\ell} \Fl{K}}\Big]^{+} \dv\dx\dt
\\
& = 2 \int_{T_1}^{T_2} \int_{\T^3} \int_{\R^3} \vint{v}^{j} (1-\Delta_{v})^{-\kappa/2} \vpran{\tilde{Q}(G,F)\,\vint{v}^{\ell} \Fl{K}} \One_{A_{K}} \dv\dx\dt
\\
&= 2 \int_{T_1}^{T_2} \int_{\T^3} \int_{\R^3} \tilde{Q}(G,F)\,\vint{v}^{\ell} \Fl{K} (1-\Delta_{v})^{-\kappa/2} \vpran{\vint{v}^{j} \One_{A_{K}}} \dv\dx\dt,
\end{align*}  
where $A_K$ is the set given by 
\begin{align} \label{def:set-A-K}
   A_{K} = \big\{(t,x,v)\in (T_1, T_2) \times \T^3 \times \R^3\, |\, (1-\Delta_{v})^{-\kappa/2}\big(\tilde{Q}(G,F)\,\vint{v}^{\ell}\,\Fl{K}\big)\geq0 \big\}.  
 \end{align} 
Using that $\tilde{Q}(G,F) = Q(G,F) + \epsilon\,L_{\alpha}(F)$,
we have
\begin{align*}
& \quad \,
    \int_{T_1}^{T_2} \int_{\T^3} \int_{\R^3} \Big| \vint{v}^{j}(1 -  \Delta_{v})^{-\kappa/2}\big( \tilde{Q}(G,F)\,\vint{v}^{\ell}\,\Fl{K} \big)\Big| \dv\dx\dt - C\,\| \vint{v}^{j/2} \Fl{K}(T_1, \cdot, \cdot) \|^{2}_{L^{2}_{x,v}} 
\\
&\leq 
   2\int_{T_1}^{T_2} \int_{\T^3} \int_{\R^3} Q(G,F)\,\vint{v}^{\ell}\,\Fl{K}\, (1-\Delta_{v})^{-\kappa/2}\big(\vint{v}^{j}\One_{A_{K}}\big) \,\dv\dx\dt 
\\
&\quad \, 
   + 2\,\epsilon\int_{T_1}^{T_2} \int_{\T^3} \int_{\R^3} L_{\alpha}(F)\,\vint{v}^{\ell}\,\Fl{K}\, (1-\Delta_{v})^{-\kappa/2}\big(\vint{v}^{j}\One_{A_{K}}\big) \,\dv\dx\dt.
\end{align*}
Denote in the following
\begin{align} \label{def:W-K}
  W_{K} = (1-\Delta_{v})^{-\kappa/2}\vpran{\vint{v}^{j}\One_{A_{K}}}\geq 0. 
\end{align} 
Then, for $\kappa > 2$ it holds for all derivatives up to second order that 
\begin{align} \label{bound:W-K-deriv}
\big| W_{K}(v) \big| +  \big| \nabla_{i} W_{K}(v) \big| +  \big| \nabla^{2}_{i,k} W_{K}(v) \big| \leq C\vint{v}^{j}\,,\qquad i, \,k = 1,2,3\,,   
\end{align}
with $C$ independent of $K$.  In fact, noting that the $\kappa^{th}$ Bessel kernel $\CalB_\kappa (w)$ in dimension $d$ satisfies
\begin{align*}
0 \le \CalB_\kappa (w) = C_{d, \kappa}\left\{ \begin{array}{ll}
|w|^{\kappa-d}(1 + o(1)), &\enskip \mbox{if} \enskip 0< \kappa <d, \\
\log \frac{1}{|w|} ( 1 + o(1)), &\enskip \mbox{if} \enskip \kappa = d,\\
(1+o(1)),& \enskip \mbox{if} \enskip \kappa > d,
\end{array} \right. \enskip \mbox{as} \enskip |w| \rightarrow 0,
\end{align*}
and 
\begin{align*}
0 \le \CalB_\kappa (w) = C'_{d, \kappa} \frac{e^{-|w|}}{ |w|^{(d+1-\kappa)/2}}( 1 + o(1)), \enskip \mbox{as} \enskip |w| \rightarrow \infty,
\, \mbox{(see, \cite[(4.2), (4.3)]{AS1961})},
\end{align*}
we have
\begin{align*}
\vint{v}^{-j} |W_K(v)| &\le \vint{v}^{-j} \left(\int_{\{|w|\le 1\}} \CalB_\kappa (w)\vint{v-w}^j\dw + \int_{\{|w| \ge 1\}} 
\CalB_\kappa (w)\vint{v}^j \vint{w}^j\dw\right) \le C.
\end{align*}
Since the inequality $|\nabla_i \CalB_\kappa (w) | \le C'(\CalB_\kappa (w) + \CalB_{\kappa-1} (w))$ holds (see \cite[(4.5)]{AS1961}),
we have the estimate of the first-order derivative. The estimate of second order is also obvious because similar inequality holds
(see (4.4), (4.1) and (3.7) of \cite{AS1961}). In this way, we are led to estimate
\begin{align} \label{two-terms-estimate-1}
 & \quad \, 
    \int_{T_1}^{T_2} \int_{\T^3} \int_{\R^3} \abs{\vint{v}^{j}(1 -  \Delta_{v})^{-\kappa/2}\big( \tilde{Q}(G,F)\,\vint{v}^{\ell}\,\Fl{K} \big)} \dv\dx\dt 
    - C\,\| \vint{v}^{j/2} \Fl{K}(T_1, \cdot, \cdot) \|^{2}_{L^{2}_{x,v}}   \nn
\\
&\leq 
   2 \int_{T_1}^{T_2} \int_{\T^3} \int_{\R^3} Q(G,F)\,\vint{v}^{\ell} \Fl{K}W_{K} \dv\dx\dt 
   + 2 \epsilon\int_{T_1}^{T_2} \int_{\T^3} \int_{\R^3} L_{\alpha}(F) \vint{v}^{\ell} \Fl{K} W_{K} \,\dv\dx\dt   \nn
\\
&\Denote 
   \int_{T_1}^{T_2} \CalQ \dt + \epsilon\, \int_{T_1}^{T_2} T_R^+ \dt.
\end{align}
We will estimate the main term $\CalQ$ and the regularising linear term $T_R^+$ separately. The proofs align with those for Proposition~\ref{thm:L2-level-set}. We start with $\CalQ$ and write it as
\begin{align} \label{decomp:CalQ}
  \int_{\T^3} \int_{\R^3} Q(G, F) \vint{v}^{\ell} \Fl{K} W_K \dv\dx
&= \int_{\T^3} \int_{\R^3} Q \vpran{G, f - \tfrac{K}{\vint{v}^\ell}} \vint{v}^{\ell} \Fl{K} W_K \dv\dx  \nn
\\
& \quad \,
+ \int_{\T^3} \int_{\R^3} Q \vpran{G, \tfrac{K}{\vint{v}^\ell}} \vint{v}^{\ell} \Fl{K} W_K \dv\dx  \nn
\\
& \quad \, 
+ \int_{\T^3} \int_{\R^3} Q \vpran{G, \mu} \vint{v}^{\ell} \Fl{K} W_K \dv\dx \Denote T_1^+ + T_2^+ + T_3^+.  
\end{align}
Then by the definition of $Q$ and the positivity of $G$, the first term $T_1^+$ satisfies
\begin{align} \label{bound:T-1-plus-1}
  T_1^+
& = \iiiint_{\T^3 \times \R^6 \times \Ss^2}
        G_\ast \vpran{f - \tfrac{K}{\vint{v}^\ell}} 
        \vpran{\Fl{K}(v') W_K(v') \vint{v'}^{\ell} - \Fl{K} W_K \vint{v}^{\ell}} 
        b(\cos\theta) |v - v_\ast|^\gamma \dbmu   \nn
\\
& \leq 
  \iiiint_{\T^3 \times \R^6 \times \Ss^2}
        G_\ast \Fl{K} \frac{1}{\vint{v}^\ell} 
        \vpran{\Fl{K}(v') W_K(v') \vint{v'}^{\ell} - \Fl{K} W_K \vint{v}^{\ell}}
        b(\cos\theta) |v - v_\ast|^\gamma \dbmu.
\end{align}
By Cauchy-Schwarz, the part of the integrand involving $\Fl{K}$ satisfies
\begin{align*}
& \quad \,
   \Fl{K} \frac{1}{\vint{v}^\ell} 
   \vpran{\Fl{K}(v') W_K(v') \vint{v'}^{\ell} - \Fl{K} W_K \vint{v}^{\ell}}
\\
&=  \frac{1}{\vint{v}^\ell} 
   \vpran{\Fl{K}(v) \Fl{K}(v')  W_K(v') \vint{v}^\ell \cos^\ell \tfrac{\theta}{2} 
   - \vpran{\Fl{K}}^2 \vint{v}^\ell W_K} 
\\
& \quad \, 
   + \frac{1}{\vint{v}^\ell} \Fl{K}(v) \Fl{K}(v') W_K(v') 
   \vpran{\vint{v'}^{\ell} - \vint{v}^{\ell} \cos^\ell \tfrac{\theta}{2}} 
\\
& \leq 
   \tfrac{1}{2} \vpran{\vpran{\Fl{K}(v')}^2 W_K(v') \cos^{2\ell} \tfrac{\theta}{2}
     - \vpran{\Fl{K}}^2 W_K}
\\
& \quad \, 
   + \frac{1}{\vint{v}^\ell} \Fl{K}(v) \Fl{K}(v') W_K(v') \vpran{\vint{v'}^{\ell} - \vint{v}^{\ell} \cos^\ell \tfrac{\theta}{2}}
  + \tfrac{1}{2} \vpran{\Fl{K} (v)}^2 \vpran{W_K' - W_K}. 
\end{align*}
Write the upper bound of $T_1^+$ correspondingly as
\begin{align} \label{decomp:T-1-plus-linear}
   T_1^+ \leq T_{1, 1}^+ + T_{1, 2}^+ + T_{1, 3}^+. 
\end{align}
Similar as in the estimates of $T_1$ in~\eqref{bound:T-1-first-1},  the bounds for $T_{1, 1}^+$ and $T_{1, 2}^+$ follow from the regular change of variables, bound~\eqref{ineq:trilinear-1} in Proposition~\ref{prop:commutator} together with Proposition~\ref{prop:strong-sing-cancellation}:
\begin{align} \label{bound:T-1-first-plus}
   T_{1,1}^+ + T_{1,2}^+
&= \tfrac{1}{2} \iiiint_{\T^3 \times \R^6 \times \Ss^2}
      G_\ast \vpran{\vpran{\Fl{K}(v')}^2 W_K(v') \cos^{2\ell} \tfrac{\theta}{2}
     - \vpran{\Fl{K}}^2 W_K} 
        b(\cos\theta) |v - v_\ast|^\gamma \dbmu   \nn
\\
& \quad \,
  + \iiiint_{\T^3 \times \R^6 \times \Ss^2}
        G_\ast \frac{\Fl{K}(v)}{\vint{v}^\ell}  \Fl{K}(v') W_K(v')
        \vpran{\vint{v'}^{\ell} - \vint{v}^{\ell} \cos^\ell \tfrac{\theta}{2}} 
        b(\cos\theta) |v - v_\ast|^\gamma \dbmu   \nn
\\
&\leq
  -\gamma_0 \vpran{1 - C\sup_x\norm{g}_{L^1_\gamma}}
   \norm{\Fl{K} \sqrt{W_K}}^2_{L^2_x L^2_{\gamma/2}}   \nn
\\
& \quad \,
  + C_\ell \vpran{1 + \sup_x \norm{g}_{L^1_{\ell+\gamma}}}
     \norm{\Fl{K}}_{L^2_{x, v}} 
      \norm{\Fl{K} W_K}_{L^2_{x, v}} \nn
\\
& \quad \,
   + C\vpran{1 + \sup_x\norm{g}_{L^1_{3+\gamma+2s} \cap L^2}}
  \norm{\Fl{K}}_{L^2_x H^{s_1}_{\gamma_1/2}}
  \norm{\Fl{K} W_K}_{L^2_x L^2_{\gamma/2}}.  \nn
\end{align}
Inserting the bound of $W_K$ in~\eqref{bound:W-K-deriv}, we get
\begin{align*}
   T_{1,1}^+ + T_{1,2}^+
&\leq
  C_\ell \vpran{1 + \sup_x \norm{g}_{L^1_{\ell+\gamma}}}
     \norm{\Fl{K}}_{L^2_{x, v}} 
      \norm{\Fl{K}}_{L^2_x L^2_j} \nn
\\
& \quad \,
   + C\vpran{1 + \sup_x\norm{g}_{L^1_{3+\gamma+2s} \cap L^2}}
  \norm{\Fl{K}}_{L^2_x H^{s}_{\gamma/2}}^2  \nn
\\
& \quad \,
   + C\vpran{1 + \sup_x\norm{g}_{L^1_{3+\gamma+2s} \cap L^2}}
  \norm{\Fl{K}}_{L^2_x L^2_{j+\gamma/2}}^2.  \nn
\end{align*}
Note that in the estimate above we have combined the cases of the mild and strong singularities. 

\Ni The bound of $T_{1,3}^+$ is derived by using Proposition~\ref{prop:symmetry-cancel}, which gives
\begin{align} 
  T_{1,3}^+
& = \tfrac{1}{2} 
       \iiiint_{\T^3 \times \R^6 \times \Ss^2}
        G_\ast \vpran{\Fl{K} (v)}^2 \vint{v}^j \vpran{W_K' - W_K}
        b(\cos\theta) |v - v_\ast|^\gamma \dbmu   \nn
\\
& \leq 
   C \iiiint_{\T^3 \times \R^6 \times \Ss^2}
       G_\ast \vpran{\Fl{K} (v)}^2 \vint{v}^j
       \vpran{\sup_{|u| \leq |v| + |v_\ast|} \abs{\nabla W_K (u)} 
        + \sup_{|u| \leq |v| + |v_\ast|} \abs{\nabla^2 W_K (u)}}
       |v - v_\ast|^{2+\gamma}    \nn
\\
& \leq 
   C \iiiint_{\T^3 \times \R^6 \times \Ss^2}
       G_\ast \vpran{\Fl{K} (v)}^2 \vint{v}^j
       \vpran{\vint{v}^{j+2+\gamma} + \vint{v_\ast}^{j+2+\gamma}}
       \dbmu   \nn
\\
& \leq
   C\vpran{1 + \sup_x\norm{g}_{L^1_v}}
  \norm{\Fl{K}}_{L^2_x L^2_{j+\gamma/2 + 1}}^2
  + C\vpran{1 + \sup_x\norm{g}_{L^1_{j+2+\gamma}}}
  \norm{\Fl{K}}_{L^2_x L^2_{j/2}}^2.  
\end{align}
Combining the estimates for $T_{1,1}^+, T_{1,2}^+, T_{1,3}^+$, we have
\begin{align} 
   T_1^+
&\leq
  C_\ell \vpran{1 + \sup_x \norm{g}_{L^1_{\ell+\gamma}}}
     \norm{\Fl{K}}_{L^2_x L^2_j}^2 \nn
   + C\vpran{1 + \sup_x\norm{g}_{L^1_{3+\gamma+2s} \cap L^2}}
  \norm{\Fl{K}}_{L^2_x H^{s}_{\gamma/2}}^2  \nn
\\
& \quad \,
  + C\vpran{1 + \sup_x\norm{g}_{L^1_{j+2+\gamma}}}
  \norm{\Fl{K}}_{L^2_x L^2_{j + \gamma/2 + 1}}^2.   
\end{align}

The estimates for $T_2^+$ and $T_3^+$ is similar to those for $T_2$ and $T_3$ in Proposition~\ref{thm:L2-level-set}. In fact comparing the forms of $T_2^+, T_3^+$ with $T_2, T_3$ defined in~\eqref{decomp:linear}, one can see that the only difference is that $\Fl{K}$ in $T_2, T_3$ is now replaced by $\Fl{K}W_{K}$. Since no particular structure of $\Fl{K}$ is used in the bounds of $T_2$ and $T_3$, we simply replace $\Fl{K}$ in those bounds by $\Fl{K}W_{K}$ and obtain
\begin{align*}
  T_2^+ + T_3^+
&= \int_{\T^3} \int_{\R^3} Q \vpran{G, \tfrac{K}{\vint{v}^\ell}} \vint{v}^{\ell} \Fl{K} W_K \dv\dx
+ \int_{\T^3} \int_{\R^3} Q \vpran{G, \mu} \vint{v}^{\ell} \Fl{K}W_K \dv\dx
\\
& \leq
  C (1+K) \vpran{1 + \sup_x\norm{g}_{L^1_{\ell + \gamma}}}
   \norm{\Fl{K}}_{L^1_x L^1_{j+\gamma}}\,.
\end{align*}
The bound of $\CalQ$ is the combination of the bounds of $T_1^+, T_2^+, T_3^+$, which writes
\begin{align}  \label{est:CalQ}
  \CalQ
& \leq 
   C_\ell \vpran{1 + \sup_x \norm{g}_{L^1_{\ell+\gamma}}}
     \norm{\Fl{K}}_{L^2_x L^2_j}^2 \nn
   + C\vpran{1 + \sup_x\norm{g}_{L^1_{3+\gamma+2s} \cap L^2}}
  \norm{\Fl{K}}_{L^2_x H^{s}_{\gamma/2}}^2  \nn
\\
& \quad \,
  + C\vpran{1 + \sup_x\norm{g}_{L^1_{j+2+\gamma}}}
  \norm{\Fl{K}}_{L^2_x L^2_{j + \gamma/2 + 1}}^2
  +   C (1+K) \vpran{1 + \sup_x\norm{g}_{L^1_{\ell + \gamma}}}
   \norm{\Fl{K}}_{L^1_x L^1_{j+\gamma}}. 
\end{align}

Next we estimate $T_R^+$ defined in~\eqref{two-terms-estimate-1}. The estimate follows the same line as the one for $T_R$ in part (b) of Proposition~\ref{thm:L2-level-set}. We start with a similar decomposition as in~\eqref{decomp:L-R}:
\begin{align} \label{decomp:L-R-plus}
   \tfrac{1}{2} T_R^+
& = \iint_{\T^3 \times \R^3}
   \vpran{- \vint{v}^\ell \Fl{K} W_K \vpran{\vint{v}^{2\alpha}  - \nabla_v \cdot (\vint{v}^{2\alpha} \nabla_v )} \mu} \dv\dx  \nn
\\
& \quad \,
     + \iint_{\T^3 \times \R^3}
   \vpran{- \vint{v}^\ell \Fl{K} W_K \vpran{\vint{v}^{2\alpha}  - \nabla_v \cdot (\vint{v}^{2\alpha} \nabla_v )} \frac{K}{\vint{v}^\ell}} \dv\dx  \nn
\\
& \quad \,
   + \iint_{\T^3 \times \R^3}
   \vpran{- \vint{v}^\ell \Fl{K} W_K \vpran{\vint{v}^{2\alpha}  - \nabla_v \cdot (\vint{v}^{2\alpha} \nabla_v )} \vpran{f - \frac{K}{\vint{v}^\ell}}} \dv\dx \nn
\\
&\Denote
    T^+_{R,1} + T^+_{R,2} + T^+_{R,3}. 
\end{align}
It is then clear that the first two terms are bounded similarly as $T_R^1$ and $T_R^2$ which results in
\begin{align*}
    T^+_{R,1} + T^+_{R,2}
 \leq 
   C_\ell \,(1+K) \norm{\Fl{K}}_{L^1_{x} L^1_j}.
\end{align*}
The third term $T^+_{R,3}$ needs more careful estimates due to the presence of $W_K$. Via integration by parts once, we have
\begin{align} \label{decomp:T-R-3-plus}
   T^+_{R,3}
&= -\iint_{\T^3 \times \R^3}
   \vint{v}^{2\alpha} \vpran{\Fl{K}}^2 W_K \dx \dv
   - \iint_{\T^3 \times \R^3}
   \vint{v}^{2\alpha} \nabla_v \vpran{\vint{v}^\ell \Fl{K} W_K}
   \cdot \nabla_v \vpran{\frac{1}{\vint{v}^\ell} \Fl{K}} \dx \dv  \nn
\\
& =
   -\iint_{\T^3 \times \R^3}
   \vint{v}^{2\alpha} \vpran{\Fl{K}}^2 W_K \dx \dv
   - \iint_{\T^3 \times \R^3}
   \vint{v}^{2\alpha} W_K \abs{\nabla_v \Fl{K}}^2 \dx\dv
   - Rem, 
\end{align}
where the remainder $Rem$ has five pieces 
\begin{align*}
   Rem
&= \iint_{\T^3 \times \R^3}
     \vint{v}^{2\alpha} W_K \vpran{\Fl{K}}^2
     \nabla_v \vint{v}^\ell \cdot \nabla_v \vpran{\vint{v}^{-\ell}} \dv\dx
\\
& \quad \,
   + \iint_{\T^3 \times \R^3}
     \vint{v}^{2\alpha} W_K \frac{\Fl{K}}{\vint{v}^\ell}
     \nabla_v \vint{v}^\ell \cdot \nabla_v \Fl{K} \dx\dv
\\
& \quad \,
    + \iint_{\T^3 \times \R^3}
        \vint{v}^{2\alpha} W_K \, \Fl{K}\,
        \nabla_v \Fl{K} \cdot \nabla_v \vpran{\vint{v}^{-\ell}} \dx\dv
\\
& \quad \,
    + \iint_{\T^3 \times \R^3}
        \vint{v}^{2\alpha} \vint{v}^\ell \vpran{\Fl{K}}^2
        \nabla_v W_K \cdot \nabla_v \vpran{\vint{v}^{-\ell}} \dx\dv
\\
& \quad \,
    + \frac{1}{2}\iint_{\T^3 \times \R^3} \vint{v}^{2\alpha}
       \nabla_v \vpran{\Fl{K}}^2 \cdot \nabla_v W_K \dx\dv
\, \Denote
   \sum_{j=1}^5 Rem_j. 
\end{align*}
It is clear that $Rem_1, Rem_2, Rem_3$ are directly bounded as
\begin{align}  \label{bound:Rem-1-2-3}
& \quad \,
   \abs{Rem_1} + \abs{Rem_2} + \abs{Rem_3}  \nn
\\
& \leq
   C_\ell \iint_{\T^3 \times \R^3}
   \vint{v}^{2\alpha-1} W_K 
   \vpran{\Fl{K}}^2 \dx\dv
   + \frac{1}{8} \iint_{\T^3 \times \R^3}
   \vint{v}^{2\alpha-1} W_K 
   \abs{\nabla_v \Fl{K}}^2 \dx\dv \nn
\\
& \leq
   \frac{1}{8} \iint_{\T^3 \times \R^3}
   \vint{v}^{2\alpha} W_K 
   \vpran{\vpran{\Fl{K}}^2 + \abs{\nabla_v \Fl{K}}^2} \dx\dv
  + C_\ell \norm{\Fl{K}}_{L^2_{x}L^2_{j/2}}^2.
\end{align}
By integrating by parts, $Rem_4$ satisfies
\begin{align} \label{bound:Rem-4}
   \abs{Rem_4}
&= \abs{\iint_{\T^3 \times \R^3}
        W_K \nabla_v \cdot \vpran{\vint{v}^{2\alpha} \vint{v}^\ell \vpran{\Fl{K}}^2
        \nabla_v \vint{v}^{-\ell}} \dx\dv}  \nn
\\
&\leq
   C_\ell \iint_{\T^3 \times \R^3}
   \vint{v}^{2\alpha-2} W_K 
   \vpran{\Fl{K}}^2 \dx\dv
   + C_\ell \iint_{\T^3 \times \R^3}
   \vint{v}^{2\alpha-1} W_K 
   \abs{\Fl{K}} \abs{\nabla_v \Fl{K}} \dx\dv    \nn
\\
&\leq
   \frac{1}{8} \iint_{\T^3 \times \R^3}
   \vint{v}^{2\alpha} W_K 
   \vpran{\vpran{\Fl{K}}^2 + \abs{\nabla_v \Fl{K}}^2} \dx\dv
  + C_\ell \norm{\Fl{K}}_{L^2_{x} L^2_{j/2}}^2.
\end{align}
The last term $Rem_5$ needs more careful treatment. Integrating by parts, we have
\begin{align} \label{bound:Rem-5-1}
  - Rem_5
 = \frac{1}{2} \iint_{\T^3 \times \R^3}
       \vpran{\Fl{K}}^2 \nabla_{v}\vint{v}^{2\alpha}\cdot\nabla_v W_K \dx\dv 
  + \frac{1}{2}\iint_{\T^3 \times \R^3}
       \vint{v}^{2\alpha}\vpran{\Fl{K}}^2 \Delta_v W_K \dx\dv. 
\end{align}
The main observation here is for any $\kappa \geq 2$,
\begin{align*}
    (I - \Delta_v) W_K
 = (I - \Delta_v)^{1 - \kappa/2} \vpran{\vint{v}^j \One_{A_K}}
 \geq 0,
\end{align*}
where the pointwise positivity is a consequence of the positivity of the Bessel potential. Hence for pointwise $t, x, v$ we have
\begin{align*}
   \Delta_v W_K \leq W_K.
\end{align*}
Applying such relation in the second term of~\eqref{bound:Rem-5-1}, we obtain that
\begin{align*}
   \frac{1}{2}\iint_{\T^3 \times \R^3}
       \vint{v}^{2\alpha}&\vpran{\Fl{K}}^2\Delta_v W_K \dx\dv
\leq 
   \frac{1}{2}\iint_{\T^3 \times \R^3}
       \vint{v}^{2\alpha}\vpran{\Fl{K}}^2 W_K \dx\dv\,,
\end{align*}
which is a leading order-term in moments.  However, it is dominated by the dissipation because it has a smaller coefficient $\frac{1}{2}$.  Using integration by parts, the first term in~\eqref{bound:Rem-5-1} satisfies
\begin{align*}
\frac{1}{2}\iint_{\T^3 \times \R^3}
     \vpran{\Fl{K}}^2 \nabla_{v}\vint{v}^{2\alpha}\cdot\nabla_v W_K \dx\dv
&=-\iint_{\T^3 \times \R^3}
       \Fl{K} \, \nabla_v\Fl{K}\cdot
       \vpran{\nabla_{v}\vint{v}^{2\alpha}} \, W_K \dx\dv
\\
 &\quad \,
     -\frac{1}{2}\iint_{\T^3 \times \R^3}
       \vpran{\Fl{K}}^2 \vpran{\Delta_{v}\vint{v}^{2\alpha}} W_K \dx\dv\,.
\end{align*}
These terms can be controlled similarly to the $Rem_{1}$ and $Rem_{2}$ as they are lower-order in moments.  Thus,
\begin{align*}
\bigg|\frac{1}{2}&\iint_{\T^3 \times \R^3}
       \vpran{\Fl{K}}^2 \nabla_{v}\vint{v}^{2\alpha}\cdot\nabla_v W_K \dx\dv\bigg|\\
   & \leq
   \frac{1}{8} \iint_{\T^3 \times \R^3}
   \vint{v}^{2\alpha} W_K 
   \vpran{\vpran{\Fl{K}}^2 + \abs{\nabla_v \Fl{K}}^2} \dx\dv
  + C_\ell \norm{\Fl{K}}_{L^2_{x} L^2_{j/2}}^2.    
\end{align*}
The conclusion is that
\begin{align} \label{bound:Rem-5}
  - Rem_5
&\leq
      \frac{1}{2}\iint_{\T^3 \times \R^3}
       \vpran{\Fl{K}}^2 \vint{v}^{2\alpha} W_K \dx\dv \nn
\\
& \quad \,
   + \frac{1}{8} \iint_{\T^3 \times \R^3}
     \vint{v}^{2\alpha} W_K 
     \vpran{\vpran{\Fl{K}}^2 + \abs{\nabla_v \Fl{K}}^2} \dx\dv
      + C_\ell \norm{\Fl{K}}_{L^2_{x} L^2_{j/2}}^2.      
\end{align}
Now combining the estimates for $Rem_1, \cdots, Rem_5$ with the dissipation terms in~\eqref{decomp:T-R-3-plus}, we have
\begin{align*}
   T^+_{R, 3}
&\leq
   - \frac{1}{8}\iint_{\T^3 \times \R^3}
   \vint{v}^{2\alpha} \vpran{\Fl{K}}^2 W_K \dx \dv\\
   & \qquad - \frac{5}{8} \iint_{\T^3 \times \R^3}
   \vint{v}^{2\alpha} W_K \abs{\nabla_v \Fl{K}}^2 \dx\dv
   + C_\ell \norm{\Fl{K}}_{L^2_{x} L^2_{j/2}}^2. 
\end{align*}
Together with the bounds for $T^+_{R, 1}$ and $T^+_{R, 2}$, we obtain that
\begin{align} \label{bound:T-plus-R}
   T^+_R
 \leq 
   C_\ell \norm{\Fl{K}}_{L^2_{x} L^2_{j/2}}^2
   + C_\ell \, (1+K) \norm{\Fl{K}}_{L^1_{x} L^1_j},
\end{align}
which, when combined with $\CalQ$, can be absorbed into the upper bound for $\CalQ$ in~\eqref{est:CalQ}. 

\smallskip
\Ni For $h=-f$, all the previous estimates follow identically except $T^+_3$, for which we apply a similar estimate for $J_3$ in the proof of Proposition~\ref{thm:level-set-minus-f} instead of $T_3$ in Proposition~\ref{thm:L2-level-set}. Then the same bound ~follows. 
\end{proof}


\subsection{Time-space-velocity energy functional}
In this subsection we complete the $L^{2}$-energy estimate for the level-set function by adding the regularisation in the spatial variable. To such end we introduce the energy functional for $s''\in(0,s) \subseteq (0, 1)$, $\ell\geq0$, $p>1$, 
\begin{align}\label{EFunctional}
     \CalE_{p}(K,T_1,T_2)
:= \sup_{ t \in [ T_1 , T_2 ] } 
       \norm{\Fl{K}}^{2}_{L^{2}_{x,v}} 
& +  \int^{T_2}_{T_1}\int_{\T^{3}}
         \norm{\vint{\cdot}^{\gamma/2}\Fl{K}}^{2}_{H^{s}_{v}} \dx\dtau \nn
\\
& + \frac{1}{C_0}\vpran{\int^{T_2}_{T_1} \norm{(1-\Delta_{x})^{\frac {s''}{2}}\vpran{\Fl{K}}^{2}}^{p}_{L^{p}_{x,v}}\dtau}^{\frac{1}{p}}. 
\end{align}
The constant $C_0$ does not play any essential role and the parameters $s''>0$, $p>1$ will be suitably chosen 
as the discussion progresses.  We start with imposing one condition on $p$: let $r(1)$ and $r(p)$ be the exponents given in Lemma \ref{app-inter-x-theta} such that
\begin{align} \label{cond:p}
   r(1) = \tilde r(s,s'',1,3) > 2, 
\qquad 
   r(p) = \tilde r(s,s'',p,3) > r(1) > 2.
\end{align}  
We require that $p$ satisfies the condition
\begin{align} \label{cond:parameter-p-1-1}
   \frac{r(p)}{2p} \frac{r(1) - 2}{r(p) - 2} > 1. 
\end{align}
Such $p$ exists, since by the continuity of $r(\cdot)$,
\begin{align*}
   \frac{r(z)}{2z} \frac{r(1) - 2}{r(z) - 2} 
\to 
   \frac{r(1)}{2} > 1
\quad 
   \text{as $z \to 1$}.
\end{align*}
Hence a sufficient condition for~\eqref{cond:parameter-p-1} to hold is by letting $p$ be close enough to 1. Since such closeness is needed for later parts, we simply enforce it here: let $p^\sharp \in (1, 2)$ be fixed and close enough to 1 such that
\begin{align} \label{cond:p-sharp}
    \min_{[1, p^\sharp]} \frac{r(p)}{2p} \frac{r(1) - 2}{r(p) - 2} > 1,
\end{align}
and in what follows we restrict to 
\begin{align} \label{cond:parameter-p-1}
   1 < p \leq p^\sharp.
\end{align}
The reason for imposing~\eqref{cond:parameter-p-1} or~\eqref{cond:parameter-p-1-1} will be clear in the proof of the following key interpolation lemma: 
%
%
\begin{lem}[Energy functional interpolation]\label{Interpolationlemma}
Let the parameters $T_1, T_2, s, s'', \ell, n$ be given such that
\begin{align*}
   0 \leq T_1 < T_2 < T,
\quad
    0<s'' < s\in(0,1),
\quad
    \ell\geq0,
\quad
    n \geq 0. 
\end{align*} 
Let $\ell_0$ be large enough with the specification in~\eqref{cond:ell-0-1-1}. Suppose 
\begin{align*}
    \sup_{t}\big\| \vint{v}^{\ell_0+\ell} f \big\|_{L^{1}_{x,v}} \leq C_1.
\end{align*}
Let $p > 1$ be fixed and satisfying~\eqref{cond:parameter-p-1} and let $\CalE_p(K, T_1, T_2)$ be the energy functional defined in~\eqref{EFunctional}. 
Then there exists a constant $q_\ast$ which is independent of $p$ and satisfies $1 < q_\ast < \frac{r(1)}{2}$
%
such that the following holds: for any $1 < q \leq q_\ast$, we can find a pair of parameters $(r_\ast, \xi_\ast)$ with the properties
\begin{align} \label{def:r-ast-xi-ast}
   r_\ast > q_\ast > q > 1, 
\qquad
   \xi_\ast > 2q_\ast > 2q > 2, 
\end{align} 
such that for any $0 \leq M < K$ and $0\leq T_{1}\leq T_{2}\leq T$, 
\begin{align}\label{estT1-1}
   \norm{\vint{\cdot}^{\frac{n}{q}} \vpran{\Fl{K}}^{2}}_{L^{q}((T_1, T_2) \times \T^3 \times \R^3)}
\leq
   \frac{C\,\CalE_{p}(M,T_{1},T_{2})^{\frac{r_\ast}{q}}}{(K-M)^{\frac{\xi_\ast-2q}{q}}},
\end{align}
where $C$ only depends on $(C_1, s, s'', q, p)$. The parameters $q_\ast, r_\ast, \xi_\ast$ are defined in~\eqref{cond:p-theta-1}, ~\eqref{def:r-ast} and~\eqref{def:xi-ast} and they only depend on $(s, s'')$. In particular, all of these parameters are independent of $K, M$, $T_1, T_2$ and ~$f$.
\end{lem}
\begin{proof}
Recall the definitions of $r(1), r(p)$ in~\eqref{cond:p}. For any $(\theta,\xi, q)$ satisfying the relation
\begin{align} \label{cond:theta}
1 < \theta < 2 < 2q < \xi < r(1) < r(p),
\end{align}  
which is depicted in Figure~\ref{parameters-1-1}, we define $\beta \in (0, 1)$ by
\begin{align} \label{def:beta-xi}
   \frac{1}{\xi}= \frac{1-\beta}{\theta} + \frac{\beta}{r(p)},
\qquad
   \beta \in (0, 1).
\end{align}  
Note that for a given pair of $(\theta, p)$, the parameters $\beta$ and $\xi$ are in one-to-one correspondence. 
%
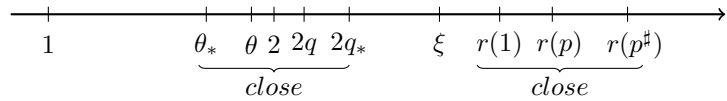
\begin{figure}[b] 
\centering
\begin{tikzpicture}
\draw[thick,->] (0,0) -- (9.5,0);
\draw (0.5,-.1) -- (0.5,.1);
\draw (2.6,-.1) -- (2.6,.1);
\draw(3.2,-.1) -- (3.2,.1);
\draw (3.5,-.1) -- (3.5,.1);
\draw(3.9,-.1) -- (3.9,.1);
\draw(4.5,-.1) -- (4.5,.1);
\draw (6.5,-.1) -- (6.5,.1);
\draw (5.7,-.1) -- (5.7,.1);
\draw (7.2,-.1) -- (7.2,.1);
\draw (8.2,-.1) -- (8.2,.1);
\node[] at (0.5,-0.4) {$1$};
\node[] at (2.6,-0.4) {$\theta_\ast$};
\node[] at (3.2,-0.4) {$\theta$};
\node[] at (3.5,-0.4) {$2$};
\node[] at (3.9,-0.4) {$2q$};
\node[] at (4.5,-0.4) {$2q_\ast$};
\node[] at (6.5,-0.4) {$r(1)$};
\node[] at (5.7,-0.4) {$\xi$};
\node[] at (7.3,-0.4) {$r(p)$};
\node[] at (8.3,-0.4) {$r(p^\sharp)$};
\draw[decoration={brace,mirror,raise=5pt},decorate]
  (2.5,-.5) -- node[below=6pt] {$close$} (4.5,-.5);
  \draw[decoration={brace,mirror,raise=5pt},decorate]
  (6.2,-.5) -- node[below=6pt] {$close$} (8.4,-.5);
\end{tikzpicture}
\caption{Choice of parameters}\label{parameters-1-1}
\end{figure}
\Ni We observe that for any $p$ satisfying~\eqref{cond:parameter-p-1}, by the definition of $\beta$, it holds that
\begin{align} \label{bound:beta-xi-upper}
 \frac{\beta \xi}{2p} < \frac{\beta \xi}{2}
= \frac{r(p)}{2} \frac{\xi - \theta}{r(p) - \theta} 
\leq
   \frac{r(p^\sharp)}{2} \frac{\xi - \theta}{r(1) - 2}
\to 
  0
\quad
  \text{as $(\theta, \xi) \to (2, 2)$}, 
\end{align}
and by~\eqref{cond:parameter-p-1-1},
\begin{align} \label{bound:beta-xi-lower}
  \frac{\beta \xi}{2} > \frac{\beta \xi}{2 p}
= \frac{r(p)}{2p} \frac{\xi - \theta}{r(p) - \theta} 
\to 
  \frac{r(p)}{2p} \frac{r(1) - 2}{r(p) - 2} > 1
\quad
  \text{as $(\theta,\xi) \to (2, r(1))$}.
\end{align}
The limits above are uniform in $p$ as long as $p$ satisfies~\eqref{cond:parameter-p-1}. By continuity, there exist $q_\ast, \theta_\ast$ such that if
\begin{align} \label{cond:p-theta-1}
   1 < q \leq q_\ast < r(1)/2
\quad \text{and} \quad
   \theta_\ast \leq \theta < 2, 
\end{align}
then for $(\beta, \xi)$ satisfying~\eqref{def:beta-xi}, we have
\begin{align*}
   \frac{\beta \xi}{2} < 1
\qquad
  \text{as $\xi \to 2 q_\ast$ and $\xi > 2q_\ast$},
\end{align*}
and 
\begin{align*}
  \frac{\beta \xi}{2} > p > 1
\quad
  \text{as $\xi \to r(1)$}.
\end{align*}
As an example, we can choose 
\begin{align} \label{def:q-ast-theta-ast}
   q_\ast
= 1 + \frac{1}{2} \frac{r(1) - 2}{r(p^\sharp)}, 
\qquad
  \theta_\ast
> 2 - \frac{1}{2} \frac{r(1) - 2}{r(p^\sharp)}. 
\end{align}
Such a choice guarantees that
\begin{align*}
  \frac{r(p^\sharp)}{2} \frac{2q_\ast - \theta_\ast}{r(1) - 2} < 1. 
\end{align*}
It is then clear that the choices of $q_\ast, \theta_\ast$ only depend on $p^\sharp, s, s''$.
By \eqref{bound:beta-xi-upper}, if $\xi_\ast$ is sufficiently close to $2q_\ast$, then $\beta \xi /2 < 1$. As a result, for any $\zeta \in (0, 1)$, there exists $\xi_\ast(\zeta)$ paired with $\beta_\ast(\zeta)$ such that
\begin{align} \label{def:xi-ast}
   \zeta \frac{\beta_\ast(\zeta) \xi_\ast (\zeta)}{2} + (1 - \zeta) \frac{\beta_\ast(\zeta) \xi_\ast (\zeta)}{2 p} = 1. 
\end{align}
The notations $\xi_\ast(\zeta), \beta_\ast(\zeta)$ are simply emphasizing the dependence of $\xi_\ast, \beta_\ast$ on $\zeta$ instead of indicating they are functions of $\zeta$.

With the preparations above, we now fix $(q, \theta)$ satisfying~\eqref{cond:p-theta-1} and let $\zeta = \tilde \alpha(s, s'', p, 3)$, where $\tilde \alpha(s, s'', p, 3)$ is the parameter in Lemma~\ref{app-inter-x-theta}. Next we fix a pair of parameters $\xi_\ast, \beta_\ast$ satisfying $(\beta_\ast, \xi_\ast) = (\beta_\ast (\zeta), \xi_\ast(\zeta))$,
such that~\eqref{def:beta-xi} holds and
\begin{align} \label{rel:convex-comb}
    \zeta \frac{\beta_\ast \xi_\ast}{2} + (1 - \zeta) \frac{\beta_\ast \xi_\ast}{2 p} = 1. 
\end{align}
With these parameters chosen we carry out various interpolations. First, \begin{align} \label{est:interpolation-1}
    \norm{\vint{\cdot}^{\frac{n}{q}}  \vpran{\Fl{K}}^{2}}^{q}_{L^{q}_{t,x,v}} 
& = \int^{T_2}_{T_1} 
     \norm{\vint{\cdot}^{\frac {n}{2q}} \Fl{K}}^{2q}_{L^{2q}_{x,v}} \dtau 
\leq 
    \frac{1}{(K-M)^{\xi_\ast-2q}}
      \int^{T_2}_{T_1} 
        \norm{\vint{\cdot}^{\frac{n}{\xi_\ast}} f^{(\ell)}_{M,+}}^{\xi_\ast}_{L^{\xi_\ast}_{x,v}}\dtau \nn
\\
&\leq 
    \frac{1}{(K-M)^{\xi_\ast - 2q}}
       \int^{T_2}_{T_1} 
         \norm{\vint{v}^{a_0} f^{(\ell)}_{M,+}}^{(1-\beta_\ast) \xi_\ast}_{L^{\theta}_{x,v}}
         \norm{\vint{v}^{-2} f^{(\ell)}_{M,+}}^{\beta_\ast \xi_\ast}_{L^{r(p)}_{x,v}}\dtau,
\end{align}
where $a_0 
= \frac{1}{1-\beta_\ast} \vpran{\frac{n}{\xi_\ast} + 2 \beta_\ast}$.
%
%
Application of Lemma \ref{app-inter-x-theta} with $(\tilde r, \eta, \eta', m) = (r(p), s, s'', p)$  and Lemma~\ref{cor:commut-homo-fraction} to~\eqref{est:interpolation-1}  gives
\begin{align} \label{est:integral-intermediate-1}
& \quad \,
   \int^{T_2}_{T_1} 
      \norm{\vint{v}^{a_0} f^{(\ell)}_{M,+}}^{(1-\beta_\ast) \xi_\ast}_{L^{\theta}_{x,v}}      
      \norm{\vint{v}^{-2} f^{(\ell)}_{M,+}}^{\beta_\ast \xi_\ast}_{L^{r(p)}_{x,v}}\dtau  \nn
\\
&\leq 
   C \int^{T_2}_{T_1} 
     \norm{\vint{v}^{a_0} f^{(\ell)}_{M,+}}^{(1-\beta_\ast) \xi_\ast}_{L^{\theta}_{x,v}}     
     \norm{(-\Delta_{v})^{\frac{s}{2}} \vpran{\vint{v}^{-2} f^{(\ell)}_{M,+}}}^{\zeta\beta_\ast \xi_\ast}_{L^{2}_{x,v}}
     \norm{\vint{v}^{-4}(1-\Delta_{x})^{ \frac{s''}{2}} \big( f^{(\ell)}_{M,+} \big)^{2}}^{\frac{1-\zeta}{2}\beta_\ast \xi_\ast}_{L^1_v L^p_x}\dtau \nn
\\
&\leq 
   C \int^{T_2}_{T_1} 
       \norm{\vint{v}^{a_0} f^{(\ell)}_{M,+}}^{(1-\beta_\ast) \xi_\ast}_{L^{\theta}_{x,v}} 
       \norm{(-\Delta_{v})^{\frac{s}{2}} f^{(\ell)}_{M,+}}^{\zeta\beta_\ast \xi_\ast}_{L^{2}_{x,v}}
       \norm{(1-\Delta_{x})^{ \frac{s''}{2}} \vpran{f^{(\ell)}_{M,+}}^{2}}^{\frac{1-\zeta}{2}\beta_\ast \xi_\ast}_{L^{p}_{x,v}} \dtau,
\end{align}
where $C = C(s, s'', p)$. 
By~\eqref{rel:convex-comb} and the H\"{o}lder's inequality, the integral term in~\eqref{est:integral-intermediate-1} is controlled~by
\begin{align*}
& \quad \,
   \vpran{\sup_{t} \norm{\vint{v}^{a_0}f^{(\ell)}_{M,+}}^{(1-\beta_\ast) \xi_\ast}_{L^{\theta}_{x,v}}}
   \vpran{\int^{T_2}_{T_1} \norm{(-\Delta_{v})^{\frac{s}{2}} f^{(\ell)}_{M,+}}^{2}_{L^{2}_{x,v}}\dtau}^{\frac{\zeta \beta_\ast \xi_\ast}{2}}
\\
&\hspace{2.5cm} \times
  \vpran{\int^{T_2}_{T_1} \norm{(1-\Delta_{x})^{ \frac{s''}{2}}   
               \vpran{f^{(\ell)}_{M,+}}^{2}}^{p}_{L^{p}_{x,v}}\dtau}^{\frac{1-\zeta}{2p} \beta_\ast \xi_\ast}
\\
&\leq
    \vpran{\sup_{t} \norm{\vint{v}^{a_0} f^{(\ell)}_{M,+}}^{(1-\beta_\ast) \xi_\ast}_{L^{\theta}_{x,v}}}
    \CalE_{p}(M,T_{1},T_{2})^{\frac{\zeta\beta_\ast \xi_\ast}{2}}\CalE_{p}(M,T_{1},T_{2})^{\frac{1-\zeta}{2} \beta_\ast \xi_\ast}.
\end{align*}
Interpolating the $L^{\theta}_{x,v}$-norm with 
\begin{align} \label{def:beta-prime}
   \frac{1}{\theta} = \frac{1-\beta'}{1} + \frac{\beta'}{2},
\qquad
  \beta' \in (0, 1),
\end{align}
it follows that
\begin{align*}
  \norm{\vint{v}^{a_0} f^{(\ell)}_{M,+}}^{(1-\beta_\ast) \xi_\ast}_{L^{\theta}_{x,v}}
&\leq
   \norm{\vint{v}^{\frac{a_0}{1-\beta'}} f^{(\ell)}_{M,+}}^{(1-\beta') (1-\beta_\ast) \xi_\ast}_{L^{1}_{x,v}}\big\| f^{(\ell)}_{M,+} \big\|^{\beta' (1-\beta_\ast) \xi_\ast}_{L^{2}_{x,v}}\\
&\leq 
   C_1^{(1-\beta') (1-\beta_\ast) \xi_\ast}
   \CalE_{p}(M,T_{1},T_{2})^{\beta' (1-\beta_\ast) \frac{\xi_\ast}{2}},
\end{align*}
by taking 
\begin{align} \label{cond:ell-0-1-1}
   \ell_0 \geq \frac{a_0}{1-\beta'}.
\end{align}
%
Overall, we have
\begin{align*}
   \int^{T_2}_{T_1} 
      \norm{\vint{v}^{a_0} f^{(\ell)}_{M,+}}^{(1-\beta_\ast) \xi_\ast}_{L^{\theta}_{x,v}}      
      \norm{\vint{v}^{-2} f^{(\ell)}_{M,+}}^{\beta_\ast \xi_\ast}_{L^{r(p)}_{x,v}}\dtau  
\leq
  C C_1^{(1-\beta')(1-\beta_\ast) \xi_\ast} \CalE_p(M,T_{1},T_{2})^{r_\ast},
\end{align*}
with
\begin{align*}
   r_\ast 
= \beta' (1-\beta_\ast) \frac{\xi_\ast}{2} 
    + \frac{\zeta \beta_\ast \xi_\ast}{2} 
    + \frac{1-\zeta}{2} \beta_\ast \xi_\ast. 
\end{align*} 
We can make $\beta'$ arbitrarily close to 1 by taking $\theta_\ast$ in~\eqref{cond:p-theta-1} close enough to 2. This way we have
\begin{align} \label{def:r-ast}
   r_\ast = \big( 1-(1-\beta')(1-\beta_\ast) \big)\frac{\xi_\ast}{2} >q_\ast > 1,
\end{align}
hence the desired bound in~\eqref{estT1-1}. 
\end{proof}

\begin{rmk} \label{rmk:parameters}
The parameters $(\theta_\ast, q_\ast, r_\ast, \xi_\ast, \beta_\ast, \beta')$  in Lemma~\ref{Interpolationlemma} can be made explicit. Here we give an example of these parameters such that Lemma~\ref{Interpolationlemma} holds. First we fix $p^\sharp$ which satisfies~\eqref{cond:p-sharp} and let 
\begin{align*}
   p = p^\sharp, 
\qquad
   q_\ast = 1 + \frac{1}{2} \frac{r(1) - 2}{r(p^\sharp)}.
\end{align*}
Note that for each $\theta \in (1,2)$, equations~\eqref{def:beta-xi} and~\eqref{rel:convex-comb} provide a system that uniquely determines $(\xi, \beta)$ in terms of $\theta$. We recall the system below
\begin{align} \label{def:xi-beta-recall-1}
  \frac{1}{\xi}= \frac{1-\beta}{\theta} + \frac{\beta}{r(p^\sharp)},
\quad \ \text{and} \quad \ 
  \zeta^\sharp \frac{\beta \xi}{2} + (1 - \zeta^\sharp) \frac{\beta \xi}{2 p^\sharp} = 1, 
\end{align}
where $\zeta^\sharp = \tilde \alpha(s, s'', p^\sharp, 3)$. For a fixed $\theta$, we can solve and obtain
\begin{align} \label{soln:xi-beta-theta}
   \xi 
= \xi(\theta)
= \frac{r(p^\sharp) - \theta}{r(p^\sharp)}
   \frac{1}{\frac{\zeta^\sharp}{2} + \frac{1 - \zeta^\sharp}{2 p^\sharp}}
   + \theta,
\qquad
  \beta
= \beta(\theta)
= \frac{1}{\xi} \frac{1}{\frac{\zeta^\sharp}{2} + \frac{1 - \zeta^\sharp}{2 p^\sharp}} \in (0, 1).
\end{align}
The condition for $\theta$ comes from the combination of~\eqref{def:beta-prime} and~\eqref{def:r-ast}. At this moment we only need $q \in (1, r(1)/2)$. Hence we require
\begin{align} \label{def:beta-prime-recall}
  \frac{1}{\theta} = \frac{1-\beta'}{1} + \frac{\beta'}{2}
\quad \ \text{and} \quad \
  \big( 1-(1-\beta')(1-\beta) \big)\frac{\xi}{2} > q_\ast.  
\end{align}
Solving $(\theta, \beta')$-system above, we obtain the condition on $\theta$ as
\begin{align} \label{cond:theta-example}
  0 < 2 - \theta < 2 - \frac{2}{1 + \frac{\xi - 2 q_\ast}{1 - \beta}}.
\end{align}
The existence issue is equivalent to whether there exists $\theta \in (1,2)$ such that~\eqref{soln:xi-beta-theta} and~\eqref{cond:theta-example} hold simultaneously. In order to check this, we note that by~\eqref{soln:xi-beta-theta}, for any $\theta \in (1, 2)$, it holds that
\begin{align*}
   \xi 
= \xi(\theta) > 2 \frac{r(p^\sharp) - \theta}{r(p^\sharp)} + \theta
= 2 + \frac{r(p^\sharp) - 2}{r(p^\sharp)} \theta
> 2 q_\ast.
\end{align*}
In particular, it holds that 
\begin{align} \label{def:c-ast}
  \lim_{\theta \to 2} \vpran{\xi(\theta)-2 q_\ast}
\geq 
  2 + 2 \frac{r(p^\sharp) - 2}{r(p^\sharp)}  - 2 q_\ast = : 2c_\ast > 0,
\qquad
  c_\ast \in (0, 1).
\end{align}
Hence the right-hand side of~\eqref{cond:theta-example} satisfies 
\begin{align} \label{def:c-ast-ast}
   \lim_{\theta \to 2} \vpran{2 - \frac{2}{1 + \frac{\xi - 2}{1 - \beta}}}
\geq
   \lim_{\theta \to 2} \vpran{2 - \frac{2}{1 + (\xi - 2)}}
\geq
  2 - \frac{2}{1 + 2 c_\ast} 
= \frac{4 c_\ast}{1 + 2 c_\ast}= : c_{\ast\ast} \in (0, 2),
\end{align}
while the middle term clearly satisfies $\lim_{\theta \to 2} (2 - \theta) = 0$. This shows there is a range of $\theta$ values that satisfy all the desired properties. For a particular example we first introduce two parameters
\begin{align*}
   c^\sharp 
= \frac{1}{\frac{\zeta^\sharp}{2} + \frac{1 - \zeta^\sharp}{2 p^\sharp}}
> 2,
\qquad
   \alpha^\sharp
= \min \left\{\frac{1}{2} \frac{c^\sharp -2}{\vpran{1 - \frac{1}{r(p^\sharp)}} \frac{2}{1 + 2c_\ast}}, \ \ \frac{1}{2} \right\},
\end{align*}
where $c_\ast$ is defined in~\eqref{def:c-ast}. 
Then use the parameter $c_{\ast\ast}$ defined in~\eqref{def:c-ast-ast} and let
\begin{align*}
  \theta_\ast = 2 - \alpha^\sharp c_{\ast\ast} \in (1, 2). 
\end{align*}
By~\eqref{soln:xi-beta-theta} we can solve and obtain
\begin{align*}
   \xi
= \frac{r(p^\sharp) - \theta}{r(p^\sharp)} c^\sharp + \theta
= \frac{r(p^\sharp) - (2 - \alpha^\sharp c_{\ast\ast})}{r(p^\sharp)} c^\sharp + (2 - \alpha^\sharp c_{\ast\ast}).
\end{align*}
Now we check that~\eqref{cond:theta-example} holds: by the definition of $\alpha^\sharp$, we have
\begin{align*} 
   \xi - 2 
&= \frac{r(p^\sharp) - (2 - \alpha^\sharp c_{\ast\ast})}{r(p^\sharp)} c^\sharp - \alpha^\sharp c_{\ast\ast}
= \frac{r(p^\sharp) - 2}{r(p^\sharp)} c^\sharp
   - \vpran{1 - \frac{1}{r(p^\sharp)}} \alpha^\sharp c_{\ast\ast}
\\
& = c_\ast c^\sharp
   - \alpha^\sharp \vpran{1 - \frac{1}{r(p^\sharp)}} 
      \frac{4 c_\ast}{1 + 2 c_\ast}
= 2 c_\ast \vpran{\frac{c^\sharp}{2} - \alpha^\sharp \vpran{1 - \frac{1}{r(p^\sharp)}} 
      \frac{2}{1 + 2 c_\ast}}
\\
&\geq
  2 c_\ast \vpran{\frac{c^\sharp}{2} - \frac{c^\sharp - 2}{2}}
= 2 c_\ast. 
\end{align*}
Hence, repeating the previous estimate, we have
\begin{align*}
   2 - \frac{2}{1 + \frac{\xi_\ast - 2}{1 - \beta_\ast}}
\geq 
  2 - \frac{2}{1 + (\xi_\ast - 2)}
\geq
  2 - \frac{2}{1+2 c_\ast}
= c_{\ast\ast} > \alpha^\sharp c_{\ast \ast} =  2 - \theta_\ast,
\end{align*}
that is, inequality~\eqref{cond:theta-example} holds. With such $(\theta_\ast, \xi)$, we obtain $\beta, \beta', r_\ast$ via formulas~\eqref{def:xi-beta-recall-1}, \eqref{def:beta-prime-recall}, and~\eqref{def:r-ast}.
\end{rmk}

\begin{rmk} 
We also make a comment regarding $\ell_0$ in Lemma~\ref{Interpolationlemma}. Note that by Remark~\ref{rmk:parameters}, $\beta'$ and $\xi_\ast, \beta_\ast$ are functions of $\theta$. Hence $\ell_0$ depends on $n, \theta$, which in term depends on $n, s, s''$.
\end{rmk}

With Lemma \ref{Interpolationlemma} at hand, we are ready to prove the precise estimate regarding the energy functional \eqref{EFunctional} in the context of the Boltzmann equation.

\begin{prop}[Energy functional interpolation inequality]\label{thm:SV-energy-functional-linear}
Let $T>0$ be fixed and let $\ell_0>0$ be sufficiently large such that it satisfies~\eqref{cond:ell-0-2}. 
Assume that the given function $G$ satisfies~\eqref{cond:coercivity-1} and
\begin{align}  \label{bi-hyp*}
   G = \mu + g \geq 0,
\qquad
    \sup_{t, x} \| g \|_{ L^{1}_{\gamma} } \leq \delta_{0},
\qquad
   \sup_{t,x}\| g \|_{ L^{\infty}_{k_0} } \leq C.
\end{align}
Fix $\ell$ such that
\begin{align*}
   8 + \gamma \leq \ell\leq k_0 - 4 -\gamma,
\end{align*}
and assume that $f$ is a solution of ~\eqref{BEe1} which satisfies
\begin{align*}
   F = \mu + f \geq 0,
\qquad 
   \sup_{ t }\| \vint{v}^{\ell_0+\ell}f(t,\cdot,\cdot) \|_{ L^{1}_{x,v} }
\leq C_1 < \infty.
\end{align*}
Then, there exist $s'' > 0$ and $p >1$ such that for any 
\begin{align*}
   0\leq T_{1} \leq T_{2}<T,
\qquad
    \epsilon\in[0,1],
\qquad 
  0 \leq M < K,
\end{align*}  
if we let $\CalE_p(M, T_1, T_2)$ be the energy functional in~\eqref{EFunctional} with the parameters $p, s''$, then
it follows that
\begin{align}\label{key-estimate-linear}
& \qquad \, \norm{\Fl{K} (T_2)}^{2}_{L^{2}_{x,v}} 
   +  \int^{T_2}_{T_1}\| \vint{v}^{\gamma/2}(1 -\Delta_{v})^{\frac {s}{2}}\Fl{K}(\tau)\|^{2}_{L^{2}_{x,v}} {\rm d}\tau  \nn
\quad 
\\
& \qquad  \quad
   + \frac{1}{C_0}\bigg(\int^{T_{2}}_{T_{1}}\big\| (1-\Delta_{x})^{\frac {s''}{2}}\vpran{\Fl{K}}^{2} \big\|^{p}_{L^{p}_{x,v}}{\rm d}\tau\bigg)^{\frac{1}{p}}   \nn
\\
& \leq 
  C \| \vint{v}^{2} \Fl{K}(T_1) \|^{2}_{L^{2}_{x,v}} 
   + C \| \vint{v}^{2}\Fl{K}(T_1) \|^{2}_{L^{2p}_{x,v}} 
   +  \frac{C\, K}{K-M}
   \sum^{4}_{i=1}\frac{\CalE_{p}(M,T_{1},T_{2})^{\beta_{i}}}{(K-M)^{a_i}}, \end{align}
where the parameters $\beta_i>1$ and $a_i>0$ are defined in~\eqref{def:a-beta-i} and $C$ is independent of $K, M, f, T_1, T_2$.
%
Furthermore, 
the estimate holds for $h=-f$, solution to equation \eqref{eq:h-1}, with $\Fl{K}$ replaced by $(-f)^{(\ell)}_{K,+}$.
\end{prop}

\begin{proof}
We start with the bound of the term that involves the $x$-derivative on the left-hand side of~\eqref{key-estimate-linear}. This constitutes the main part of the proof. To this end, fix $\sigma \in (0, 1/2)$. From equation \eqref{BEe1} one has that
\begin{align} \label{eq:Fl-k-2}
   \frac{\text{d}}{\dt} \vpran{\Fl{K}}^{2} 
   + v\cdot\nabla_{x} \vpran{\Fl{K}}^{2} 
= 2\tilde{Q}(G,F)\vint{v}^{\ell}\Fl{K}
\Denote
   (1-\Delta_{x} - \del^{2}_{t})^{\frac\sigma2}(1-\Delta_{v} )^{\frac\sigma2+\frac{\kappa}{2}}\CalG^{(\ell)}_{K},
\qquad
   \kappa > 2,
\end{align}
that is, we have defined $\CalG^{(\ell)}_K$ as
\begin{align*}
   \CalG^{(\ell)}_K
= 2 (1-\Delta_{x} - \del^{2}_{t})^{-\frac\sigma2}
    (1-\Delta_{v} )^{-\frac{\sigma}{2} - \frac{\kappa}{2}}
    \vpran{\tilde Q(G, F) \vint{v}^{\ell}\Fl{K}},
\qquad
  \kappa > 2,
\end{align*}
where $\kappa$ can be any number larger than 2. In what follows we take
\begin{align} \label{assump:kappa-sigma}
   \sigma + \kappa \leq 3. 
\end{align}
Choose the parameters in Proposition \ref{average-lemma-p} as
\begin{align*}
   m=\kappa+\sigma, 
\quad
   \beta \in (0, s),
\quad
   s^\flat = \frac{(1- 2\sigma)\beta_-}{2(1+\sigma + \kappa + \beta)} =: s'' < \min \left\{\beta, \,\, \frac{(1- \sigma p)\beta_-}{p(1+\sigma + \kappa + \beta)} \right\},
\quad
   r = \sigma,
\quad
  \kappa > 2,
\end{align*}
where $1 < p < 2$ is chosen to be close enough to 1 such that ~\eqref{cond:parameter-p-1} holds and
\begin{align} \label{cond:p-p-prime}
   \sigma p < 1,  
\qquad
   1 < p < \frac{p}{2-p} < q_\ast,
\qquad
    \sigma p^\ast = \sigma p /(p-1)> 6,
\end{align}
where $q_\ast$ is defined in~\eqref{def:q-ast-theta-ast} and the third condition guarantees that
\begin{align} \label{embedding-1}
   H^{-\sigma,p} \vpran{\T^3_x \times \R^3_v}
\supseteq 
   L^{1} \vpran{\T^3_x \times \R^3_v}
\quad \text{since} \quad  
   H^{\sigma, p^\ast} \vpran{\T^3_x \times \R^3_v}
\subseteq 
   L^{\infty} \vpran{\T^3_x \times \R^3_v}.
\end{align}

\Ni With the choices of these parameters and~\eqref{assump:kappa-sigma}, we now apply Proposition \ref{average-lemma-p} and obtain
that
\begin{align} \label{ineq:x-reg-level}
  \norm{(1-\Delta_{x})^{\frac {s''}{2}} \vpran{\Fl{K}}^{2}}_{L^p_{t,x,v}} 
&\leq 
  C \Big(\norm{\vint{v}^{4} \vpran{\Fl{K}(T_{1})}^{2}}_{L^p_{x,v}} 
             + \norm{\vint{v}^{4} (I - \Delta_v)^{-\kappa/2}\vpran{\Fl{K}(T_{2})}^{2}}_{H^{-\sigma, p}_{x, v}} 
                    \nn
\\
&\qquad
   + \norm{\vpran{\Fl{K}}^{2}}_{L^p_{t,x,v}}
   + \norm{(-\Delta_{v})^{\frac{\beta}{2}}\vpran{\Fl{K}}^{2}}_{L^p_{t,x,v}} 
   + \norm{\vint{v}^{1+\sigma+\kappa}\CalG^{(\ell)}_{K}}_{L^{p}} \Big) \nn
\\
& \leq 
   C \Big(\norm{\vint{v}^2 \vpran{\Fl{K}(T_{1})}}_{L^{2p}_{x,v}}^2 
              + \norm{\vint{v}^4 (I - \Delta_v)^{-\kappa/2} \vpran{\Fl{K}(T_{2})}^2}_{L^1_{x,v}}  \nn
\\
&\qquad
                 + \norm{\vpran{\Fl{K}}^{2}}_{L^p_{t,x,v}}
   + \norm{(-\Delta_{v})^{\frac{\beta}{2}}\vpran{\Fl{K}}^{2}}_{L^p_{t,x,v}} 
   + \norm{\vint{v}^{4}\CalG^{(\ell)}_{K}}_{L^{p}} \Big).
\end{align}
In what follows, we bound the terms on the right-hand side of~\eqref{ineq:x-reg-level} in order with the bound for $\Fl{K}(T_2)$ left to the end. Let $n = 0$ in Lemma~\ref{Interpolationlemma}. Then the third term on the right-hand side of~\eqref{ineq:x-reg-level} is bounded as
\begin{align}\label{estT1-2}
   \norm{\vpran{\Fl{K}}^{2}}_{L^p_{t,x,v}}
\leq 
   \frac{\tilde{C}_{0} \,\CalE_{p}(M,T_{1},T_{2})^{\frac{r_\ast}{p}}}{(K-M)^{\frac{\xi_\ast-2p}{p}}},
\qquad
     \text{with $r_\ast > p$ and $\xi_\ast > 2p$.}
\end{align}
For the fourth term one invokes Lemma \ref{app-square-f} with $p'=p/(2-p)$ to get
\begin{align*}
& \hspace{-2mm}
    \norm{(-\Delta_{v})^{\frac{\beta}{2}} \vpran{\Fl{K}}^{2}}^{p}_{L^p_{t,x,v}} 
 = \int^{T_2}_{T_1}\int_{\T^{3}} 
      \norm{(-\Delta_{v})^{\frac{\beta}{2}} \vpran{\Fl{K}}^{2}}^{p}_{L^p_{v}}\dx\,\dtau
\\
&\leq  
  C  \int^{T_2}_{T_1} \int_{\T^{3}} 
      \vpran{\norm{(-\Delta_{v})^{\frac{s}{2}} \Fl{K}}^{p}_{L^{2}_{v}}
                  \norm{\vpran{\Fl{K}}^{2}}^{\frac{p}{2}}_{L^{p'}_{v}} 
                  + \norm{\vpran{\Fl{K}}^{2}}^{p}_{L^{p}_{v}}} \dx\dtau
\\
&\leq 
  C \vpran{\int^{T_2}_{T_1} \norm{(-\Delta_{v})^{\frac{s}{2}} \Fl{K}}^{2}_{L^{2}_{x,v}}\dtau}^{\frac{p}{2}}
   \vpran{\int^{T_2}_{T_1} \norm{\vpran{\Fl{K}}^{2}}^{p'}_{L^{p'}_{x,v}}\dtau}^{\frac{2-p}{2}} 
   + C \int^{T_2}_{T_1} \norm{\vpran{\Fl{K}}^{2}}^{p}_{L^{p}_{x,v}} \dtau.
\end{align*}
Since 
~\eqref{cond:p-p-prime} holds, we can apply Lemma \ref{Interpolationlemma} in the $p$ and $p'$ norms  with $n=0$ to obtain that
\begin{align}\label{estT2-1}
   \norm{(-\Delta_{v})^{\frac{s'}{2}}\vpran{\Fl{K}}^{2}}_{L^p_{t,x,v}} 
\leq 
   \tilde{C}_0 \vpran{ \frac{\CalE_{p}(M,T_{1},T_{2})^{\frac{1}{2}\vpran{1+r_\ast'/p'}}}{(K-M)^{a_2}} + \frac{\CalE_{p}(M,T_{1},T_{2})^{ \frac{r_\ast}{p}}}{(K-M)^{a_1}}},
\end{align}  
where the parameters satisfy
\begin{align} \label{def:a-1-2}
 &   r'_\ast > p', 
\qquad
  \tfrac{1}{2} \vpran{1 + r'_\ast/p'} > 1,
\qquad
   r_\ast > p,   \nn
\\
& \hspace{-3mm}
  a_1 = \vpran{\xi_\ast - 2 p}/p > 0.
\qquad
  a_2 = \vpran{\xi'_\ast - 2 p'}/p' > 0.
\end{align}
So far for Lemma~\ref{Interpolationlemma} to apply, we need
\begin{align} \label{cond:ell-0-1}
   \ell_0 
\geq 
  \frac{a_0(s, s'')}{1-\beta'(s, s'')}, 
\qquad
  a_0(p, s) = \frac{2\beta_\ast(s, s'')}{1-\beta_\ast(s, s'')}, 
\end{align}
where $\beta', \beta_\ast$ are defined in the proof of Lemma~\ref{Interpolationlemma}. 

\Ni Next we bound the last term on the right-hand side of~\eqref{ineq:x-reg-level}.  Using Lemma~\ref{prop:equivalence}, the embedding in~\eqref{embedding-1}, the assumptions for $G$ in~\eqref{bi-hyp*} and Proposition \ref{T1} with $j=4$, we get
\begin{align} \label{bound:G-K-ell-1}
   \norm{\vint{v}^{4} \CalG^{(\ell)}_{K}}_{L^p_{t,x,v}} 
&= 2 \vpran{\int^{T_2}_{T_1} \norm{\vint{v}^{4} (1-\Delta_{x} - \del_{t}^2)^{-\frac\sigma2}(1-\Delta_{v})^{-\frac\sigma2-\frac{\kappa}{2}}
    \vpran{\tilde{Q}(G,F)\vint{v}^{\ell}\Fl{K}}}^{p}_{L^{p}_{x,v}}\dtau}^{\frac 1p}  \nn
\\
&\leq  
  2 C_{\sigma} \int^{T_2}_{T_1} \big\| (1-\Delta_{v})^{-\kappa/2}\big(\tilde{Q}(G,F)\vint{v}^{\ell+4}\Fl{K}\big)\big\|_{L^{1}_{x,v}}\dtau \nn
\\
&\leq 
    C\,\| \vint{v}^{2} \Fl{K}(T_{1}) \|^{2}_{L^{2}_{x,v}}
   + C \int^{T_2}_{T_1} \|\Fl{K}\|^{2}_{L^2_x H^{s}_{\gamma/2 }} \dt  \nn
\\
& \quad \,
   + C_\ell \int_{T_1}^{T_2}
         \|\vint{v}^6 \Fl{K}\|^{2}_{L^{2}_{x, v}} \dt 
   + C _\ell (1 + K) \int^{T_2}_{T_1} \|\vint{v}^5 \Fl{K}\|_{L^{1}_{x, v}} \dt.
\end{align}
Letting $n=12$ and $n=5$ respectively in Lemma \ref{Interpolationlemma}, we can bound the last two terms in~\eqref{bound:G-K-ell-1} as
\begin{align} \label{bound:level-L-2-1}
   \int^{T_2}_{T_1} \|\vint{v}^6 \Fl{K}\|^{2}_{L^{2}_{x,v}} \dt 
&\leq 
   \frac{2^{2p-2}}{(K-M)^{2p-2}} \int^{T_2}_{T_1}\|\vint{\cdot}^{\frac{12}{2p}}f^{(\ell)}_{\frac{K+M}{2},+}\|^{2p}_{L^{2p}_{x,v}}\dtau 
\leq 
   \tilde C_0 \frac{ \CalE_{p}(M,T_{1},T_{2})^{ r_\ast } }{(K-M)^{\xi_\ast-2}}, \nn
\\
   \int^{T_2}_{T_1} \|\vint{v}^5 \Fl{K}\|_{L^{1}_{x,v}} \dt
& \leq 
   \frac{2^{2p-1}}{(K-M)^{2p-1}} \int^{T_2}_{T_1}\|\vint{\cdot}^{\frac{5}{2p}}f^{(\ell)}_{\frac{K+M}{2},+}\|^{2p}_{L^{2p}_{x,v}}\dtau 
\leq 
   \tilde C_0 \frac{ \CalE_{p}(M,T_{1},T_{2})^{ r_\ast } }{(K-M)^{\xi_\ast-1}},
\end{align}
where for such estimates to hold, we require that
\begin{align} \label{cond:ell-0-2}
  \ell_0 
\geq 
   \frac{a_0(s, s'')}{1-\beta'(s, s'')}, 
\qquad
  a_0(s, s'') = \frac{1}{1-\beta_\ast(s, s'')} \vpran{\frac{12}{\xi_\ast(s, s'')} + 2\beta_\ast(s, s'')}, 
\qquad
  \beta' = \beta' (s, s''),
\end{align}
where again $\beta', \beta_\ast$ are defined in the proof of Lemma~\ref{Interpolationlemma}. Then we are led to
\begin{align*}
  \norm{\vint{v}^{4}\,\CalG^{(\ell)}_{K}}_{L^{p}} 
&\leq 
  C\,\| \vint{v}^{2} \Fl{K}(T_{1}) \|^{2}_{L^{2}_{x,v}}
   + C \int^{T_2}_{T_1} \|\Fl{K}\|^{2}_{L^2_x H^{s}_{\gamma/2 }} \dt \nn
\\
& \quad \,
   + \tilde C_0 \big(1 + K \big) \frac{ \CalE_{p}(M,T_{1},T_{2})^{ r_\ast } }{(K-M)^{\xi_\ast-1}}
   + \tilde C_0 \frac{ \CalE_{p}(M,T_{1},T_{2})^{ r_\ast } }{(K-M)^{\xi_\ast-2}},
\qquad
   \xi_\ast > 2.
\end{align*}   
Since $\frac{K}{K-M}\geq1$, we have that
\begin{align*}
\big(1 + K \big) \frac{ \CalE_{p}(M,T_{1},T_{2})^{ r_\ast } }{(K-M)^{\xi_\ast-1}} &= \frac{ \CalE_{p}(M,T_{1},T_{2})^{ r_\ast } }{(K-M)^{\xi_\ast-1}} + \frac{K}{K-M}\frac{ \CalE_{p}(M,T_{1},T_{2})^{ r_\ast } }{(K-M)^{\xi_\ast-2}}\\
&\leq \frac{K}{K-M}\bigg(\frac{ \CalE_{p}(M,T_{1},T_{2})^{ r_\ast } }{(K-M)^{\xi_\ast-1}} +\frac{ \CalE_{p}(M,T_{1},T_{2})^{ r_\ast } }{(K-M)^{\xi_\ast-2}}\bigg).
\end{align*}   
Therefore, we conclude that
\begin{align}\label{estT3-1}
\norm{\vint{v}^{4}\,\CalG^{(\ell)}_{K}}_{L^{p}} 
 & \leq 
  C\,\| \vint{v}^{2} \Fl{K}(T_{1}) \|^{2}_{L^{2}_{x,v}}
   + C \int^{T_2}_{T_1} \! \|\Fl{K}\|^{2}_{L^2_x H^{s}_{\gamma/2 }} \dt \nn
\\
& \quad \,
    +\frac{\tilde{C}_0\,K}{K-M}\bigg(\frac{ \CalE_{p}(M,T_{1},T_{2})^{ r_\ast } }{(K-M)^{\xi_\ast-1}} +\frac{ \CalE_{p}(M,T_{1},T_{2})^{ r_\ast } }{(K-M)^{\xi_\ast-2}}\bigg), 
\end{align}
with constants $C, \tilde C_0$ independent of $\epsilon\in[0,1]$. 

\Ni Finally we bound the second term on the right-hand side of~\eqref{ineq:x-reg-level}. By the positivity of the Bessel potential and Fubini's theorem,
\begin{align*}
   \norm{\vint{v}^4 (I - \Delta_v)^{-\kappa/2} \vpran{\Fl{K}(T_{2})}^2}_{L^{1}_{x,v}}
= \int_{\R^3} \vint{v}^4 (I - \Delta_v)^{-\kappa/2}
       \vpran{\int_{\T^3} \vpran{\Fl{K}(T_{2})}^2 \dx} \dv. 
\end{align*}
Integrating equation~\eqref{eq:Fl-k-2} first in $x$ and then in $t, v$ gives
\begin{align*}
& \quad \,
   \int_{\R^3} \vint{v}^4 (I - \Delta_v)^{-\kappa/2}
       \vpran{\int_{\T^3} \vpran{\Fl{K}(T_{2})}^2 \dx} \dv
\\
&\leq \int_{\T^3} \int_{\R^3} \vint{v}^4 
       \vpran{\Fl{K}(T_{1})}^2 \dv\dx
      + \int_{T_1}^{T_2} \int_{\T^3} \int_{\R^3}
           \vint{v}^4 (I - \Delta_v)^{-\kappa/2} 
           \vpran{\tilde Q(G, F) \vint{v}^\ell \Fl{k} }
         \dv\dx\dt
\\
& \leq 
   \norm{\vint{v}^2 \Fl{K}(T_1)}_{L^2_{x,v}}^2
      + \int_{T_1}^{T_2} \norm{(1-\Delta_{v})^{-\kappa/2}\vpran{\tilde{Q}(G,F)\vint{v}^{\ell+4}\Fl{K}}}_{L^{1}_{x,v}} \dt,
\end{align*}
where the last term satisfies the same bound as in~\eqref{bound:G-K-ell-1}. Hence the term involving $\Fl{K}(T_2)$ does not add new terms to the bound. Overall, we obtain from \eqref{ineq:x-reg-level}, \eqref{estT1-2}, \eqref{estT2-1}, \eqref{estT3-1} that
\begin{align}\label{estT4-1}
   \frac{1}{C_0} 
   \norm{(1 - \Delta_{x})^{\frac {s''}{2}}\vpran{\Fl{K}}^{2}}_{L^p_{t,x,v}} 
&\leq 
   \frac{C}{C_0} \vpran{\norm{\vint{v}^2 \Fl{K}(T_{1})}^{2}_{L^{2p}_{x,v}}  
      + \big\| \vint{v}^{2} \Fl{K}(T_{1})\big\|^{2}_{L^{2}_{x,v}}}  \nn
\\
&\quad \,
  + \frac{C_{\ell}}{C_0} \int^{T_2}_{T_1} \big\|\vint{v}^{\gamma/2}(1 -\Delta_{v})^{\frac s2}\Fl{K} \big\|^{2}_{L^{2}_{x,v}} \dtau \nn
  \\
&\quad \, 
   + \frac{\tilde{C}_0}{C_0} 
      \frac{K}{K-M}
   \sum^{4}_{i=1}\frac{ \CalE_{p}(M,T_{1},T_{2})^{\beta_{i}} }{(K-M)^{a_i}},\end{align}
where the constants $C, C_\ell, C_0, \tilde C_0$ are independent of $f, K, M, T_1, T_2$ and as a summary, 
\begin{align} \label{def:a-beta-i}
&  \hspace{1cm} \beta_1 = r_\ast/p, 
\qquad
   \beta_2 = \tfrac{1}{2} \vpran{1 + r'_\ast/p'}, 
\qquad
  \beta_3 = r_\ast,  
\qquad
  \beta_4 = r_\ast,
  \nn
\\
 & a_1 = (\xi_\ast-2p)/p, 
\qquad
   a_2 = \vpran{\xi'_\ast - 2 p'}/p', 
\qquad
   a_3 = \xi_\ast - 1,
\qquad
   a_4 = \xi_\ast - 2.
\end{align}
Note that $\beta_i > 1$ and $a_i > 0$ for all $i =1, \cdots, 4$.

\smallskip
\Ni For the first two terms on the left-hand side of~\eqref{key-estimate-linear} we invoke Proposition \ref{thm:L2-level-set} to get 
\begin{align*}
& \quad \, 
  \norm{\Fl{K}(T_2)}^{2}_{L^{2}_{x,v}} 
  + \frac{c_0 \delta_4}{4} 
      \int^{T_2}_{T_1} \norm{\Fl{K}(\tau)}^{2}_{L^{2}_{x} H^s_{\gamma/2}}\dtau 
 \\
& \leq 
   \norm{\Fl{K}(T_1)}^{2}_{L^{2}_{x,v}}  
   + C_{\ell} \int^{T_2}_{T_1} \norm{\Fl{K}(\tau)}^{2}_{L^{2}_{x,v}}\dtau
   + C_\ell (1 + K) \int^{T_2}_{T_1} \norm{\Fl{K}(\tau)}_{L^1_x L^1_\gamma} \dtau
\\
&\leq 
   \norm{\Fl{K}(T_1)}^{2}_{L^{2}_{x,v}}  
   + \frac{\tilde C_0\, K}{K-M}\bigg(\frac{ \CalE_{p}(M,T_{1},T_{2})^{ r_\ast } }{(K-M)^{\xi_\ast-1}}
   + \frac{ \CalE_{p}(M,T_{1},T_{2})^{ r_\ast}}{(K-M)^{\xi_\ast-2}}\bigg),
\end{align*}
where the last step follows from similar bounds as in~\eqref{bound:level-L-2-1} and the subsequent procedure that led to~\eqref{estT3-1}. Together with \eqref{estT4-1} and by choosing $C_0=C_0(s,\ell)>0$ sufficiently large such that 
\begin{align*}
     \frac{C_{\ell}}{C_0}
\leq
    \frac{c_0 \delta_4}{8},
\end{align*}
we obtain the desired estimate in~\eqref{key-estimate-linear}. 

\smallskip
\Ni Since $(-f)^{(\ell)}_{K, +}$ satisfies the same bound as $\Fl{K}$ in Proposition~\ref{T1}, the same estimate for $(-f)^{(\ell)}_{K, +}$ as in~\eqref{key-estimate-linear} holds with Proposition~\ref{thm:level-set-minus-f} replacing Proposition~\ref{thm:L2-level-set}. 
\end{proof}

Before showing the $L^\infty$-bound of $f$, we need a closed $L^2$-bound of the zeroth level energy $\CalE_0$ given ~by
\begin{align} \label{def:CalE-0}
   \CalE_0 := \CalE_p(0, 0, T)
&= \sup_{ t \in [0 , T] } \norm{f^{(\ell)}_+}^{2}_{L^{2}_{x,v}} 
+  \int^{T}_{0}\int_{\T^{3}} \norm{\vint{\cdot}^{\gamma/2}f^{(\ell)}_+}^{2}_{H^{s}_{v}} \dx\,\dtau  \nn
\\
& \hspace{4cm}
+ \frac{1}{C_0} \vpran{\int^{T}_{0}\big\| (1-\Delta_{x})^{\frac {s''}{2}}\vpran{f^{(\ell)}_+}^{2} \big\|^{p}_{L^{p}_{x,v}}\dtau}^{\frac{1}{p}},
\end{align}
where $f_+$ denotes the positive part of $f$ and
\begin{align*}
   f^{(\ell)}_+ = \vint{v}^\ell f_+.
\end{align*}
\begin{prop} \label{prop:bound-CalE-0}
Let $T>0$ be fixed and $\Eps \in [0, 1]$, $s \in (0, 1)$.  Assume that the given function $G$ satisfies~\eqref{cond:coercivity-1} and
\begin{align} \label{bi-hyp*-1}
   G = \mu + g \geq 0, 
\qquad
    \sup_{t, x} \| g \|_{ L^{1}_{\gamma} } \leq \delta_{0},
\qquad
   \sup_{t,x}\| g \|_{ L^{\infty}_{k_0} } \leq C.
\end{align}
Fix $\ell$ such that
\begin{align*}
   \max\{8+\gamma, 3 + 2\alpha\} \leq \ell\leq k_0 - 5 -\gamma,
\end{align*}
and assume that $f$ is a solution of ~\eqref{BEe1} which satisfies $\mu + f \geq 0$.
Then for any $0 < s' < \frac{s}{2(s+3)}$, there exist $s'' \in (0, \,\, s' \frac{\gamma}{2 \ell + \gamma})$ and $p^\flat:=p^\flat(\ell,\gamma,s,s') > 1$ such that for any $1 < p < p^\flat$, we have
\begin{align}\label{bi-initial-E0-linear}
     \CalE_0
 \leq 
    C_\ell e^{C_\ell \,T}\max_{j \in\{1/p, p'/p\}}
      \vpran{\norm{\vint{\cdot}^{\ell} f_0}^{2j}_{L^{2}_{x,v}} 
      + \sup_{t,x} \|g\|^{2j}_{L^{\infty}_{k_0}} T^{j} + \epsilon^{2j} T^j},
 \qquad
    p' = p/(2-p).
\end{align}
The same estimate holds for $(-f)^{\ell}_+$ and its associated $\CalE_0$. 
\end{prop}
\begin{proof}
 The estimate for $\CalE_0$ follows from the basic energy estimates and the averaging lemma in earlier sections. 
By Corollary \ref{cor:bi-cor-basic-energy-estimate-linear}, the first two terms in $\CalE_0$ satisfy
\begin{align} \label{est:CalE-0-1}
& \quad \,
    \sup_{t \in [0,T)} \norm{f^{(\ell)}_+ (t)}^{2}_{L^{2}_{x,v}}  
 + \int^{T}_{0} \int_{\T^{3}}
        \norm{\vint{\cdot}^{\gamma/2}f^{(\ell)}_+}^{2}_{H^{s}_{v}} \dx\dt  \nn
\\
&\leq
       \sup_{t \in [0,T)} \norm{\vint{v}^\ell f(t)}^{2}_{L^{2}_{x,v}}  
 + \int^{T}_{0} \int_{\T^{3}}
        \norm{\vint{\cdot}^{\ell + \gamma/2} f}^{2}_{H^{s}_{v}} \dx\dt  \nn
\\
&\leq 
     C_\ell e^{C_\ell \,T} 
       \vpran{\norm{\vint{\cdot}^{\ell} f_0}^{2}_{L^{2}_{x,v}}  
                    + \sup_{t,x} \|g\|^{2}_{L^{\infty}_{k_0}} T
                    + \Eps^{2} T}
\Denote 
   C_\ell e^{C_\ell T} \CalD,
\end{align}
since by~\eqref{bi-hyp*-1} $0\leq\Sigma(g)\leq 1+C$.
%
Let us concentrate on the term with the spatial fractional differentiation.  Invoking Lemma \ref{app-square-f}, it follows that for $p\in(1,2)$, $0<s''<\beta\in(0,s')$,
\begin{align}\label{bi-smallE0-1}
& \quad \,
    \int^{T}_{0} 
    \norm{(1 - \Delta_x)^{\frac {s''}{2}} \vpran{f^{(\ell)}_+}^2}^p_{L^p_{x,v}} \dtau  \nn
\\
& \leq 
  C \int^{T}_{0} 
       \norm{(-\Delta_{x})^{\frac {\beta}{2}} f^{(\ell)}_+}^{2}_{L^2_{x,v}} \dtau 
 + C \int^{T}_{0} 
        \vpran{\norm{f^{(\ell)}_+}^{2p'}_{ L^{2p'}_{x,v}}  
                     + \norm{f^{(\ell)}_+}^{2p}_{ L^{2p}_{x,v} }} \dtau,
\quad \quad 
    p'= \frac{p}{2-p}>1.
\end{align}
The controls of the $L^{2p}$- and $L^{2p'}$-norms of $f^{(\ell)}_+$ are similar and both through suitable interpolations. First,
\begin{align} \label{def:interp-xi}
   \norm{f^{(\ell)}_+}_{ L^{2p}_{x,v}} 
 &\leq 
    \norm{f^{(\ell)}_+}^{1-\beta_p}_{L^{2}_{x,v}} 
    \norm{f^{(\ell)}_+}^{\beta_p}_{L^{\xi(p)}_{x,v}},
\quad \text{where} \quad 
   \xi(p)= \frac{2}{2-p}>2, 
\quad 
   \beta_{p}=\frac{1}{p}.
\end{align}
For any $\beta > 0$, let $(\eta, \eta') = (s, \beta)$ in Lemma~\ref{app-inter-x-theta-1}. Then by choosing $\xi(p) = r(s, \beta, 3)$ in that Lemma we have
\begin{align}\label{bi-smallE0-2-p}
  \norm{f^{(\ell)}_+}_{ L^{\xi(p)}_{x,v}} 
\leq   
   C \vpran{\int_{\T^3} \norm{f^{(\ell)}_+ (x,\cdot)}^2_{H^{s}_{v}} \dx}^{\frac{1}{2}} 
   + C \vpran{\int_{\R^3} \norm{f^{(\ell)}_+ (\cdot, v)}^2_{H^{\beta}_x} \dv}^{\frac{1}{2}}.
\end{align}
Consequently,  one is led to
\begin{align*}
\big\| f^{(\ell)}_+ \big\|^{2p}_{ L^{2p}_{x,v} }  \leq C\,\big\| f^{(\ell)}_+ \big\|^{2(p - 1)}_{ L^{2}_{x,v} }  \Big( \big\| (1 -\Delta_v)^{\frac{s}{2}}f^{(\ell)}_+\big\|^{2}_{L^{2}_{x,v}} + \big\|(1-\Delta_x)^{\frac{\beta}{2}} f^{(\ell)}_+\big\|^{2}_{L^{2}_{x,v}}\Big).
\end{align*}
If $\beta$ is in the range 
\begin{align} \label{range:beta-CalE-0}
    \beta \in \Big(0,  \ \frac{\gamma}{2 \ell + \gamma} s' \Big),
\end{align}
then we have the following interpolation
\begin{align}\label{bi-smallE0-2}
   \norm{(1-\Delta_x)^{\frac{\beta}{2}} f^{(\ell)}_+}^{2}_{L^2_{x,v}} 
\leq 
   C_{\ell,\gamma} \vpran{\norm{\vint{v}^{\gamma/2} f^{(\ell)}_+}^{2}_{ L^{2}_{x,v} } 
   +  \norm{(1 - \Delta_x)^{ \frac{s'}{2}} f^{(\ell)}_{+}}^{2}_{ L^2_{x,v}}}.
\end{align}
This can be seen by using the Plancherel and Young's inequalities:  
\begin{align*}
   \norm{(1-\Delta_x)^{\frac{\beta}{2}} f^{(\ell)}_+}^{2}_{L^2_{x,v}}
&= \int_{\R^3} \sum_{\eta \in \ZZ^3}
      \vint{v}^{2\ell} \vint{\eta}^{2\beta} \abs{\CalF_x\vpran{f^{(\ell)}_+}}^2 \dv
\\
& \leq
   \int_{\R^3} \sum_{\eta \in \ZZ^3}
      \vpran{\frac{1}{q} \vint{v}^{2\ell q} + \vpran{1 - 1/q} \vint{\eta}^{2\beta \frac{q}{q-1}}} \abs{\CalF_x\vpran{f^{(\ell)}_+}}^2 \dv.
\end{align*}
Take $q = \frac{2 \ell + \gamma}{2 \ell}$. Then the restriction on $\beta$ is that
\begin{align*}
   \beta < s' \vpran{1 - 1/q} = s' \frac{\gamma}{2 \ell + \gamma}.
\end{align*}
Since $\xi$ is an increasing function in $\beta$, we obtain the corresponding range for $\xi$ and for $p$ by~\eqref{def:interp-xi} as
\begin{align} \label{range:beta-p-flat}
  \beta \in \vpran{2, \ \ r \vpran{s, s' \tfrac{\gamma}{2\ell + \gamma}, 3}}
  =: (2, \ r^\flat), 
\qquad
   p \in (1, \ 2 - 2/ r^\flat)
   = : (1, \ p^\flat),
\end{align}
where $r(\cdot, \cdot, \cdot)$ is defined in Lemma~\ref{app-inter-x-theta-1}. It is clear by its definition that $p^\flat$ depends on $\ell, \gamma, s, s'$.
Using such parameters and combining the previous estimates, we obtain that
\begin{align*}
   \norm{f^{(\ell)}_+}^{2p}_{ L^{2p}_{x,v}}   
\leq 
  C \norm{f^{(\ell)}_+}^{2(p - 1)}_{ L^{2}_{x,v} }
     \vpran{\norm{\vint{v}^{\gamma/2} f^{(\ell)}_+}^{2}_{L^2_x H^s_v} 
     + \norm{(1 - \Delta_x)^{\frac{s'}{2}} f^{(\ell)}_{+}}^{2}_{L^{2}_{x,v}}}.
\end{align*}
Integrating this estimate in $t\in(0,T)$ and invoking Corollary \ref{cor:bi-cor-basic-energy-estimate-linear}, with $\ell-$moments, one is led to
\begin{align} \label{bi-smallE0-3}
    \int^{T}_{0} \norm{f^{(\ell)}_+}^{2p}_{ L^{2p}_{x,v} } \dtau 
 \leq  
    C \CalD^p,\qquad p\in(1,p^\flat).
\end{align}
Note that by making $p$ close enough to 1, we have $p' \in (1, p^\flat)$ where $p'$ is defined in~\eqref{bi-smallE0-1}. Therefore~\eqref{bi-smallE0-3} also holds with $p$ replaced by $p'$. 
Furthermore, integrating \eqref{bi-smallE0-2} in $t\in(0,T)$ and invoking Corollary \ref{cor:bi-cor-basic-energy-estimate-linear} once more, it holds that
\begin{align}\label{bi-smallE0-4}
   \int^{T}_{0} 
   \norm{(1-\Delta_x)^{\frac{\beta}{2}}f^{(\ell)}_+}^{2}_{L^{2}_{x,v}} \dtau 
\leq 
   C \int^{T}_{0} \vpran{\norm{f^{(\ell)}_+ \vint{v}^{\gamma/2}}^{2}_{ L^{2}_{x,v} } 
     + \norm{(1 - \Delta_x)^{ \frac{s'}{2}} f_{+}}^{ 2 }_{ L^{2}_{x,v}}} \dtau 
\leq 
   C \CalD.
\end{align}
Using the estimates \eqref{bi-smallE0-3}-\eqref{bi-smallE0-4} in the estimate \eqref{bi-smallE0-1},  we conclude that
\begin{align*}
   \vpran{\int^{T}_{0} 
     \norm{(1-\Delta_{x})^{\frac {s''}{2}} \vpran{f^{(\ell)}_+}^{2}}^{p}_{L^{p}_{x,v}} \dtau}^{\frac{1}{p}} 
\leq  
   C \vpran{\CalD^{\frac{1}{p}}+ \CalD^{\frac{p'}{p}}},
\end{align*}
which combined with~\eqref{est:CalE-0-1} gives~\eqref{bi-initial-E0-linear}. 

The same estimate holds for $(-f)^{\ell}_+$ and its associated $\CalE_0$
since Corollary~\ref{cor:bi-cor-basic-energy-estimate-linear} applies to the absolute value of $f$, which dominates both the negative and positive parts of $f$. 
\end{proof}

We are now ready to build the main $L^\infty$-estimate for the linear equation~\eqref{BEe1}.
 
\begin{thm}[Linear case]\label{central-bilinear-T}
Suppose $G = \mu + g \geq 0$ satisfies that
\begin{align*}
    \inf_{t, x} \norm{G}_{L^1_{v}} \geq D_0 > 0, 
\qquad
   \sup_{t, x} \vpran{\norm{G}_{L^1_2} + \norm{G}_{L\log L}} 
   < E_0 < \infty.
\end{align*}
Let $F = \mu + f\geq0$ be a solution to equation~\eqref{BEe1} with $s \in (0, 1)$.  Assume the following holds:
\begin{align*}
   \sup_{t,x} \norm{\vint{v}^{\gamma}g}_{L^{1}_{v} } 
\leq \delta_{0},
\qquad 
   \sup_{t,x} \|g\|_{L^{\infty}_{k_0}}
\leq C,
\qquad
    \max\{8+\gamma, 3 + 2 \alpha \} < \ell \leq k_0 - 5 - \gamma .
\end{align*}
Assume that the initial data satisfies
\begin{align} \label{cond:initial-L-2-infty-linear}
    \norm{\vint{v}^{\ell+2} f_0}_{ L^{2}_{x,v} }<\infty,
\quad 
    \norm{\vint{v}^{\ell} f_0}_{L^{\infty}_{x,v}} < \infty.	
\end{align}
Additionally, assume that the solution satisfies
\begin{align*}
\sup_{t}\| \langle v \rangle^{\ell_0+\ell}&f \|_{ L^{1}_{x,v} } \leq C,
\end{align*}
where $\ell_0$ satisfies the bound in Proposition~\ref{thm:SV-energy-functional-linear} (more precisely, ~\eqref{cond:ell-0-2}). Then it follows that 
\begin{align*}
    \sup_{t \in [0,T]} \norm{\vint{v}^\ell f}_{L^{\infty}_{x,v}} 
\leq 
    \max\Big\{ 2 \norm{\vint{v}^\ell f_0}_{L^{\infty}_{x,v}},  K^{lin}_0\Big\},
\end{align*}
where
\begin{align}\label{K-initial-E0-linear-1}
   K^{lin}_0:= 
  C_\ell e^{C_\ell \,T} \max_{1 \leq i \leq 4} \max_{j\in\{1/p, p'/p\} }
  \vpran{\norm{\vint{v}^{\ell} f_0}^{2j}_{L^{2}_{x,v}} 
             + \sup_{t,x} \|g\|^{2j}_{L^{\infty}_{k_0}}T^{j} + \Eps^{2j} T^j}^{\frac{\beta_i -1}{a_i}},
\qquad
   p' = \frac{p}{2-p}.
\end{align}
\end{thm} 
\begin{proof}
Choose $(p, s'')$ close enough to $(1, 0)$ so that 
\begin{align*}
   s'' < s' \frac{\gamma}{2 \ell + \gamma},
\qquad
   s' < \frac{s}{2(s+3)},
\qquad
  p < \min\{p^\sharp, p^\flat\}, 
\end{align*}
where $p^\sharp$ and $p^\flat$ are defined in~\eqref{cond:parameter-p-1} and 
~\eqref{range:beta-p-flat}
respectively. Such $p, s''$ guarantee that Lemma~\ref{Interpolationlemma}, Proposition~\ref{thm:SV-energy-functional-linear} and Proposition~\ref{prop:bound-CalE-0} hold. We implement a classical iteration scheme to prove an estimation of the $L^{\infty}$-norm for solutions. To this end, fix $K_0>0$ which will be specified later and introduce the increasing levels $M_k$ as
\begin{equation*}
M_{k}:=K_0(1-1/2^k),\qquad k=0,1,2,\cdots. 
\end{equation*}
Take $T_{2}\in (0, T)$ with $T>0$ fixed in the analysis.  In order to simplify the notation, denote
\begin{align*}
    f_{k} := f^{(\ell)}_{M_{k},+} 
\quad \text{and} \quad  
    \CalE_{k} := \CalE_{p}(M_{k},0,T),\qquad k=0,1,2,\cdots.
\end{align*}
Choose $M=M_{k-1}<M_{k}=K$ and $T_{1}=0$ in Proposition~\ref{thm:SV-energy-functional-linear}. Then
\begin{align*}
   \CalE_{p}(M_{k-1},0,T_{2})
\leq 
  \CalE_{p}(M_{k-1},0,T) =\CalE_{k-1},
\qquad 
   k=1,2,\cdots, 
\end{align*}
and
\begin{align*}
& \quad \,
     \norm{f_{k}(T_2)}^{2}_{L^{2}_{x,v}} 
   + \int^{T_{2}}_{0} \norm{\vint{v}^{\frac{\gamma}{2}}(1 -\Delta_{v})^{\frac {s}{2} } f_{k}(\tau)}^{2}_{L^{2}_{x,v}} \dtau 
    + \frac{1}{C_0} \vpran{\int^{T_{2}}_{0} \norm{(1 - \Delta_{x})^{\frac{s''}{2}} \vpran{f_{k}}^{2}}^p_{L^p_{x,v}} \dtau}^{\frac{1}{p}} 
\\
& \leq 
   C \norm{\vint{v}^{2}f_{k}(0)}^{2}_{L^{2}_{x,v}} 
   + \norm{\vint{v}^{2}f_{k}(0)}^{2}_{L^{2p}_{x,v}} 
   + C \sum^{4}_{i=1}\frac{2^{k (a_{i}+1)}\,\CalE_{k-1}^{ \beta_{i} }}{K^{a_i}_0}.
\end{align*}
Taking supremum in $T_{2}\in[0,T]$ one arrives at
\begin{align} \label{iteration:E-k-infty}
    \CalE_{k} 
\leq  
   C \norm{\vint{v}^{2}f_{k}(0)}^{2}_{L^{2}_{x,v}} 
   + C \norm{\vint{v}^{2} f_{k}(0)}^{2}_{L^{2p}_{x,v}} 
   + C \sum^{4}_{i=1}\frac{2^{k (a_{i}+1)}\,\CalE_{k-1}^{ \beta_{i} }}{K^{a_i}_0} .
\end{align}
Terms related to the initial data vanish by setting
\begin{align*}
   K_0 
\geq 
    2 \norm{\vint{v}^{\ell} f_0}_{L^\infty_{x,v}}.
\end{align*}  
Then we are led to
\begin{align} \label{DeG-ineq}
     \CalE_{k} 
 \leq 
    C \sum^{4}_{i=1} \frac{2^{k (a_{i}+1)} \CalE_{k-1}^{\beta_{i}}}{K^{a_i}_0}, 
\qquad 
     K_0 \geq 2 \norm{\vint{v}^{\ell} f_0}_{\infty}.
\end{align}
Let 
\begin{align*}
   Q_0 = \max_{1\leq i \leq 4}\Big\{ 2^{ \frac{ a_{i}+1}{\beta_{i}-1} } \Big\},
\qquad
   \CalE^{*}_{k} =\CalE_0 (1/Q_0)^{k},
\qquad
   \text{for $k=0,1,2,\cdots$},
\end{align*}
and
\begin{align} \label{def:K-0-CalE-0}
  K_0 
\geq 
   K_0(\CalE_0) 
:=  \max_{1\leq i \leq 4} \Big\{4 \, C^{\frac{1}{a_{i}}} \CalE^{\frac{\beta_{i} - 1}{a_{i}}}_0 Q_0^{\frac{\beta_{i}}{a_{i}}} \Big\}.
\end{align}
Then one can check via a direct computation that $\CalE^\ast_k$ satisfies
\begin{align*}
   \CalE^\ast_0 = \CalE_0,
\qquad
    \CalE^\ast_{k} 
 \geq 
    C\sum^{4}_{i=1} \frac{2^{k (a_{i}+1)} \vpran{\CalE^\ast_{k-1}}^{\beta_{i}}}{K^{a_i}_0},
\qquad
  k = 0, 1, 2, \cdots.
\end{align*}
By a comparison principle (since $\CalE_{0} = \CalE^{*}_{0}$) one obtains that
\begin{align*}
    \CalE_{k}\leq \CalE^{*}_{k}\rightarrow 0 \quad \text{as} \quad  k \rightarrow \infty,
\end{align*}
since $\beta_{i}>1$ (so that $Q_0>1$).  In particular, we can infer that 
\begin{align}\label{Koinfinito}
  \sup_{t\in[0,T)} \norm{f^{(\ell)}_{K_0,+}(t, \cdot, \cdot)}_{L^{2}_{x,v}} = 0
\quad \text{for}\quad 
   K_0 =\max\Big\{2 \norm{\vint{v}^{\ell} f_0}_{L^\infty_{x,v}},  K_0(\CalE_0) \Big\},
\end{align}
which implies that
\begin{align}\label{Linfinito}
   \sup_{t\in[0,T)} \norm{\vint{v}^{\ell} f_+(t, \cdot, \cdot)}_{L^{\infty}_{x,v}} \leq K_0.
\end{align}
Thanks to the estimates on $\CalE_0$ given by Proposition~\ref{prop:bound-CalE-0}, it follows that 
\begin{align*}
   K_0(\CalE_0) 
\leq 
  C_\ell e^{C_\ell \,T} \max_{1 \leq i \leq 4} \max_{j\in\{1/p, p'/p\} }
  \vpran{\norm{\vint{\cdot}^{\ell} f_0}^{2j}_{L^{2}_{x,v}} 
             + \sup_{t,x} \|g\|^{2j}_{L^{\infty}_{k_0}}T^{j} + \Eps^{2j} T^j}^{\frac{\beta_i -1}{a_i}}
=: K^{lin}_0, 
\quad 
  p' = \frac{p}{2-p}.
\end{align*}
Thus, given \eqref{Koinfinito} and \eqref{Linfinito}, one is led to
\begin{align*}
    \sup_{t} \norm{\vint{v}^\ell f_{+}(t, \cdot, \cdot)}_{L^{\infty}_{x,v}} 
\leq 
    \max\Big\{2 \|\vint{v}^\ell f_0\|_{L^\infty_{x,v}},  K^{lin}_0 \Big\}.\end{align*}
A similar bound is also valid for $-f$ since Lemma~\ref{Interpolationlemma}, Proposition~\ref{thm:SV-energy-functional-linear} and Proposition~\ref{prop:bound-CalE-0} all have their counterparts for $-f$.
\end{proof}

\begin{rmk}
In fact, since the negative part $f_-$ satisfies $f_{-}\leq\mu$, it has a Gaussian tail and
\begin{align*}
   \sup_{t \in [0,T]} \norm{\frac{f_{-}}{\sqrt{\mu}}}^{2}_{L^{\infty}_{x,v}} &\leq \sup_{t\in[0,T]}\big\|  f_{-} \big\|_{L^{\infty}_{x,v}} \leq \max\Big\{2\big\| \langle v \rangle^{\ell}f_0 \big\|_{\infty},  K_0(\CalE_0)\Big\}.
\end{align*}
\end{rmk}

%
%


\section{Linear Local Well-posedness} \label{Sec:linear-wellposedness}
%
%

In this section we establish the local well-posedness of a modified linearized Boltzmann equation. 
The ambient space for contraction is 
\begin{align} \label{def:CalX}
   \CalX_k = L^\infty(0, T; L^2_x L^2_k(\T^3 \times \R^3)),
\end{align}
where conditions on $k$ will naturally appear along the argument and the weight is only in $v$. 
We will find a solution in the subset $\CalH_k$ defined by
\begin{align} \label{def:CalH}
   \CalH_k = \{g \in \CalX_k | \, \ \mu + g \geq 0\} . 
\end{align}
The precise form of the equation under consideration in this section is
\begin{align} \label{eq:linear-reg}
   \del_t f + v \cdot \nabla_x f 
= \Eps L_\alpha (\mu + f) 
    + Q (\mu + g \chi(\vint{v}^{k_0} g), f) + Q(g \chi(\vint{v}^{k_0} g), \mu),
\end{align}
where $g \in \CalH_k$ and we recall the definition of the operator $L_\alpha$ defined in~\eqref{def:L-alpha}:
\begin{align} \label{def:L-alpha-recall}
  L_\alpha\psi
= -\Big( \langle v \rangle^{2\alpha} \psi - \nabla_{v} \cdot \big( \langle v \rangle^{2\alpha} \nabla_{v}\psi\big) \Big) \,,\qquad \alpha\geq0,
\end{align}
where $\alpha$, to be specified later, is chosen to close the energy estimate. 
The cutoff function $\chi$ satisfies
\begin{align} \label{def:chi}
   \chi(a)
= \begin{cases}
       1, & |a| \leq \delta_0, \\[2pt]
       0, & |a| \geq 2\delta_0, \\[2pt]
       \text{smooth}, & \text{for all $a \in \R$,}
   \end{cases}
\qquad
  0 \leq \chi \leq 1.
\end{align}
Note that since $g \in \CalH_k$, we have $\mu + g \chi(\vint{v}^{k_0} g) \geq 0$.

The main well-posedness statement for the linear equation~\eqref{eq:linear-reg} is
\begin{thm} \label{thm:linear-local}
Suppose $s \in (0, 1)$ and $\Eps \in [0, 1]$. Let $g \in \CalH_k$ and let $\chi$ be the cutoff function defined in~\eqref{def:chi}. 

\smallskip

\Ni (a) Let $T > 0$ be arbitrary but fixed. Suppose the initial data $f_0 \in \CalH_k$ and assume that 
\begin{align} \label{cond:k-0-k}
    k_0 > \max\{7 + \gamma, \  (k-\alpha)^++\gamma+3+2s\}, 
\qquad
   k > 8 + \gamma,
\end{align}
where $(k-\alpha)^+$ is the positive part of $k-\alpha$.
Suppose $\delta_0$ is small enough such that~\eqref{cond:delta-0-1} is satisfied. Then equation~\eqref{eq:linear-reg} has a unique solution $f \in \CalH_k$. 

\smallskip

\Ni (b) In addition to the assumptions in part (a) we further assume that
$\delta_0$ satisfies~\eqref{assump:delta-0-2-linear} and
\begin{align} \label{cond:k-0-b}
    k_0 > \max \left\{\ell_0 + 15 + 2\gamma, \ \ell_0 + 10 + 2\alpha + \gamma \right\}, 
\end{align}
where $\ell_0$ is the weight chosen in Theorem~\ref{central-bilinear-T} (more precisely, ~\eqref{cond:ell-0-2}). Then there exist $\Eps_\ast$ and $\delta_{\ast}$ small enough such that for any $T \in (0, 1)$, if the initial data satisfies
\begin{align} \label{cond:initial-data-linear}
   \norm{\vint{v}^{k_0 - \ell_0 - 5 - \gamma} f_0}_{L^2_{x,v}} < \infty,
\qquad
   \norm{\vint{v}^{k_0 - \ell_0 - 7 - \gamma} f_0}_{L^\infty_{x,v}} < \delta_{\ast},
\end{align}
then for any $0 \leq \Eps \leq \Eps_\ast$, the solution obtained in part (a) satisfies
\begin{align*}
   \norm{\vint{v}^{k_0 - \ell_0 - 7 - \gamma} f}_{L^\infty{\vpran{[0, T] \times \T^3 \times \R^3}}} \leq \delta_0,
\qquad
  \forall \, T \in (0, 1).
\end{align*}
The choice of $\Eps_\ast, \delta_{\ast}$ only depends on $\gamma, s, k_0$. 
\end{thm}
\begin{proof}
(a) We will use a similar strategy of applying the Hahn-Banach theorem as in \cite{AMSY} to obtain a solution in $\CalH_k$. 
Denote $\CalT$ as the operator
\begin{align*}
   \CalT h
= -\del_t h - v \cdot \nabla_x h
   - \vpran{\Eps L_\alpha 
    + Q (\mu + g \chi(\vint{v}^{k_0} g), \cdot)}^\ast h,
\end{align*}
where the adjoint is taken with respect to the inner product of $ L^2_x L^2_k(\T^3 \times \R^3)$ for each time $t$. 

The main step is to show the coercivity of $\CalT$ on test functions. 
Let $\mathcal{S}$ be the test function space given by
\begin{align*}
    \mathcal{S} = C^\infty_0((-\infty, T]; C^\infty(\T^3; C^\infty_c(\R^3))).
\end{align*}
Let $h \in \mathcal{S}$. Then
\begin{align} \label{term:coercivity}
 \viint{\CalT h}{h}_k
&= -\frac{1}{2} \frac{\rm d}{\dt} \norm{h}_{L^2_x L^2_k}^2
   + \Eps \iint_{\T^3 \times \R^3} \vint{v}^{2k}\vpran{\vint{v}^{2\alpha} h - \nabla_v \cdot (\vint{v}^{2\alpha} \nabla_v )h} h  \dx\dv  \nn
\\
& \quad \,
    - \iint_{\T^3 \times \R^3} Q (\mu + g\chi(\vint{v}^{k_0} g), h) h \vint{v}^{2k}  \dx\dv.
\end{align}
The bound of each term is as follows. First,
\begin{align} \label{bound:term-1}
& \quad \,
  \Eps \iint_{\T^3 \times \R^3} \vint{v}^{2k}\vpran{\vint{v}^{2\alpha} h - \nabla_v \cdot (\vint{v}^{2\alpha} \nabla_v )h} h   \nn
\\
&\geq \frac{\Eps}{2} \norm{f}_{L^2_x L^2_{k+\alpha}}^2
   + \frac{\Eps}{2} \iint_{\T^3 \times \R^3}
       \abs{\vint{v}^{\alpha+k} \nabla_v h}^2 \dx\dv
   - C_k \Eps \norm{h}_{L^2_x L^2_k}^2. 
\end{align}
For ease of notation, denote
\begin{align} \label{def:T-0-ast}
   T_0^\ast = \iint_{\T^3 \times \R^3} Q (\mu + g \chi(\vint{v}^{k_0} g), h) h \vint{v}^{2k} \dx\dv.
\end{align}
It is clear that $T_0^\ast$ has a similar structure as $T_0$ in~\eqref{def:T-0}. Hence we first get a similar bound as in~\eqref{est:T-0-1-1}:
\begin{align} \label{bound:T-ast-1}
   T_0^\ast
&\leq
     -\vpran{\gamma_0 - C_{k} \sup_x \norm{g \chi}_{L^1_\gamma}}
     \norm{\vint{v}^{k + \gamma/2} h}_{L^2_{x, v}}^2  \nn
\\
& \quad \,
   + \IntTRRS 
       b(\cos\theta) |v - v_\ast|^\gamma 
       \vpran{\mu_\ast + g_\ast \chi_\ast} \frac{|h|\vint{v}^k}{\vint{v}^k} |h'| \vint{v'}^{k}
       \vpran{\vint{v'}^{k} - \vint{v}^{k} \cos^{k} \tfrac{\theta}{2}} 
       \dbmu.
\end{align}
However, unlike~\eqref{est:T-0-1}, we cannot apply Proposition~\ref{prop:commutator} directly since having a bound depending on an $L^1_k$-norm of $g$ is unwarranted. Instead we revise the proof of Proposition~\ref{prop:commutator} to obtain a proper bound. To this end, we make a similar decomposition as in~\eqref{decomp:general} by using Lemma~\ref{lem:diff-v-k}:
\begin{align} \label{decomp:linear-k}
   \IntTRRS 
       b(\cos\theta) |v - v_\ast|^\gamma 
       \vpran{\mu_\ast + g_\ast \chi_\ast} \frac{|h|\vint{v}^k}{\vint{v}^k} |h'| \vint{v'}^{k}
       \vpran{\vint{v'}^{k} - \vint{v}^{k} \cos^{k} \tfrac{\theta}{2}} 
       \dbmu
= \sum_{n=1}^5 \Gamma_n^\ast.
\end{align}
Estimates for $\Gamma_1^\ast, \Gamma_4^\ast$ and $\Gamma_5^\ast$ are the same as in the proof of Proposition~\ref{prop:commutator}, which combined with part (b) of Proposition~\ref{prop:strong-sing-cancellation} gives
\begin{align} \label{bound:Gamma-145-ast}
   \abs{\Gamma_1^\ast}
+ \abs{\Gamma_4^\ast}
+ \abs{\Gamma_5^\ast}
\leq 
  C_k \vpran{1 + \sup_x \norm{g \chi}_{L^1_{4+\gamma}}} \norm{\vint{v}^k h}_{L^2_{x,v}}^2
\leq
  C_k \vpran{1 + \sup_x \norm{g \chi}_{L^\infty_{k_0}}} \norm{\vint{v}^k h}_{L^2_{x,v}}^2,
\end{align}
by taking $k_0$ large enough such that $k_0 > 7 + \gamma$.
The bounds for $\Gamma_2^\ast$ and $\Gamma_3^\ast$ are trickier, since as mentioned before we want to avoid introducing the $L^1_k$-norm of $g$. We do so by using the extra weight $\vint{v}^\alpha$ from the regularizing term in~\eqref{eq:linear-reg} to compensate for the loss of weights. 
The term $\Gamma_2^\ast$ is given by
\begin{align*}
   \Gamma_2^\ast
&= \iiiint_{\T^3 \times \R^6 \times \Ss^2_+} 
   b(\cos\theta) |v - v_\ast|^\gamma \vint{v_\ast}^k \sin^k\vpran{\tfrac{\theta}{2}}
   (\mu_\ast + g_\ast \chi_\ast) h \vpran{\vint{v'}^k h'} \dsigma \dv_\ast \dv \dx
\\
& = \iiiint_{\T^3 \times \R^6 \times \Ss^2_+} 
   b(\cos\theta) |v - v_\ast|^\gamma \vint{v_\ast}^k \sin^k\vpran{\tfrac{\theta}{2}}
   \mu_\ast  h \vpran{\vint{v'}^k h'} \dsigma \dv_\ast \dv \dx
\\
& \quad \, 
  + \iiiint_{\T^3 \times \R^6 \times \Ss^2_+} 
   b(\cos\theta) |v - v_\ast|^\gamma \vint{v_\ast}^k \sin^k\vpran{\tfrac{\theta}{2}}
   g_\ast \chi_\ast h \vpran{\vint{v'}^k h'} \dsigma \dv_\ast \dv \dx
\\
& \leq
  C_k \norm{\vint{v}^k h}^2_{L^2_{x, v}}
  + C_k \vpran{\sup_x \norm{\vint{v}^{k_0 - 2^+}g \chi}_{L^1_v}}
  \norm{h}_{L^2_{x} L^2_{k+2\gamma+2^+-k_0}}
  \norm{h}_{L^2_x L^2_k}
\\
& \quad \,
  + C_k \vpran{\sup_x \norm{\vint{v}^{k_0 - 3^+}g \chi}_{L^1_v}}
  \norm{h}_{L^2_{x} L^2_{2k+\gamma+3^+-k_0}} \norm{h}_{L^2_x L^2_\gamma},
\end{align*}
where $3^+$ denotes any number close to and larger than 3. In the above estimate we have applied the bound
\begin{align*}
  \vint{v_\ast} \sin\tfrac{\theta}{2}
\leq
  2\vint{v} + \vint{v'},
\qquad
  k_0 > 4 + \gamma.
\end{align*}
Since by~\eqref{cond:k-0-k}, 
\begin{align} \label{cond:k-0-2}
   k_0 \geq 7 + \gamma > 5 + \gamma,
\qquad
    k + \gamma + 3 - \alpha  < k_0,
\end{align}
it holds that
\begin{align} \label{bound:Gamma-2-ast}
   \Gamma_2^\ast
& \leq
  C_k \norm{\vint{v}^k h}^2_{L^2_{x, v}}
  + C_k \vpran{\sup_x \norm{g \chi}_{L^\infty_{k_0}}}
  \norm{\vint{v}^k h}_{L^2_{x,v}}^2  \nn
\\
& \quad \,
  + C_k \vpran{\sup_x \norm{g \chi}_{L^\infty_{k_0}}}
  \norm{\vint{v}^{k+\alpha} h}_{L^2_{x, v}} \norm{\vint{v}^\gamma h}_{L^2_{x, v}}  \nn
\\
&\leq
   \vpran{C_k + C_{k, \Eps} \sup_{x} \norm{g \chi}_{L^\infty_{k_0}}} \norm{\vint{v}^k h}^2_{L^2_{x, v}}
  + \frac{\Eps}{2} \sup_{\T^3} \norm{g \chi}_{L^\infty_{k_0}}
     \norm{\vint{v}^{k+\alpha} h}^2_{L^2_{x, v}}.
\end{align}
The same bound applied to $\Gamma_3^\ast$, which combined with~\eqref{bound:Gamma-145-ast} and~\eqref{bound:Gamma-2-ast} gives
\begin{align*}
  T_0^\ast
&\leq
    -\vpran{\gamma_0 - C_k \sup_{\T^3} \norm{g \chi}_{L^\infty_{k_0}}}
  \norm{\vint{v}^{k+\gamma/2} h}_{L^2_{x,v}}^2
\\
& \qquad \,
  + C_{k, \Eps} \vpran{1 + \sup_{x} \norm{g \chi}_{L^\infty_{k_0}}}\norm{\vint{v}^k h}^2_{L^2_{x, v}}
  + \frac{\Eps}{2} \sup_{\T^3} \norm{g \chi}_{L^\infty_{k_0}}
     \norm{\vint{v}^{k+\alpha} h}^2_{L^2_{x, v}}.
\end{align*}
Combining estimates of all the three terms in $\viint{\CalT h}{h}_k$ we obtain that
\begin{align} \label{bound:energy-T-0-ast}
  \viint{\CalT h}{h}_k
&\geq
  -\frac{1}{2} \frac{\rm d}{\dt} \norm{h}_{L^2_k(\T^3 \times \R^3)}^2
  + \vpran{\gamma_0 - C_k \sup_{\T^3} \norm{g \chi}_{L^\infty_{k_0}(\R^3)}}
  \norm{\vint{v}^{k+\gamma/2} h}_{L^2_{x,v}}^2  \nn
\\
& \quad \, 
  + \frac{\Eps}{2} \vpran{1 - \sup_{\T^3} \norm{g \chi}_{L^\infty_{k_0}}} \norm{\vint{v}^{k+\alpha} h}_{L^2_x H^{1}_v}^2
  - C_{k, \Eps} \norm{h}^2_{L^2_x L^2_{k}}  \nn
\\
& \geq
   -\frac{1}{2} \frac{\rm d}{\dt} \norm{h}_{L^2_x L^2_k}^2
   + \frac{\Eps}{4} \norm{\vint{v}^{k+\alpha} h}_{L^2_x H^{1}_v}^2
   - C_{k, \Eps} \norm{h}^2_{L^2_x L^2_{k}},
\end{align}
by taking 
\begin{align} \label{cond:delta-0-1}
   \delta_0 < \min \{1/2, \gamma_0/(2C_k) \}.
\end{align} 
Note that the restriction of $\delta_0$ is independent of $\Eps$. By Gronwall's inequaliy, we have
for any $t \in [0, T]$,
\begin{align} \label{bound:lower-Th}
  \int_t^T e^{2C_{k, \Eps} \tau} \viint{\CalT h}{h}_k \dtau
\geq
  \frac{1}{2} e^{2C_{k, \Eps} t} \norm{h(t, \cdot, \cdot)}_{L^2_x L^2_k}^2
  + \frac{\Eps}{4} \int_t^T e^{2C_{k, \Eps} \tau} \norm{\vint{v}^{k+\alpha} h}_{L^2_x H^{1}_v}^2 \dtau.
\end{align}
Note the following bounds:
\begin{align*}
   \int_t^T e^{2C_{k, \Eps} \tau} \viint{\CalT h}{h}_k \dtau
\leq
  \vpran{\int_0^T e^{2C_{k, \Eps} \tau} \norm{\CalT h}^2_{L^2_x H^{-1}_{k-\alpha}} \dtau}^{1/2}
  \vpran{\int_t^T e^{2C_{k, \Eps} \tau} \norm{\vint{v}^{k+\alpha} h}_{L^2_x H^1_v}^2 \dtau}^{1/2}
\end{align*}
and
\begin{align*}
  \int_t^T e^{2C_{k, \Eps} \tau} \viint{\CalT h}{h}_k \dtau
\leq
  \vpran{\sup_{t \in [0, T]} \norm{\vint{v}^k h}_{L^2_{x, v}}}
  \int_0^T e^{2C_{k, \Eps} \tau} \norm{\CalT h}_{L^2_x L^2_{k}} \dtau.
\end{align*}
These together with~\eqref{bound:lower-Th} give
\begin{align*}
   \sup_{t \in [0, T]} \norm{h(t, \cdot, \cdot)}_{L^2_x L^2_k}
\leq
  2 \int_0^T e^{2C_{k, \Eps} \tau} \norm{\CalT h}_{L^2_x L^2_k} \dtau,
\end{align*}
which further implies that
\begin{align*}
  \int_t^T e^{2C_{k, \Eps} \tau}\viint{\CalT h}{h}_k \dtau
\leq
  2 \vpran{\int_0^T e^{2C_{k, \Eps} \tau} \norm{\CalT h}_{L^2_x L^2_k} \dtau}^2.
\end{align*}
As a consequence,
\begin{align*}
  \frac{\Eps}{4} \int_t^T e^{2C_{k, \Eps} \tau} \norm{\vint{v}^{k+\alpha} h}_{L^2_x H^{1}_v}^2 \dtau
\leq
  2 \vpran{\int_0^T e^{2C_{k, \Eps} \tau}\norm{\CalT h}_{L^2_x L^2_k} \dtau}^2. 
\end{align*}
Moreover, we have
\begin{align*}
  \frac{\Eps^2}{16} \int_0^T e^{2C_k \tau} \norm{\vint{v}^{k+\alpha} h}_{L^2_x H^{1}_v}^2 \dtau
\leq
  \int_0^T e^{2C_k \tau} \norm{\CalT h}^2_{L^2_x H^{-1}_{k-\alpha}} \dtau.
\end{align*}
This also implies that 
\begin{align*}
  \sup_{t \in [0, T]} \norm{h(t, \cdot, \cdot)}_{L^2_k(\T^3 \times \R^3)}^2
&\leq
  2 \int_0^T e^{2C_{k, \Eps} \tau} \viint{\CalT h}{h}_k \dtau
\\
&\leq
  2 \vpran{\int_0^T e^{2C_k \tau} \norm{\vint{v}^{k+\alpha} h}_{L^2_x H^{1}_v}^2 \dtau}^{1/2}
  \vpran{\int_0^T e^{2C_k \tau} \norm{\CalT h}^2_{L^2_x H^{-1}_{k-\alpha}} \dtau}^{1/2}
\\
&\leq
  \frac{8}{\Eps} \int_0^T e^{2C_k \tau} \norm{\CalT h}^2_{L^2_x H^{-1}_{k-\alpha}} \dtau.
\end{align*}
Define 
\begin{align*}
  \mathcal{W}
= \CalT \mathcal{S}
= \left\{w \big | \, w = \CalT h, h \in \mathcal{S} \right\}.
\end{align*}
Then $\mathcal{W}$ is a subspace of
\begin{align*}
   \mathcal{Y}_1 
= L^1(0, T; L^2_x L^2_k(\T^3 \times \R^3))
\quad \text{and} \quad
   \mathcal{Y}_2 
= L^2(0, T; L^2_x H^{-1}_{k-\alpha}(\T^3 \times \R^3)). 
\end{align*}
Note that if we let 
\begin{align*}
   \mathcal{X}^{(1)} = \mathcal{X}_k
= L^\infty(0, T; L^2_x L^2_k(\T^3 \times \R^3))
\quad \text{and} \quad
   \mathcal{X}^{(2)} 
= L^2(0, T; L^2_x H^{1}_{k+\alpha}(\T^3 \times \R^3)). 
\end{align*}
Then 
\begin{align*}
    \mathcal{Y}_1^\ast = \mathcal{X}^{(1)}, 
\qquad
    \mathcal{Y}_2^\ast = \mathcal{X}^{(2)}, 
\end{align*}
where the adjoint is taken in the weighted space $L^2(0, T; L^2_x L^2_k(\T^3 \times \R^3))$. Thus for any test function $h \in \CalS$, we have shown that 
\begin{align}  \label{bound:h-w} 
   \norm{h}_{\CalX^{(1)}} + \norm{h}_{\CalX^{(2)}} 
\leq 
  C_\Eps \norm{w}_{\CalY_1}, 
\qquad
   \norm{h}_{\CalX^{(1)}} + \norm{h}_{\CalX^{(2)}} 
\leq 
  C_\Eps\norm{w}_{\CalY_2}.
\end{align}
Denote
\begin{align*}
  R_\Eps = - \Eps \vpran{\vint{v}^{2\alpha} \Id - \nabla_v \cdot (\vint{v}^{2\alpha} \nabla_v )} \mu. 
\end{align*}
Define the linear mapping on $\mathcal{W}$ 
\begin{align*}
   G (w) = \vint{h_0, f_0}_k  + \int_0^T \vint{h, Q(g \chi, \mu)}_k \dtau
   + \int_0^T \vint{h, R_\Eps}_k \dtau,
\qquad
   \text{for any $w \in \mathcal{W}$ with $\CalT h = w$}.
\end{align*}
Then by~\eqref{bound:h-w}, 
\begin{align*}
   \abs{\vint{h_0, f_0}_k}
\leq 
  \norm{h_0}_{L^2_x L^2_v}
  \norm{f_0}_{L^2_x L^2_v}
\leq
  \norm{h}_{\CalX^{(1)}}
  \norm{f_0}_{L^2_x L^2_v}
\leq
  C_{T, k, \Eps} \norm{f_0}_{L^2_x L^2_v}
  \min\{\norm{w}_{\mathcal{Y}_1}, \norm{w}_{\mathcal{Y}_2}\}
\end{align*}
and
\begin{align*}
   \abs{\int_0^T \vint{h, R_\Eps}_k \dtau}
\leq
  C_{T} \norm{h}_{L^2(0, T; L^2_{x,v})}
\leq
  C_{T} \norm{h}_{\CalX^{(2)}}
\leq
  C_{T, k, \Eps} \min\{\norm{w}_{\mathcal{Y}_1}, \norm{w}_{\mathcal{Y}_2}\}.
\end{align*}
By the trilinear estimate in Proposition~\ref{prop:trilinear} and~\eqref{bound:h-w}, the forcing term involving $Q(g \chi, \mu)$ satisfies
\begin{align*}
  \abs{\int_0^T \vint{h, Q(g \chi, \mu)}_k \dtau}
&= \abs{\int_0^T \!\! \iint_{\T^3 \times \R^3}
    Q(g \chi, \mu) h \vint{v}^{2k} \dv\dx\dtau}
\\
&\leq
  C_k \int_0^T \!\! \int_{\T^3} \norm{g \chi}_{L^1_{(k-\alpha)^++\gamma+2s} \cap L^2} \norm{\vint{v}^{k+\alpha} h}_{L^2_v} \dx\dtau
\\
&\leq
  C_{T, k} \vpran{\sup_{t, x}\norm{g \chi}_{L^\infty_{k_0}}}
  \norm{h}_{\CalX^{(2)}}
\\
&\leq
  C_{T, k, \Eps} \vpran{\sup_{t, x}\norm{g \chi}_{L^\infty_{k_0}}}
  \min\{\norm{w}_{\mathcal{Y}_1}, \norm{w}_{\mathcal{Y}_2}\},
\end{align*}
provided
\begin{align} \label{cond:k-0-3}
   (k-\alpha)^+ + \gamma + 2s + 3 < k_0. 
\end{align}
This shows $G$ is a well-defined bounded linear functional on $\mathcal{W}$, which then can be extended to $\mathcal{Y}_1$ and $\mathcal{Y}_2$. Therefore, there exists $f \in \CalX^{(1)} \cap \CalX^{(2)}$ such that 
\begin{align*}
   \vint{h_0, f_0}_k + \vint{h, Q(g, \mu)}_k + \vint{h, R_\Eps}_k = \vint{w, f}
\qquad
   \text{for any $w \in \mathcal{W}$},
\end{align*}
with the norm of $f$ satisfying
\begin{align} \label{bound:energy-1}
   \max\{\norm{f}_{\CalX^{(1)}}, \norm{f}_{\CalX^{(2)}}\}
\leq
  C_{T, k, \Eps} \norm{f_0}_{L^2_x L^2_k}
  + C_{T, k, \Eps} \vpran{1 + \sup_{t, x}\norm{g \chi}_{L^\infty_{k_0}}}.
\end{align}

To show that $f \in \CalH_k$, we need to prove that $\mu + f \geq 0$. This can be done similarly as in \cite{AMSY}. Let $F = \mu + f$ and $G = \mu + g \chi$. Then $G \geq 0$ and $F$ satisfies
\begin{align} \label{eq:F-positivity}
   \del_t F + v \cdot \nabla_x F 
= - \Eps \vpran{\vint{v}^{2\alpha} \Id - \nabla_v \cdot (\vint{v}^{2\alpha} \nabla_v )} F + Q (G, F). 
\end{align}
Let $\eta: \R \to \R^+ \cup \{0\}$ be the convex and decreasing $W^{2, \infty}$-function given by 
\begin{align*}
   \eta(x) = \frac{1}{2} (x_-)^2,
\qquad
   x_- = \min\{x, 0\}.
\end{align*}
Multiply~\eqref{eq:F-positivity} by $\vint{v}^{2k} \eta'(F) = \vint{v}^{2k} F_-$. This gives
\begin{align*}
   \frac{1}{2} \vint{v}^{2k} \vpran{\del_t (F_-)^2 + v \cdot \nabla_x (F_-)^2}
= - \Eps \vint{v}^{2\alpha+2k} \vpran{FF_-}
   + \vint{v}^{2k} \Eps F_- \nabla_v \cdot (\vint{v}^{2\alpha} \nabla_v ) F
   + \vint{v}^{2k} F_- Q(G, F).
\end{align*}
The term involving $Q(G, F)$ is estimated in the same way as in Section 7.1 of \cite{AMSY}. We only need to check the regularizing terms. After integration they satisfy
\begin{align*}
   \int_{\R^3} \vpran{- \Eps \vint{v}^{2\alpha+2k} \vpran{FF_-}} \dv
= -\Eps \norm{\vint{v}^{\alpha + k} F_-}^2_{L^2_v}, 
\end{align*}
and
\begin{align*}
   \int_{\R^3} \vpran{\vint{v}^{2k} \Eps F_- \nabla_v \cdot (\vint{v}^{2\alpha} \nabla_v  F)} \dv 
&= -\Eps \int_{\R^3} \nabla_v\vpran{\vint{v}^{2k} F_-} \cdot (\vint{v}^{2\alpha} \nabla_v  F) \dv
\\
& \hspace{-2cm} = -\Eps \int_{\R^3}  \vint{v}^{2\alpha + 2k} \eta''(F) |\nabla_v  F|^2 \dv
      -\Eps \int_{\R^3} F_- \nabla_v\vpran{\vint{v}^{2k}} \cdot (\vint{v}^{2\alpha} \nabla_v  F) \dv
\\
& \hspace{-2cm} = -\Eps \int_{\R^3}  \vint{v}^{2\alpha + 2k} \eta''(F) |\nabla_v  F|^2 \dv
      -\Eps \int_{\R^3} F_- \nabla_v\vpran{\vint{v}^{2k}} \cdot (\vint{v}^{2\alpha} \nabla_v  F_-) \dv
\\
& \hspace{-2cm} = -\Eps \int_{\R^3}  \vint{v}^{2\alpha + 2k} \eta''(F) |\nabla_v  F|^2 \dv
      +\frac{1}{2} \Eps \int_{\R^3} F_-^2 \, \nabla_v \cdot \vpran{\vint{v}^{2\alpha} \nabla_v \vint{v}^{2k}} \dv
\\
& \hspace{-2cm}
\leq
   C_{k} \Eps \norm{\vint{v}^{\alpha + k - 1} F_-}_{L^2_{v}}^2
\leq
   \frac{\Eps}{2} \norm{\vint{v}^{\alpha + k} F_-}_{L^2_{v}}^2
   + C_k \norm{\vint{v}^k F_-}_{L^2_v}^2, 
\end{align*}
which only adds to the lower-order term in the energy estimate. Hence the similar estimate as in \cite{AMSY} holds and gives $F_- = 0$, that is, $F$ is non-negative. Combined with~\eqref{bound:energy-1} we have that $f \in \CalH_k$. The uniqueness of $f$ is guaranteed by the basic energy estimate in Proposition~\ref{bilinear-zero-level}.
%

\smallskip
\Ni (b) Although~\eqref{bound:energy-1} gives a regularization in $v$ which can induce an $L^\infty$-bound of the solution by Theorem~\ref{central-bilinear-T}, the bound is undesirable since it depends on $\Eps$. Now we show the derivation of a uniform-in-$\Eps$ bound for a smaller weight by using the De Giorgi method in Theorem~\ref{central-bilinear-T}. 

The main step is to show that by letting $\ell = k_0 - \ell_0 - 7 - \gamma$
in Theorem~\ref{central-bilinear-T}, the solution from part (a) satisfies
\begin{align*}
   \sup_{t \in (0, 1)}\| \langle v \rangle^{\ell_0+\ell}&f \|_{ L^{1}_{x,v} } \leq C,
\end{align*}
where $C$ is independent of $\Eps$. The main reason that $\Eps$ enters the energy estimate in part (a) is because, in the estimates of $\Gamma_2^\ast$ and $\Gamma_3^\ast$, we have to make use of the artificial regularization $\Eps L_\alpha$ to help us control the weighted $L^\infty$-norm of $g \chi$. To avoid this difficulty, we lower the weight by introducing
\begin{align} \label{def:k-1}
   k_1 = k_0 - 5 - \gamma. 
\end{align}
By taking $\ell = k_1$ in Proposition~\ref{bilinear-zero-level}, we obtain the energy estimate
\begin{align*}
     \frac{\rm d}{\dt} \norm{\vint{v}^{k_1}  f}_{L^2_{x,v}}^2
&\leq 
    -\vpran{\frac{\gamma_0}{2} - C_{k_1} \sup_{x} \norm{g \chi}_{L^1_{\gamma}}}
  \norm{\vint{v}^{k_1+\gamma/2} f}_{L^2_{x,v}}^2
   + C_{k_1} \vpran{1 + \sup_{x} \norm{g \chi}_{L^1_{k_1+\gamma}}} \norm{\vint{v}^{k_1} f}^2_{L^2_{x, v}}  \nn
\\
& \quad \,
  -\frac{c_0 \delta_2}{4} \int_{\T^3} \norm{\vint{v}^{k_1} f}^2_{H^s_{\gamma/2}} \dx
  - \frac{\Eps}{2} \norm{\vint{v}^{k_1+\alpha} f}_{L^2_x H^1_v}^2   \nn
\\
& \quad \,
   + C_{k_1} \vpran{1 + \sup_x \norm{g \chi}_{L^1_{k_1+\gamma+2s} \cap L^2}}
   \norm{\vint{v}^{k_1} f}_{L^2_{x,v}}, 
\qquad
   k_1 > 8 + \gamma.
   \nn
\end{align*}
By the embedding of weighted $L^1$-norms into $L^\infty_{k_0}$, we get
\begin{align} \label{bound:basic-L-2-linear-1}
    \frac{\rm d}{\dt} \norm{\vint{v}^{k_1}  f}_{L^2_{x,v}}^2
&\leq 
    -\vpran{\frac{\gamma_0}{4} - C_{k_1} \sup_{x} \norm{g \chi}_{L^\infty_{k_0}}}
  \norm{\vint{v}^{k_1+\gamma/2} f}_{L^2_{x,v}}^2
   + C_{k_1} \vpran{1 + \sup_{x} \norm{g \chi}_{L^\infty_{k_0}}} \norm{\vint{v}^{k_1} f}^2_{L^2_{x, v}}  \nn
\\
& \quad \,
  - \frac{c_0 \delta_2}{4} \int_{\T^3} \norm{\vint{v}^{k_1} f}^2_{H^s_{\gamma/2}} \dx
   + C_{k_1} \vpran{1 + \sup_x \norm{g \chi}_{L^\infty_{k_0}}}
   \norm{\vint{v}^{k_1} f}_{L^2_{x,v}}  
\\
&\leq 
    -\frac{\gamma_0}{4} 
  \norm{\vint{v}^{k_1+\gamma/2} f}_{L^2_{x,v}}^2
   + C_{k_1} 
       \norm{\vint{v}^{k_1} f}^2_{L^2_{x, v}} 
  - \frac{c_0 \delta_2}{4}  \norm{\vint{v}^{k_1} f}^2_{L^2_x H^s_{\gamma/2}}
   + C_{k_1} 
   \norm{\vint{v}^{k_1} f}_{L^2_{x,v}}, \nn
\end{align}
by letting
\begin{align}  \label{assump:delta-0-2-linear}
   \delta_0 
\leq 
  \min \left\{\frac{1}{2}, \ \frac{\gamma_0}{4C_{k_1}} \right\}.
\end{align}
Therefore for $T \leq 1$, there exists $C_1$ such that 
\begin{align*} 
   \sup_{t \in [0, T]} \norm{\vint{v}^{k_1} f}_{L^2_{x, v}}
\leq 
   C_1 \vpran{\norm{\vint{v}^{k_1} f_0}_{L^2_{x, v}} + 1}
< \infty.
\end{align*}
The constant $C_1$ only depends on $k_0, \gamma, s$. By interpolation we obtain the bound
\begin{align*} 
   \sup_{t \in [0, T]} \norm{\vint{v}^{k_1-2} f}_{L^1_{x, v}}
\leq 
   10 \, C_1 \vpran{\norm{\vint{v}^{k_1} f_0}_{L^2_{x, v}} + 1}
< \infty.
\end{align*}
Given~\eqref{cond:k-0-b}, or equivalently, 
\begin{align*}
      k_1  - \ell_0 -2 > \max\{8+\gamma, \, 3+2\alpha \},
\end{align*}
we now apply Theorem~\ref{central-bilinear-T} 
to obtain that
\begin{align} \label{bound:Linfty-f}
    \sup_{t, x, v} \norm{\vint{v}^{k_1 - \ell_0 -2}  f}_{L^\infty_{x, v}} 
\leq 
    \max\Big\{ 2 \norm{\vint{v}^{k_1-\ell_0-2} f_0}_{L^{\infty}_{x,v}}, \  K^{lin}_0\Big\},
\end{align}
where $K^{lin}_0$ is defined in~\eqref{K-initial-E0-linear-1}. From the definition of $K^{lin}_0$, it is clear that there exist $\Eps_\ast, T, \delta_{\ast}$ such that if they are small enough, then 
\begin{align*}
   \sup_{t, x, v} \norm{\vint{v}^{k_1 - \ell_0 -2}  f}_{L^\infty_{x, v}} 
< \delta_0.
\end{align*}
Specifically, we require that $T < 1$ and
\begin{align*}
    C_{k_1} e^{C_{k_1}} \max_{1 \leq i \leq 4} \max_{j\in\{1/p, p'/p\} }
  \vpran{{\delta_\ast}^{2j} 
             + {\delta_0}^{2j} + \Eps_\ast^{2j}}^{\frac{\beta_i -1}{a_i}}
< \delta_0,
\qquad
   \delta_\ast < \tfrac{1}{2} \delta_0.
\end{align*}
It is then clear that the bounds of $T, \delta_{\ast}$ are all independent of $\Eps$. 
\end{proof}


\section{Nonlinear Local Theory} \label{Sec:nonlinear-local}
In this section we establish the local existence of solutions to the nonlinear Boltzmann equation
\begin{align*}
   \del_t f + v \cdot \nabla_x f 
= Q (\mu + f, \mu + f),
\qquad
   f|_{t=0} = f_0(x, v).
\end{align*}
The proof is divided into three steps: first, we show the local existence of the regularized modified nonlinear Boltzmann equation which has the same form as \eqref{eq:linear-reg} with $g$ replaced by $f$. Next, we use the De Giorgi method to show the $L^\infty_{k_0}$-bound of the solution, thus automatically removing the cutoff function. Finally, we use strong compactness to pass to the limit in $\Eps$ to recover the solution to the original Boltzmann equation. This whole process will be carried out into three subsections. Note that in this section we need to restrict ourselves to the weak singularity case with $s \in (0, 1/2)$. This is due to insufficient regularization provided by the operator $\Eps L_\alpha$ in the contraction argument (see the last step in \eqref{reason-mild-sing}). 

\subsection{Local Existence to the modified Boltzmann equation (MBE)} The modified equation has the form
\begin{align} \label{eq:MBE-nonlinear}
   \del_t f + v \cdot \nabla_x f 
= \Eps L_\alpha (\mu + f) 
    + Q (\mu + f \chi(\vint{v}^{k_0} f), \mu + f) ,
\end{align}
where $L_\alpha$ and $\chi$ are the same as in the linear equation~\eqref{eq:linear-reg}. The local existence of solutions to~\eqref{eq:MBE-nonlinear} will be shown by applying the fixed-point argument in $\CalX_k$ to the linear equation~\eqref{eq:linear-reg} with a suitable $k$. 

\begin{thm} \label{thm:local-exist-MBE}
Suppose $s \in (0, 1/2)$ and 
\begin{align*}
 & k_0 > \max \left\{\ell_0 + 15 + 2\gamma, \ \ell_0 + 10 + 2\alpha + \gamma, \ 
   k - \alpha + 2\gamma + 2s + 9 + \ell_0 \right\},
\\[1mm]
& \hspace{2cm}
   k > \max\{8 + \gamma, \alpha\}, 
 \qquad
   \alpha > \gamma + 2s + 2, 
\end{align*}
where $\ell_0$ is the same weight in Theorem~\ref{thm:linear-local} (precise statement in~\eqref{cond:ell-0-2}). 
Suppose $\Eps, \delta_0, f_0$ satisfy the assumptions in both part (a)  and part (b) in Theorem~\ref{thm:linear-local}. For each such $\Eps$, if $T$ is small enough (which may depend on $\Eps$) then ~\eqref{eq:MBE-nonlinear} has a solution $f \in L^2_k((0, T) \times \T^3 \times \R^3)$. Moreover, $f$ satisfies the bound
\begin{align} \label{bound:L-infty-lower-f-MBE-2}
   \norm{\vint{v}^{k_0 - \ell_0 - 7 - \gamma} f}_{L^\infty{\vpran{(0, T) \times \T^3 \times \R^3}}} \leq \delta_0.
\end{align}
\end{thm}
\begin{proof}
Let $k > \max\{8 + \gamma, \, \alpha\}$ and $\CalH_k \subseteq \CalX_k$ be the set defined in~\eqref{def:CalH}. For a given $g \in \CalH_k$, define the map 
\begin{align*}
   \Gamma: \CalH_k \to \CalH_k, 
\qquad
   \Gamma g = f,
\end{align*}
where $f \in \CalH_k$ is the solution to~\eqref{eq:linear-reg}. Theorem~\ref{thm:linear-local} guarantees that $\Gamma$ is well-defined provided $\delta_0, \Eps, T$ are small enough. Moreover, if we choose $k > \alpha$, then the assumptions in Theorem~\ref{thm:linear-local} require that 
\begin{align*}
     k_0 > \max \left\{\ell_0 + 15 + 2\gamma, \ \ell_0 + 10 + 2\alpha + \gamma, \ k - \alpha + \gamma + 2 + 2s \right\},
 \qquad
   k > \max\{8 + \gamma, \, \alpha\}.
\end{align*}

Our goal is to show that $\Gamma$ is a contraction mapping on the space $\CalX_k = L^\infty(0, T; L^2_x L^2_k(\T^3 \times \R^3))$ for $T$ small enough. 
Let $g, h \in \CalH_k$ and $f_g, f_h$ be the corresponding solutions such that
\begin{align*}
   \del_t f_g + v \cdot \nabla_x f_g 
= - \Eps L_\alpha f_g 
    + Q (\mu + g \chi(\vint{v}^{k_0} g), f_g) + Q(g \chi(\vint{v}^{k_0} g), \mu),
\\
   \del_t f_h + v \cdot \nabla_x f_h 
= - \Eps L_\alpha f_h 
    + Q (\mu + h \chi(\vint{v}^{k_0} h), f_h) + Q(h \chi(\vint{v}^{k_0} h),\mu).
\end{align*}
The difference of the two equations reads
\begin{align} \label{eq:difference-1}
& \quad \,
   \del_t \vpran{f_g - f_h} + v \cdot \nabla_x \vpran{f_g - f_h}  \nn
\\
&= - \Eps L_\alpha \vpran{f_g - f_h}  
   + Q (\mu + g \chi(\vint{v}^{k_0} g), f_g - f_h)  \nn
\\
& \hspace{1cm}
   + Q(g \chi(\vint{v}^{k_0} g) - h \chi(\vint{v}^{k_0} h), f_h)
    + Q(g \chi(\vint{v}^{k_0} g)-h \chi(\vint{v}^{k_0} h), \mu), 
\end{align}
with the zero initial data for $f_g - f_h$. Given sufficient regularity obtained in Theorem~\ref{thm:linear-local}, we can now apply direct energy estimates.  Multiply~\eqref{eq:difference-1} by $(f_g - f_h) \vint{v}^{2k}$ and integrate in $x, v$. Then by similar estimates as for~\eqref{bound:energy-T-0-ast}, we have
\begin{align} \label{bound:contraction}
& \quad \,
  \frac{1}{2}\frac{\rm d}{\dt} \norm{f_g - f_h}^2_{L^2_x L^2_k(\T^3 \times \R^3)}
  + \frac{\Eps}{4} \norm{\vint{v}^{\alpha + k} (f_g - f_h)}^2_{L^2_xH^1_v}  \nn
\\
&\leq
  C_{k, \Eps} \norm{f_g - f_h}^2_{L^2_x L^2_k(\T^3 \times \R^3)}
  + \iint_{\T^3 \times \R^3}
       Q(g \chi(\vint{v}^{k_0} g) - h \chi(\vint{v}^{k_0} h), f_h) (f_g - f_h) \vint{v}^{2k}\dx\dv
\\
& \quad \, 
  + \iint_{\T^3 \times \R^3}
      Q(g \chi(\vint{v}^{k_0} g)-h \chi(\vint{v}^{k_0} h), \mu)
      (f_g - f_h) \vint{v}^{2k}\dx\dv.  \nn
\end{align}
By the trilinear estimate in Proposition~\ref{prop:trilinear}, we have, for $k > \alpha$,
\begin{align} \label{reason-mild-sing}
& \quad \,
   \iint_{\T^3 \times \R^3}
       Q(g \chi(\vint{v}^{k_0} g) - h \chi(\vint{v}^{k_0} h), f_h) (f_g - f_h) \vint{v}^{2k} \dx\dv \nn
\\
&\leq
   \int_{\T^3} \norm{g \chi(\vint{v}^{k_0} g) - h \chi(\vint{v}^{k_0} h)}_{L^1_{\gamma+2s+k-\alpha} \cap L^2} \norm{f_h}_{L^2_{\gamma + 2s + k - \alpha}} \norm{f_g - f_h}_{H^{2s}_{k+\alpha}} \dx  \nn
\\
&\leq
  C \vpran{\sup_x \norm{f_h}_{L^2_{\gamma + 2s+k-\alpha}}}
  \norm{g - h}_{L^2_x L^2_k } \norm{f_g - f_h}_{L^2_x H^{2s}_{k+\alpha}}  \nn
\\
& \leq
  \frac{\Eps}{16} \norm{\vint{v}^{k+\alpha}(f_g - f_h)}_{L^2_x H^1_v}^2
  + C_\Eps \norm{g - h}_{L^2_x L^2_k }^2,
\end{align}
where the last step is precisely the (only) reason that we have to restrict to the weak singularity in this section. 
The interpolations in the estimates above require that
\begin{align*} 
  \gamma + 2s + k - \alpha \leq k_0 -\ell_0 - 9 - \gamma,
\qquad
  \gamma + 2s + k - \alpha + 2 \leq k,
\qquad
   k > 8 + \gamma. 
\end{align*}
A sufficient condition is
\begin{align} \label{cond:k-0-5}
  \alpha \geq \gamma + 2s + 2, 
\qquad
  \alpha < k \leq k_0 + \vpran{\alpha - \vpran{2\gamma + 2s + 9 + \ell_0}}.
\end{align}
By Proposition 3.4 in \cite{AMSY}, the last term in~\eqref{bound:contraction} is bounded as
\begin{align*}
& \quad \,
  \abs{\iint_{\T_x^3 \times \R_v^3}
      Q(g \chi(\vint{v}^{k_0} g)-h \chi(\vint{v}^{k_0} h), \mu)
      (f_g - f_h) \vint{v}^{2k}\dx\dv}
\\
& \leq
  C_k \norm{g \chi(\vint{v}^{k_0} g)-h \chi(\vint{v}^{k_0} h)}_{L^2_xL^2_k}
  \norm{f_g - f_h}_{L^2_xL^2_{k+\gamma}}
\\
& \leq
   C_k \norm{g-h}_{L^2_xL^2_k}  \norm{f_g - f_h}_{L^2_xL^2_{k + \alpha}}
\\
& \leq
  C_{k, \Eps} \norm{g-h}_{L^2_xL^2_k}^2
  + \frac{\Eps}{16} \norm{f_g - f_h}_{L^2_xL^2_{k + \alpha}}^2,
\end{align*}
where we have written $(f_g - f_h) \vint{v}^{2k} = \vpran{(f_g - f_h) \vint{v}^\gamma} \vint{v}^{2k-\gamma}$ when applying Proposition 3.4 from \cite{AMSY}. 
Combining the inequalities above, we obtain
\begin{align*}
& \quad \,
   \frac{1}{2}\frac{\rm d}{\dt} \norm{f_g - f_h}^2_{L^2_x L^2_k}
  + \frac{\Eps}{8} \norm{\vint{v}^{\alpha + k} (f_g - f_h)}^2_{L^2_xH^1_v}
\leq
  C_{\Eps} \norm{f_g - f_h}^2_{L^2_x L^2_k}
  + C_{k,\Eps} \norm{g - h}_{L^2_x L^2_k }^2,
\end{align*}
which, by choosing $T$ small enough which may depend on $\Eps$, gives
\begin{align*}
   \norm{f_g - f_h}^2_{L^\infty(0, T; L^2_x L^2_k)}
\leq
  \frac{1}{2} \norm{g - h}^2_{L^\infty(0, T; L^2_x L^2_k)}.
\end{align*}
Therefore, $\Gamma$ is a contraction mapping and we obtain a unique solution to the modified equation~\eqref{eq:MBE-nonlinear}. The uniform bound in \eqref{bound:L-infty-lower-f-MBE-2} is a direct consequence of Theorem~\ref{thm:linear-local}. 
\end{proof}

\subsection{$L^\infty_{k_0}$-bound of solutions to MBE}

In this part we show that the solution obtained in Theorem~\ref{thm:local-exist-MBE} is in fact a solution to the regularized Boltzmann equation
\begin{align} \label{eq:reg-Boltzmann}
      \del_t f + v \cdot \nabla_x f 
= \Eps L_\alpha (\mu + f) 
    + Q (\mu + f, \mu + f),
\qquad
   f|_{t=0} = f_0(x, v).
\end{align}
 The main step is to prove that such a solution satisfies 
 \begin{align} \label{bound:L-infty-f-reg}
   \norm{\vint{v}^{k_0} f}_{L^\infty([0, T) \times \T^3 \times \R^3)} 
\leq 
   \delta_0. 
 \end{align}
 This way the cutoff function automatically vanishes and we recover a solution to~\eqref{eq:reg-Boltzmann}.

Note that $f$ already satisfies a uniform-in-$\Eps$ bound in~\eqref{bound:L-infty-lower-f-MBE-2}. Our goal is to enhance the weight to $\vint{v}^{k_0}$. For a large part, the proof of~\eqref{bound:L-infty-f-reg} parallels that of Theorem~\ref{central-bilinear-T} for the linear case. The central  difference, which will manifest itself repeatedly in the proofs below, is that the moment requirement on $f$ for the quadratic problem \eqref{eq:MBE-nonlinear} is substantially lessened in comparison to that of the linear equation \eqref{eq:linear-reg}. This is due to the quadratic structure of the collision operator which permits us to strategically allocate moments to the appropriate entry of the collision operator. Similar as in Section~\ref{Sec:a-priori-linear}, the $L^\infty_{k_0}$-estimate is built upon various $L^2$-estimates of the solution $f$ and its level-set functions. Hence we will need to lay the ground by proving several propositions before showing the $L^\infty_{k_0}$-estimate.

\subsubsection{Local in time $L^{2}$-estimates} As the first step we show a uniform-in-$\Eps$ weighted $L^2$-bound of $f$, the solution to~\eqref{eq:MBE-nonlinear}. The following proposition is the analog to Proposition \ref{bilinear-zero-level}. 

\begin{prop}[Nonlinear uniform-in-$\Eps$ estimate]\label{quadratic-zero-level} 
Let $f$ be a solution to equation~\eqref{eq:MBE-nonlinear} with singularity $s \in (0, 1)$. Suppose 
\begin{align} \label{Q:cond:coercivity}
    \inf_{t,x} \norm{\mu + f \chi(\vint{v}^{k_0} f)}_{L^1_{v}} \geq D_0 > 0, 
\qquad
    \sup_{t,x} \norm{\mu + f \chi(\vint{v}^{k_0} f)}_{L^1_2 \cap L \log L} < E_0 < \infty\,.
\end{align}
Then for any $\ell \geq \frac{37 + 5\gamma}{2}$, the solution $f$ satisfies, for $\delta_{5}>0$ sufficiently small,
\begin{align} \label{Q:ineq:energy-basic-1}
\frac{\rm d}{\dt}\norm{\vint{v}^{\ell}  f}_{L^2_{x,v}}^2
&\leq  -\Big(\frac{\gamma_0}{4} - \delta_{5}\sup_{x}\| f \|_{L^{1}_{\gamma}}\Big)
  \norm{\vint{v}^{\ell+\gamma/2} f}_{L^2_{x,v}}^2
  - \frac{\Eps}{4} \norm{\vint{v}^{\ell+\alpha} f}_{L^2_x H^1_v}^2  \nn
\\
& \quad \, 
   -\Big(\frac{c_0}{4} \delta_5  - C_{\ell}\sup_{x}\|f\|_{L^{1}_{3+\gamma+2s}\cap L^{2}}\Big) \int_{\T^3} \norm{\vint{v}^{\ell} f}^2_{H^s_{\gamma/2}} \dx  \nn
\\
& \quad \, 
   + \vpran{ C_{\ell} + C \sup_{x} \norm{f}_{L^1_{1+\gamma}}} \norm{\vint{v}^{\ell} f}^2_{L^2_{x, v}} 
+ C_{\ell} \,\Eps\, \norm{f}_{L^2_{x,v}}.
\end{align}
In particular, if the following additional conditions hold:
\begin{align}\label{smallness}
    \sup_{t,x} \norm{f}_{L^1_{3+\gamma+2s}\cap L^{2}} 
\leq 
   \delta_0 
 < \frac{c_0\delta_5}{8 C_{\ell}}, 
\qquad
   \ell > \max\{\tfrac{37 + 5\gamma}{2}, \ 3 + 2\alpha\},
\end{align}
then for any $\Eps < 1$ and $t \in [0,T)$, we have  
\begin{align} \label{bound:energy-basic-nonlinear}
   \norm{\vint{\cdot}^{\ell} f(t)}^{2}_{L^{2}_{x,v}} 
  + \frac{c_0\delta_5}{8}  \int_0^t \int_{\T^3} \norm{\vint{v}^{\ell} f}^2_{H^s_{\gamma/2}} \dx \dtau
\leq 
   e^{C_\ell \,t} \vpran{\norm{\vint{\cdot}^{\ell} f_0}^{2}_{L^{2}_{x,v}} + \Eps^2 T},
\end{align}
where the constants $c_0, \delta_5, C_\ell$ are all independent of $\Eps$.
Furthermore, we have the regularisation in $(t, x)$ as
\begin{align}\label{Q:bound:velocity-avg-basic}
& 
    \int^T_0
     \norm{(1 - \Delta_{t})^{s'/2}  f}^{2}_{L^{2}_{x,v}} \dtau
   + \int^{T}_0
     \norm{(1 - \Delta_{x})^{s'/2}  f}^{2}_{L^{2}_{x,v}} \dtau  \nn
\\
& \hspace{1cm} \leq 
   C\int^T_0 
   \vpran{\Eps^2 \norm{\vint{v}^{3+2\alpha} f}^{2}_{L^{2}_{x,v}}
   + \norm{(1 -\Delta_{v})^{s/2} f}^{2}_{L^{2}_{x,v}}} \dt  \nn
\\
& \hspace{1cm} \quad \, 
  + C\int^T_0 \norm{ \langle v \rangle^{5+\gamma+2s}f}_{L^2_{x, v}}^2 \dt
  + C \norm{\vint{v}^9 f_0}^{2}_{L^{2}_{x,v}} + C\Eps^{2} \,T,
\end{align}
for any $s'< \frac{s}{2(s+3)}$.
\end{prop}
\begin{proof}
This is a direct consequence of  \cite[Proposition 3.2]{AMSY}, \cite[Proposition 3.4]{AMSY} and \cite[Step 1 in Theorem 6.1]{AMSY}. Note that the cutoff function $\chi$ does not change the proofs in Propositions 3.2 and 3.4 in~\cite{AMSY}, since the coercivity is guaranteed by~\eqref{Q:cond:coercivity} and the upper bounds follow from
\begin{align*}
   \abs{f \chi} \leq \abs{f}.
\end{align*}
Bounds for the regularising term $\Eps L_\alpha$ and the $(t,x)$-smoothing in~\eqref{Q:bound:velocity-avg-basic} are both handled in the same way as in the proofs of Proposition \ref{bilinear-zero-level} and Corollary~\ref{cor:bi-cor-basic-energy-estimate-linear}.
\end{proof}

The uniform $L^2$-bound in Proposition~\ref{quadratic-zero-level} is the first place that one observes the weight difference in the $\sup_x$-norm compared with the linear case: the weight $\vint{v}^\ell$ does not appear in the $\sup_x$-norm in~\eqref{Q:ineq:energy-basic-1} as opposed to~\eqref{ineq:energy-basic-1} in Proposition~\ref{bilinear-zero-level}. 

\smallskip

\subsubsection{A priori $L^2$-estimates for level sets}
Let us proceed to show the nonlinear counterpart for the \textit{a priori} estimates for the level sets.  We recall that it is a building block for the energy functional interpolation.
\begin{prop}\label{thm:L2-level-set-nonlinear}
Set $F = \mu + f \geq 0$ and $s \in (0, 1)$. Suppose $k_0, \delta_0$ in the definition of $\chi$ in~\eqref{def:chi} satisfy that $k_0 > 8 + \gamma$ and $\delta_0$ small enough such that~\eqref{cond:delta-0-nonlinear-level} holds and
\begin{align*} 
    \inf_{t,x} \norm{\mu + f \chi(\vint{v}^{k_0} f)}_{L^1_{v}} \geq D_0 > 0, 
\qquad
    \sup_{t,x} \norm{\mu + f \chi(\vint{v}^{k_0} f)}_{L^1_2 \cap L \log L} < E_0 < \infty\,.
\end{align*}
Then for any $8 + \gamma < \ell \leq k_0$,
\begin{align} \label{est:level-set-2}
& \quad \, 
      \int_{\T^3} \int_{\R^3} Q(\mu + f \chi(\vint{v}^{k_0} f), \, F) \Fl{K} \vint{v}^{\ell} \dv\dx \nn
\\
& \hspace{1cm} \leq
   -\frac{c_0 \Eps_3}{4}
   \norm{\Fl{K}}^2_{L^2_x H^s_{\gamma/2}}  
   + C_\ell \vpran{\sup_x \norm{f}_{L^1_{1+\gamma}}}
   \norm{\Fl{K}}_{L^1_{x} L^1_\gamma}
   + C_\ell (1 + K) \norm{\Fl{K}}_{L^1_{x} L^1_\gamma},  
\end{align}
where $\Eps_3$ is a constant with the bound in~\eqref{bound:Eps-2-2}.
\end{prop}

\begin{proof}
\smallskip
\Ni The proof follows from a similar argument to that of Proposition \ref{thm:L2-level-set} for the linear case.  We focus on removing the high moment dependence, such as in the norm $L^\infty_x L^1_{\ell+\gamma}$, in estimate \eqref{est:level-set-1}. First we make a similar decomposition to that of \eqref{decomp:linear}:
\begin{align} \label{decomp:nonlinear}
  \int_{\T^3} \int_{\R^3} Q(\mu + f \chi, F) \Fl{K} \vint{v}^{\ell} \dv\dx
&= \int_{\T^3} \int_{\R^3} Q \vpran{\mu + f \chi, f - \tfrac{K}{\vint{v}^\ell}} \Fl{K} \vint{v}^{\ell} \dv\dx  \nn
\\
& \quad \,
+ \int_{\T^3} \int_{\R^3} Q \vpran{\mu + f \chi, \tfrac{\mu \vint{v}^\ell + K}{\vint{v}^\ell}} \Fl{K} \vint{v}^{\ell} \dv\dx   
\Denote \tilde T_1 + \tilde T_2,
\end{align}
where we have abbreviated $f \chi(\vint{v}^{k_0} f)$ as $f \chi$.
Similar as for $T_1$ in~\eqref{decomp:linear} (with $G$ there now replaced by $\mu + f \chi$), by the regular change of variables together with~\eqref{ineq:trilinear-1-f-minus} in Propositions~\ref{prop:commutator} and~\ref{prop:strong-sing-cancellation}, we bound $\tilde T_1$ as
\begin{align} \label{bound:T-1-first}
  \tilde T_1
&\leq
 \frac{1}{2} \iiiint_{T^3 \times \R^6 \times \Ss^2}
        (\mu_\ast + f_\ast \chi_\ast)  \vpran{\vpran{\Fl{K}(v')}^2 \cos^{2\ell} \tfrac{\theta}{2}
     - \vpran{\Fl{K}}^2}
        b(\cos\theta) |v - v_\ast|^\gamma \dbmu   \nn
\\
& \quad \,
  + \iiiint_{T^3 \times \R^6 \times \Ss^2}
        (\mu_\ast + f_\ast \chi_\ast) \frac{\Fl{K}(v)}{\vint{v}^\ell}  \Fl{K}(v') \vpran{\vint{v'}^{\ell} - \vint{v}^{\ell} \cos^\ell \tfrac{\theta}{2}} 
        b(\cos\theta) |v - v_\ast|^\gamma \dbmu   \nn
\\
& \leq
  -\frac{1}{2} \gamma_0 \vpran{1 - C \sup_x\norm{f \chi}_{L^1_{3+\gamma+2s} \cap L^2}}
   \norm{\Fl{K}}^2_{L^2_x L^2_{\gamma/2}}  
   +    C_\ell \vpran{1 + \delta_0}
   \norm{\Fl{K}}^2_{L^2_{x,v}}    \nn
\\
& \quad \,
   + C_\ell \vpran{1+\delta_0} 
   \vpran{\sup_x \norm{f}_{L^1_{1+\gamma}}}
   \norm{\Fl{K}}_{L^1_{x} L^1_\gamma}
  + C_\ell \vpran{1 + \sup_x \norm{f \chi}_{L^1_{3+\gamma+2s} \cap L^2}} \norm{\Fl{\ell}}_{H^{s_1}_{\gamma_1/2}}^2,
\end{align}
where $s_1, \gamma_1$ are defined in~\eqref{def:s-1-gamma-1} with $s_1 < s$ and $\gamma_1 < \gamma$. If we impose that 
\begin{align} \label{cond:delta-0-nonlinear-level}
    \delta_0 < \min \left\{1, \ \tfrac{1}{2C} \right\}, 
\end{align}
then
\begin{align} \label{bound:tilde-T-1}
   \tilde T_1
\leq
   - \frac{1}{4} \gamma_0 \norm{\Fl{K}}^2_{L^2_x L^2_{\gamma/2}}
   + C_\ell \norm{\Fl{K}}^2_{L^2_{x,v}}
   + C_\ell \vpran{\sup_x \norm{f}_{L^1_{1+\gamma}}}
   \norm{\Fl{K}}_{L^1_{x} L^1_\gamma}
   + C_\ell \norm{\Fl{\ell}}_{H^{s_1}_{\gamma_1/2}}^2.
\end{align}

\smallskip
\noindent
Next we estimate $\tilde T_2$ by writing it as
\begin{align} \label{bound:tilde-T-2}
   \tilde T_2
& = \iiiint_{\T^3 \times \R^6 \times \Ss^2}
       (\mu_\ast + f_\ast \chi_\ast) 
       \tfrac{\mu \vint{v}^\ell + K}{\vint{v}^\ell}
       \vpran{\Fl{K}(v') \vint{v'}^{\ell} - \Fl{K} \vint{v}^{\ell}} 
       b(\cos\theta) |v - v_\ast|^\gamma \dbmu   \nn
\\
& = \iiiint_{\T^3 \times \R^6 \times \Ss^2}
       (\mu_\ast + f_\ast \chi_\ast) 
       \vpran{\mu \vint{v}^\ell + K}
       \vpran{\Fl{K}(v') - \Fl{K}(v) } 
       b(\cos\theta) |v - v_\ast|^\gamma \dbmu   \nn
\\
& \quad \, 
   +  \iiiint_{\T^3 \times \R^6 \times \Ss^2}
        (\mu_\ast + f_\ast \chi_\ast)
        \Fl{K}(v') \frac{\mu \vint{v}^\ell + K}{\vint{v}^\ell}
       \vpran{\vint{v'}^\ell - \vint{v}^\ell \cos^\ell \tfrac{\theta}{2}} 
       b(\cos\theta) |v - v_\ast|^\gamma \dbmu   \nn
\\
& \quad \,
   -  \iiiint_{\T^3 \times \R^6 \times \Ss^2}
        (\mu_\ast + f_\ast \chi_\ast)
      \vpran{\mu \vint{v}^\ell + K}\Fl{K}(v') 
       \vpran{1 - \cos^\ell \tfrac{\theta}{2}} 
       b(\cos\theta) |v - v_\ast|^\gamma \dbmu   \nn
\\
& \Denote 
   \tilde T_{2, 1} + \tilde T_{2, 2} + \tilde T_{2, 3}.
\end{align}
By \eqref{ineq:trilinear-1-f-minus} in Proposition~\ref{prop:commutator} and~\eqref{bound:cancellation-3} in Proposition~\ref{prop:strong-sing-cancellation}, we have
\begin{align*}
   \tilde T_{2,2}
&\leq
   C_\ell \vpran{1 + \sup_x \norm{f \chi}_{L^1_{4+\gamma}}}
   \norm{\Fl{K}}_{L^1_x L^1_\gamma}
   + C_\ell (1 + \delta_0) (1 + K) \norm{\Fl{K}}_{L^1_x L^1_\gamma}
\\
&\leq
   C_\ell (1 + K) \norm{\Fl{K}}_{L^1_x L^1_\gamma}.
\end{align*}
The third term $\tilde T_{2,3}$ is directly bounded as
\begin{align*}
  \abs{\tilde T_{2,3}}  
\leq  
  C_\ell (1 + K) \norm{\Fl{K}}_{L^1_{x,v}}.
\end{align*}
In order to bound $\tilde T_{2,1}$, we use~\eqref{bound:L-infty-Q-F-mu-weight} and a regular change of variables to obtain that
\begin{align*}
   \tilde T_{2,1}
&= \iiiint_{\T^3 \times \R^6 \times \Ss^2}
       (\mu_\ast + f_\ast \chi_\ast) 
       \mu \vint{v}^\ell
       \vpran{\Fl{K}(v') - \Fl{K}(v) } 
       b(\cos\theta) |v - v_\ast|^\gamma \dbmu
\\
& \quad \, 
   + K \iiiint_{\T^3 \times \R^6 \times \Ss^2}
       (\mu_\ast + f_\ast \chi_\ast) 
       \vpran{\Fl{K}(v') - \Fl{K}(v) } 
       b(\cos\theta) |v - v_\ast|^\gamma \dbmu
\\
& \leq
   \iint_{T^3 \times \R^3} Q(\mu + f \chi, \, \mu \vint{v}^\ell) \Fl{K} \dv\dx
    + C K \norm{\Fl{K}}_{L^1_{x} L^1_\gamma}
\\
& \leq
   C (1 + K) \norm{\Fl{K}}_{L^1_{x} L^1_\gamma}. 
\end{align*}
Overall we have
\begin{align*}
   \tilde T_2
\leq
   C_\ell (1 + K) \norm{\Fl{K}}_{L^1_{x} L^1_\gamma}.
\end{align*}
Combining the estimates for $\tilde T_1$ and $\tilde T_2$, we obtain the first bound for the right-hand side as
\begin{align} \label{bound:RHS-1-1}
   \int_{\T^3} \int_{\R^3} Q(F, F) \Fl{K} \vint{v}^{\ell} \dv\dx
&\leq
   -\frac{1}{4} \gamma_0 
   \norm{\Fl{K}}^2_{L^2_x L^2_{\gamma/2}}  \nn 
   + C_\ell \vpran{\sup_x \norm{f}_{L^1_{1+\gamma}}}
   \norm{\Fl{K}}_{L^1_{x} L^1_\gamma}
\\
& \quad \,
   + C_\ell \norm{\Fl{\ell}}_{H^{s_1}_{\gamma_1/2}}^2
   + C_\ell (1 + K) \norm{\Fl{K}}_{L^1_{x} L^1_\gamma}.
\end{align}
%

Next we derive the second bound with the $H^s$-norm. To this end, we only need to re-estimate $\tilde T_1$ as
\begin{align} \label{bound:RHS-2-2}
  \tilde T_1
& \leq
    \iiiint_{T^3 \times \R^6 \times \Ss^2}
        (\mu_\ast + f_\ast \chi_\ast)
        \Fl{K} \frac{1}{\vint{v}^\ell} \vpran{\Fl{K}(v') \vint{v'}^{\ell} - \Fl{K} \vint{v}^{\ell}} 
        b(\cos\theta) |v - v_\ast|^\gamma \dbmu  \nn
\\
& \leq
    \iiiint_{T^3 \times \R^6 \times \Ss^2}
        (\mu_\ast + f_\ast \chi_\ast)
        \Fl{K} \vpran{\Fl{K}(v') - \Fl{K} } 
        b(\cos\theta) |v - v_\ast|^\gamma \dbmu   \nn
\\
& \quad \,
  +     \iiiint_{T^3 \times \R^6 \times \Ss^2}
        (\mu_\ast + f_\ast \chi_\ast)
        \Fl{K} \Fl{K}(v') \frac{1}{\vint{v}^\ell}\vpran{\vint{v'}^\ell -\vint{v}^\ell \cos^\ell \tfrac{\theta}{2}}               
        b(\cos\theta) |v - v_\ast|^\gamma \dbmu  \nn
\\
& \quad \,
  +     \iiiint_{T^3 \times \R^6 \times \Ss^2}
        (\mu_\ast + f_\ast \chi_\ast)
        \Fl{K} \Fl{K}(v') \vpran{1 - \cos^\ell \tfrac{\theta}{2}}               
        b(\cos\theta) |v - v_\ast|^\gamma \dbmu   \nn
\\
& \leq
   - \frac{c_0}{2} \norm{\Fl{K}}^2_{L^2_x H^s_{\gamma/2}}
   + C_\ell \vpran{\sup_x \norm{f}_{L^1_{1+\gamma}}}
   \norm{\Fl{K}}_{L^1_x L^1_\gamma}
   + C_\ell \norm{\Fl{K}}^2_{L^2_x L^2_{\gamma/2}} \,.
\end{align}
Let $\Eps_3$ be a constant such that
\begin{align} \label{bound:Eps-2-2}
   C_\ell \Eps_3
\leq 
  c_0/8.
\end{align}
The desired bound in~\eqref{est:level-set-2} is obtained by multiplying~\eqref{bound:RHS-1-1} by a small enough $\Eps_3$,  
adding it to~\eqref{bound:RHS-1-1} and then interpolating $L^2_x H^{s_1}_{\gamma_1/2}$ between $L^2_x H^s_{\gamma/2}$ and $L^2_{x,v}$. 
\end{proof}

\begin{rmk}
Although in the proof of Proposition~\ref{thm:L2-level-set-nonlinear} it seems that the cutoff function plays an essential role in removing the $\vint{v}^\ell$-dependence in the $L^\infty_x$-norm, the above estimates in fact hold (with some modifications) even when we treat the original Boltzmann operator $Q(F, F)$. There are two ways to achieve this goal: first, if $f$ is the solution to the modified equation obtained in Theorem~\ref{thm:local-exist-MBE} and $\ell = k_0$ (which is the case when we apply Proposition~\ref{thm:L2-level-set-nonlinear} in the later analysis), we can use the $L^\infty_{t,x,v}$-bound of $f$ with a lower weight $k_0 - \ell_0 - \gamma - 6$. Then the majority of the weight can be transferred to the first component of $Q(F, F)$. Thus it eliminates the need for a high moment in the $L^\infty_x$ term. The second way is even more general, in the sense that we do not need any a prior $L^\infty$-bound on $f$. Instead we make use of the nonlinear structure and decompose the first entry in $Q(F, F)$ into 
\begin{align*}
   F = \mu + \vpran{f - K/\vint{v}^\ell} + K/\vint{v}^\ell,
\end{align*}
and allocate all the $\vint{v}^\ell$ to such term and bound it using $\Fl{\ell}$. The price to pay here is to have an extra $K$ in the coefficient in the upper bound. It does not generate any essential problem since $K$ is the upper bound of $f$ which will eventually be small. However, it is more in line with the linear estimates to have homogeneity in $K$. Hence we opt to use the special structure of $\chi$ in the proof of Proposition~\ref{thm:L2-level-set-nonlinear}. 
\end{rmk}

\subsubsection{Level Estimate for $-f$} Similar as the linear case, we need to show that not only $f \vint{v}^\ell < \delta_0$ but also
\begin{align*}
   - f \vint{v}^\ell < \delta_0.
\end{align*}
Hence we establish the counterpart estimates for the level set of $-f$.

\begin{prop} \label{thm:level-set-minus-f-nonlinear}
Let $h = -f$. Suppose $F = \mu - h \geq 0$. Suppose $k_0, \delta_0$ in the definition of $\chi$ in~\eqref{def:chi} satisfy that $k_0 > 8 + \gamma$ and $\delta_0$ small enough such that~\eqref{cond:delta-0-nonlinear-level} holds and
\begin{align*} 
    \inf_{t,x} \norm{\mu - h \chi(\vint{v}^{k_0} h)}_{L^1_{v}} \geq D_0 > 0, 
\qquad
    \sup_{t,x} \norm{\mu - h \chi(\vint{v}^{k_0} h)}_{L^1_2 \cap L \log L} < E_0 < \infty\,.
\end{align*}
Then for any $s \in (0, 1)$ and $8 + \gamma < \ell \leq k_0$, 
the nonlinear estimate has the form
\begin{align} \label{est:level-set-2-minus-f}
 &    - \int_{\T^3} \int_{\R^3} Q(\mu - h \chi, F) \Hl{K} \vint{v}^{\ell} \dv\dx \nn
\\
& \hspace{1cm} \leq
   -\frac{c_0 \Eps_3}{4}
   \norm{\Fl{K}}^2_{L^2_x H^s_{\gamma/2}}  
   + C_\ell \vpran{\sup_x \norm{f}_{L^1_{1+\gamma}}}
   \norm{\Fl{K}}_{L^1_{x} L^1_\gamma}
   + C_\ell (1 + K) \norm{\Fl{K}}_{L^1_{x} L^1_\gamma},
\end{align}
where $\Eps_3$ is the same constant in~\eqref{bound:Eps-2-2}.
\end{prop}
\begin{proof}
Decompose the term of interest in a similar way as in~\eqref{decomp:nonlinear}:
\begin{align*}
  - \int_{\T^3} \int_{\R^3} Q(\mu - h \chi, F) \Hl{K} \vint{v}^{\ell} \dv\dx
&= \int_{\T^3} \int_{\R^3} Q \vpran{\mu - h \chi, h - \tfrac{K}{\vint{v}^\ell}} \Hl{K} \vint{v}^{\ell} \dv\dx  \nn
\\
& \quad \,
+ \int_{\T^3} \int_{\R^3} Q \vpran{\mu - h \chi, \tfrac{K}{\vint{v}^\ell} - \mu} \Hl{K} \vint{v}^{\ell} \dv\dx. 
\end{align*}
Since we have 
\begin{align*}
    \mu - h \chi \geq 0, 
\qquad
    - h \vint{v}^{k_0} \chi(\vint{v}^{k_0} h) \leq \delta_0,
\end{align*}
the same estimates in~\eqref{decomp:nonlinear} and~\eqref{bound:T-1-first} apply to obtain~\eqref{est:level-set-2-minus-f}. 
\end{proof}

\subsubsection{Level-set estimate for $L^1$-norm of the collisional operator: Quadratic version}

\begin{prop}\label{T1-nonlinear}
Let $f$ be a solution to equation \eqref{eq:MBE-nonlinear} and denote $F = \mu + f$.  Then, for any $T > 0$ and 
\begin{align*}
   s\in(0,1),
\quad
    \epsilon \geq 0,
\quad 
   0 \leq j < k_0 - 5 - \gamma,
\quad 
   8 + \gamma < \ell \leq k_0,
\quad 
   \kappa > 2, 
\quad
   K > 0,
\end{align*} 
it holds that
\begin{align}\label{Qlevelaverage}
& \quad \,
    \int_0^T \int_{\T^3} \int_{\R^3} \abs{\vint{v}^{j}(1 - \Delta_{v})^{-\kappa/2}\big( \tilde{Q}(\mu + f \chi, F)\,\vint{v}^{\ell}\,\Fl{K} \big)} \dv\dx\dt \nn
\\
&\leq 
  C\,\| \vint{v}^{j/2} \Fl{K}(0) \|^{2}_{L^{2}_{x,v}}
  + C_\ell \norm{\Fl{K}}_{L^2_{t,x} L^2_j}^2 
  + C_\ell \norm{\Fl{K}}_{L^2_{t,x} H^{s}_{\gamma/2}}^2  \nn
\\
& \quad \,
  + C_\ell \norm{\Fl{K}}_{L^2_{t,x} L^2_{j+\gamma/2+1}}^2
  + C_\ell \vpran{1 + K + \sup_{t,x} \norm{f}_{L^1_{1+\gamma}}}
   \norm{\Fl{K}}_{L^1_{t,x} L^1_{j+\gamma}},
\end{align}
where the coefficients $C, C_\ell$ are independent of $T$ and recall that 
\begin{align*}
   \tilde Q(\mu + f \chi, F) = Q(\mu + f \chi, F) + \Eps L_\alpha F.
\end{align*}
Furthermore, an identical estimate holds if $\Fl{K}$ is replaced by $(-f)^{(\ell)}_{K,+}$.
\end{prop}

\begin{proof}
The proof is a slight modification of that of Proposition~\ref{T1}. We only need to show the bound of 
\begin{align*}
   \int_0^T \int_{\T^3} \int_{\R^3} Q(\mu + f \chi,F)\,\vint{v}^{\ell} \Fl{K}W_{K} \dv\dx\dt, 
\end{align*}
with the aim to remove the $\ell$-moment dependence in the $\sup_x$-norm in~\eqref{Qlevelaverage-2}. The definition of $W_K$ is in~\eqref{def:W-K}. Similar to the linear case, write
\begin{align} \label{decomp:CalQ-quad}
 \CalQ^{quad}:= \int_{\T^3} \int_{\R^3} Q(\mu + f \chi, F) \vint{v}^{\ell} \Fl{K} W_K \dv\dx
&= \int_{\T^3} \int_{\R^3} Q \vpran{\mu + f \chi, f - \tfrac{K}{\vint{v}^\ell}} \vint{v}^{\ell} \Fl{K} W_K \dv\dx  \nn
\\
& \quad \,
+ \int_{\T^3} \int_{\R^3} Q \vpran{\mu + f \chi, \mu + \tfrac{K}{\vint{v}^\ell}} \vint{v}^{\ell} \Fl{K} W_K \dv\dx  \nn
\\
& \Denote \tilde T_1^+ + \tilde T_2^+.  
\end{align}
Decompose the upper bound of the first term $\tilde T_1^+$ similarly as in~\eqref{decomp:T-1-plus-linear} with $G$ replaced by $\mu+f \chi$:
\begin{align} \label{bound:T-1-plus-1}
  \tilde T_1^+
\leq 
   \tilde T_{1, 1}^+ + \tilde T_{1, 2}^+ + \tilde T_{1, 3}^+. 
\end{align}
The estimate for $\tilde T_{1, 3}^+$ remains the same as for  $T_{1,3}^+$ in Proposition~\ref{T1}, which gives
\begin{align*}
   \tilde T_{1, 3}^+
&\leq
  C\vpran{1 + \sup_x\norm{f \chi}_{L^1_v}}
  \norm{\Fl{K}}_{L^2_x L^2_{j+\gamma/2 + 1}}^2
  + C\vpran{1 + \sup_x\norm{f \chi}_{L^1_{j+2+\gamma}}}
  \norm{\Fl{K}}_{L^2_x L^2_{j/2}}^2
\\
&\leq
  C \norm{\Fl{K}}_{L^2_x L^2_{j+\gamma/2 + 1}}^2, 
\qquad
   \text{provided} \,\, j < k_0 - 5 - \gamma. 
\end{align*}

Similar to the estimates of $\tilde T_1$ in~\eqref{bound:T-1-first},  the bounds for $\tilde T_{1, 1}^+$ and $\tilde T_{1, 2}^+$ follow from the regular change of variables together with~\eqref{ineq:trilinear-1-f-minus} in Proposition~\ref{prop:commutator} and Proposition~\ref{prop:strong-sing-cancellation}, which has the form
\begin{align*} 
& \tilde T_{1,1}^+ + \tilde T_{1,2}^+   \nn
\\
&  
  = \tfrac{1}{2} \iiiint_{\T^3 \times \R^6 \times \Ss^2}
      (\mu_\ast + f_\ast \chi_\ast) \vpran{\vpran{\Fl{K}(v')}^2 W_K(v') \cos^{2\ell} \tfrac{\theta}{2}
     - \vpran{\Fl{K}}^2 W_K} 
        b(\cos\theta) |v - v_\ast|^\gamma \dbmu   \nn
\\
& 
\quad \,
  + \iiiint_{\T^3 \times \R^6 \times \Ss^2}
       (\mu_\ast + f_\ast \chi_\ast)  \frac{\Fl{K}(v)}{\vint{v}^\ell}  \Fl{K}(v') W_K(v')
        \vpran{\vint{v'}^{\ell} - \vint{v}^{\ell} \cos^\ell \tfrac{\theta}{2}} 
        b(\cos\theta) |v - v_\ast|^\gamma \dbmu   \nn
\\
&\leq
  -\frac{1}{2}c_0 \vpran{1 - C \delta_0}
   \norm{\Fl{K} \sqrt{W_K}}^2_{L^2_x L^2_{\gamma/2}}  
   +    C_\ell \vpran{1 + \delta_0}
   \norm{\Fl{K}}_{L^2_{x,v}}
    \norm{\Fl{K} W_K}_{L^2_{x,v}}    \nn
\\
& \quad \,
   + C_\ell \vpran{1+\delta_0} 
   \vpran{\sup_x \norm{f}_{L^1_{1+\gamma}}}
   \norm{\Fl{K} W_K}_{L^1_{x} L^1_\gamma}
  + C_\ell \vpran{1 + \delta_0} \norm{\Fl{K}}_{H^{s_1}_{\gamma_1/2}}
  \norm{\Fl{K} W_K}_{L^2_{\gamma/2}}.
\end{align*}
Inserting the definition of $W_K$, we get
\begin{align*}
   \tilde T_{1,1}^+ + \tilde T_{1,2}^+
&\leq
  C_\ell \norm{\Fl{K}}_{L^2_x L^2_j}^2 
  + C_\ell \norm{\Fl{K}}_{L^2_x H^{s}_{\gamma/2}}^2  \nn
  + C_\ell \norm{\Fl{K}}_{L^2_x L^2_{j+\gamma/2}}^2
\\
& \quad \,
  + C_\ell \vpran{\sup_x \norm{f}_{L^1_{1+\gamma}}}
   \norm{\Fl{K}}_{L^1_{x} L^1_{j+\gamma}}. \nn
\end{align*}
Combining the estimates for $\tilde T_{1,1}^+, \tilde T_{1,2}^+, \tilde T_{1,3}^+$, we have
\begin{align} 
   \tilde T_1^+
&\leq
    C_\ell \norm{\Fl{K}}_{L^2_x L^2_j}^2 
  + C_\ell \norm{\Fl{K}}_{L^2_x H^{s}_{\gamma/2}}^2  \nn
  + C_\ell \norm{\Fl{K}}_{L^2_x L^2_{j+\gamma/2+1}}^2
  + C_\ell \vpran{\sup_x \norm{f}_{L^1_{1+\gamma}}}
   \norm{\Fl{K}}_{L^1_{x} L^1_{j+\gamma}}.
\end{align}
The estimate for $\tilde T_2^+$ is similar to those for $\tilde T_2$ in~\eqref{bound:tilde-T-2} with $\Fl{K}$ replaced by $\Fl{K} W_K$. This gives
\begin{align*}
  \tilde T_2^+ 
\leq
  C_\ell (1 + K) \norm{\Fl{K} W_K}_{L^1_{x} L^1_\gamma}
\leq
  C_\ell (1 + K) \norm{\Fl{K}}_{L^1_{x} L^1_{j+\gamma}}.
\end{align*}
The bound of $\CalQ^{quad}$ is the combination of the bounds of $\tilde T_1^+, \tilde T_2^+$, which writes
\begin{align*}
    \CalQ^{quad}
&\leq
    C_\ell \norm{\Fl{K}}_{L^2_x L^2_j}^2 
  + C_\ell \norm{\Fl{K}}_{L^2_x H^{s}_{\gamma/2}}^2  \nn
  + C_\ell \norm{\Fl{K}}_{L^2_x L^2_{j+\gamma/2+1}}^2
\\
& \quad \,
  + C_\ell \vpran{1 + K + \sup_x \norm{f}_{L^1_{1+\gamma}}}
   \norm{\Fl{K}}_{L^1_{x} L^1_{j+\gamma}}.
\end{align*}
The regularizing term $L_\alpha$ is bounded in the same way as in~\eqref{bound:T-plus-R}, which will be absorbed into the estimate for $\CalQ^{quad}$. Combining the estimates for $\CalQ^{quad}$ and $L_\alpha$ and integrating in $t$ gives~\eqref{Qlevelaverage}. 
\end{proof}

\subsubsection{Time-space-velocity energy functional: Quadratic version.}

Now we establish the key iterative inequality for the quadratic case which is the counterpart of Proposition~\ref{thm:SV-energy-functional-linear}. 
\begin{prop}[Energy functional interpolation inequality]\label{SV-energy-functional-nonlinear}
Let $T > 0$ and $\ell_0 > 0$ be the same weight as in Theorem~\ref{thm:linear-local} (precise statement in~\eqref{cond:ell-0-2}).  Let $8 + \gamma < \ell \leq k_0$. Suppose $f$ is a solution to~\eqref{eq:MBE-nonlinear} which satisfies
\begin{align*} 
  \sup_{t,x} \norm{f}_{L^1_{1+\gamma}} \leq \delta_0,
\qquad
   \sup_{ t }\| \vint{v}^{\ell_0+\ell}f(t,\cdot,\cdot) \|_{ L^{1}_{x,v} } \leq C, 
\end{align*}
where $\delta_0$ satisfies the smallness condition in~\eqref{cond:delta-0-nonlinear-level}. Furthermore, suppose that 
\begin{align*} 
    \inf_{t,x} \norm{\mu + f \chi(\vint{v}^{k_0} f)}_{L^1_{v}} \geq D_0 > 0, 
\qquad
    \sup_{t,x} \norm{\mu + f \chi(\vint{v}^{k_0} f)}_{L^1_2 \cap L \log L} < E_0 < \infty\,.
\end{align*}
Then there exist $p, s''$ such that for any $0 \leq T_1 < T_2 \leq T$, $\epsilon\in[0,1]$, $\alpha\geq0$ and $0 < M < K$, 
\begin{align}\label{key-estimate}
\begin{split}
& \qquad \, 
  \|\Fl{K} (T_2)\|^{2}_{L^{2}_{x,v}} + \int^{T_2}_{T_1}\| \vint{v}^{\gamma/2}(1 -\Delta_{v})^{\frac {s}{2} }\Fl{K}(\tau)\|^{2}_{L^{2}_{x,v}} {\rm d}\tau
\\
& \hspace{2cm} + \frac{1}{C} \bigg(\int^{T_2}_{T_1} \big\| (1-\Delta_{x})^{\frac {s''}{2}}\big(\Fl{K}\big)^{2} \big\|^{p}_{L^{p}_{x,v}}{\rm d}\tau\bigg)^{\frac{1}{p}} 
\\
& \leq 
   2\| \vint{v}^{2} \Fl{K}(T_1) \|^{2}_{L^{2}_{x,v}} 
   + \| \vint{v}^{2}\Fl{K}(T_1) \|^{2}_{L^{2p}_{x,v}} 
   + \frac{CK}{K-M}\sum^{4}_{i=1}\frac{\CalE_{p}(M,T_1, T_2)^{\beta_{i}}}{(K-M)^{a_i}},
\end{split}
\end{align}
for constants $c_0:=c(\ell,s,\gamma)$ and $C:=C(\ell,s,\gamma,\alpha)$. In particular, $C$ does not depend on $T_1, T_2, T$. The parameters $s'', p, \beta_i, a_i$ depend on $\ell, \gamma, s$ in the same way as in 
Proposition~\ref{thm:SV-energy-functional-linear}. 

\smallskip
\noindent
Furthermore, the estimate holds for $(-f)$ 
with $\Fl{K}$ replaced by $(-f)^{(\ell)}_{K,+}$.
\end{prop}

\begin{proof}
By replacing Propositions~\ref{thm:L2-level-set} and~\ref{T1} with Proposition~\ref{thm:L2-level-set-nonlinear} and Proposition~\ref{T1-nonlinear}, the proof is the same as that of Proposition~\ref{thm:SV-energy-functional-linear}.  
\end{proof}


\subsubsection{Baseline level $\mathcal{E}_0$ and level set iteration: Quadratic case.} 

Similar to Proposition~\ref{prop:bound-CalE-0}, we now show the boundedness of the baseline case $\CalE_0$ which prepares the ground for the $\CalE_k$-iteration. 

\begin{prop} \label{prop:initial-E0-nonlinear}
Suppose $s \in (0,1)$, $T > 0$ and $\frac{37+5\gamma}{2} < \ell \leq k_0$. Suppose $f$ is a solution to equation~\eqref{eq:MBE-nonlinear} and \eqref{Q:cond:coercivity}, \eqref{smallness} hold. Then the baseline energy functional $\CalE_0$ defined in~\eqref{def:CalE-0} satisfies
\begin{equation}\label{bound:initial-E0-nonlinear}
  \CalE_{0} 
\leq 
  C_\ell e^{C_\ell \,T} \max_{j \in\{1/p, \, p'/p\} } \vpran{\norm{\vint{\cdot}^{\ell} f_0}^{2j}_{L^{2}_{x,v}} + \Eps^{2 j} T^{j}},
\qquad
   p' = p/(2-p).
\end{equation}
\end{prop}
\begin{proof}
The proof follows a similar line as for Proposition~\ref{prop:bound-CalE-0}. We only need to replace the linear energy estimate in Corollary \ref{cor:bi-cor-basic-energy-estimate-linear} with its nonlinear counterpart in Proposition~\ref{quadratic-zero-level}. The proof of the $x$-regularizing term in Proposition~\ref{prop:bound-CalE-0} applies directly since it holds for general functions rather than merely solutions to any equation.
\end{proof}

The $L^\infty_{k_0}$-bound now follows:
\begin{prop} \label{thm:L-infty-k-0-nonlinear}
Let $T > 0$ and let $f(t, \cdot, \cdot)$ be a solution to~\eqref{eq:MBE-nonlinear} with 
\begin{align*}
   k = k_0 + \ell_0 + 2, 
\qquad
   s \in (0, 1),
\qquad
   t \in [0, T].
\end{align*} Let $\ell_0$ be the weight in Theorem~\ref{thm:linear-local} (precise statement in~\eqref{cond:ell-0-2}). Suppose
\begin{align*}
   k_0 > \max\{\ell_0 + 2 \gamma + 2s + 13, \ \tfrac{37+5\gamma}{2}\}.
\end{align*}
Moreover, suppose 
\begin{align} \label{assump:lower-L-infty-1}
   \norm{\vint{v}^{k_0} f_0}_{L^\infty_{x,v}} \leq \delta_\ast,
\qquad
  \norm{\vint{v}^{k_0 + \ell_0 + 2} f_0}_{L^2_{x,v}} < \infty,
\qquad
   \norm{\vint{v}^{k_0 - \ell_0 - 7 - \gamma} f}_{L^\infty_{t, x,v}} < \delta_0.
\end{align}
For any $T < 1$, if $\delta_0$ satisfying the assumptions for Theorem~\ref{thm:linear-local}, Proposition~\ref{quadratic-zero-level} and Proposition~\ref{thm:L2-level-set-nonlinear} (more precisely, ~\eqref{cond:delta-0-1}, ~\eqref{assump:delta-0-2-linear}, ~\eqref{smallness} and~\eqref{cond:delta-0-nonlinear-level}) and $\delta_\ast, \Eps$ are chosen small enough (which depends on $\delta_0$), then we have
\begin{align} \label{bound:L-infty-nonlinear-delta-0-1}
   \norm{\vint{v}^{k_0} f}_{L^\infty_{t, x,v}} < \delta_0. 
\end{align}
The smallness of $\delta_\ast$ is independent of $\Eps$.
\end{prop}
\begin{proof}
Take $\ell = k_0$ in Proposition~\ref{SV-energy-functional-nonlinear} and Proposition~\ref{prop:initial-E0-nonlinear} and we only need to show that the assumptions in these two propositions hold. First, by the $L^\infty$-bound in~\eqref{assump:lower-L-infty-1}, we have
\begin{align} \label{bound:L-infty-small}
   & \hspace{4cm}
  \sup_{t,x} \norm{f}_{L^1_{3+\gamma+2s} \cap L^2} \leq \delta_0,
\\
  & \inf_{t,x} \norm{\mu + f \chi}_{L^1_{v}} 
\geq 
   \norm{\mu}_{L^1_v} - \norm{\vint{v}^{-4}}_{L^1_v} \norm{\vint{v}^4 f \chi}_{L^\infty_{t,x,v}} 
\geq 
 8 \pi \vpran{\tfrac{1}{8\pi} - \delta_0}  
> 0,    \nn
\end{align}
and
\begin{align*}
  \sup_{t,x} \vpran{\norm{F}_{L^1_2} + \norm{F}_{L\log L}} 
 <  \sup_{t,x} \vpran{\norm{\mu}_{L^1_2} + \norm{\mu}_{L\log L}} 
     + \sup_{t,x} \vpran{\norm{f \chi}_{L^1_2} + \norm{f \chi}_{L\log L}} 
< C_0 (1 + \delta_0),
\end{align*}
since $k_0 - \ell_0 - 7 - \gamma > 6 + \gamma + 2s$. We are left to show that
\begin{align} \label{bound:finite-L-2-eventual}
    \sup_t \norm{\vint{v}^{\ell_0 + k_0} f}_{L^1_{x,v}} < \infty.
\end{align}
To this end, we apply~\eqref{bound:energy-basic-nonlinear} in Proposition~\ref{quadratic-zero-level} and get
\begin{align*}
   \sup_t \norm{\vint{v}^{\ell_0 + k_0} f}_{L^1_{x,v}}
\leq
   C \sup_t \norm{\vint{v}^{\ell_0 + k_0+2} f}_{L^2_{x,v}}
\leq
  C_T \vpran{1 + \norm{\vint{v}^{\ell_0 + k_0+2} f_0}_{L^2_{x,v}}}
< \infty.
\end{align*}
Note that for~\eqref{bound:energy-basic-nonlinear} to hold, we only need the bound in~\eqref{bound:L-infty-small}. In particular, the weight in~\eqref{bound:L-infty-small} is independent of $k_0$, which again marks the essential difference between the linear equation and the nonlinear one. Combining Proposition~\ref{SV-energy-functional-nonlinear} and Proposition~\ref{prop:initial-E0-nonlinear} with the same argument in Theorem~\ref{central-bilinear-T}, we obtain that
\begin{align} \label{bound:Linfty-interm}
   \sup_{t} \norm{\vint{v}^{k_0} f}_{L^{\infty}_{x,v}} 
\leq 
   \max \Big\{ 2 \|\vint{v}^{k_0} f_0\|_{L^{\infty}_{x,v}}, \  K^{quad}_{0}(\CalE_0) \Big\}\,,
\end{align}
where
\begin{align*}
   K^{quad}_{0}(\CalE_{0}) 
=  C_{k_0}e^{C_{k_0}\,T}\max_{1 \leq i \leq 4}\max_{j\in\{1/p,p'/p\} }\Big(\| \vint{\cdot}^{k_0} f_0 \|^{2 j}_{L^{2}_{x,v}} + \epsilon^{2 j} T^{j}\Big)^{\frac{\beta_i - 1}{a_i}},
\qquad
   p' = p/(2-p).
\end{align*}
Hence for any $0 < T < 1$, \eqref{bound:L-infty-nonlinear-delta-0-1} holds by taking $\delta_\ast$ and $\Eps$ small enough.
\end{proof}

\begin{rmk}
The result in Proposition~\ref{thm:L-infty-k-0-nonlinear} applies to the full range of singularity where $s \in (0, 1)$. This provides the basis for extending the well-posedness result below from the mild to the strong singularity. 
\end{rmk}

To pass in the limit in $\Eps$, we need to show that the time interval of existence is independent of $\Eps$. To this end, we need to find an explicit relation between the smallness of the initial data and the solution, that is, the relation between $\delta_\ast$ and $\delta_0$. This relation is derived from~\eqref{bound:Linfty-interm} by setting
\begin{align*}
   \delta_0 
\geq \max \Big\{ 2 \|\vint{v}^{k_0} f_0\|_{L^{\infty}_{x,v}}, \  K^{quad}_{0}(\CalE_0) \Big\}. 
\end{align*}
Take $T < 1$. Since $\Eps_\ast < 1$, $\delta_0 < 1$, $2/p > 1$ and $2p'/p >1$, we get
\begin{align*}
   K^{quad}_{0}(\CalE_0)
\leq
  C_{k_0} e^{C_{k_0}} \max_{1 \leq i \leq 4} \vpran{\delta_\ast + \Eps_\ast}^{\frac{\beta_i - 1}{a_i}}
= C_{k_0} e^{C_{k_0}} \vpran{\delta_\ast + \Eps_\ast}^{\eta_0},
\qquad
   \eta_0  = \min_{1 \leq i \leq 4} \frac{\beta_i - 1}{a_i}.
\end{align*}
Hence, we set 
\begin{align*}
   \delta_\ast 
< \min \left\{\tfrac{1}{2} \delta_0, \  \frac{1}{2 (C_{k_0} e^{C_{k_0}})^{1/\eta_0}} \delta_0^{\frac{1}{\eta_0}} \right\}.
\end{align*}
Denote the function
\begin{align} \label{def:frakH}
    \mathfrak{H} 
= \mathfrak{H}(x) 
= \frac{1}{4} \min \left\{x, \  \frac{1}{(C_{k_0} e^{C_{k_0}})^{1/\eta_0}} x^{\frac{1}{\eta_0}} \right\}. 
\end{align}
Then $\mathfrak{H}$ is invertible on $[0, 1]$. With this setup we have the following corollary of Proposition~\ref{thm:L-infty-k-0-nonlinear}:
\begin{cor} \label{cor:local-well-posedness}
Let $0 < T < 1$ and let $f(t, \cdot, \cdot)$ be a solution to~\eqref{eq:MBE-nonlinear} with $k = k_0 + \ell_0 + 2$, $s \in (0, 1)$ and $t \in [0, T]$. Let $\ell_0$ be the weight in Theorem~\ref{thm:linear-local} (precise statement in~\eqref{cond:ell-0-2}). Suppose
\begin{align*}
   k_0 > \max\{\ell_0 + 2 \gamma + 2s + 13, \ \tfrac{37+5\gamma}{2}\}.
\end{align*}
Suppose $\delta_0$ satisfies the same bounds~\eqref{cond:delta-0-1}, ~\eqref{assump:delta-0-2-linear}, ~\eqref{smallness} and~\eqref{cond:delta-0-nonlinear-level} as in Theorem~\ref{thm:L-infty-k-0-nonlinear}. Let
\begin{align} \label{def:Eps-00}
   \delta_\ast = \mathfrak{H}^{-1}(\delta_0/2), 
\end{align}
where $\mathfrak{H}$ is defined in~\eqref{def:frakH}. 
Moreover, suppose 
\begin{align} \label{assump:lower-L-infty}
   \norm{\vint{v}^{k_0} f_0}_{L^\infty_{x,v}} \leq \delta_{\ast},
\qquad
  \norm{\vint{v}^{k_0 + \ell_0 + 2} f_0}_{L^2_{x,v}} < \infty,
\qquad
   \norm{\vint{v}^{k_0 - \ell_0 - 7 - \gamma} f}_{L^\infty_{t, x,v}} < \delta_0.
\end{align}
Then it holds that
\begin{align} \label{bound:L-infty-nonlinear-delta-0}
   \norm{\vint{v}^{k_0} f}_{L^\infty_{t, x,v}} \leq \delta_0/2 < \delta_0. 
\end{align}
\end{cor}

We summarize the above results and state the local well-posedness of the regularized equation~\eqref{eq:reg-Boltzmann}.

\begin{thm} \label{thm:local-well-posedness}
Suppose $s \in (0, 1/2)$ and 
\begin{align*}
  k_0 > \max \left\{\ell_0 + 15 + 2\gamma, \ \ell_0 + 10 + 2\alpha + \gamma, \tfrac{37+2\gamma}{2} \right\}, 
\qquad
   \alpha > 2 \ell_0 + 2 \gamma + 2s + 11, 
\end{align*}
where $\ell_0$ is the same weight in Theorem~\ref{thm:linear-local}, 
which only depends on $s$ (precise statement in~\eqref{cond:ell-0-2}). Suppose 
\begin{align*}
     \norm{\vint{v}^{k_0} f_0}_{L^\infty_{x,v}} \leq \delta_{\ast},
\qquad
  \norm{\vint{v}^{k_0 + \ell_0 + 2} f_0}_{L^2_{x,v}} < \infty.
\end{align*}
with $\delta_\ast$ defined in~\eqref{def:Eps-00} and $\delta_0$ satisfying the same bounds as in Theorem~\ref{thm:L-infty-k-0-nonlinear}. Then for any $T < 1$, there exists $\Eps_\ast$ such that for any $\Eps \leq \Eps_\ast$, equation~\eqref{eq:reg-Boltzmann} has a solution $f$
satisfying 
\begin{align} \label{bound:L-infty-lower-f-MBE-1}
   \norm{\vint{v}^{k_0} f}_{L^\infty{\vpran{[0, T) \times \T^3 \times \R^3}}} \leq \delta_0/2 < \delta_0.
\end{align}
\end{thm}
\begin{proof}
Take $k = k_0 + \ell_0 + 2$ in Theorem~\ref{thm:local-exist-MBE}. Then the combination of Theorem~\ref{thm:local-exist-MBE} and Proposition~\ref{thm:L-infty-k-0-nonlinear} shows that there exists $T_\Eps$, which may depend on $\Eps$, such that~\eqref{eq:reg-Boltzmann} has a solution $f$ which satisfies~\eqref{bound:L-infty-lower-f-MBE-1}. We claim that such $T_\Eps$ can be extended to $T$ independent of $\Eps$. Indeed, by Corollary~\ref{cor:local-well-posedness} and Theorem~\ref{thm:local-exist-MBE} we first extend $T_\Eps$ to $\tilde T_\Eps$, where $\tilde T_\Eps$ is the largest interval such that
\begin{align*}
   \norm{\vint{v}^{k_0} f}_{L^\infty{\vpran{[0, \tilde T_\Eps) \times \T^3 \times \R^3}}}  < \delta_0.
\end{align*}
Such a bound, together with the basic $L^2$-estimate in Proposition~\ref{quadratic-zero-level}, the $L^2$-level-set estimate in Proposition~\ref{thm:L2-level-set-nonlinear} and the $L^\infty$-estimate in Proposition~\ref{thm:L-infty-k-0-nonlinear} that are all independent of $\Eps$, gives
\begin{align*}
   \norm{\vint{v}^{k_0} f}_{L^\infty{\vpran{[0, \tilde T_\Eps) \times \T^3 \times \R^3}}} \leq \delta_0/2 < \delta_0.
\end{align*}
Hence $\tilde T_\Eps$ can be continued to the maximal interval $[0, T)$ for any $T < 1$.
\end{proof}

We are ready to pass to the limit and obtain a local solution to the original Boltzmann equation~\eqref{eq:Nonlinear-Boltzmann-intro}. 

\begin{thm} \label{thm:local-nonlinear}
Suppose $s \in (0, 1/2)$ and 
\begin{align*}
  k_0 > 5\ell_0 + 32 + 5\gamma+4s, 
\end{align*}
where $\ell_0$ is the same weight in Theorem~\ref{thm:linear-local} (precise statement in~\eqref{cond:ell-0-2}). 
Suppose $f_0$ satisfies
\begin{align*}
     \norm{\vint{v}^{k_0} f_0}_{L^\infty_{x,v}} \leq \delta_{\ast},
\qquad
  \norm{\vint{v}^{k_0 + \ell_0 + 2} f_0}_{L^2_{x,v}} < \infty,
\end{align*}
where $\delta_{\ast}$ is defined in~\eqref{def:Eps-00} with $\delta_0$ satisfying the same bounds as in Theorem~\ref{thm:L-infty-k-0-nonlinear}. Then for any $T < 1$, the nonlinear Boltzmann equation~\eqref{eq:Nonlinear-Boltzmann-intro} has a solution $f \in L^\infty(0, T; L^2_{k_0+\ell_0+2}(\T^3 \times \R^3))$. Moreover, $f$ satisfies the bound
\begin{align} \label{bound:L-infty-lower-f-MBE}
   \norm{\vint{v}^{k_0} f}_{L^\infty{\vpran{[0, T_0] \times \T^3 \times \R^3}}} \leq \delta_0/2 < \delta_0.
\end{align}
\end{thm}
\begin{proof}
Denote $f^\Eps$ as the local solution to~\eqref{eq:reg-Boltzmann}. 
By Proposition~\ref{quadratic-zero-level} and \eqref{bound:L-infty-lower-f-MBE-1}, we obtain the uniform-in-$\Eps$ bound of $f^\Eps$ in the following space:
\begin{align*}
   L^\infty_{k_0}((0, T) \times \T^3 \times \R^3) 
\cap L^\infty(0, T; L^2_{k_0+\ell_0+2}(\T^3 \times \R^3)) 
\cap  H^{s'}((0, T) \times \T^3; H^s_{k_0+\ell_0+2}(\R^3))),
\quad
  s' < \tfrac{s}{2(s+3)}.
\end{align*}
We can extract a subsequence, still denoted as $f^\Eps$ such that
\begin{align} \label{bound:unif-L-infty-f-Eps}
   \norm{f^\Eps}_{L^\infty_{k_0}((0, T) \times \T^3 \times \R^3) } \leq \delta_0/2.
\end{align}
By the uniform polynomial decay and a diagonal argument, we have
\begin{align} \label{convg:strong-L-p}
   f^\Eps \to f
\quad 
  \text{strongly in $L^2_{t, x} L^2_{\ell_0 + 2} ((0, T) \times \T^3 \times \R^3)$}.
\end{align}
Such strong convergence then implies the convergence of $Q(f^\Eps, f^\Eps)$ to $Q(f, f)$ as distributions. Indeed, if we take $\phi \in C^\infty_c(\R^3_v)$ as a test function, then
\begin{align*}
& \quad \,
   \norm{\int_{\R^3}
     Q(f^\Eps, f^\Eps) \phi(v) \dv 
   - \int_{\R^3} Q(f, f) \phi(v) \dv}_{L^2_{t, x}}
\\
& = \norm{\iiint_{\R^6 \times \Ss^2} b(\cos\theta)
               \vpran{f^\Eps(v_\ast)f^\Eps(v) - f(v_\ast)f(v)}
               \vpran{\phi(v') - \phi(v)} |v - v_\ast|^\gamma \dsigma \dv_\ast \dv}_{L^2_{t, x}}
\\
& \leq
    \norm{\nabla_v \phi}_{L^\infty_v}
    \norm{\iint_{\R^3 \times \R^3} \abs{f^\Eps(v_\ast)f^\Eps(v) - f(v_\ast)f(v)} |v-v_\ast|^\gamma \dv_\ast\dv}_{L^2_{t, x}}
\\
& \leq
    C \norm{\nabla_v \phi}_{L^\infty_v}
    \norm{\iint_{\R^3 \times \R^3} \abs{f^\Eps(v_\ast) - f(v_\ast)} |f(v)| |v-v_\ast|^\gamma \dv_\ast\dv}_{L^2_{t, x}}
\\
& \quad \,
+ C \norm{\nabla_v \phi}_{L^\infty_v}
    \norm{\iint_{\R^3 \times \R^3} \abs{f^\Eps(v) - f(v)} |f(v_\ast)| |v-v_\ast|^\gamma \dv_\ast\dv}_{L^2_{t, x}}
\\
& \leq
  C \norm{\nabla_v \phi}_{L^\infty_v} \vpran{\sup_{t, x}\norm{f^\Eps}_{L^\infty_\gamma}}
  \norm{f^\Eps - f}_{L^2_{t, x}L^2_\gamma}
\to 0
\qquad \text{as $\Eps \to 0$.}
\end{align*}
Therefore we obtain a solution $f$ to the nonlinear Boltzmann equation~\eqref{eq:Nonlinear-Boltzmann-intro},
where $f$ lives in the space
\begin{gather*}
      L^\infty_{k_0}((0, T) \times \T^3 \times \R^3) 
\cap L^\infty(0, T; L^2_k(\T^3 \times \R^3)) 
\cap  H^{s'}((0, T) \times \T^3; H^s_{k_0 + \ell_0 + 2}(\R^3))).  \qedhere
\end{gather*}
\end{proof}


\section{Nonlinear Global Theory} \label{sec:nonlinear-global}

%
In this section we extend the local-in-time result of the previous section to global, thus proving the main theorem for the weakly singular case. The key step is to use the spectral theory of the linearised Boltzmann operator for the hard potential case. 

\smallskip
\noindent
In the sequel $\CalL$ stands for the operator
\begin{equation*}
\CalL f = Q(\mu,f) + Q(f,\mu) - v\cdot\nabla_{x}f\,.
\end{equation*} 
The nonlinear Boltzmann equation is recast as
\begin{equation}\label{BE}
\partial_{t}f = \CalL f + Q(f,f), \qquad (t,x,v)\in (0,T)\times\T^3\times\R^{3}.
\end{equation}
We recall the consequence of the spectral property of $\CalL$ shown in Theorem 5.8 in \cite{AMSY}:
\begin{thm}[\cite{AMSY}] \label{thm:weighted-H-s-L-2}
Let $h$ be the solution to the linear equation
\begin{align*}
   \del_t h = \CalL h, 
\qquad
   h |_{t=0} = h^{in},
\end{align*}
where $h^{in}$ has zero mass, momentum and energy. 
Let $\ell > \frac{5 \gamma + 37}{2}$ so that the spectral gap of $\CalL$ (Theorem 4.4 in \cite{AMSY}) holds. Then there exists $T_0 > 0$  such that
\begin{align} \label{reg:T-0-L-2}
   \int^{T_0}_0 
   \norm{\vint{v}^{\ell}  f(t, \cdot, \cdot)}_{L^2_{x,v}}^2 \dt
\leq
  C \norm{(I - \Delta_v)^{-s/2} \vpran{\vint{v}^\ell h^{in}}}_{L^2_{x,v}}^2,
\end{align}
and for any $t \geq T_0$,
\begin{align} \label{reg:T-0-pointwise}
    \norm{\vint{v}^{\ell}  f(t, \cdot, \cdot)}_{L^2_{x,v}}
\leq
    C\vpran{\frac{1}{\sqrt{T_0}} + 1} e^{-\lambda t} 
    \norm{(I - \Delta_v)^{-s/2} \vpran{\vint{v}^\ell h^{in}}}_{L^2_{x,v}}.
\end{align}
Here $\lambda >0$ is the same decay rate as in the spectral gap estimate in Theorem 4.4 in \cite{AMSY}.  
\end{thm}

Using Theorem~\ref{thm:weighted-H-s-L-2} we show a lemma which is an intermediate step in establishing the global $L^2$-bound. 
\begin{lem}\label{semi-source}
Assume that $h \in L^{2}\vpran{0,T;L^{2}_x L^2_\ell}$ has zero total mass, momentum, and energy:
\begin{equation*}
\int_{\T^3}\int_{\R^{3}} 
h(t,x,v)
\left(\begin{array}{c}
1 \\ v \\ |v|^{2}
\end{array}\right)\dv\dx =0.
\end{equation*}
Then, for any $s\in(0,1)$, $\ell\geq \frac{5 \gamma + 37}{2}$ and $t > 0$, it follows that
\begin{equation*}
   \int^{t}_{0} \norm{\vint{v}^{\ell} \int^{w}_{0} e^{\CalL (w - \tau)} h(\tau)  \dtau}^{2}_{L^{2}_{x,v}} \dw 
\leq 
  C(1+\lambda^{-2})\int^{t}_{0} \big\|\vint{v}^{\ell} h(\tau) \big\|^{2}_{L^{2}_{x}H^{-s}_{v}}\dtau,
\end{equation*} 
where $\lambda>0$ is the spectral gap of $\CalL$ in $L^{2}_x L^2_\ell$.
\end{lem}
\begin{proof} Assume first that $t\leq T_{0}$ with $T_{0}$ defined in Theorem~\ref{thm:weighted-H-s-L-2}. Then
\begin{align}\label{semi-source-e1}
\begin{split}
  \int^{t}_{0} \norm{\vint{v}^{\ell} \int^{w}_{0} e^{\CalL(w - \tau)} h(\tau) \dtau}^{2}_{L^{2}_{x,v}} \dw
&\leq 
   T_0\int^{t}_{0} \int^{w}_{0} \norm{\vint{v}^{\ell} e^{\CalL(w - \tau)}h(\tau)}^{2}_{L^{2}_{x,v}} \dtau \, \dw
\\
&\leq 
  T_0\int^{t}_{0} \int^{\tau + T_0}_{\tau} \big\|\vint{v}^{\ell}  e^{\CalL(w - \tau)} h(\tau) \big\|^{2}_{L^{2}_{x,v}} \dw \, \dtau
\\
&\leq 
   C\,T_{0} \int^{t}_{0}\big\|(1 - \Delta_{v})^{-s/2}\vint{v}^{\ell} h(\tau) \big\|^{2}_{L^{2}_{x,v}}  \dtau ,
\end{split}
\end{align}
where for the latter inequality we used the time invariance of the semigroup and~\eqref{reg:T-0-L-2}.
For the case $t>T_0$ split the integration as
\begin{align*}
  \int^{t}_{0}  \norm{\vint{v}^{\ell} \int^{w}_{0} e^{\CalL(w - \tau)} h(\tau)  \dtau}^{2}_{L^{2}_{x,v}} \dw 
= \vpran{\int^{T_0}_{0} + \int^{t}_{T_0}} 
   \norm{\vint{v}^{\ell}\int^{w}_{0} e^{\CalL(w - \tau)} h(\tau)  \dtau}^{2}_{L^{2}_{x,v}} \dw.
\end{align*}
The integral in $(0,T_0)$ falls into the previous case.  For the interval $(T_0,t)$ one has
\begin{align}\label{semi-source-e2}
   \int^{t}_{T_0} \norm{\vint{v}^{\ell} \int^{w}_{0} e^{\CalL(w - \tau)} F(\tau)  \dtau}^{2}_{L^{2}_{x,v}} \dw 
 =  \int^{t}_{T_0} \norm{\vint{v}^{\ell} \bigg(\int^{w - T_0}_{0}  +  \int^{w}_{w - T_0}\bigg) e^{\CalL(w - \tau)}F(\tau)  \dtau}^{2}_{L^{2}_{x,v}} \dw.
\end{align}
Note that for the first integral in the right side of \eqref{semi-source-e2} one has that $w - \tau \geq T_0$. Invoking~\eqref{reg:T-0-pointwise} one has
\begin{equation*}
  \norm{\vint{v}^{\ell} e^{\CalL(w - \tau)}F(\tau)}_{L^{2}_{x,v}} 
\leq 
   C \vpran{\frac{1}{\sqrt{T_0}} + 1} e^{-\lambda(w-\tau)}
   \big\|(1 - \Delta_{v})^{-s/2} \vint{v}^{\ell} F(\tau) \big\|^{2}_{L^{2}_{x,v}},
\end{equation*} 
where $\lambda>0$ is the spectral gap of $\CalL$.  As a consequence,
\begin{align*}
 \int^{t}_{T_0} \norm{\vint{v}^{\ell} \int^{w - T_0}_{0}  e^{\CalL(w - \tau)}  F(\tau)  \dtau}^{2}_{L^{2}_{x,v}} \dw 
 &\leq  \int^{t}_{T_0} \bigg(\int^{w - T_0}_{0} \big\|\vint{v}^{\ell} e^{\CalL(w - \tau)}F(\tau) \big\|_{L^{2}_{x,v}} \dtau\bigg)^{2}  \dw
 \\
& \hspace{-2cm} \leq 
   C \vpran{\frac{1}{T_0} + 1} \int^{t}_{T_0} \bigg(\int^{w - T_0}_{0}e^{-\lambda(w-\tau)} \big\|(1 - \Delta_{v})^{-s/2} \vint{v}^{\ell} F(\tau) \big\|_{L^{2}_{x,v}} \dtau\bigg)^{2}  \dw
\\
& \hspace{-2cm} \leq 
  \frac{C}{\lambda} \vpran{\frac{1}{T_0} + 1} \int^{t}_{T_0} \int^{w - T_0}_{0}e^{-\lambda(w-\tau)} \big\|(1 - \Delta_{v})^{-s/2} \vint{v}^{\ell} F(\tau) \big\|^{2}_{L^{2}_{x,v}} \dtau\,  \dw
\\
& \hspace{-2cm} \leq 
  \frac{C}{\lambda^{2}} \vpran{\frac{1}{T_0} + 1} \int^{t}_{0} \big\|(1 - \Delta_{v})^{-s/2} \vint{v}^{\ell} F(\tau) \big\|^{2}_{L^{2}_{x,v}} \dtau,
\end{align*}
where we have used the Cauchy-Schwarz inequality and changed the order of integration for the last two steps.  Finally, for the latter integral in \eqref{semi-source-e2} one simply has that
\begin{align*}
  &\int^{t}_{T_0}\Big\| \vint{v}^{\ell} \int^{w}_{w - T_0} e^{\CalL(w - \tau)} F(\tau)  \dtau \Big\|^{2}_{L^{2}_{x,v}} \dw 
\leq 
  T_0 \int^{t}_{T_0} \int^{w}_{w - T_0} \big\|\vint{v}^{\ell} e^{\CalL(w - \tau)} F(\tau) \big\|^{2}_{L^{2}_{x,v}} \dtau \, \dw
\\
&\quad\leq  
  T_{0} \int^{t}_{0} \int^{\tau + T_0}_{\tau} \big\| \vint{v}^{\ell} e^{\CalL(w - \tau)} F(\tau) \big\|^{2}_{L^{2}_{x,v}} \dw \, \dtau 
\leq 
   C\,T_{0}\int^{t}_{0} \big\| (1 - \Delta_{v})^{-s/2}\vint{v}^{\ell}F(\tau) \big\|^{2}_{L^{2}_{x,v}} \dtau.
\end{align*}
Overall, we conclude for the case $t>T_0$ that
\begin{equation}\label{semi-source-e3}
  \int^{t}_{0} \norm{\vint{v}^{\ell} \int^{w}_{0} e^{\CalL(w - \tau)}F(\tau)  \dtau}^{2}_{L^{2}_{x,v}} \dw 
\leq 
  C_{T_0} \int^{t}_{0} \big\| (1 - \Delta_{v})^{-s/2}\vint{v}^{\ell}F(\tau) \big\|^{2}_{L^{2}_{x,v}} \dtau.
\end{equation}
Estimates \eqref{semi-source-e1} and \eqref{semi-source-e3} prove the theorem.
\end{proof}

\begin{prop}\label{TT2}
Let $F=\mathcal{\mu} + f\geq0$ be a solution of the Boltzmann equation \eqref{BE}. Assume that
\begin{equation*}
     \sup_{t,x}\| f \|_{L^{1}_{\ell}\cap L^2}\leq \delta_0,
\qquad 
   \|\vint{\cdot}^{\ell}f_0\|_{L^{2}_{x,v}}<+\infty,
\end{equation*}
with $\ell, \delta_0$ satisfying~\eqref{smallness}. In addition, suppose 
\begin{align*}
   \ell \geq 5 \gamma + 37.
\end{align*}
Then it follows that  
\begin{equation}\label{TT2-e1}
  \norm{\vint{\cdot}^{\ell} f(t)}^{2}_{L^{2}_{x,v}} \leq C\,\| \vint{\cdot}^{\ell} f_0 \|^{2}_{L^{2}_{x,v}}\,e^{-\lambda'\,t},\qquad t\in[0,T),
\end{equation}
for a constant $C:=C_{\ell}(\lambda')$.  The time relaxation rate $\lambda' \in (0,\lambda]$ where $\lambda>0$ is the spectral gap in $L^{2}_x L^2_\ell$ of the linearised Boltzmann operator.  Furthermore,
\begin{align}\label{TT2-e2}
  \int^{T}_{0}\big\| (1 - \Delta_{x})^{s'/2} & f \big\|^{2}_{L^{2}_{x,v}}\dtau 
 + \int^{T}_{0}\big\| \vint{v}^{\ell+\gamma/2} (1 -\Delta_{v})^{s/2}f  \big\|^{2}_{L^{2}_{x,v}}\dtau 
\leq 
   C\,\| \vint{\cdot}^{\ell} f_0 \|^{2}_{L^{2}_{x,v}}, 
\end{align}
for a constant $C:=C_{\ell}(\lambda')$.  All constants are independent of $T>0$.
\end{prop}
\begin{proof}
Set $g(t) = e^{\lambda' t}f(t) $ with $\lambda'>0$ to be chosen. Then $g$ satisfies
\begin{align*}
   \del_t g + v \cdot \nabla_x g 
= e^{\lambda' t} \vpran{Q(\mu + f, f) + Q(f, \mu)}
   + \lambda' g,
\qquad
   g = e^{\lambda' t} f. 
\end{align*}
Since $\ell, \delta_0$ satisfy~\eqref{smallness}, by multiplying estimate \eqref{Q:ineq:energy-basic-1} (with $\Eps=0$) by $e^{2 \lambda' t}$, we get
\begin{align*}
   \frac{\rm d}{\dt} \norm{\vint{v}^{\ell}g}^{2}_{L^{2}_{x,v}} 
   + \vpran{\frac{\gamma_0}{8}-\lambda'}
        \int_{\T^3} \norm{\vint{v}^{\ell}g}^{2}_{L^{2}_{\gamma/2}} \dx  
    + c_2 \int_{\T^3} \norm{\vint{v}^{\ell}g}^{2}_{H^{s}_{\gamma/2}} \dx  
\leq 
   C\int_{ \T^3 } \big\| g  \big\|^{2}_{L^{2}} \dx,
\end{align*}
where $c_2=\frac{c_0\delta_{8}}{8}$.  Hence, integrating in time, one gets
\begin{align}\label{TT2-e2.1}
  & \norm{\vint{v}^{\ell}g(t)}^{2}_{L^{2}_{x,v}} 
      + \vpran{\frac{\gamma_0}{8}-\lambda'}
          \int^{t}_{0}\int_{\T^3} \norm{\vint{v}^{\ell}g}^{2}_{L^{2}_{\gamma/2}}   
          \dx\dtau 
      + c_2 \int^{t}_{0}\int_{\T^3} \norm{\vint{v}^{\ell}g}^{2}_{H^{s}_{\gamma/2}} \dx\dtau  \nn
\\
& \hspace{2cm} \leq 
   \norm{\vint{v}^{\ell}f_0}^{2}_{L^{2}_{x,v}} + C\int^{t}_{0}\int_{ \T^3 } \big\| g  \big\|^{2}_{L^{2}} \dx\dtau. 
\end{align}
Let us estimate the right side of \eqref{TT2-e2.1}. The equation for $g$ can also be viewed as
\begin{equation*}
\frac{\text{d}g}{\text{d}t} = \big(\CalL+\lambda'\,I\big)g  + Q(f,g) =: \tilde{\CalL}g + Q(f,g).
\end{equation*}
Then, we can write
\begin{equation}\label{small-e1}
g(t) = e^{\tilde{\CalL}t}f_0 + \int^{t}_{0}e^{\tilde{\CalL}(t-\tau)}Q\big( f(\tau) , g(\tau) \big)\text{d}\tau.
\end{equation}
By \cite[Theorem 4.4]{AMSY}, the operator $\CalL$ has an spectral gap $\lambda$ in $L^{2}_x L^2_{\ell/2}$, provided 
\begin{align*}
   \frac{\ell}{2} > \frac{5\gamma + 37}{2}. 
\end{align*}
Then,
\begin{align}\label{small-e2}
   \norm{\vint{\cdot}^{\ell/2}e^{\tilde{\CalL}t}f_0}_{L^{2}_{x,v}} 
\leq 
   C\,e^{-(\lambda - \lambda') t}\| \vint{\cdot}^{\ell/2}f_0 \|_{L^{2}_{x,v}}.
\end{align}
Furthermore, $Q\big( f(t) , g(t) \big)$ has total zero mass, momentum, and energy for all $t\in(0,T)$.  Then, Lemma \ref{semi-source} implies that for any $\lambda'\in(0,\lambda)$ it holds that
\begin{align}\label{small-e3}
   \int^{t}_{0} \norm{\vint{v}^{\ell/2} \int^{w}_{0} e^{\tilde{\CalL}(w - \tau)} Q\big( f(\tau) , g(\tau) \big) \dtau}^{2}_{L^{2}_{x,v}} \dw
\leq 
  C\int^{t}_{0} \norm{\vint{v}^{\ell/2} Q\vpran{f(\tau) , g(\tau)}}^{2}_{L^{2}_x H^{-s}_v}\dtau.
\end{align} 
By Proposition~\ref{prop:trilinear}, it follows that
\begin{align*}
& \quad \,
  \norm{(1 - \Delta_{v})^{-s/2}\vint{v}^{\ell/2}  Q \vpran{f(\tau) , g(\tau)}}^{2}_{L^{2}_{x,v}}
\\
& \qquad \quad \leq 
  \vpran{\|\vint{\cdot}^{\ell/2+\gamma+2s} f(\tau)\|_{L^{1}_{v}} + \| f(\tau) \|_{L^{2}_{v}}}^{2}
   \norm{\vint{\cdot}^{\ell/2+\gamma+2s} g(\tau)}^{2}_{H^{s}_{v}} \leq \delta^{2}_0\| \vint{\cdot}^{\ell} g(\tau) \|^{2}_{H^{s}_{v}}.
\end{align*}
Consequently, from \eqref{small-e1}, \eqref{small-e2}, and \eqref{small-e3} one is led to
\begin{align}\label{small-e4}
   \int^{t}_{0}\int_{ \T^3 } \norm{g}^{2}_{L^{2}} \dx\dtau 
\leq 
   C\, \norm{\vint{\cdot}^{\ell}f_0}^{2}_{L^{2}_{x,v}} 
   + C\delta^{2}_0\int^{t}_0\| \vint{\cdot}^{\ell} g(\tau) \|^{2}_{H^{s}_{v}} \dtau.
\end{align}
Take $\lambda'<\min\{\frac{\gamma_{0}}{8},\lambda\}$ and use estimate \eqref{small-e4} in estimate \eqref{TT2-e2.1} to conclude that
\begin{align} \label{bound:g-basic}
  \norm{\vint{v}^{\ell} g(t)}^{2}_{L^{2}_{x,v}} 
   + \vpran{c_2 - C\delta^{2}_{0}}
      \int^{t}_{0}\int_{\T^3} \norm{\vint{v}^{\ell}g(\tau)}^{2}_{H^{s}_{\gamma/2}} \dx \dtau 
\leq 
   C\, \norm{\vint{v}^{\ell} f_{0}}^{2}_{L^{2}_{x,v}}. 
\end{align}
Choose $\delta_{0}>0$ such that 
\begin{align} \label{cond:delta-0-sec-7}
   \sqrt{\frac{c_2}{C}} \geq \delta_{0}.
\end{align}
Then~\eqref{bound:g-basic} leads to
\begin{equation*}
\big\| \vint{v}^{\ell} f(t) \big\|_{L^{2}_{x,v}} \leq C\big\| \vint{v}^{\ell} f_{0} \big\|_{L^{2}_{x,v}}e^{-\lambda'\,t},\qquad t\in[0,T).
\end{equation*}
Plugging this estimate in \eqref{TT2-e2.1} and \eqref{Q:bound:velocity-avg-basic} (with $\Eps = 0$), one obtains~\eqref{TT2-e2} and concludes the proof.
\end{proof}

We now have all the ingredients to show the main theorem for the weak singularity and it states
\begin{thm}[Global Existence] \label{thm:global-mild}
Let $s \in (0, 1/2)$ and $\gamma \in (0, 1)$. Suppose $\delta_0$ is a constant small enough such that bounds in Theorem~\ref{thm:L-infty-k-0-nonlinear} and~\eqref{cond:delta-0-sec-7} are 
satisfied. 
Let $\ell_0$ be the same weight in Theorem~\ref{thm:linear-local} and $k_0$ be a constant satisfying
\begin{align*}
  k_0 > 5\ell_0 + 35 + 5\gamma + 4s.
\end{align*}
Let $\delta_\ast^\natural$, defined in~\eqref{def:frakH-Eps-ast}, be the constant measuring the smallness of the data. Suppose the initial data $f_0$ has zero mass, momentum and energy and satisfies
\begin{align} \label{assump:main-f-0}
   \norm{\vint{v}^{k_0} f_0}_{L^\infty_{x, v} \cap L^2_{x,v}} < \delta_{\ast}^\natural, 
\qquad
   \norm{\vint{v}^{k_0 + \ell_0 + 2} f_0}_{L^2_{x, v}} < \infty. 
\end{align}
Then the Boltzmann equation~\eqref{eq:Nonlinear-Boltzmann-intro} has a solution $f \in L^\infty(0, \infty; L^2_x L^2_{k_0 + \ell_0 + 2}(\T^3 \times \R^3))$. Moreover, $f$ satisfies
\begin{align*}
  \norm{\vint{v}^{k_0} f}_{L^\infty(0, \infty; \T^3 \times \R^3)}
\leq
  \delta_0/2 < \delta_0. 
\end{align*}
\end{thm}
\begin{proof}
The reason that Theorem~\ref{thm:local-well-posedness} (or Theorem~\ref{thm:local-nonlinear}) can only treat a short-time existence is because that the bound in~\eqref{bound:Linfty-interm} (with $\Eps = 0$) relies on $T$. It will exceed $\delta_0$ if $T$ is large, which will render the $L^2$-estimates invalid. Such dependence of $T$ is through $K_0^{quad}(\CalE_0)$ since $\CalE_0$ grows with $T$ (see \eqref{bound:initial-E0-nonlinear}) when the spectral gap is not used.  
Equipped now with Proposition~\ref{TT2} we can replace Proposition~\ref{quadratic-zero-level} in the proof of~\eqref{bound:initial-E0-nonlinear} with~\eqref{TT2-e1} and~\eqref{TT2-e2} to get 
\begin{align*}
   \CalE_0 
\leq 
   C_{k_0} \max_{j \in\{1/p, \, p'/p\} } \norm{\vint{\cdot}^{k_0} f_0}^{2j}_{L^{2}_{x,v}}
\leq 
   C_{k_0} \norm{\vint{\cdot}^{k_0} f_0}_{L^{2}_{x,v}}. 
\end{align*}
As a result, there exists $C_{k_0}$ independent of $T$ such that
\begin{align*}
  K_0^{quad}(\CalE_0)
\leq
  C_{k_0} \norm{\vint{\cdot}^{k_0} f_0}_{L^{2}_{x,v}}^{\eta_0}, 
\qquad
  \eta_0  = \min_{1 \leq i \leq 4} \frac{\beta_i - 1}{a_i}.
\end{align*}
Similarly as in~\eqref{def:frakH} and~\eqref{def:Eps-00}, define 
\begin{align} \label{def:frakH-Eps-ast}
    \mathfrak{H}_\ast 
= \mathfrak{H}_\ast(x) 
= \frac{1}{4} \min \left\{x, \  \frac{1}{C_{k_0}^{1/\eta_0}} x^{\frac{1}{\eta_0}} \right\}, 
\qquad
  \delta_\ast^\natural = \mathfrak{H}_\ast^{-1}(\delta_0/2).
\end{align}
%
Under the smallness assumption in~\eqref{assump:main-f-0}, we obtain in the same way as in Theorem~\ref{thm:local-well-posedness} that 
\begin{align*}
  \norm{\vint{v}^{k_0} f}_{L^\infty(0, T; \T^3 \times \R^3)}
\leq
  \delta_0/2 < \delta_0,
\qquad
  \text{for all $T > 0$. }
\end{align*}
This shows for any $T > 0$, the solution can be extended beyond $T$, thus giving the global existence. 
\end{proof}


\section{Strong Singularity} \label{Sec:strong-singularity}
In this part we show the existence for the nonlinear Boltzmann equation with a strong singularity. The only reason we have to restrict to the weak singularity in Sections~\ref{Sec:nonlinear-local}-~\ref{sec:nonlinear-global} is because that in the construction of solutions in Theorem~\ref{thm:local-exist-MBE}, when using the fixed-point argument in~\eqref{reason-mild-sing}, the regularizing term $\Eps L_\alpha$ needs to be used to control the $H^{2s}$-norm. All the a priori estimates are performed for the full range of $s \in (0, 1)$.

To circumvent the difficulty mentioned above when constructing approximate solutions in the strong singularity case, our strategy is to smooth the collision kernel into a weakly singular one and repeat the process in Sections~\ref{Sec:a-priori-linear}-\ref{sec:nonlinear-global} to find approximate solutions uniformly bounded in the smoothing parameter $\eta$. Note that we can as well simply regularize the kernel into a cutoff one by removing all its singularities. But that will require introducing new estimates for cutoff kernels. Since all the tools are available for the weakly singular kernel in the previous sections, we take a weak-singularity smoothing. The weak singularity itself will not play an essential role. 

Without extra means, we will not be able to obtain uniform bounds in $\eta$. This is because the regularity gained by part of the collision term cannot compensate, uniformly in $\eta$, the loss of derivatives in the rest of the terms. Consequently, many estimates will not close for the nonlinear Boltzmann operator with the smoothed collision term. To overcome this difficulty, 
we temporarily add a dissipation term $\Eps L_\alpha$ as in the previous sections and will remove it after obtaining a local well-posedness for the nonlinear Boltzmann equation (with $\Eps L_\alpha$) with the strong singularity. 

Recall the original Boltzmann equation with a strong singularity $s \in [1/2, 1)$:
\begin{align} \label{eq:BE-strong}
   \del_t f + v \cdot \nabla_x f = Q(\mu + f, \mu + f), 
\qquad
   f \big|_{t=0} = f_0(x, v),    
\end{align}
whose collision kernel satisfies
\begin{align*}
   b(\cos\theta) \sim \frac{1}{\theta^{2+2s}}, 
\quad
  \text{for $\theta$ near $0$ and $s \in [1/2, 1)$}. 
\end{align*}
Fix $s_\ast \in (0, 1/2)$ such that
\begin{align} \label{def:s-ast}
   2s - 2 s_\ast < 1. 
\end{align}
For any $\eta \in (0, 1)$, let $Q_\eta$ be the approximate operator with the collision kernel
\begin{align} \label{bound:kernel-essential}
   \frac{\alpha_0}{\theta^{2+2 s_\ast}}  
\leq 
  b_\eta(\cos\theta)
  = \frac{b(\cos\theta) \theta^{2+2s}}{\theta^{2+2s_\ast} (\theta + \eta)^{2s-2s_\ast}}
\leq b(\cos\theta).
\end{align}
Although the lower bound will not be used in the subsequent proof, we note that the coefficient $\alpha_0$ is independent of $\eta$ since
\begin{align*}
   \frac{1}{(\theta + \eta)^{2s - 2s_\ast}} 
\geq
   \frac{1}{(\pi + 1)^{2s - 2s_\ast}}, 
\qquad
  \text{for all $\theta \in (0, \pi)$ and $\eta \in (0, 1)$}.
\end{align*}
The uniform upper bound in~\eqref{bound:kernel-essential} is the key for uniform estimates in $\eta$. Consider the regularized 
equation
\begin{align} \label{eq:BE-eta}
   \del_t f_\eta + v \cdot \nabla_x f_\eta 
 = \Eps L_\alpha f_\eta + Q_\eta(\mu + f_\eta, \mu + f_\eta),
\qquad
   f_\eta \big|_{t=0} = f_0(x, v),
\end{align}
where $\Eps \in (0, 1)$ and $L_\alpha$ is the same operator as in~\eqref{def:L-alpha}. 
First we note that due to the uniform bounds in~\eqref{bound:kernel-essential}, the constant in the trilinear estimate 
is independent of $\eta$. This is summarized as
\begin{lem} \label{lem:trilinear-coerc-unif}
Let $b$ be the original collision kernel with $s \in [1/2, 1)$ and $b_\eta$ be the one defined in~\eqref{bound:kernel-essential}. Then there exists $C$ independent of $\eta$ such that
\begin{align} \label{ineq:trilinear-Q-eta}
   \abs{\int_{\R^3} Q_\eta(f, g) h  \dv}
\leq
  C \vpran{\norm{f}_{L^1_{\vpran{m-\gamma/2}^+ + \gamma + 2s} \cap L^2}} \norm{g}_{H^{s-\sigma}_{\gamma/2 + 2s + m}} \norm{h}_{H^{s+\sigma}_{\gamma/2 - m}}
\end{align}
for any $\sigma \in [\min\{s-1, -s\}, s]$, $m \in \R$, $\gamma \geq 0$ and $0 < s < 1$.
\end{lem}

As mentioned at the beginning of this section, our plan is to repeat the process of proving the well-posedness of~\eqref{eq:reg-Boltzmann}, with the goal to obtain a local existence result for~\eqref{eq:BE-eta} over a time interval uniform in $\eta$. 
We show that the sequence of intermediate results from Proposition~\ref{bilinear-zero-level} to Corollary~\ref{cor:local-well-posedness} can be modified (with indispensable help from $\Eps L_\alpha$) in the way that their coefficients are all independent of $\eta$. The main idea is that in all these estimates, we only rely on the upper bound of the collision kernel with no further structures required. 
We start with the modified equation with the cutoff function in~\eqref{def:chi} (with its solution still denoted as $f_\eta$):
\begin{align} \label{eq:MBE-Strong-eta}
      \del_t f_\eta + v \cdot \nabla_x f_\eta 
 = \Eps L_\alpha (\mu + f_\eta) 
    + Q_\eta(\mu + f_\eta \chi(\vint{v}^{k_0} f_\eta), \mu + f_\eta),
\qquad
   f_\eta \big|_{t=0} = f_0(x, v).    
\end{align}
Its linearized version is 
\begin{align} \label{eq:MBE-Strong-linear}
      \del_t f_\eta + v \cdot \nabla_x f_\eta 
 = \Eps L_\alpha (\mu + f_\eta) + Q_\eta(\mu + g \chi(\vint{v}^{k_0} g), \mu + f_\eta)
=: \tilde Q_\eta(\mu + g \chi, \mu + f_\eta),
\end{align}
with the initial $f_\eta \big|_{t=0} = f_0(x, v)$.
Choices of weights remain the same as in the previous sections. 
We will show the details for the basic energy estimates for the linearized  equation to illustrate how to use~\eqref{bound:kernel-essential} to derive uniform-in-$\eta$ bounds. The rest of the steps are parallel to those in Sections~\ref{Sec:a-priori-linear}-\ref{sec:nonlinear-global} and their details will be either sketched or omitted. The regularization $\Eps L_\alpha$ helps to simplify the estimates, since for each fixed $\Eps$, the gain of velocity regularity (and subsequently the hypoellipticity) now comes from $\Eps L_\alpha$ instead of $Q$.

\begin{prop} \label{prop:L-2-MBE-strong}
Suppose $G = \mu + g \geq 0$ and $\delta_0$ in the cutoff function is small enough such that 
$G_\chi = \mu + g \chi \geq 0$ satisfies
\begin{align} \label{cond:coercivity}
    \inf_{t, x} \norm{G_\chi}_{L^1_{v}} \geq D_0 > 0, 
\qquad
   \sup_{t, x} \vpran{\norm{G_\chi}_{L^1_2} + \norm{G_\chi}_{L\log L}} 
   < E_0 < \infty.
\end{align}
Suppose $s \in [1/2, 1)$. Let $F_\eta = \mu + f_\eta$ be a solution to equation~\eqref{eq:MBE-Strong-linear}. 
Then for any 
\begin{align*}
  \max\{3 + 2\alpha, \ 8 + \gamma\} < \ell < k_0 - 5 - \gamma,
\qquad
  \alpha > \gamma + 2s,
\end{align*} 
the solution $f_\eta$ satisfies 
\begin{align} \label{bi-cor-bas-e1-strong}
   \norm{\vint{\cdot}^{\ell} f_\eta(t)}^{2}_{L^{2}_{x,v}} 
   + \frac{\Eps}{4} \int_0^t \norm{\vint{v}^{\ell+\alpha} f_\eta}_{L^2_x H^1_v}^2
\leq 
    C_\ell e^{C_{\ell, \Eps} t}
    \vpran{\norm{\vint{\cdot}^{\ell} f_0}^{2}_{L^{2}_{x,v}}  + t},
\end{align}
where $C_{\ell, \Eps}$ is independent of $\eta$ but does depend on $\Eps$ and $C_\ell$ is independent of both $\Eps$ and $\eta$.
Furthermore, for any $0 \leq T_1 < T_2 < T$ and any $s' \leq \frac{1}{8}$, we have the regularisation in $t, x$ as
\begin{align} \label{bound:velocity-avg-basic}
     \int^{T_2}_{T_{1}}
     \norm{(1 - \del^2_{t})^{s'/2}  f_\eta}^{2}_{L^{2}_{x,v}} \!\!\dtau
  + \int^{T_2}_{T_{1}}
     \norm{(1 - \Delta_{x})^{s'/2}  f_\eta}^{2}_{L^{2}_{x,v}} \!\!\dtau 
\leq
  C e^{C_{\ell, \Eps} T} 
  \vpran{\norm{\vint{\cdot}^{\ell} f_0}^{2}_{L^{2}_{x,v}} + T}, 
\end{align}
where the coefficient $C_{\ell, \Eps}$ is independent of $\eta$ and $C$ is independent of $\Eps$.
\end{prop}
\begin{proof}
By~\eqref{est:L-alpha-1}, the regularizing term $\Eps L_\alpha$ gives
\begin{align*}
   \int_{\T^3} \int_{\R^3} \Eps L_\alpha(\mu + f_\eta) (f_\eta) \vint{v}^{2\ell} \dv\dx
\leq
   - \frac{\Eps}{2} \norm{\vint{v}^{\ell + \alpha} f_\eta}_{L^2_x H^1_v}^2
   + C_{\ell} \Eps \norm{\vint{v}^{\ell} f_\eta}_{L^2_{x, v}}^2
   + C_{\ell} \Eps \norm{\vint{v}^{\ell} f_\eta}_{L^2_{x, v}}. 
\end{align*}
Since $\Eps L_\alpha$ will provide the dominating term in both the weight and the regularity, we can bound the collision term in a more direct way via the trilinear estimate in Lemma~\ref{lem:trilinear-coerc-unif}: for $\ell < k_0 - 5 - \gamma$, it holds that
\begin{align*}
& \quad \,
   \int_{\T^3} \int_{\R^3}
   Q_\eta(\mu + g \chi(\vint{v}^{k_0} g), \mu + f_\eta) f_\eta \vint{v}^{2\ell} \dv\dx
\\
&\leq
  C_{\ell} \norm{\vint{v}^\ell f_\eta}_{L^2_{x,v}}
  + C_{\ell} \norm{\vint{v}^{\ell+\gamma/2+s} f_\eta}_{L^2_x H^s_v}^2
\\
&\leq
  C_{\ell} \norm{\vint{v}^\ell f_\eta}_{L^2_{x,v}}
  + \frac{\Eps}{4} \norm{\vint{v}^{\ell+\alpha} f_\eta}^2_{L^2_x H^1_{v}}
  + C_{\ell, \Eps} \norm{\vint{v}^\ell f_\eta}_{L^2_{x,v}}^2,
\qquad
  \alpha > \gamma/2 + s.
\end{align*}
Combining the two estimates above and apply the Gronwall's inequality gives~\eqref{bi-cor-bas-e1-strong}.

Next, we apply the averaging lemma in Proposition~\ref{average-lemma-p} to obtain the regularisation in $x$. In light of equation \eqref{eq:MBE-Strong-linear}, if we invoke Proposition \ref{average-lemma-p} with 
\begin{align*}
    \beta=1, 
\quad m=2, 
\quad r=0,
\quad p=2, 
\quad s' < 1/8,
\end{align*} 
then for any $0\leq T_{1} \leq T_{2}<T$, 
\begin{align*}
& \quad \,
   \int^{T_{2}}_{T_{1}} \norm{(1 - \del_{t}^2)^{s'/2} f_\eta}^{2}_{L^{2}_{x,v}} \dtau   
   + \int^{T_{2}}_{T_{1}} \norm{(1 - \Delta_{x})^{s'/2} f_\eta}^{2}_{L^{2}_{x,v}} \dtau 
\\
&\leq 
     C \norm{\vint{v}^3 f_\eta(T_1)}^{2}_{L^{2}_{x,v}} 
     + C \norm{\vint{v}^3 f_\eta(T_2)}^{2}_{L^{2}_{x,v}} 
     + C \int^{T_2}_{T_{1}} 
          \norm{(1 -\Delta_{v})^{1/2} f_\eta}^{2}_{L^{2}_{x,v}} \dtau
\\
& \quad \,
   + C \int_{T_1}^{T_2} \norm{\vint{v}^{3} (1 -\Delta_{v})^{-1}\tilde{Q}(\mu + g \chi, \mu + f_\eta)}^{2}_{L^{2}_{x,v}} \dtau.
\end{align*}
By the trilinear estimate in Lemma~\ref{lem:trilinear-coerc-unif}, 
it follows that
\begin{align*}
& \quad \, 
  \norm{\langle v \rangle^{3} (1 - \Delta_{v})^{-1}\tilde{Q}(\mu + g \chi, \mu + f_\eta)}_{L^{2}_{v}}
\\
& \leq 
  \norm{\vint{v}^{3} (1 -\Delta_{v})^{-1} \vpran{Q(\mu + g \chi,f_\eta) + Q(g \chi,\mu)}}_{L^{2}_{v}} 
  + \Eps \norm{\vint{v}^{3} (1 -\Delta_{v})^{-1}L_{\alpha} (\mu + f_\eta)}_{L^{2}_{v}} 
\\
& \leq 
   C \norm{f_\eta}_{L^{2}_{3+\gamma+2s}}
   + C \delta_0
   + \Eps \, C\,\| f_\eta \|_{L^{2}_{3+2\alpha}}
   + C \Eps.
\end{align*}
Applying~\eqref{bi-cor-bas-e1-strong} we get
\begin{align}\label{L2prop-e2}
&  \int^{T_{2}}_{T_{1}} \norm{(1 - \Delta_{t})^{s'/2} f_\eta}^{2}_{L^{2}_{x,v}} \dtau
  + \int^{T_2}_{T_{1}}
     \norm{(1 - \Delta_{x})^{s'/2}  f_\eta}^{2}_{L^{2}_{x,v}} \dtau  \nn
\\
& \hspace{1cm} \leq 
  C\int^{T_2}_{T_1} 
   \vpran{\Eps^2 \norm{\vint{v}^{3+2\alpha} f_\eta}^{2}_{L^{2}_{x,v}}
   + \norm{(1 -\Delta_{v})^{1/2} f_\eta}^{2}_{L^{2}_{x,v}}} \dt  \nn
  + C \int^{T_2}_{T_1} \norm{\vint{v}^{3+\gamma+2s} f_\eta}_{L^2_{x, v}}^2 \dt   
\\
& \hspace{1.4cm} 
  + C \norm{\vint{v}^3 f_\eta(T_1)}^{2}_{L^{2}_{x,v}}
  + C \norm{\vint{v}^3 f_\eta(T_2)}^{2}_{L^{2}_{x,v}}
  + C \vpran{\Eps^2 + \delta_0^2} (T_2 - T_1)   \nn
\\
& \hspace{1cm} \leq
  C e^{C_{\ell, \Eps} T} 
  \vpran{\norm{\vint{\cdot}^{\ell} f_0}^{2}_{L^{2}_{x,v}} + T + \Eps^{2} T},
\end{align}
which is the desired inequality showing the spatial regularisation of $f_\eta$.
\end{proof}

The basic $L^2$-level-set estimate parallel to Proposition~\ref{thm:L2-level-set} is
\begin{prop}\label{thm:L2-level-set-strong}
Suppose $G = \mu + g \geq 0$
and 
\begin{align*}
   8 + \gamma < \ell < k_0 - 5 - \gamma,
\qquad
   \alpha > \gamma + 2s.
\end{align*} 
Then the level-set function satisfies
\begin{align} \label{est:level-set-1-strong}
& \quad \,
   \int_{\T^3} \int_{\R^3} Q_\eta(\mu + g \chi, \mu + f) \Fl{K} \vint{v}^{\ell} \dv\dx
   + \int_{\T^3} \int_{\R^3} \Eps L_{\alpha}(\mu + f) \Fl{K} \vint{v}^{\ell} \dv\dx  \nn
\\
&\leq   
   -\frac{\Eps}{4} \norm{\vint{v}^{\alpha} \Fl{K}}^2_{L^2_{x}H^{1}_{v}}
   + C_\Eps \norm{\Fl{K}}_{L^2_{x, v}}^2
  + C (1+K) \norm{\Fl{K}}_{L^1_{x, v}},
\end{align}
where the constants $C_\Eps, C$ are independent of $\eta$.
\end{prop}

\begin{proof}
Recall the decomposition in~\eqref{decomp:linear}:
\begin{align} \label{decomp:linear-eta}
  \int_{\T^3} \int_{\R^3} Q_\eta(\mu + g\chi, \mu + f) \Fl{K} \vint{v}^{\ell} \dv\dx
&= \int_{\T^3} \int_{\R^3} Q_\eta \vpran{\mu + g\chi, f - \tfrac{K}{\vint{v}^\ell}} \Fl{K} \vint{v}^{\ell} \dv\dx   \nn
\\
& \quad \,
+ \int_{\T^3} \int_{\R^3} Q_\eta \vpran{\mu + g\chi, \mu + \tfrac{K}{\vint{v}^\ell}} \Fl{K} \vint{v}^{\ell} \dv\dx,
\end{align}
where by the positivity of $\mu + g\chi$ and the same upper bound for $T_1$ in~\eqref{bound:T-1-levl-set}, we have
\begin{align*}
& \quad \,
  \int_{\T^3} \int_{\R^3} Q_\eta \vpran{\mu + g\chi, f - \tfrac{K}{\vint{v}^\ell}} \Fl{K} \vint{v}^{\ell} \dv\dx
\\
&\leq
    \iiiint_{\T^3 \times \R^6 \times \Ss^2}
        (\mu_\ast + g_\ast \chi_\ast) \Fl{K} \frac{1}{\vint{v}^\ell} \vpran{\Fl{K}(v') \vint{v'}^{\ell} - \Fl{K} \vint{v}^{\ell}} 
        b_\eta(\cos\theta) |v - v_\ast|^\gamma \dbmu
\\
&= \int_{\T^3} \int_{\R^3} Q_\eta \vpran{\mu + g\chi, \Fl{K}/\vint{v}^\ell} \Fl{K} \vint{v}^{\ell} \dv\dx
\\
& \leq
  C \vpran{1 + \sup_x\norm{g\chi}_{L^1_{\ell + \gamma + 2s} \cap L^2}}
  \norm{\Fl{K}}_{L^2_x H^s_{\gamma+2s}}^2.
\end{align*}
By Lemma~\ref{lem:trilinear-coerc-unif}, the coefficient $C$ in the inequality above is independent of $\eta$. By interpolation, 
\begin{align} \label{bound:level-Q-eta-strong}
   \int_{\T^3} \int_{\R^3} Q_\eta \vpran{\mu + g\chi, f - \tfrac{K}{\vint{v}^\ell}} \Fl{K} \vint{v}^{\ell} \dv\dx
\leq
   \frac{\Eps}{4} \norm{\vint{v}^{\alpha} \Fl{K}}_{L^2_x H^1_{v}}^2
   + C_\Eps \norm{\Fl{K}}^2_{L^2_{x,v}}.
\end{align}
The second term on the right-hand side of~\ref{decomp:linear-eta} satisfies the same bound as for $T_2$ in~\eqref{decomp:linear} since only the upper bound of $b_\eta$ is needed in the estimates. Moreover, the regularizing term $\Eps L_\alpha$ satisfies the same bound as in~\eqref{est:level-set-3}, which combined with~\eqref{bound:level-Q-eta-strong} gives~\eqref{thm:L2-level-set-strong}. 
\end{proof}

The counterpart of Proposition~\ref{T1} states
\begin{prop}\label{T1-strong}
Let $G= \mu + g \geq 0$ and $F = \mu + f$ satisfying equation \eqref{eq:MBE-Strong-linear}. Denote 
\begin{align*}
    \tilde Q_\eta (\mu + g \chi, \mu + f)
= Q_\eta (\mu + g \chi, \mu + f) + \Eps L_\alpha(\mu + f).
\end{align*}
Then, for any $T > 0$ and 
\begin{align*}
   s \in [1/2,1),
\quad
    \epsilon \in [0, 1],
\quad 
   j \geq 0,
\quad 
   8 + \gamma < \ell < k_0 - 5 - \gamma,
\quad 
   \kappa > 2, 
\quad
   K > 0,
\quad  
   \alpha > j + \gamma + 2s,
\end{align*} 
it follows that
\begin{align}\label{Qlevelaverage-1}
& \quad \,
    \int_0^T \int_{\T^3} \int_{\R^3} \abs{\vint{v}^{j}(1 - \Delta_{v})^{-\kappa/2}\big( \tilde{Q}_\eta(G,F)\,\vint{v}^{\ell}\,\Fl{K} \big)} \dv\dx\dt  \nn
\\
&\leq 
  C\,\| \vint{v}^{j/2} \Fl{K}(0, \cdot, \cdot) \|^{2}_{L^{2}_{x,v}}
  +   C  \norm{\vint{v}^\alpha \Fl{K}}_{L^2_x H^1_v}^2
  + C (1+K) \norm{\Fl{K}}_{L^1_x L^1_{j+\gamma}},
\end{align}
where $C, C_\ell$ are independent of $\eta$. 
Identical estimate holds for $\tilde Q^-_\eta(\mu+g\chi, -\mu + h)$ with $\Fl{K}$ replaced by ~$\Hl{K}$.
\end{prop}

\begin{proof}
As in the proof of Proposition~\ref{T1}, we only need to control $\CalQ$ in~\eqref{two-terms-estimate-1} with $b$ replaced by $b_\eta$. By the same decomposition in~\eqref{decomp:CalQ} and a similar argument in Proposition~\ref{thm:L2-level-set-strong}, we have 
\begin{align*}
   \CalQ_\eta
&\leq
   \iiiint_{\T^3 \times \R^6 \times \Ss^2}
        (\mu_\ast + g_\ast \chi_\ast) \Fl{K} \frac{1}{\vint{v}^\ell} 
        \vpran{\Fl{K}(v') W_K(v') \vint{v'}^{\ell} - \Fl{K} W_K \vint{v}^{\ell}}
        b_\eta(\cos\theta) |v - v_\ast|^\gamma \dbmu
\\
& \quad \, 
   + \int_{\T^3} \int_{\R^3} Q_\eta \vpran{\mu + g \chi, \tfrac{K}{\vint{v}^\ell}} \vint{v}^{\ell} \Fl{K} W_K \dv\dx 
+ \int_{\T^3} \int_{\R^3} Q_\eta \vpran{\mu + g \chi, \mu} \vint{v}^{\ell} \Fl{K} W_K \dv\dx 
\\
& \leq
  C  \norm{\Fl{K}}_{L^2_x H^s_{j+\gamma/2+s}}^2
  + C (1+K) \norm{\Fl{K}}_{L^1_{x} L^1_{j+\gamma}}
\\
& \leq
  C  \norm{\vint{v}^\alpha \Fl{K}}_{L^2_x H^1_v}^2
  + C (1+K) \norm{\Fl{K}}_{L^1_x L^1_{j+\gamma}}. 
\end{align*}
The estimate for $T_R^+$ in~\eqref{two-terms-estimate-1} remains the same. 
\end{proof}

Lemma~\ref{Interpolationlemma} stays the same since it is independent of the collision kernel. Same with Proposition~\ref{thm:SV-energy-functional-linear}. Using the energy bound in Proposition~\ref{prop:L-2-MBE-strong} which is similar to Corollary~\ref{cor:bi-cor-basic-energy-estimate-linear}, we obtain a similar bound for $\CalE_0$ (with $s=1$ and $s' < 1/8$) as in Proposition~\ref{prop:bound-CalE-0}:
\begin{prop} \label{prop:bound-CalE-0-strong}
Let $T>0$ be fixed. Suppose $\Eps \in [0, 1]$ and $s \in [1/2, 1)$.  Assume that the given function $G = \mu + g \geq 0$. 
Suppose $\ell$ satisfies
\begin{align*}
   \max\{8+\gamma, 3 + 2\alpha\} \leq \ell < k_0 - 5 -\gamma
\end{align*}
and assume that $f$ is a solution of ~\eqref{eq:MBE-Strong-linear} which satisfies $\mu + f \geq 0$.
Then for any $s' < 1/8$, there exist $s'' > 0$ and $p^{\flat}:=p^{\flat}(\ell,\gamma,s,s') > 1$ such that if $s'' < s' \frac{\gamma}{2 \ell + \gamma}$ and $1 < p < p^{\flat}$, then we have
\begin{align}\label{bi-initial-E0}
     \CalE_0
 \leq 
    C e^{C_\Eps \,T}\max_{j \in\{1/p, p'/p\}}
      \vpran{\norm{\vint{\cdot}^{\ell} f_0}^{2j}_{L^{2}_{x,v}} 
      + T^{j}},
 \qquad
    p' = p/(2-p),
\end{align}
where $C_\Eps$ is independent of $\eta$. The same estimate holds for $(-f)^{\ell}_+$ and its associated $\CalE_0$. 
\end{prop}

Since all the building blocks leading to Theorem~\ref{central-bilinear-T} agree, we have a similar statement for the a priori $L^\infty$-bound:
\begin{thm}[Linear case]\label{central-bilinear-T-strong}
Suppose $G = \mu + g \geq 0$. 
Let $F = \mu + f_\eta \geq 0$ be a solution to equation~\eqref{eq:MBE-Strong-linear} with $s \in [1/2, 1)$.  Assume that $\ell$ satisfies
\begin{align*}
    \max\{8+\gamma, 3+2\alpha \} \leq \ell < k_0 - 5 - \gamma,
\qquad
   \alpha > 2 + \gamma + 2s.
\end{align*}
Assume that the initial data satisfies
\begin{align} \label{cond:initial-L-2-infty-linear}
    \norm{\vint{v}^{\ell+2} f_0}_{ L^{2}_{x,v} }<\infty,
\quad 
    \norm{\vint{v}^{\ell} f_0}_{L^{\infty}_{x,v}} < \infty.	
\end{align}
Additionally, assume that the solution satisfies
\begin{align*}
\sup_{t}\| \langle v \rangle^{\ell_0+\ell}&f \|_{ L^{1}_{x,v} } \leq C_1,
\end{align*}
where $\ell_0$ satisfies the bound in Proposition~\ref{thm:SV-energy-functional-linear} (or~\eqref{cond:ell-0-2}). Then it follows that 
\begin{align*}
    \sup_{t \in [0,T]} \norm{\vint{v}^\ell f}_{L^{\infty}_{x,v}} 
\leq 
    \max\Big\{ 2 \norm{\vint{v}^\ell f_0}_{L^{\infty}_{x,v}},  K^{lin}_0\Big\},
\end{align*}
where
\begin{align}\label{K-initial-E0-linear}
   K^{lin}_0:= 
  C e^{C_\Eps\,T} \max_{1 \leq i \leq 4} \max_{j\in\{1/p, p'/p\} }
  \vpran{\norm{\vint{v}^{\ell} f_0}^{2j}_{L^{2}_{x,v}} 
             + T^j}^{\frac{\beta_i -1}{a_i}}.
\end{align}
Here $C_\Eps$ is independent of $\eta$ and $C$ is independent of both $\Eps$ and $\eta$.
\end{thm} 

It is clear from Theorem~\ref{central-bilinear-T-strong} that for each $\Eps > 0$,  if we let $T$ be small enough (with smallness depending on $\Eps, \delta_0$ only) and 
$\norm{\vint{v}^{\ell} f_0}_{L^2_{x,v}\cap L^\infty_{x,v}}$ small enough (with smallness independent of both $\Eps$ and $\eta$), then 
\begin{align*}
   \sup_{t \in [0,T]} \norm{\vint{v}^\ell f}_{L^{\infty}_{x,v}} \leq \delta_0.
\end{align*}

We can now combine the linear and nonlinear theory in Theorem~\ref{thm:linear-local} and Theorem~\ref{eq:MBE-nonlinear}
to obtain the local existence of~\eqref{eq:MBE-Strong-eta} as follows.
\begin{thm} \label{thm:local-exist-MBE-strong}
Suppose $s \in [1/2, 1)$ and let $b_\eta$ be the regularized collision kernel. Suppose
\begin{align*}
 & k_0 > \max \left\{\ell_0 + 15 + 2\gamma, \ \ell_0 + 10 + 2\alpha +\gamma, \ 
   k - \alpha + 2\gamma + 2s + 9 + \ell_0 \right\},
\\[1mm]
& \hspace{2cm}
   k > \max\{8 + \gamma, \alpha\}, 
 \qquad
   \alpha > 2 + \gamma + 2s, 
\end{align*}
where $\ell_0$ is the same weight in Theorem~\ref{thm:linear-local} (precise statement in~\eqref{cond:ell-0-2}). 
Suppose $\Eps, \delta_0, f_0$ satisfy the assumptions in both part (a)  and part (b) in Theorem~\ref{thm:linear-local}. Then for each such $\Eps$, if $T$ is small enough (which only depends on $\Eps$) then ~\eqref{eq:MBE-Strong-eta} has a solution $f \in L^\infty_t ((0, T); L^2_x L^2_k(\T^3 \times \R^3))$. Moreover, $f$ satisfies the bound
\begin{align} \label{bound:L-infty-lower-f-MBE}
   \norm{\vint{v}^{k_0 - \ell_0 - 7 - \gamma} f}_{L^\infty{\vpran{(0, T) \times \T^3 \times \R^3}}} \leq \delta_0.
\end{align}

\end{thm}
\begin{proof}
The proof is the combination of the proofs of Theorem~\ref{thm:linear-local} and Theorem~\ref{eq:MBE-nonlinear}. 
When applying the fixed-point argument as in~\eqref{reason-mild-sing}, we note that the coefficients obtained will depend on $\eta$. This is the place that the regularization of $b$ in~\eqref{bound:kernel-essential} takes effect. Specifically, the counterpart of~\eqref{reason-mild-sing} is 
\begin{align*} 
& \quad \,
   \iint_{\T^3 \times \R^3}
       Q_\eta(g \chi(\vint{v}^{k_0} g) - h \chi(\vint{v}^{k_0} h), f_h) (f_g - f_h) \vint{v}^{2k} \dx\dv \nn
\\
&\leq
   \frac{C}{\eta^{2s-2s^\ast}} \int_{\T^3} \norm{g \chi(\vint{v}^{k_0} g) - h \chi(\vint{v}^{k_0} h)}_{L^1_{\gamma+2s^\ast+k-\alpha} \cap L^2} \norm{f_h}_{L^2_{\gamma + 2s^\ast + k - \alpha}} \norm{f_g - f_h}_{H^{2s^\ast}_{k+\alpha}} \dx  \nn
\\
&\leq
  C_\eta \vpran{\sup_x \norm{f_h}_{L^2_{\gamma + 2s^\ast+k-\alpha}}}
  \norm{g - h}_{L^2_x L^2_k } \norm{f_g - f_h}_{L^2_x H^{2s^\ast}_{k+\alpha}}  \nn
\\
& \leq
  \frac{\Eps}{16} \norm{\vint{v}^{k+\alpha}(f_g - f_h)}_{L^2_x H^1_v}^2
  + C_{\Eps,\eta} \norm{g - h}_{L^2_x L^2_k }^2.
\end{align*}
Thus the first time interval of existence obtained depends on both $\Eps$ and $\eta$. However, since the a priori estimates are independent of $\eta$, such a solution can be extended to $T$ independent of $\eta$.
\end{proof}

Once the existence of $f_\eta$ is shown, we can pass to the limit and return to the original operator $Q$ (with~$\chi$).
\begin{thm} \label{thm:strong-reg-Boltzmann}
Suppose $s \in [1/2, 1)$ and
\begin{align*}
 & k_0 > \max \left\{\ell_0 + 15 + 2\gamma, \ \ell_0 + 10 + 2\alpha +\gamma, \ 
   k - \alpha + 2\gamma + 2s + 9 + \ell_0 \right\},
\\[1mm]
& \hspace{2cm}
   k > \max\{8 + \gamma, \alpha\}, 
 \qquad
   \alpha > 2 + \gamma + 2s, 
\end{align*}
where $\ell_0$ is the same weight in Theorem~\ref{thm:linear-local} (precise statement in~\eqref{cond:ell-0-2}). 
Suppose $\Eps, \delta_0, f_0$ satisfy the assumptions in both part (a)  and part (b) in Theorem~\ref{thm:linear-local}. Then for each such $\Eps$, if $T$ is small enough (which only depend on $\Eps$) then the equation
\begin{align} \label{eq:Boltzmann-reg}
          \del_t f + v \cdot \nabla_x f 
 = \Eps L_\alpha (\mu + f) 
    + Q(\mu + f \chi(\vint{v}^{k_0} f), \mu + f),
\qquad
   f \big|_{t=0} = f_0(x, v)   
\end{align}
has a solution $f \in L^2_{t, x} L^2_k((0, T) \times \T^3 \times \R^3)$. Moreover, $f$ satisfies the bound
\begin{align} \label{bound:L-infty-lower-f-MBE}
   \norm{\vint{v}^{k_0 - \ell_0 - 7 - \gamma} f}_{L^\infty{\vpran{(0, T) \times \T^3 \times \R^3}}} \leq \delta_0.
\end{align}
\end{thm}
\begin{proof}
By Theorem~\ref{prop:L-2-MBE-strong} and Theorem~\ref{thm:local-exist-MBE-strong}, equation~\eqref{eq:MBE-Strong-eta} has a solution $f_\eta$ satisfying
\begin{align*}
     \norm{\vint{v}^{k_0 -\ell_0 - 7 -\gamma} f_\eta}_{L^\infty_{t, x, v}} < \delta_0,
\qquad
    \norm{f_\eta}_{H^{s'}_{t, x} H^1_{k+\alpha}} < C_0 < \infty,
\qquad
   s' < 1/8. 
\end{align*}
Given the uniform polynomial decay and a diagonal argument, we can extract a subsequence, still denoted as $f_\eta$, such that
\begin{align*}
   f_\eta \to f
\quad 
   \text{strongly in $L^{2}_{t,x,v}((0, T) \times \T^3 \times \R^3)$}.
\end{align*}
Our goal is to show that
\begin{align} \label{converg:Q-eta}
   Q_\eta(f_\eta \chi(\vint{v}^{k_0} f_\eta), f_\eta) 
\to 
   Q(f \chi(\vint{v}^{k_0} f), f)
\quad 
   \text{in $\CalD'$.}
\end{align}
Using a test function $\phi$, we consider the difference
\begin{align*}
  Q_\eta(f_\eta, f_\eta) - Q(f, f)
 &=  \int_{\R^3} \! \int_{\R^3} \! \int_{\Ss^2}
 b_\eta(\cos\theta) \vpran{f'_{\eta, \ast} \chi'_{\eta, \ast} f'_\eta - f_{\eta, \ast} \chi_{\eta, \ast} f_\eta} |v - v_\ast|^\gamma \phi(v) \dsigma \dv_\ast \dv
\\
& \quad \,
  - \int_{\R^3} \! \int_{\R^3} \! \int_{\Ss^2}
 b(\cos\theta) \vpran{f'_{\ast} \chi'_{\ast} f' - f_{\ast} \chi_\ast f} |v - v_\ast|^\gamma \phi(v) \dsigma \dv_\ast \dv
\\
& = \int_{\R^3} \! \int_{\R^3} \! \int_{\Ss^2}
 b_\eta(\cos\theta) \vpran{f_{\eta, \ast} \chi_{\eta, \ast} f_\eta - f_{\ast} \chi_\ast f} |v - v_\ast|^\gamma \vpran{\phi(v') - \phi(v)} \dsigma \dv_\ast \dv
\\
& \quad \,
  + \int_{\R^3} \! \int_{\R^3} \! \int_{\Ss^2}
 \vpran{b(\cos\theta) - b_\eta(\cos\theta)} f_{\ast} \chi_\ast f |v - v_\ast|^\gamma \vpran{\phi(v') - \phi(v)} \dsigma \dv_\ast \dv
\\
& \Denote E_1 + E_2.
\end{align*}
By Proposition~\ref{prop:symmetry-cancel} and the upper bound of $b_\eta$ in~\eqref{bound:kernel-essential}, $E_1$ is bounded as
\begin{align*}
  \abs{E_1}
&\leq
  \abs{\int_{\R^3} \! \int_{\R^3} \! \int_{\Ss^2}
 b_\eta(\cos\theta) \vpran{f_{\eta, \ast} \chi_{\eta, \ast} f_\eta - f_{\ast} \chi_\ast f} |v - v_\ast|^\gamma \vpran{\phi(v') - \phi(v)} \dsigma \dv_\ast \dv}
\\
&\leq
  \int_{\R^3} \int_{\R^3}
  \abs{f_{\eta, \ast} \chi_{\eta, \ast} f_\eta - f_{\ast} \chi_\ast f} |v - v_\ast|^\gamma
  \abs{\int_{\Ss^2} b_\eta(\cos\theta) \vpran{\phi(v') - \phi(v)} \dsigma} \dv_\ast \dv
\\
&\leq
  C \norm{\phi}_{W^{2, \infty}} \int_{\R^3} \int_{\R^3}
  \abs{f_{\eta, \ast} \chi_{\eta, \ast} f_\eta - f_{\ast} \chi_\ast f} |v - v_\ast|^{2+\gamma} \dv_\ast \dv,
\end{align*}
where $C$ is independent of $\eta$. The integrand in the inequality above satisfies
\begin{align*}
   \abs{f_{\eta, \ast} \chi_{\eta, \ast} f_\eta - f_{\ast} \chi_\ast f} |v - v_\ast|^{2+\gamma}
&\leq
  \abs{f_{\eta, \ast} \chi_{\eta, \ast} - f_{\ast} \chi_\ast} \vint{v_\ast}^{2+\gamma} \abs{f_\eta} \vint{v}^{2+\gamma}
  + \abs{f_\eta - f} \vint{v}^{2+\gamma} \abs{f_{\eta, \ast}} \vint{v_\ast}^{2+\gamma}
\\
& \leq
  \abs{f_{\eta, \ast} - f_{\ast}} \vint{v_\ast}^{2+\gamma} \abs{f_\eta} \vint{v}^{2+\gamma}
  + \abs{f_\eta - f} \vint{v}^{2+\gamma} \abs{f_{\eta, \ast}} \vint{v_\ast}^{2+\gamma}. 
\end{align*}
Therefore,
\begin{align*}
   \norm{E_1}_{L^2_{t,x}}
\leq
  C \norm{\phi}_{W^{2, \infty}} \norm{f_\eta - f}_{L^2_{t,x}L^2_{4+\gamma}} \to 0,
\qquad
  \eta \to 0.
\end{align*}
%
%
To estimate $E_2$, note that by symmetry (or more precisely, anti-symmetry) and Taylor expansion, $E_2$ satisfies
\begin{align*}
& \quad \, 
  \abs{\int_{\R^3} \int_{\R^3} \int_{\Ss^2} \vpran{b(\cos\theta) - b_\eta(\cos\theta)} f_{\ast} \chi_\ast f |v - v_\ast|^\gamma \vpran{\phi(v') - \phi(v)} \dsigma \dv_\ast \dv}
\\
&\leq 
  \abs{\int_{\R^3} \int_{\R^3} \int_{\Ss^2} \vpran{b(\cos\theta) - b_\eta(\cos\theta)} f_{\ast} \chi_\ast f |v - v_\ast|^\gamma (v - v') \cdot \nabla_v \phi(v) \dsigma \dv_\ast \dv}
\\
& \quad \, 
  + \abs{\frac{1}{2} \int_{\R^3} \int_{\R^3} \int_{\Ss^2} \vpran{b(\cos\theta) - b_\eta(\cos\theta)} f_{\ast} \chi_\ast f |v - v_\ast|^\gamma (v - v') \otimes (v - v') \cdot \nabla_v^2 \phi(\bar v) \dsigma \dv_\ast \dv}
\\
&\leq 
  \int_{\R^3} \int_{\R^3} \int_{\Ss^2} (1 - \cos\theta)\abs{b(\cos\theta) - b_\eta(\cos\theta)} f_{\ast} \chi_\ast f |v - v_\ast|^{1+\gamma} \abs{\nabla_v \phi(v)} \dsigma \dv_\ast \dv
\\
& \quad \, 
  + \frac{1}{2} \int_{\R^3} \int_{\R^3} \int_{\Ss^2} \sin^2\theta \abs{b(\cos\theta) - b_\eta(\cos\theta)} f_{\ast} \chi_\ast f |v - v_\ast|^{2+\gamma}\abs{\nabla_v^2 \phi(\bar v)} \dsigma \dv_\ast \dv,
\end{align*}
where by~\eqref{bound:kernel-essential}, the integrands of the last two terms satisfy the uniform bounds
\begin{align*}
   (1 - \cos\theta)\abs{b(\cos\theta) - b_\eta(\cos\theta)} f_{\ast} f |v - v_\ast|^{1+\gamma} \abs{\nabla_v \phi(v)}
\leq 
  2 \norm{\phi}_{W^{1, \infty}} (1 - \cos\theta) b(\cos\theta) f_{\ast} f |v - v_\ast|^{1+\gamma}
\end{align*}
and
\begin{align*}
  \sin^2\theta \abs{b(\cos\theta) - b_\eta(\cos\theta)} f_{\ast} f |v - v_\ast|^{2+\gamma}\abs{\nabla_v^2 \phi(\bar v)}
\leq
  2 \norm{\phi}_{W^{2, \infty}} \sin^2\theta b(\cos\theta) f_{\ast} f |v - v_\ast|^{2+\gamma}.
\end{align*}
Since the right-hand sides of the inequalities above are integrable, we can apply the Lebesgue Dominated Convergence Theorem and obtain that $E_2 \to 0$ as $\eta \to 0$. Hence~\eqref{converg:Q-eta} holds. 
\end{proof}

Recall that the only place that the restriction of a weak singularity enters is when we apply the fixed-point argument (see~\eqref{reason-mild-sing}) to obtain an approximate solution to equation~\eqref{eq:Boltzmann-reg}. Once such restriction is bypassed via Theorem~\ref{thm:strong-reg-Boltzmann} , the rest of the results from Proposition~\ref{quadratic-zero-level} to Theorem~\ref{thm:global-mild} all hold, since they are all proved for $s \in (0, 1)$.  
This leads us to the main theorem of this paper.
\begin{thm}[Global Existence] \label{thm:main-global}
Let $s \in (0, 1)$ and $\gamma \in (0, 1)$. Suppose $\delta_0$ is a constant small enough such that bounds in Theorem~\ref{thm:L-infty-k-0-nonlinear} and~\eqref{cond:delta-0-sec-7} are 
satisfied. 
Let $\ell_0$ be the same weight in Theorem~\ref{thm:linear-local} and $k_0$ be a constant satisfying
\begin{align*}
  k_0 > 5\ell_0 + 35 + 5\gamma + 4s. 
\end{align*}
Let $\delta_\ast^\natural$, defined in~\eqref{def:frakH-Eps-ast}, be the constant measuring the smallness of the data. Suppose the initial data $f_0$ has zero mass, momentum and energy and satisfies
\begin{align} \label{assump:main-f-0}
   \norm{\vint{v}^{k_0} f_0}_{L^\infty_{x, v} \cap L^2_{x,v}} < \delta_{\ast}^\natural, 
\qquad
   \norm{\vint{v}^{k_0 + \ell_0 + 2} f_0}_{L^2_{x, v}} < \infty. 
\end{align}
Then the Boltzmann equation~\eqref{eq:Nonlinear-Boltzmann-intro} has a solution $f \in L^\infty(0, \infty; L^2_x L^2_{k_0 + \ell_0 + 2}(\T^3 \times \R^3))$. Moreover, there exist $\lambda' > 0$ such that 
\begin{align*}
  \norm{\vint{v}^{k_0} f}_{L^\infty(0, \infty; \T^3 \times \R^3)}
\leq
  \delta_0/2 < \delta_0, 
\qquad
   \norm{f(t, \cdot, \cdot)}_{L^2_x L^2_{k_0+\ell_0+2}} 
\leq 
  C \norm{f_0}_{L^2_x L^2_{k_0+\ell_0+2}} e^{-\lambda' t},
\quad
  t \geq 0.
\end{align*}
\end{thm}

Finally, based on the global result and the exponential decay of the $L^2$-norm in Theorem~\ref{thm:main-global}, we can show an exponential decay in the $L^\infty$-norm of the solution. 

\begin{thm}
Suppose $k_0$ and the initial data $f_0$ satisfy the same conditions in Theorem~\ref{thm:main-global}. Then there exists $C_{k_0}, \eta_0 > 0$ such that for any $t > 1$ the solution obtained in Theorem~\ref{thm:main-global} satisfies
\begin{align} \label{decay:L-infty-final}
   \norm{\vint{v}^{k_0} f(t, \cdot, \cdot)}_{L^{\infty}_{x,v}} \leq 
   C_{k_0} \norm{\vint{v}^{k_0} f_0}_{L^2_{x, v}}^{2 \eta_0/p} e^{- \frac{2\lambda' \eta_0}{p}t},
\end{align}
where $\lambda'$ is the same decay rate in Theorem~\ref{thm:main-global}. 
\end{thm}
\begin{proof}
For any $K, t_1 > 0$, let $\CalE_p$ be the energy functional similar as in~\eqref{EFunctional}:
\begin{align*}
     \CalE_{p}(K,t_1,\infty)
:= \sup_{ t \geq t_1 } 
       \norm{\Fk{\ell}{K}(t, \cdot, \cdot)}^{2}_{L^{2}_{x,v}} 
& +  \int^{\infty}_{t_1}\int_{\T^{3}}
         \norm{\vint{\cdot}^{\gamma/2}\Fl{K}}^{2}_{H^{s}_{v}} \dx\dtau \nn
\\
& + \frac{1}{C_0}\vpran{\int^{\infty}_{t_1} \norm{(1-\Delta_{x})^{\frac {s''}{2}}\vpran{\Fl{K}}^{2}}^{p}_{L^{p}_{x,v}}\dtau}^{\frac{1}{p}}.
\end{align*}
Note that by its definition $\CalE_p$ is decreasing in $t_1$ and $K$. The global bounds of $f$ developed in Theorem~\ref{thm:main-global} guarantee that $\CalE_p(K, t_1, \infty)$ is well-defined. Moreover, for any $T \geq 0$ and $\ell \leq k_0 + \ell_0 + 2$,
\begin{align} \label{bound:E-0-infty-decay}
   \CalE_p(0, T, \infty)
&= \sup_{ t \geq T} 
       \norm{\vint{v}^\ell f_+(t, \cdot, \cdot)}^{2}_{L^{2}_{x,v}} 
 +  \int^{\infty}_{T}\int_{\T^{3}}
         \norm{\vint{\cdot}^{\ell + \gamma/2} f_+}^{2}_{H^{s}_{v}} \dx\dtau \nn
\\
& \hspace{1cm}
   + \frac{1}{C_0}\vpran{\int^{\infty}_{T} \norm{(1-\Delta_{x})^{\frac {s''}{2}} \vpran{\vint{v}^{2\ell} f^2_+}}^{p}_{L^{p}_{x,v}}\dtau}^{\frac{1}{p}} \nn
\\
&\leq 
  C_\ell \norm{\vint{v}^\ell f_0}_{L^2_{x, v}}^{2/p} e^{- \frac{2\lambda' }{p}T},
\qquad
  \ell \leq k_0 + \ell_0 + 2,
\end{align}
where $\lambda'$ is the decay rate in Theorem~\ref{thm:main-global}.

Our main goal is to remove the dependence on the weighted $L^\infty$-norm of $f_0$ in~\eqref{bound:Linfty-interm} (with $\Eps = 0$) so that the exponential decay in the weighted $L^2$-norm of $f$ can be transferred to exponential decay in the $L^\infty$-norm. To this end, define the levels
\begin{align*}
M_{k}:=K_0\big(1-1/2^k\big),\qquad k=0,1,2,\cdots.
\end{align*}
Setting $ f_{k} := f^{(\ell)}_{M_{k},+}$ and proceeding as in the proof of Theorem \ref{central-bilinear-T}, we arrive at
\begin{align}\label{EFunctional-infty}
    \CalE_{p}(M_{k},t_1,\infty)
\leq  
   C \norm{\vint{v}^{2}f_{k}(t_{1})}^{2}_{L^{2}_{x,v}} 
  &  + C \norm{\vint{v}^2 f_{k}(t_{1})}^{2}_{L^{2p}_{x,v}} 
 +  C \sum^{4}_{i=1}\frac{2^{k (a_{i}+1)}}{K^{a_i}_0}\,\CalE_{p}(M_{k-1},t_{1},\infty)^{ \beta_{i} },
\end{align}
for $k=1,2,\cdots$.  The parameters $a_i, \beta_i$ are the same as in Theorem \ref{central-bilinear-T}. Fix $T > 1$ and let $T_k$ be the increasing time sequence
\begin{align*}
T_{k}:=T\big(1-1/2^{k+1}\big),\qquad k=0,1,2,\cdots.
\end{align*}
We further denote $\CalE_k$ as
\begin{align*}
   \CalE_k = \CalE_p(M_k, T_k, \infty).
\end{align*}
Integrate \eqref{EFunctional-infty} in $t_{1}\in [T_{k-1},T_{k}]$ to obtain that
\begin{align*}
    \CalE_{k} 
& = \CalE_p(M_k, T_k, \infty)
\\
& \leq  
   C \big(T_{k} - T_{k-1}\big)^{-1}\bigg(\int^{T_{k}}_{T_{k-1}}\norm{\vint{v}^{2}f_{k}(t_{1})}^{2}_{L^{2}_{x,v}} {\rm d}t_{1} 
   + \int^{T_{k}}_{T_{k-1}}\norm{\vint{v}^2 f_{k}(t_{1})}^{2}_{L^{2p}_{x,v}} {\rm d}t_{1}\bigg)
\\
& \quad \,
  +  C \sum^{4}_{i=1}\frac{2^{k (a_{i}+1)}\,\CalE_{k-1}^{ \beta_{i} }}{K^{a_i}_0},
\end{align*}
where we have applied the monotonicity of $\CalE_p(\cdot, \cdot, \cdot)$ in its first and second variables.
By similar estimates as in \eqref{bound:level-L-2-1} with the same definitions for $r_\ast, \xi_\ast$, we have
\begin{align*}
   \int^{T_k}_{T_{k-1}} \norm{\vint{v}^2 f_{k}(t_{1})}^{2}_{L^{2}_{x,v}} {\rm d}t_{1} 
\leq 
   \tilde C_0 \frac{\CalE_{p}(M_{k},T_{k-1},T_{k})^{ r_\ast } }{(M_{k} - M_{k-1})^{\xi_\ast-2}} 
\leq 
   \tilde C_0 \frac{2^{k(\xi_\ast-2)} \CalE_{k-1}^{ r_\ast } }{K_0^{\xi_\ast-2}},
\end{align*}
and by \eqref{estT1-1}, it holds that
\begin{align*}
   \int^{T_k}_{T_{k-1}} 
& \norm{\vint{v}^2 f_{k}(t_{1})}^{2}_{L^{2p}_{x,v}} {\rm d}t_{1} 
= \int^{T_k}_{T_{k-1}} \norm{\vint{v}^4 \vpran{f_{k}(t_{1})}^2}_{L^{p}_{x,v}} {\rm d}t_{1}
\\
&\leq 
  (T_{k} - T_{k-1})^{\frac{p-1}{p}} \norm{\vint{v}^4 \vpran{f_{k}}^2}_{L^{p}_{t,x,v}} 
\leq 
  \tilde{C}_{0}\,(T_{k} - T_{k-1})^{\frac{p-1}{p}} \frac{ 2^{k \frac{\xi_\ast-2p}{p} }\CalE_{k-1}^{ \frac{r_\ast}{p} } }{K_0^{\frac{\xi_\ast-2p}{p} } }.
\end{align*}
Since we are interested in the long time behaviour we may take $T\geq1$ to derive that an analogous estimate to \eqref{DeG-ineq} with $a_{i}, \beta_{i}$ defined in \eqref{def:a-beta-i} holds:
\begin{align} \label{EFunctional-infty-key}
    \CalE_{k} 
 \leq 
    C_\ell \sum^{4}_{i=1}\frac{2^{k (a_{i}+1)}\,\CalE_{k-1}^{ \beta_{i} }}{K^{a_i}_0},\qquad k=1,2,\cdots, \qquad T\geq1.
\end{align}
The key difference between~\eqref{EFunctional-infty-key} and~\eqref{DeG-ineq} is that $K_0$ in~\eqref{EFunctional-infty-key} is independent of $f_0$. Applying the De Giorgi iteration to~\eqref{EFunctional-infty-key} we conclude similarly as in \eqref{bound:Linfty-interm} (with $\Eps = 0$) that
\begin{align*}
   \sup_{t\geq T} \norm{\vint{v}^{\ell} f_+(t, \cdot, \cdot)}_{L^{\infty}_{x,v}} \leq 
     K_0:= K_0(\CalE_0)
\leq 
  C_\ell \max_{1 \leq i \leq 4} \CalE_0^{\frac{\beta_i - 1}{a_i}},
\quad
\ell \leq k_0,
\end{align*}
where $\CalE_0:=\CalE_{p}(0,T/2,\infty)$.  
Hence by the bound in~\eqref{bound:E-0-infty-decay}, we have
\begin{align*}
   \sup_{t\geq T} \norm{\vint{v}^{\ell} f_+(t, \cdot, \cdot)}_{L^{\infty}_{x,v}} \leq 
   C_\ell \norm{\vint{v}^\ell f_0}_{L^2_{x, v}}^{2 \eta_0/p} e^{- \frac{2\lambda' \eta_0}{p}T}, 
\qquad
  \eta_0 = \min_{1 \leq i \leq 4} \frac{\beta_i - 1}{a_i}. 
\end{align*}
In particular, the above inequality holds for $\ell = k_0$,  which is the desired bound in~\eqref{decay:L-infty-final} for the positive part of $f$. Analogous computation can be performed for the negative part of $f$ which finishes the bound in~\eqref{decay:L-infty-final}. 
\end{proof}

\appendix


\section{Proofs of Lemma~\ref{prop:equivalence} and Lemma~\ref{cor:commut-homo-fraction}} \label{appendix:lemmas}

In this appendix we show the proofs of Lemma~\ref{prop:equivalence} and Lemma~\ref{cor:commut-homo-fraction},
starting with Lemma~\ref{prop:equivalence}.

\begin{proof}[Proof of Lemma~\ref{prop:equivalence}]
For the proof it suffices to show, for $u \in L^p({\mathbb{R}}^d)$,  
\begin{align}\label{seco-inequality}
\norm{\vint{v}^\ell \vint{D_v}^\theta \vint{v}^{-\ell} \vint{D_v}^{-\theta} u}_{ L^p_{v} } 
&\leq 
   C \norm{u}_{ L^p_{v} },\\
\label{first-ineq}
\norm{ \vint{D_v}^\theta \vint{v}^\ell\vint{D_v}^{-\theta} \vint{v}^{-\ell} u }_{ L^p_{v} } 
&\leq 
   C \norm{u}_{ L^p_{v} }. \
\end{align}
To show \eqref{seco-inequality}, we use the expansion formula of pseudo-differential operators (Ex., \cite[Theorem 3.1]
{Kumanogo}), 
\[
\vint{D_v}^\theta \vint{v}^{-\ell} =  \vint{v}^{-\ell} \vint{D_v}^\theta + \sum_{0 < |\alpha| < N}\frac{1}{\alpha !}
(\vint{v}^{-\ell} )_{(\alpha)}(\vint{D_v}^\theta )^{(\alpha)} + r_N(v, D_v)\,,
\]
where $p_{(\beta)}^{(\alpha)}(v,\xi) = \partial_\xi^\alpha (-i\partial_v)^\beta p(v,\xi)$ for the symbol $p(v,\xi)$. 
If $N > d+1 + |\ell| + |\theta|$ then $\tilde 
r_N(v, \xi) \Denote \vint{v}^\ell r_N(v, \xi) \vint{\xi}^{-\theta}$ belongs to
the symbol class $S_{1,0}^{-d-1}$. In fact, it follows from \cite[Theorem 3.1]{Kumanogo} that
\begin{align*}
&r_N(v,\xi) = N \sum_{|\alpha| = N} \int_0^1
\frac{(1-\tau)^{N-1}}{\alpha !} r_{N,\tau,\alpha}(v,\xi) d
\tau,\\
&r_{N,\tau,\alpha}(v,\xi) = \mbox{Os}- \int \int e^{-iy\cdot\eta}
(\vint{ \xi + \tau \eta }^{\theta} )^{(\alpha)}(\vint{ v + y }^{-\ell})_{(\alpha)} 
\frac{dyd\eta}{(2\pi)^d}.
\end{align*}
Using the elementary identities
\[e^{-iy \cdot \eta} =\vint{ \eta }^{-2m}(1-\Delta_y)^m e^{-iy\cdot \eta}, \quad\quad
e^{-iy\cdot \eta} = \vint{ y }^{-2k}(1-\Delta_{\eta})^k e^{-iy \cdot \eta},
\]
we have, for $m, k \in {\mathbb N}$ sufficiently large, 
\begin{equation*}
 \begin{split}
&
r_{N,\tau,\alpha}(v,\xi)
\\
& \quad = \int \Big(
\int e^{-i y\cdot \eta} \vint{ y }^{-2k} (1-\Delta_{\eta})^k
\left\{
 \vint{ \eta }^{-2m}(\vint{ \xi + \tau \eta }^\theta)^{(\alpha)}
(1-\Delta_{y})^m  (\vint{ v + y }^{-\ell})_{(\alpha)} 
\right\}  \frac{d\eta}{(2\pi)^d}\Big) dy
\\
& \quad = 
 \int \{(1-\Delta_{y})^m(\vint{ v + y }^{-\ell})_{(\alpha)}\}
\Big(  \int e^{-i y \cdot \eta}
(1-\Delta_{\eta})^k\{
 \vint{ \eta }^{-2m}
(\vint{ \xi + \tau \eta }^{\theta})^{(\alpha)} \}  \frac{d\eta}{(2\pi)^d}\Big)
\frac{dy}{\vint{ y}^{2k}  }
\\
 & \quad =  \int \displaystyle
\{(1-\Delta_{y})^m  (\vint{ v + y }^{-\ell})_{(\alpha)} \}
\Big(  \int_{|\eta| \leq \frac{\vint {\xi
}}{2}}  \{\cdots \} \frac{d\eta}{(2\pi)^d} +\int_{|\eta| \geq
\frac{ \vint{ \xi }}{2}}
 \{\cdots \}  \frac{d\eta}{(2\pi)^d}\Big)
\frac{dy}{ \vint {y }^{2k}  }
\\
& \quad \Denote  \int \displaystyle
\{(1-\Delta_{y})^m   (\vint{ v + y }^{-\ell})_{(\alpha)} \} \Big( I_1(\xi;y) +
I_2(\xi,y) \Big) \frac{dy}{\vint{y}^{2k}  }.
\end{split}
\end{equation*}
Since $\vint{\xi}$ and $\vint{\xi+ \tau \eta}$ are equivalent in
$I_1$, it follows that
\begin{align*}
|I_1| \leq C \vint{ \xi }^{\theta-N},
\end{align*}
 and moreover the same bound for $|I_2|$ holds if  $2m > N- \theta+d$. 
Using
$
\vint{v+y}^{-1}  \vint{y}^{-1} \le \vint {v}^{-1}$, and taking $k$ satisfying 
$2k > N+ |\ell| +d$,  we see that 
$\tilde r_N(v, \xi)$ belongs to the desired symbol class.  If we put $K(v,z) = \int e^{i z \cdot \xi}\tilde r_N(v, \xi)d\xi/(2\pi)^d$ then we have $\tilde r_N(v,D_v) u (v)
= \int K(v, v-y)u(y)dy$ and 
\[
\sup_v |K(v, z)| \le \vint{ z }^{-2d}
\int \sup_v \left |e^{iz\cdot \xi} (1-\Delta_\xi)^d 
\tilde r_N(v, \xi)\right|d\xi \le C \vint{ z }^{-2d},
\]
which concludes that $\tilde r_N(v, D_v)$ is $L^p$ bounded 
operator for $p \in [1,\infty]$.  Next we consider the the $L^p$ boundedness of terms 
$\vint{v}^\ell (\vint{v}^{-\ell} )_{(\alpha)}(\vint{D_v}^\theta
 )^{(\alpha)} \vint{D_v}^{-\theta}$ for $0 \le |\alpha|< N-1$.
Since the term for $\alpha =0$ is identity, its $L^p$ boundedness is trivial.  Note that
the multiplication $\vint{v}^\ell (\vint{v}^{-\ell} )_{(\alpha)}$ is $L^p$ bounded operator.  
If we put 
$Q_\alpha(\xi) = (\vint{\xi}^\theta
 )^{(\alpha)} \vint{\xi}^{-\theta}$ for $\alpha \ne 0$, then the proof of \eqref{seco-inequality} is completed by the fact that 
the Fourier multiplier $Q_\alpha(D_v)$ is $L^p$
bounded.  Indeed, one can see that  
$K_\alpha(z) \Denote \int e^{i z \cdot \xi} Q_\alpha(\xi) d\xi/(2\pi)^d \in L^1$, more precisely, 
$|K_{\alpha}(z)| \le
C |z|^{-d+1}$ if $|z| < 1$
and 
$|K_\alpha(z)| \le
C_m|z|^{-2m}$ if $|z| \ge 1$ for any $m \in {\mathbb N}$ satisfying $2m > d$.  To obtain these estimates, take a cutoff function 
$\varphi(\xi) \in C_0^\infty({\mathbb R}^d)$
satisfying $\varphi =1$ for $|\xi| \le 1$ and $\varphi =0$ for $|\xi| \ge 2$, and decompose 
\begin{align*}
K_\alpha(z) &= \int e^{i z \cdot \xi} \varphi\big(\frac{\xi}{A}\big) Q_\alpha(\xi) \frac{d\xi}{(2\pi)^d}
+  |z|^{-2m}
\int e^{i z \cdot \xi} (-\Delta_\xi)^m\Big( (1-\varphi\big(\frac{\xi}{A}\big)) Q_\alpha(\xi) \Big)\frac{d\xi}{(2\pi)^d} \\
&\Denote K_{1,\alpha}(z) + K_{2,\alpha}(z)\,,
\end{align*}
for  any $A > 0$.   Then we have 
\begin{align*}
|K_{1,\alpha}(z)| &\le C \int_{\{|\xi| \le 2A\}} \vint{\xi}^{-1} d \xi \le C' A^	{d-1} ,\\
|K_{2,\alpha}(z)| & \le C_m |z|^{-2m} \int_{\{|\xi| \ge A\}} \vint{\xi}^{-2m-1} d\xi \le C'_m  |z|^{-2m} A^{-2m-1+d},
\end{align*}
because $\varphi(\xi/A) \in S^0_{1,0}$ and $Q_\alpha(\xi) \in S^{-1}_{1,0}$. 
Choosing $A = |v|^{-1}$, we have the desired estimate for $K_\alpha$ when  $|z| \le 1$, and another estimate is
obvious by considering the same formula without the cutoff function $\varphi$. 

For the proof of \eqref{first-ineq}  we use the expansion formula twice. First expansion is 
\[
\vint{D_v}^{-\theta }\vint{v}^{-\ell} =  \vint{v}^{-\ell} \vint{D_v}^{-\theta }+ \sum_{0 < |\alpha| < N}\frac{1}{\alpha !}
(\vint{v}^{-\ell} )_{(\alpha)}(\vint{D_v}^{-\theta} )^{(\alpha)} + r_{1,N}(v, D_v)\,,
\]
where $r_{1,N}(v, \xi)$ satisfies $\vint{v}^{\ell} \vint{\xi}^{\ell} r_{1,N}(v,\xi) \in S_{1,0}^{-d-1)}$ if $N$ is chosen sufficiently large.
This implies that the symbol of $\vint{D_v}^\theta \vint{v}^{\ell} r_{1,N}(v,D_v)$ belongs to $S_{1,0}^{-d-1}$, and hence one can show that
$\vint{D_v}^\theta \vint{v}^{\ell} r_{1,N}(v,D_v)$ is 
$L^p$ bounded,  by the same way as 
before.  Since $\vint{D_v}^\theta \vint{v}^\ell \vint{v}^{-\ell} \vint{D_v}^{-\theta } = Id$, it suffices to consider the $L^p$
boundedness of $\vint{D_v}^\theta \vint{v}^\ell  (\vint{v}^{-\ell} )_{(\alpha)}(\vint{D_v}^{-\theta })^{(\alpha)}$ for $\alpha \ne 0$.  
Use the expansion formula again
\[
\vint{D_v}^\theta \Big(\vint{v}^\ell  (\vint{v}^{-\ell} )_{(\alpha)}\Big) =
\sum_{0 \le |\beta|< \tilde N} \frac{1}{\beta !} \Big(\vint{v}^\ell  (\vint{v}^{-\ell} )_{(\alpha)}\Big)_{(\beta)}
(\vint{D_v}^\theta)^{(\beta)} + r_{2,\tilde N}(v, D_v).
\]
If $\tilde N$ is large enough, then $r_{2,\tilde N}(v, D_v)(\vint{D_v}^{-\theta })^{(\alpha)}$ is $L^p$ bounded because  its symbol
belongs to $S^{-d-1}_{1,0}$. On the other hand, since $\Big(\vint{v}^\ell  (\vint{v}^{-\ell} )_{(\alpha)}\Big)_{(\beta)}$ is bounded 
function and since $(\vint{D_v}^\theta)^{(\beta)} (\vint{D_v}^{-\theta })^{(\alpha)}$ is a Fourier multiplier with its symbol
in $S_{1,0}^{-1}$, we see their product is $L^p$ bounded operator. Thus we obtain \eqref{first-ineq} .
\end{proof}

Next we show the proof of Lemma~\ref{cor:commut-homo-fraction}. 

\begin{proof}[Proof of Lemma~\ref{cor:commut-homo-fraction}]
By one of the definitions of the fractional Laplacian, we have
\begin{align*}
  \norm{(-\Delta_v)^{\alpha/2} \vpran{\vint{v}^{-2} f}}_{L^2(\R^3_v)}^2
&= C \iint_{\R^6}
     \frac{\abs{\vint{v'}^{-2} f(v') - \vint{v}^{-2} f(v)}^2}{|v' - v|^{3 + 2\alpha}} \dv' \dv
\\
& \hspace{-2cm}
\leq 2C \iint_{\R^6}
        \vint{v'}^{-4} \frac{\abs{f(v') - f(v)}^2}{|v' - v|^{3 + 2\alpha}} \dv' {\!\dv}
   + 2C \iint_{\R^6}
     \frac{\abs{\vint{v'}^{-2} - \vint{v}^{-2}}^2}{|v' - v|^{3 + 2\alpha}} |f(v)|^2 \dv' {\!\dv}, 
\end{align*}
where the first term on the right-hand side is readily bounded by $C \norm{(-\Delta_v)^{\alpha/2} f}_{L^2(\R^3_v)}^2$. Hence we focus on the second term, which satisfies
\begin{align} \label{bound:neg-homo-commute}
  \iint_{\R^6}
     \frac{\abs{\vint{v'}^{-2} - \vint{v}^{-2}}^2}{|v' - v|^{3 + 2\alpha}} |f(v)|^2 \dv' {\!\dv}  \nn
&= \iint_{\R^6}
     \frac{1}{\vint{v'}^4} \frac{1}{\vint{v}^4} 
     \frac{\abs{|v|^2 - |v|^{2}}^2}{|v' - v|^{3 + 2\alpha}} |f(v)|^2 \dv' {\!\dv}
\\
&\leq
   \iint_{\R^6}
     \frac{1}{\vint{v'}^4} \frac{1}{\vint{v}^4} 
     \frac{|v|^2 + |v'|^2}{|v' - v|^{1 + 2\alpha}} |f(v)|^2 \dv' {\!\dv}   \nn
\\
&\leq
   \int_{\R^3} \frac{1}{\vint{v}^2} |f(v)|^2
      \vpran{\int_{\R^3} \frac{1}{\vint{v'}^2} \frac{1}{|v' - v|^{1 + 2\alpha}} \dv'} \dv.  
\end{align}
For any $v \in \R^3$, make the separation of the domain as
\begin{align*}
    \R^3 
= \{v' \, | \, |v'| > 2|v| \ \text{or} \ |v'| < |v|/2\}
  \cup 
   \{v' \, | \, |v|/2 \leq |v'| \leq 2|v|\}       
\Denote
   \Omega_1 \cup \Omega_2.
\end{align*}
Then the $v'$-integration in~\eqref{bound:neg-homo-commute} satisfies
\begin{align*} 
   \int_{\R^3} \frac{1}{\vint{v'}^2} \frac{1}{|v' - v|^{1 + 2\alpha}} \dv'
&= \int_{\Omega_1} \frac{1}{\vint{v'}^2} \frac{1}{|v' - v|^{1 + 2\alpha}} \dv'
       + \int_{\Omega_2} \frac{1}{\vint{v'}^2} \frac{1}{|v' - v|^{1 + 2\alpha}} \dv' \nn
\\
& \leq
   C \int_{\Omega_1} \frac{1}{\vint{v'}^2} \frac{1}{|v'|^{1 + 2\alpha}} \dv'
+ \frac{C}{\vint{v}^2}\int_{|v' - v| \leq 3 \vint{v}} \frac{1}{|v' - v|^{1 + 2\alpha}} \dv'
\\
& \leq
  C + \frac{C}{\vint{v}^2} \vint{v}^{2-2\alpha}
\leq 
  2C < \infty,
\end{align*}
where $C$ is independent of $v$. Hence by letting $p \in (2, 6)$ be the exponent in the Sobolev embedding 
\begin{align*}
  \norm{f}_{L^p(\R^3_v)} 
\leq 
   C \norm{(-\Delta_v)^{\alpha/2} f}_{L^2(\R^3_v)}, 
\end{align*}
we can bound the term on the right-hand side of~\eqref{bound:neg-homo-commute} as follows:
\begin{align} \label{bound:1-frac}
& \int_{\R^3} \frac{1}{\vint{v}^2} |f(v)|^2
     \vpran{\int_{\R^3} \frac{1}{\vint{v'}^2} \frac{1}{|v' - v|^{1 + 2\alpha}} \dv'} \dv \nn
\\
& \hspace{1cm}
\leq
  C \int_{\R^3} \frac{1}{\vint{v}^2} |f(v)|^2 \dv
\leq
 C \vpran{\int_{\R^3} \frac{1}{\vint{v}^{2 q}} \dv}^{2/q}  
  \norm{f}_{L^p(\R^3_v)}^{2}
\leq
  C \norm{(-\Delta_v)^{\alpha/2} f}_{L^2_v}^2,
\end{align}
where $q = (p/2)' = p/(p-2) > 3/2$ since $p \in (2, 6)$. We therefore get
\begin{align*}
   \norm{(-\Delta_v)^{\alpha/2} \vpran{\vint{v}^{-2} f}}_{L^2(\R^3_v)}^2
\leq
  C \norm{(-\Delta_v)^{\alpha/2} f}_{L^2(\R^3_v)}^2.
\end{align*}
The lemma holds by a further integration in $x$.
\end{proof}

\Ni {\bf Acknowledgement.} R. Alonso gratefully acknowledges the support from Conselho Nacional de Desenvolvimento Científico e Tecnológico - CNPq, grant Bolsa de Produtividade em Pesquisa (303325/2019-4). The research of Y. Morimoto was supported by JSPS Kakenhi Grant No.17K05318. The research of W. Sun was supported by the NSERC Discovery Grant R611626. T. Yang's research was supported by the General Research Fund of Hong Kong CityU No 11304419. 

\end{document}